\documentclass{amsbook}

\usepackage{amsmath}
\usepackage[english]{babel}
\usepackage{amsfonts}
\usepackage{amssymb}
\usepackage{dsfont}
\usepackage{mathrsfs}
\usepackage{graphicx}
\usepackage{fancyhdr}
\usepackage{multicol}
\usepackage{enumitem}
\usepackage{amsthm}
\usepackage{empheq}
\usepackage{cases}
\usepackage[all]{xy}
\usepackage{stmaryrd}
\usepackage[colorlinks, citecolor=blue, linkcolor=red]{hyperref}
\usepackage{mathrsfs}
\usepackage{tikz,tikz-cd}

\tikzset{
	subset/.style={
		draw=none,
		edge node={node [sloped, allow upside down, auto=false]{$\subset$}}},
	Subset/.style={
		draw=none,
		every to/.append style={
			edge node={node [sloped, allow upside down, auto=false]{$\subset$}}}
	}
}
\tikzset{
	labl/.style={anchor=south, rotate=90, inner sep=.50mm}
}

\newcommand{\erre}{\mathds{R}}

\newcommand{\enne}{\mathds{N}}

\newcommand{\Mid}{\bigg\vert}

\newcommand{\cinf}{C^{\infty}(M)}

\newcommand{\diver}{\operatorname{div}}

\newcommand{\id}{\operatorname{id}}

\newcommand{\KN}{\mathbin{\bigcirc\mspace{-15mu}\wedge\mspace{3mu}}}

\newcommand{\ra}{\rightarrow}

\newcommand{\set}[1]{{\left\{#1\right\}}}               
\newcommand{\pa}[1]{{\left(#1\right)}}                  
\newcommand{\sq}[1]{{\left[#1\right]}}                  
\newcommand{\abs}[1]{{\left|#1\right|}}                 

\newcommand{\eps}{\varepsilon}                           

\newcommand{\ol}[1]{\overline{#1}}
\renewcommand{\hat}[1]{\widehat{#1}}
\renewcommand{\tilde}[1]{\widetilde{#1}}

\newcommand{\hess}{\operatorname{Hess}}
\newcommand{\p}{\varphi}
\newcommand{\hs}{\mathrm{Hess}}
\newcommand{\trric}{\mathring{\mathrm{Ric}}}
\newcommand{\ric}{\mathrm{Ric}}

\newcommand{\se}{\mathrm{II}}
\newcommand{\tr}{\mathrm{tr}}

\newtheorem{theorem}{\textbf{Theorem}}[chapter]
\newtheorem{lemma}[theorem]{\textbf{Lemma}}
\newtheorem{proposition}[theorem]{\textbf{Proposition}}

\newtheorem{cor}[theorem]{\textbf{Corollary}}
\newtheorem{defi}[theorem]{\textbf{Definition}}
\theoremstyle{remark}
\newtheorem{rem}[theorem]{\textbf{Remark}}

\numberwithin{section}{chapter}
\numberwithin{equation}{chapter}

\allowdisplaybreaks
\title[]
{On the Geometry of $\p$-Static Perfect Fluid Space-Times}

\author[L. Branca]{Letizia Branca}
\address[Letizia Branca]{Dipartimento di Matematica, Universit\`{a} degli Studi di Milano, Via Saldini 50, 20133 Italy.}
\email[]{letizia.branca@unimi.it}

\author[G. Colombo]{Giulio Colombo}
\address[Giulio Colombo]{Dipartimento di Matematica, Universit\`{a} degli Studi di Milano, Via Saldini 50, 20133 Italy.}
\email[]{giulio.colombo@unimi.it}

\author[P. Mastrolia]{Paolo Mastrolia}
\address[Paolo Mastrolia]{Dipartimento di Matematica, Universit\`{a} degli Studi di Milano, Via Saldini 50, 20133 Italy.}
\email[]{paolo.mastrolia@unimi.it}

\author[F. Mastropietro]{Filippo Mastropietro}
\address[Filippo Mastropietro]{Dipartimento di Matematica, Universit\`{a} degli Studi di Milano, Via Saldini 50, 20133 Italy.}
\email[]{filippo.mastropietro@unimi.it}

\author[M. Rigoli]{Marco Rigoli}
\address[Marco Rigoli]{Dipartimento di Matematica, Universit\`{a} degli Studi di Milano, Via Saldini 50, 20133 Italy.}
\email[]{marco.rigoli@unimi.it}

\dedicatory{In loving memory of Rebo (2015-2024)}
\date{\today}

\thanks{The first and third authors are members of the
	Gruppo Nazionale per le Strutture Algebriche, Geometriche e loro Applicazioni
	(GNSAGA) of INdAM (Istituto Nazionale di Alta Matematica). The third and fifth authors are partially funded by 2022 PRIN project 20225J97H5 ``Differential Geometric Aspects of Manifolds via Global Analysis''.}
\translator{}

\subjclass[2010]{}

\begin{document}
	\maketitle

	\begin{abstract}
		In this paper we study the geometry of $\p$-static perfect fluid space-times ($\p$-SPFST, for short). In the context of Einstein's General Relativity, they arise from a space-time whose matter content is described by a perfect fluid in addition to a nonlinear field expressed by a smooth map $\p$ with values in a Riemannian manifold. Considering the Lorentzian manifold $\hat{M}$ in the form of a static warped product, we derive the fundamental equations \emph{via} reduction of Einstein's Field Equations to the factors of the product.
		 To set the stage for our main results, we discuss the validity of the classical Energy Conditions in the present setting and we introduce the formalism of $\p$-curvatures, which is a fundamental tool to merge the geometry of the manifold with that of the smooth map $\p$. We then present several mathematical settings in which similar structures arise.
		After computing two integrability conditions, we apply them to prove a number of rigidity results, both for manifolds with or without boundary.
		In each of the aforementioned results, the main assumption  is given by the vanishing of some $\p$-curvature tensors and the conclusion is a local splitting of the metric into a warped product.
		Inspired by the classical Cosmic No Hair Conjecture of Boucher, Gibbons and Horowitz, we find sharp sufficient conditions on a compact $\p$-SPFST with boundary to be isometric to the standard hemisphere. We then describe the geometry of relatively compact domains in $M$ subject to an upper bound on the mean curvature of their boundaries. Finally, we study non-existence results for $\p$-SPFSTs, both \emph{via} the existence of zeroes of the solutions of an appropriate ODE and with the aid of a suitable integral formula generalizing in a precise sense the well-known Kazdan-Warner obstruction.
	\end{abstract}
\tableofcontents
	\date{\today}
	\chapter{Introduction}
\section{Mathematical Settings and Main Results}
Over the last decades, there has been an ever-growing interest of the mathematical and physical communities towards the study of Riemannian manifolds endowed with special metrics, arising, for instance, as critical metrics of curvature functionals, or  with special structures coming, for example, as soliton solutions of particular geometric flows or as special solutions of some (physically and geometrically) relevant equations. Four important examples are, respectively, \emph{Einstein metrics}, which are critical points of the Einstein-Hilbert functional with constrained volume (see e.g. the classical \cite{Besse}), \emph{Ricci solitons}, that is, self-similar solutions of the celebrated Ricci flow (see e.g. \cite{CLN2006} and references therein), \emph{static perfect fluid space-times} (SPFSTs, for short; see e.g. \cite{HawkingEllis} and references therein), characterized by the system
\begin{equation}\label{SPFSTsystem}
	\begin{cases}
		\hess(u)-u\pa{\ric-\frac{1}{m-1}\pa{\frac{S}{2}-p}g}=0,\\
		\Delta u=\frac{u}{m-1}\pa{mp+\frac{m-2}{2}S},\\
		\mu=\frac{1}{2}S,\\
		\pa{\mu+p}\nabla u=-u\nabla p,
	\end{cases}
\end{equation}
where $u\in C^\infty(M)$, $\ric$ and $S$ denote the Ricci tensor and the scalar curvature of $(M, g)$, respectively,   $\mu$ represents the \emph{energy density} and $p$ the pressure of a perfect fluid, and, finally, \emph{vacuum static spaces}  (also known as \emph{static triples}) $(M, g, u)$ (see e.g. \cite{QY2013}, \cite{BL2018}, \cite{CDLR2020}, \cite{LPR2020} and the very recent \cite{CDPR2023}), which consist of a Riemannian manifold $(M, g)$ endowed with a positive, smooth solution $u$ on $M$ of the equation
\begin{equation}\label{Eq_VSS}
	\hs(u) - u\pa{\ric-\frac{S}{m-1}g}=0.
\end{equation}
In the last two cases (and this will be explicitated for an extension of \eqref{SPFSTsystem}), the set of equations comes from a reduction of the Einstein equations, governing the corresponding phenomenon in General Relativity, on a Lorentzian warped product $\overline{M} = \erre  \prescript{}{e^{-f}}{\times} M$ to the Riemannian ``space'' component $(M, g)$ and to the ``time'' component $\erre$ of $\overline{M}$.

In this paper we consider a further generalization of SPFSTs, that is, we study the geometry of an $m$-dimensional, connected, Riemannian manifold $(M,g)$, possibly with boundary,  carrying a structure that we shall call a  $\p$-\emph{static perfect fluid space-time} (from now on, $\mathbf{\p}$-SPFST). This means that for some smooth map $\p:(M,g) \ra (N,h)$ and  smooth functions $U :N\ra \erre$ and $\mu, p: M\ra \erre$, $(M,g)$ admits a non-negative smooth solution $u$ of the $\p$-SPFST system
\begin{align}\label{Gianny1}
	\begin{cases}
		i)\, \hs(u)-u\set{\ric^\p-\frac{1}{m-1}\pa{\frac{S^\p}{2}-p+U(\p)}g}=0,\\
		ii)\,\Delta u=\frac{u}{m-1}\sq{mp-mU(\p)+\frac{m-2}{2}S^\p},\\
		iii)\,u\tau(\p)=-d\p(\nabla u)+\frac{u}{\alpha}(\nabla U)(\p),\\
		iv)\,\mu+U(\p)=\frac{1}{2}S^\p,\\
		v)\,(\mu+p)\nabla u=-u\nabla p,
	\end{cases}
\end{align}
where we are using the shorthand notation $$\ric^{\p}:=\ric-\alpha\p^*h,$$ $\alpha \in \erre\setminus \set{0}$, for the $\p$-\emph{Ricci tensor} (introduced in \cite{M} and studied, for instance, in \cite{Wang2016}, \cite{ACR} and \cite{CMR2022}), $S^\p$ for the $\p$-\emph{scalar curvature}, i.e. the trace of  $\ric^{\p}$ with respect to $g$, and $\tau(\p)$ for the \emph{tension field} of the map $\p$. We recall that $\tau(\p)$ is  the trace, with respect to $g$, of the generalized second fundamental tensor $\nabla d\p$ of the map $\p$: it plays a central role in the geometry of smooth, harmonic maps, extending the Laplace operator to the general case where $\p$ takes values in a curved space, and it is responsible of the non-linearity of the theory (see e.g. the classical \cite{EL}, \cite{EL1988} and \cite{ES1964}).\\
We also require, motivated by physical reasons,
\begin{align}\label{equaz_0.7}
	u>0 \quad \text{on } \mathrm{int}(M)
\end{align}
and
\begin{align}\label{Eq0.4}
	\partial M=u^{-1}\pa{\set{0}}
\end{align}
in case  $\partial M \neq \emptyset$.
We refer to  Chapter \ref{Sect_Preliminaries} for the derivation of the system \emph{via} the reduction of the Einstein equations and for the interpretation of the various terms (see also the next subsection for some remarks on the physical motivations and, again,  Chapter \ref{Sect_Preliminaries} for more details on $\p$-curvatures).

Note that, when $\varphi$ is constant and $U\equiv 0$,  system \eqref{Gianny1} reduces to system \eqref{SPFSTsystem}, that is the classical SPFST system.


We observe that, for a  vacuum static space, in 1984 Boucher, Gibbons and Horowitz (see \cite{BGH} and \cite{BG}) conjectured that the only $m$-dimensional, simply connected, \emph{static triple} $(M, g, u)$ with a single horizon (that is, a connected boundary $\partial M\neq \emptyset$) and positive constant scalar curvature $S$
is isometric to a  Euclidean hemisphere $\mathbb{S}^m_+(r)$, for some appropriate radius $r>0$; this conjecture, which is known as the \emph{Cosmic No-Hair Conjecture}, suggests that, under certain assumptions, the evolution of the universe leads to the dominance of a highly symmetric geometry, getting rid of more complicated, ``hairy'' solutions. The conjecture has been confirmed under different  further hypotheses, but disproved for $\mathrm{dim}(M) \geq 4$ (for more information, see, for instance, \cite{Gall2002}, \cite{GHP}, \cite{CDPR2023}). \\

\label{key}

A great deal of the present work is devoted to the study, under various assumptions, of this type of rigidity results. In fact, in doing so, we go through a number of intermediate steps, in particular  a local description of the metric $g$ of $M$ as a (local) warped product;
this is achieved by a careful study of the geometry of the regular level sets of $u$.
To achieve our goal we use a number of new or relatively new concepts; for instance, the next definition will be needed  in a shortwhile to state some of our results:
\begin{defi}\label{defi: Harmonic Einstein}
	Let $(M,g)$ and $(N,h)$ be two Riemannian manifolds. For a smooth map $\p:(M,g)\to (N,h)$ and for a constant $\alpha \in \erre$, $\alpha \neq 0$, we say that $(M,g)$ is \emph{harmonic-Einstein} if, for some $\Lambda \in \erre$,
	\begin{align}\label{1.10}
		\begin{cases}
			\ric^\p=\ric-\alpha\p^*h=\Lambda g,\\						\tau(\p)=0.
		\end{cases}
	\end{align}
\end{defi}
\begin{rem}
  Note that, instead of $(M,g)$ to be harmonic-Einstein, one should refer to $(M,g,\p,\alpha,\Lambda)$ to be harmonic-Einstein; however, for the sake of readability we shall avoid this, and possibly other, cumbersome notation, the correct meaning being clear from the context.
\end{rem}

Observe that, for $m\geq 3$, an analogous of Schur's lemma for Einstein metrics holds: indeed, if $\Lambda$ is a function on $M$, then (\ref{1.10})
automatically implies that $\Lambda$ is constant (see e.g. \cite{ACR}). Harmonic-Einstein manifolds are strictly related to harmonic-Einstein solitons, which are special solutions of the Ricci-harmonic flow introduced by B. List in \cite{List2008EvolutionOA}.
\\


The present paper can be roughly divided into three parts: in the first we describe the setting and we analyze some special cases and motivations, in the second part we concentrate on rigidity results, while in the third we provide some sufficient conditions for the non-existence of positive solutions $u$ on $M$ for a large class of systems containing (\ref{Gianny1}).
We now describe some of our main results.\\
The first, which is local in nature, is motivated by the work \cite{KO} of Kobayashi and Obata (in Chapter \ref{Sect_KO} we shall go into details to justify the genesis of  our result and the reason why it is an extension of \cite{KO}).
We need one more definition:
\begin{defi}\label{definition of G harmonic}
	Let $\p:(M,g)\to (N,h)$ be a smooth map and $G:N\to \erre$ a smooth function. Then $\p$ is said to be \emph{$G$-harmonic} if
	\begin{align}\label{mappa G-armonica}
		\tau (\p)=\pa{\nabla G}(\p).
	\end{align}
\end{defi}
\noindent Note that, for $G$ constant, we obtain the case of harmonic maps.
The previous definition goes back to the work of Fardoun, Ratto and Regbaoui (\cite{Fardoun1997HarmonicMW}, \cite{FRRharmonicPot}, \cite{Ratto}); it comes from a variational setting that has been vastly analyzed by Lemaire in his Ph.D. Thesis, (\cite{Lemaire1977OnTE}).
Indeed, (\ref{mappa G-armonica}) is the Euler-Lagrange equation of the functional
\begin{align*}
	E(\p)=\frac{1}{2}\int_{\Omega}\set{\abs{d\p}^2+2G(\p)}\,dV_g,
\end{align*}
where $\Omega$ is a relatively compact domain in $M$ (see \cite{Lemaire1977OnTE} for the details).

\noindent The \textbf{stress-energy tensor} associated to $\p$ and $G$ is
\begin{align*}
	\mathcal{S}^G=\set{\frac{\abs{d\p}^2}{2}+G(\p)}g-\p^*h;
\end{align*}
a simple computation shows that
\begin{align*}
	\diver \mathcal{S}^G=h\pa{d\p, \pa{\nabla G}(\p)-\tau(\p)},
\end{align*}
so that a $G$-harmonic map $\p$ is in particular \emph{$G$-conservative}, that is
\begin{align*}
	\diver \mathcal{S}^G=0.
\end{align*}
Thus, for $G$ constant, a map $\p$ is conservative if and only if
\begin{align*}
	h\pa{\tau(\p),d\p}=0.
\end{align*}

If, for some $\psi\in C^{\infty}(M)$, $\psi>0$, we perform a conformal change of metric $\tilde{g}=\psi^2g$, we will set $\tilde{\p}$ to indicate the map $\p$ but now considered as a map from $(M,\tilde{g})$ to $(N,h)$.
In general, ``tilded'' quantities will refer to the metric $\tilde{g}$
(thus, for example, $\tilde{S}^{\tilde{\p}}$ is the $\tilde{\p}$-scalar curvature of $(M,\tilde{g})$).\\
In many instances, setting $f=-\log u$ when $u>0$, we will transform system (\ref{Gianny1}) into a particular case of a system of the general type
\begin{align}
	\begin{cases}\label{introduzione: system KO}
		&\ric^\p+\hess(f)-\eta df\otimes df=\lambda g,\\
		&\tau(\p)=d\p(\nabla f)+\frac{1}{\alpha}\pa{\nabla U}(\p)
	\end{cases}
\end{align}
for some $\alpha, \eta \in \erre$, $\alpha \neq 0$, $\lambda \in C^{\infty}(M)$.\\
In the following, we will often consider a regular level set $\Sigma$ of $f$ endowed with the metric $g_{\Sigma}=g_{|_{\Sigma}}$; we will set $\ric^{\p_{|_{\Sigma}}}$ and $S^{\p_{|_{\Sigma}}}$ to denote the $\p$-Ricci tensor and the $\p$-Scalar curvature of $(\Sigma,g_{\Sigma})$, respectively.

We are now ready to state our first
\begin{theorem}\label{thm 018_KO}
	Let $(M,g)$ be a manifold of dimension $m\geq 3$, with $\partial M=\emptyset$. Let $\p:(M,g)\to (N,h)$ be a smooth map, $U:N\to \erre$ be a smooth function, $\alpha \in \erre\setminus \set{0}$, $\lambda \in C^{\infty}(M)$ and let $f\in C^{\infty}(M)$ be a solution of system (\ref{introduzione: system KO}) on $M$, with $\eta\neq -\frac{1}{m-2}$. Assume that $\p$ is $\frac{1}{\alpha}U $-harmonic and suppose that, for the conformal change of metric
	\begin{align}\label{intro: def conforme}
		\tilde{g}=e^{-\frac{2}{m-2}f}g,
	\end{align}
	we have
	\begin{align}\label{intro: KO: tilde C uguale divergenza}
		2(m-1)\tilde{C}^{\tilde{\p}}=-\diver_1\pa{U(\p)g\KN g},
	\end{align}
	where $\tilde{C}^{\tilde{\p}}$ is the $\tilde{\p}$-Cotton tensor of the metric $(M,\tilde{g})$.
	Then, for each $p\in \Sigma$, where $\Sigma$ is a regular level set of $f$, there exists an open set $A\subset M$ with $p\in A$ such that $g_{|_A}$ is a warped product metric. Moreover, $U(\p)$ and $S^{\p_{|_{\Sigma}}}$ are constant on $\Sigma$ and we have
	\begin{align}
		\begin{cases}\label{introduzione: KO: harmonic einstein}
			&\ric^{\p_{|_{\Sigma}}}=\frac{S^{\p_{|_{\Sigma}}}}{m-1}g_{\Sigma},\\
			& h(\tau(\p_{|_{\Sigma}}),d\p_{|_{\Sigma}})=0
		\end{cases}
	\end{align}
	on $(\Sigma, g_{\Sigma})$.
\end{theorem}

A few remarks are in order to explain the statement of the Theorem.
Assume that $(A,g_{|_A})$ can be identified with the warped product $I\times_{\rho} \pa{\Sigma \cap A}$, with metric $g_{|_A}=dr^2+\rho^2g_{|_{\Sigma}}$. Then one can prove that, under this identification, the signed distance function $r:A\to \erre$ from $\Sigma \cap A$ coincides with  the projection $p:I\times (\Sigma\cap A)\to I$. If, for some $t\in I$, we consider the hypersurface $\{t\}\times \pa{\Sigma\cap A}$,  then  its mean curvature with respect to the inner unit normal is constant; we will denote it by $H(t)$. With these notation in mind, we can express the warping factor $\rho$ by means of the relation $\rho (r)=e^{-\int_0^rH(t)dt}$. This, and an explicit description of $S^{\p_{|_{\Sigma}}}$, is the content of Theorem \ref{almost KO} below.


The  tensor $\tilde{C}^{\tilde{\p}}$ is defined in Chapter \ref{Sect_Preliminaries}, Section \ref{Sect: phi-curvatures}, while the notation $\diver_1 T$, for some tensor $T$ (for instance, a 4-times covariant tensor) means, in a local orthonormal coframe,
\begin{align*}
	(\diver_1 T)_{jkt}=T_{ijkt,i}.
\end{align*}
Note that it is necessary to make this type of distinctions because $T$ might not have special symmetries, and thus $\diver_1T$ could be different from $\diver_2 T$; as a matter of fact, for instance,
\begin{align*}
	\diver_1 \pa{U(\p)g\KN g}=-\diver_2 \pa{U(\p)g\KN g}
\end{align*}
(here $\KN$ denotes the Kulkarni-Nomizu product).\\
Note also that, if we suppose that $\p_{|_{\Sigma}}$ is a harmonic submersion defined on an open set $B\subset \Sigma$, then $\tau(\p_{|_{\Sigma}})=0$ on $B$ and, by the unique continuation property for harmonic maps, $\tau (\p)=0$, see \cite{Sampson}. In this case, $\Sigma$ becomes a harmonic-Einstein manifold; indeed, system \eqref{introduzione: KO: harmonic einstein} reduces to
\begin{align}\label{intro: harmonic-Einstein 2}
	\begin{cases}
		&\ric^{\p_{|_{\Sigma}}}=\frac{S^{\p_{|_{\Sigma}}}}{m-1}g_{\Sigma},\\
		& \tau(\p_{|_{\Sigma}})=0.
	\end{cases}
\end{align}
Observe that, in case $m\geq 4$, equation \eqref{intro: harmonic-Einstein 2} forces $S^{\p_{|_{\Sigma}}}$ to be constant (see \cite{ACR} for a proof).\\
Replacing hypothesis \eqref{intro: KO: tilde C uguale divergenza} with a different set of assumptions yields the following theorem, which asserts a stronger conclusion than Theorem \ref{thm 018_KO}:
\begin{theorem}
	Let $(M,g)$ be a complete manifold of dimension $m\geq 3$ and with $\partial M=\emptyset$. Let $\p:(M,g)\to (N,h)$ be a smooth map, $U:N\to \erre$ be a smooth function, $\alpha \in \erre\setminus \set{0}$, $\lambda \in C^{\infty}(M)$ and let $f\in C^{\infty}(M)$ be a solution on $M$ of system \eqref{introduzione: system KO}. Let $\Sigma$ be a regular level set of $f$. Assume the validity of the following assumptions:
	\begin{itemize}
		\item $f$ is proper;
		\item either $\alpha >0$ and $\eta>-\frac{1}{m-2}$ or $\alpha<0$ and $\eta<-\frac{1}{m-2}$;
		\item $B^{\p}(\nabla f,\nabla f)=0$, where $B^\p$ is the $\p$-Bach tensor;
		\item for each regular $p\in M$, $\nabla f_p$ is an eigenvector of $\ric^{\p}_p$;
		\item $\p$  is $\frac{1}{\alpha} U$-harmonic.
	\end{itemize}
	Then, for each $p\in \Sigma$, there exists an open set $A$ in $M$, with $p\in A$, such that $g_{|_A}$ is a warped product metric. Moreover, $U(\p)$ and $S^{\p_{|_{\Sigma}}}$ are constant on $\Sigma$ and $(\Sigma, g_{\Sigma})$ is harmonic-Einstein, that is, it satisfies
	\begin{align}
		\begin{cases}
			&\ric^{\p_{|_{\Sigma}}}=\frac{S^{\p_{|_{\Sigma}}}}{m-1}g_{\Sigma},\\
			&\tau(\p_{|_{\Sigma}})=0.
		\end{cases}
	\end{align}
	
\end{theorem}
\noindent
The definition of $B^\p$ will be given in formula (\ref{1.20Bachphi}).\\
The request that $\nabla f_p$ is an eigenvector of $\ric_p$ at any regular point $p$ of $f$ is necessary, as we will see in Corollary \ref{cor: equivalence betw autov e rett}. \\
We mention  here a further result, where we prove that the - geometrically significant - tensor $\ol{D}^\p$, defined in (\ref{hat D phi}), is identically null. As a consequence of Theorem \ref{almost KO}, this fact will give a local splitting of the metric as in the previous Theorems. 
\begin{theorem}\label{intro: teo: div tot cotton 1}
	Let $(M,g)$ be a complete manifold of dimension $m\geq 3$ and with $\partial M=\emptyset$. Let $\p:(M,g)\to (N,h)$ be a smooth map, $\lambda \in C^{\infty}(M)$ and let $f\in C^{\infty}(M)$ be a solution on $M$ of system
	\begin{align}\label{intro: teo: prima del sistema}
		\ric^\p+\hess(f)-\eta df\otimes df=\lambda g
	\end{align}
	for some $\alpha,\eta\in \erre, \alpha \neq 0$ and $\lambda \in C^{\infty}(M)$.
	Suppose the validity of the following assumptions:
	\begin{itemize}
		\item $\eta\neq-\frac{1}{m-2}$;
		\item if $M$ is non-compact, $f(x)\ra +\infty \text{ as } x\ra +\infty \text{ in $M$}$;
		\item $\p$ is conservative, that is
		\begin{align*}
			h(\tau(\p),d\p)=0,
		\end{align*}
		\item $\diver^3 C^\p\equiv 0.$
	\end{itemize}
	Then there are two possibilities:
	\begin{itemize}
		\item[i)] if $\eta\neq 0$ and we further assume
		\begin{align}
			W^{\p}(\nabla f, \cdot,\cdot,\cdot)\equiv 0, \label{intro: radial weyl}
		\end{align}
		then we have
		\begin{align*}
			C^\p\equiv 0.
		\end{align*}
		Moreover, if, for some smooth $U:N\to \erre$ we also assume the validity of
		\begin{align}\label{intro: seconda del sistema}
			\tau(\p)=d\p(\nabla f)+\frac{1}{\alpha }\pa{\nabla U}(\p),
		\end{align}
		then we obtain
		\begin{align*}
			\ol{D}^\p\equiv 0.
		\end{align*}
		\item[ii)] If $\eta=0$, then we have
		\begin{align*}
			C^\p\equiv 0.
		\end{align*}
		Moreover, if we also assume both (\ref{intro: radial weyl}) and (\ref{intro: seconda del sistema}), then we have
		\begin{align*}
			\ol{D}^\p\equiv 0.
		\end{align*}
	\end{itemize}

\end{theorem}
$C^\p$ and $W^\p$, the $\p$-Cotton tensor and the $\p$-Weyl tensor, are defined respectively in (\ref{weyl phi}) and (\ref{phi Cotton}), while the symbol $\diver^3C^\p$ stands for the "total divergence"
\begin{align*}
	\diver^3 C^\p=C^{\p}_{ijk,kji}
\end{align*}
(pay attention to the order of indices).\\
Note that the conservativity of $\p$ and $\diver^3 C^\p\equiv 0$ are both necessary conditions for the validity of $C^\p\equiv 0$.
It is also worth to observe that condition (\ref{intro: radial weyl}), which for $\p$ constant takes the form
\begin{align*}
	W(\nabla f,\cdot,\cdot,\cdot)=0,
\end{align*}
in the recent literature is called the
\emph{zero radial Weyl curvature assumption}, and has in fact a precise geometric meaning.
Indeed, if we pointwise conformally deform $g$ to $\tilde{g}$ as in (\ref{intro: def conforme}), we obtain
\begin{align}\label{intro: def conf di phi cotton}
	\tilde{C}^{\tilde{\p}}=C^\p+W^{\p}(\nabla f,\cdot,\cdot,\cdot);
\end{align}
thus, $W^{\p}(\nabla f,\cdot,\cdot,\cdot)$ is the obstruction to $C^\p$ being invariant under the conformal deformation (\ref{intro: def conforme}). This also shows that the case $\eta=-\frac{1}{m-2}$, for $m\geq 3$, is special.
Indeed, as we shall see in Chapter \ref{Sect_Preliminaries}, the system
\begin{align}
	\begin{cases}
		&\ric^\p+\hess(f)+\frac{1}{m-2}df\otimes df=\lambda g,\\
		&\tau(\p)=d\p(\nabla f)
	\end{cases}
\end{align}
is obtained \textit{via} the conformal deformation (\ref{intro: def conforme}) of the metric of a harmonic-Einstein manifold (see Proposition \ref{prop su cambi conf per ricavare sist} below).
It is immediate to verify that for a harmonic-Einstein manifold the $\p$-Cotton tensor is identically null; thus, formula (\ref{intro: def conf di phi cotton}) shows that, when $\eta=-\frac{1}{m-2}$,
\begin{align*}
	C^\p=-W^\p(\nabla f,\cdot,\cdot,\cdot)
\end{align*}
and therefore
\begin{align*}
	C^\p\equiv 0 \text{ if and only if } W^\p(\nabla f, \cdot,\cdot,\cdot)\equiv 0.
\end{align*}
Interpreting Theorem \ref{intro: teo: div tot cotton 1} in case of a \textit{Ricci-harmonic soliton}, that is, for a smooth solution $f$ of the system
\begin{align}
	\begin{cases}
		&\ric^\p+\hess(f)=\lambda g,\\
		&\tau(\p)=d\p(\nabla f)
	\end{cases}
\end{align}
with $\alpha,\lambda\in \erre,\alpha \neq 0$, we deduce
\begin{cor}
	Let $(M,g)$ be a complete Ricci-harmonic soliton of dimension $m\geq 3$ such that $\p$ is conservative and $\diver^3C^\p\equiv 0$. Then $C^\p\equiv 0$.
\end{cor}
Similarly, in the case of a \textit{Ricci almost soliton} $(M,g)$ (see \cite{RicciAlmostSolitons}), we have a solution $f\in C^{\infty}(M)$ of the system
\begin{align*}
	\ric+\hess(f)=\lambda g,
\end{align*}
with $\lambda \in C^{\infty}(M)$, and we deduce the
\begin{cor}\label{intro: corollario catino e mastrolia}
	Let $(M,g,f,\lambda)$ be a complete almost Ricci soliton of dimension $m\geq 4$ such that
	\begin{align}\label{intro: divergenza tot di weyl è o}
		\diver^3\pa{\diver_1 W}\equiv 0.
	\end{align}
	Then $(M,g)$ has harmonic Weyl tensor.
\end{cor}
\begin{rem}
	According to the notation we just introduced, expression (\ref{intro: divergenza tot di weyl è o}) reads, in a local orthonormal coframe,
	\begin{align*}
		W_{tijk,tkji}\equiv 0.
	\end{align*}
\end{rem}
\noindent
As for the proof of Corollary \ref{intro: corollario catino e mastrolia}, we observe that, in this case, (\ref{intro: divergenza tot di weyl è o}) is equivalent to $\diver^3 C\equiv 0$, and the conclusion $C\equiv 0$ implies $W_{tijk,t}\equiv 0$, (this follows by equation (\ref{phi weyl e phi cotton}) when $\p$ is constant), that is, the Weyl tensor is harmonic.\
Thus, if we assume that $\lambda$ is a positive constant, or, in other words, that $(M,g,f,\lambda)$ is a complete, shrinking Ricci soliton, then $(M,g)$ is either Einstein or it is isometric to a finite quotient of $P^{m-k}\times \erre^k, k\in \enne, k>0$, the product of an Einstein manifold $P$ with the Gaussian shrinking soliton $(\erre^k,g_E,\frac{\lambda}{2}|x|^2,\lambda)$. Note that Corollary \ref{intro: corollario catino e mastrolia}, with $\lambda$ constant, is due to Catino, Mastrolia, Monticelli \cite{catino2016gradientriccisolitonsvanishing}, while the classification of Weyl-harmonic complete, shrinking Ricci solitons is due to Fernandez-Lopez, Garcia-Rio \cite{Fernandez} and Munteanu, Sesum \cite{Sesum}.\newline
The proof of Theorem \ref{intro: teo: div tot cotton 1} is based on a basic formula related to (\ref{intro: teo: prima del sistema}) and proved in Proposition \ref{other rigidity: prop: equazione fonamentale} of Chapter \ref{Other rigidity results}; however, rearranging the formula under different assumptions, we obtain the following slightly different version:

\begin{theorem}\label{intro: teo: div tot cotton 2}
	Let $(M,g)$ be a complete Riemannian manifold  of dimension $m\geq 3$, with $\partial M=\emptyset$, supporting system \eqref{introduzione: system KO}, that is,
	\begin{align*}
		\begin{cases}
			\ric^\p+\hs(f)-\eta df\otimes df=\lambda g,\\
			\tau(\p)=d\p(\nabla f)+\frac{1}{\alpha}(\nabla U)(\p)
		\end{cases}
	\end{align*}
for some $\alpha, \eta \in \erre$, $\alpha \neq 0$, and $\lambda \in C^\infty(M)$.
	Assume that
	\begin{itemize}
		\item $\eta\neq -\frac{1}{m-2}$;
		\item  if $M$ is non compact, $f(x)\ra +\infty \text{ as } x\ra \infty \text{ on }M$;
		\item $\diver^3 C^{\p}=0$;
		\item $\p$ is $\frac{1}{\alpha}U$-harmonic;
		\item $\nabla f_p$ is an eigenvector of $\ric^{\p}_p$, for each regular point p of $f$.
	\end{itemize}
	Then we have two possibilities:
	\begin{itemize}
		\item[i)] if $\eta\neq 0$ and  we further assume
		\begin{align}\label{intro: radial weyl flatness}
			W^{\p}(\nabla f, \cdot,\cdot,\cdot)=0,
		\end{align}
		then we have that
		\begin{align}\label{intro: C U p uguale zero}
			C^{\p}=-\frac{1}{2(m-1)}\diver_1\pa{U(\p)g\KN g}
		\end{align}
		and
		\begin{align*}
			\overline{D}^{\p}=0.
		\end{align*}
		\item[ii)] If $\eta=0$, (\ref{intro: C U p uguale zero}) holds
		and, if we further assume \eqref{intro: radial weyl flatness}, we also infer
		\begin{align*}
			\overline{D}^{\p}=0.
		\end{align*}
	\end{itemize}
	
\end{theorem}
\begin{rem}
	We note that, since $\ol{D}^\p\equiv 0$ and $\p$ is $\frac{1}{\alpha}U$-harmonic, we can apply the local description of the metric given in Theorem \ref{thm 018_KO}.
\end{rem}

In Chapter \ref{Other rigidity results} we will give a version of Theorem \ref{intro: teo: div tot cotton 2} for $(M,g)$ compact with non empty-boundary (see Theorem \ref{other rigidity results: the boundary case: analogo thm 1.44}); it will be necessary to identify $\partial M$ with $u^{-1}(\{0\})$, and the Theorem will therefore be expressed in terms of $u$ (and not of $f$).
Furthermore, the system we consider is the more restrictive
\begin{align}
	\begin{cases}
		&\eta u \ric^\p-\hess(u)=\frac{1}{m}\pa{\eta uS^\p-\Delta u}g,\\
		&\eta u\tau(\p)=-d\p(\nabla u)+\frac{\eta}{\alpha}u \pa{\nabla U}(\p),
	\end{cases}
\end{align}
which can be obtained from \eqref{introduzione: system KO} \textit{via} the change of variable $u=e^{-\eta f}$.

In Chapter \ref{Sect_proof1.3}, one of the main results goes in the direction of the \emph{Cosmic No-Hair Conjecture} in the case of SPFST's, which is a case more general than static triples; this result is strictly related  to the Null Energy Condition and the Strong Energy Condition
(see Section \ref{Subsection: EC} in Chapter \ref{Sect_Preliminaries}). Indeed, in this setting, the \textit{Null Energy Condition} is implied by
\begin{align}\label{intro: NEC}
	\mu+p\geq 0
\end{align}
while the conditions
\begin{align}\label{intro: SEC}
	p+\mu \geq 0, \ \ (m-2)\mu+mp\geq 2U(\p)
\end{align}
imply the validity of the \textit{Strong Energy Condition}.
\begin{rem}
  The second requirement in \eqref{intro: SEC} is also necessary; moreover, we underline that, when $\p$ is constant and $U\equiv 0$, (\ref{intro: NEC}) and (\ref{intro: SEC}) are both necessary, respectively, for the Null Energy Condition and the Strong Energy Condition.
\end{rem}

The following simple proposition will make  assumption \eqref{intro: Borg e Mazz: stima} in Theorem \ref{thm A} below meaningful. Note that \eqref{intro: Borg e Mazz: stima} in the literature is called a \emph{Boundary Gravity Condition} or a \emph{Shear Stress Condition}: we shall discuss its sharpness after the proof of Theorem \ref{thm A di nuovo}.

\begin{proposition}\label{prop: violazione della SEC}
	Let $(M,g)$ be a compact $\p$-SPFST such that $\partial M\neq \emptyset$. Then there exists $p\in \mathrm{int}(M)$ such that
	\begin{align*}
		(m-2)\mu +mp<2U(\p) \text{ at $p$}.
	\end{align*}
\end{proposition}


\begin{proof}
	Inserting (\ref{Gianny1}) iv) into (\ref{Gianny1}) ii) we obtain
	\begin{align*}
		\Delta u=\frac{1}{m-1}\sq{(m-2)\mu+mp-2U(\p)}u;
	\end{align*}
	thus, if by contradiction we assume
	\begin{align*}
		(m-2)\mu+mp\geq 2U(\p),	
	\end{align*}
	then $u\geq 0$ imply
	\begin{align*}
		\Delta u\geq 0 \ \text{ on } M.
	\end{align*}
	Since $u=0$ on $\partial M$ and $u>0$ on $\text{int} M$, by the maximum principle we infer $u\equiv 0$, which is a contradiction.
\end{proof}


\begin{theorem}\label{thm A}
	Let $(M,g)$ be an $m$-dimensional
	compact $\p$-SPFST with connected, non-empty boundary and $\alpha>0$. Assume that
	\begin{align*}
		(m-2)\mu +mp
	\end{align*}
	is constant, that
	\begin{align}
		p+\mu\geq 0
	\end{align}
	and that
	\begin{align}\label{intro: Borg e Mazz: stima}
		m(m-1)|\nabla u|^2_{|_{\partial M}}\leq \max_M\set{\sq{2U(\p)-\pa{(m-2)\mu+mp}}},
	\end{align}
	with $U$ weakly convex.
	Then $\p, \mu, p$ and $S$ are constant on $M$, with $\mu$ and $S$ positive and $\mu=-p$; moreover, $(M,g)$ is isometric to the hemisphere
	\begin{align}
		S^m_+\pa{\frac{S}{m(m-1)}}\subset \erre^{m+1}.
	\end{align}
\end{theorem}
\noindent

\begin{rem}
  The weak convexity assumption means that the Hessian of the function $U$ is positive semi-definite (see Definition \ref{definition: weakly convex} in Chapter \ref{Sect_SufficientConditions}).
\end{rem}
\begin{rem}
  The constancy assumption in Theorem \ref{thm A} and in Corollary \ref{intro:cor: quasi Borg e Mazz} below can be relaxed to an appropriate subharmonicity of the function itself (see Theorem \ref{thm A di nuovo}).
\end{rem}

\begin{rem}
  In the ``classical'' case $\p$ constant and $U\equiv 0$, a request as in (\ref{intro: Borg e Mazz: stima})
is a \textit{surface gravity} condition  (see e.g. \cite{BM2018}).
\end{rem}

Finally, we exploit the weak convexity of $U$, in order to establish the inequality
\begin{align*}
	\pa{\hess(u)}\pa{d\p(e_i),d\p(e_i)}\geq 0,
\end{align*}
where $\{e_i\}$ denotes an orthonormal basis for $T_pM$.
In case we deal with system (\ref{SPFSTsystem}) we obtain the following
\begin{cor}\label{intro:cor: quasi Borg e Mazz}
	Let $(M,g)$ be an $m$-dimensional, $m\geq 3$, compact, SPFST with connected, non-empty boundary.
	Assume that
	\begin{align*}
		(m-2)\mu+mp
	\end{align*}
	is constant, that
	\begin{align*}
		p+\mu\geq 0
	\end{align*}
	and that
	\begin{align*}
		m(m-1)|\nabla u|^2_{|_{\partial M}}\leq \max_M\set{-\pa{(m-2)\mu+mp}u^2}.
	\end{align*}
	Then $\mu,p$ and $S$ are constant with $\mu$ and $S$ positive and $\mu=-p$; moreover,  $(M,g)$ is isometric to
	\begin{align*}
		S^m_+\pa{\frac{S}{m(m-1)}}\subset \erre^{m+1}.
	\end{align*}
\end{cor}
\begin{rem}
Having chosen $S=m(m-1)$, $\mu=\frac{1}{2}m(m-1)$ and
\begin{align*}
	p=-\mu,
\end{align*}
system (\ref{SPFSTsystem}) yields
\begin{align*}
	\begin{cases}
		&u\ric-\hess(u)=mug,\\
		&\Delta u=-mu,
	\end{cases}
\end{align*}
while assumption (\ref{intro: Borg e Mazz: stima}) becomes
\begin{align*}
	|\nabla u|^2_{|_{\partial M}}\leq \max_M u^2.
\end{align*}
Thus, Corollary \ref{intro:cor: quasi Borg e Mazz} recovers Theorem 4.2 of Borghini and Mazzieri (\cite{BM2018}).
\end{rem}

\begin{rem}
  We underline the basic fact that $M$ compact, $\partial M\neq \emptyset$ and $\mu$ constant always
imply, by Proposition \ref{Prop2.12},
\begin{align*}
	p+\mu=0.
\end{align*}
\end{rem}

\
The next result, motivated by the work of Fogagnolo and Pinamonti in a different setting (\cite{FP}),  describes the geometry of relatively compact domains $\Omega$ in $\mathrm{int}(M)$ with smooth  boundary $\partial \Omega$ subject to an upper bound $-\ol{H}$ on the mean curvature $H$ of $\partial \Omega$ in the direction of the inward unit normal. Precisely, we have the following
\begin{theorem}\label{thm B}
	Let $(M,g)$ be a $\p$-SPFST of dimension $m\geq 2$ and let $\Omega \subset\subset \mathrm{int}(M)$ with smooth boundary.
	Let
	\begin{align}\label{intro: ol H}
		\ol{H}=\frac{1}{m}\frac{\int_{\partial \Omega}u}{\int_{\Omega u}}
	\end{align}
	and assume
	\begin{align}\label{intro: teo curv media: H geq ol H}
		H\leq-\ol{H},
	\end{align}
	where $H$ is the mean curvature of $\partial \Omega$ in the direction of the inward unit normal. Furthermore, suppose
	\begin{align*}
		\mu+p\geq 0 \text{ on } M.
	\end{align*}
	Then
	\begin{align*}
		i:\partial\Omega \hookrightarrow M
	\end{align*} 
	is totally umbilical and $\mu$ and $p$ are constant on $\Omega$, with $\mu=-p$.
\end{theorem}
\begin{rem}
	If we compute $H$ with respect to the outward unit normal, (\ref{intro: teo curv media: H geq ol H}) becomes
	\begin{align*}
		H\geq \ol{H}.
	\end{align*}
\end{rem}
As we said before, the third part of the paper is devoted to non-existence results. The idea is simple: the existence of a positive solution of the $\p$-SPFST system (\ref{Gianny1}) implies the existence of $u$ solving
\begin{align}\label{intro: sistema non existence}
	\begin{cases}
		&\Delta u+\frac{1}{m-1}\pa{2U(\p)-mp-(m-2)\mu}u=0,\\
		&u>0 \ \ \text{ on }\mathrm{int}(M),
	\end{cases}
\end{align}
thus we only need to give sufficient conditions on the geometry of $M$ and on the coefficient
\begin{align}\label{intro: non existence: coefficiente}
	2U(\p)-mp-(m-2)\mu
\end{align}
to ensure that \eqref{intro: sistema non existence} has no positive solutions.
Since, for $M$ complete with empty boundary,the existence of a positive solution implies that the differential operator $L$ defined by
\begin{align*}
	Lv=\Delta v+\frac{1}{m-1}\pa{2U(\p)-mp-(m-2)\mu}v
\end{align*}
has non-negative spectral radius (see e.g. \cite{Tshirt_col_brie}, \cite{MossPiepenbrink}), to prove non-existence we look, under appropriate assumptions, for a contradiction to this property.
Towards this aim, we study conditions for the existence of a first zero or an oscillatory behaviour of a solution of a Cauchy problem of the type
\begin{align}\label{intro: non-existence: prob di Cauchy}
	\begin{cases}
		&\pa{v(t)z'}'+A(t)v(t) z=0  \ \text{ on }\erre^+, \\
		&z(0^+)=z_0>0,\\
		&\pa{vz'}(0^+)=0,
	\end{cases}
\end{align}
where we set
\begin{align}\label{intro: non exist: v(t)}
	v(t)=\mathrm{Vol}\pa{\partial B_t},
\end{align}
with $B_t$ the geodesic ball of radius $t$ with a fixed origin $o\in M$ and $A(t)$ is given by
\begin{align*}
	A(t)=\frac{1}{v(t)}\int_{\partial B_t}\frac{1}{m-1}\pa{2U(\p)-mp-(m-2)\mu}
\end{align*}
(see e.g. \cite{BMR} for more details on \eqref{intro: non-existence: prob di Cauchy}). We then compare a solution of \eqref{intro: non-existence: prob di Cauchy} having a first zero with $u$ to obtain the desired contradiction.
We refer the reader to Section \ref{Sect_ non existence} for the details;
let us however state one of our results to have some flavor on the subject.
\begin{theorem}
	Let $(M,g)$ be a complete manifold of dimension $m\geq 2$ with $\partial M=\emptyset$. Assume
	\begin{align}\label{intro: non exist: teo: volume}
		\frac{1}{\mathrm{Vol}\pa{\partial B_t}}\notin L^1(+\infty)
	\end{align}
	and, for some $t_0>0$,
	\begin{align}\label{intro: non exist: teo: condizione su A}
		\lim_{t\to +\infty}\int_{B_t\backslash B_{t_0}}\sq{2U(\p)-mp-(m-2)\mu}=+\infty.
	\end{align}
	Then there is no positive solution $u$ of the $\p$-SPFST system \eqref{Gianny1}.
\end{theorem}
We should point out that condition \eqref{intro: non exist: teo: volume} is not so restrictive as it may appear. Indeed, one can construct examples of complete manifolds with, for instance, exponential growth of the volume of geodesic balls but still satisfying \eqref{intro: non exist: teo: volume}.
As for assumption \eqref{intro: non exist: teo: condizione su A}, we observe that it does not imply that $2U(\p)>(m-2)\mu+mp$ on the entire manifold, but only "at infinity", so that $(M,g)$ could in fact even satisfy the Strong Energy Condition in a region $\Omega$ of $M$  provided that, on $\Omega$, the negative contribution of the integrand in \eqref{intro: non exist: teo: condizione su A} does not affect the infinite limit.
The paper ends with an intrinsic (that is, independent of the existence of a conformal vector field on $M$) Kazdan-Warner-type obstruction on the compact manifold $M$, related to the elementary symmetric functions of the tensor
\begin{align*}
	A^\p-\frac{U(\p)}{m-1}g
\end{align*}
provided the latter is Codazzi, that is, equivalently,
\begin{align*}
	2(m-1)C^\p=-\diver_1\pa{U(\p)g\KN g}.
\end{align*}

\section{Some Remarks on the Physical Motivations}
At the beginning of the last century, Einstein was led to postulate that the gravitational forces should be expressed by the curvature of a Lorentzian metric $\bar{g}$ on a $4$-dimensional manifold, on the basis of four (by now famous) principles, that is, the principles of Relativity, Equivalence, General Covariance and Causality. He also postulated, maybe inspired by Riemann and Herbart, that space-time is curved in itself, and that its curvature is locally determined by the distribution of the sources encoded in what is called the \emph{stress-energy tensor} $\overline{T}$. As a first attempt to describe the dynamic of gravitational forces, he and Grossman proposed the equation
\begin{align}\label{0.1 in introduzione 2}
\overline{\ric}=k\overline{T},
\end{align}
where $k$ is a positive coupling constant, but Einstein soon realized that equation \eqref{0.1 in introduzione 2} was not satisfactory, both from the physical and the mathematical point of view. Indeed, if $\overline{T}$ describes any reasonable kind of matter, for instance a perfect fluid for which
\[
\overline{T} = \pa{p+\mu}du\otimes du - p\bar{g},
\]
where $p$ is the pressure, $\mu$ the density of the fluid and $du\otimes du$ a comoving observer, then the continuity equation for matter requires
\begin{align}\label{continuity eq SPFST}
\diver\overline{T}=0,
\end{align}
which is compatible with \eqref{0.1 in introduzione 2} only in case $(\overline{M},\bar{g})$ has \emph{identically zero} scalar curvature $\bar{S}$. Moreover, from the mathematical point of view, equation \eqref{0.1 in introduzione 2}, as pointed out by Hilbert (but note that this is not completely correct), has no variational origin. To overcome this latter difficulty, Einstein and Hilbert, independently, concluded that \eqref{0.1 in introduzione 2} has to be replaced by
\begin{align}\label{0.3 in introduzione 2}
\overline{\ric}-\frac{1}{2}\bar{S}\bar{g}=k\overline{T},
\end{align}
that, because of the Bianchi identities, is compatible with \eqref{continuity eq SPFST} and has a left-hand side member coinciding with the left-hand side  of the Euler-Lagrange equations for the Einstein-Hilbert functional
\begin{align}\label{EH}
\int_{\Omega}\bar{S}dV_{\bar{g}}, \quad\Omega\subset\subset\overline{M},
\end{align}
while the right-hand side  comes from the action of a matter Lagrangian with ``variational derivative'' $\overline{T}$   (see for instance Hawking and Ellis, \cite{HawkingEllis}).

Equation \eqref{0.3 in introduzione 2} is the ``simplest choice'', both from the physical and mathematical point of view; however, this is by no means the only possible choice, and recent astrophysical observations and current cosmological hypothesis  indicate that equation \eqref{0.3 in introduzione 2} is inadequate to describe the gravitational forces at the level of extra-galactic and cosmic scale, unless we are keen to admit the existence, in the universe, of some kind of unknown matter and energy, that is, the \emph{dark matter} and the \emph{dark energy}. A second approach to solve the above problem could rest on modifying the left-hand side of the field equations by considering an extended action functional,   more complex than \eqref{EH} (these are usually called  \emph{extended} $f(R)$ \emph{theories of gravitation}, see e.g. \cite{capozziello2008extended}). However, we remark that taking powers of $\overline{S}$ in the action, which seems to be the first simplest choice to test, might not be the right one, indeed for $\bar{\varphi} : \overline{M} \ra \pa{N, h}$ the geometry associated to $\int_{\Omega}\pa{\bar{S}-\frac{1}{2}\abs{d\bar{\p}}^2}^2dV_{\bar{g}}$ is quite rigid (see the Riemannian case treated in \cite{MR4639007}). On the other hand, \emph{Cotton gravity}, introduced by Harada \cite{Harada} and based on a different action functional, has been able to describe the rotational motion of eightyfour galaxies using the gravitational potential obtained from a solution of the field equations, well fitting with the experimental data, without having to use dark matter (see \cite{PhysRevD.106.064044}).
As explained before, in this paper the
setting we shall consider falls in the realm of dark matter. 
	\chapter{Preliminaries}\label{Sect_Preliminaries}

\noindent
The aim of this chapter is to present a number of facts and results that, at the same time, justify the analysis of system \eqref{Gianny1} and will reveal to be important investigation tools in the rest of the paper. In particular, we will be dealing with the following subjects:\\
\begin{itemize}

	 \item given the extended Einstein field equations in a warped product Lorentzian context, we derive system \eqref{Gianny1} by  splitting a suitable stress-energy tensor with respect to the "temporal" and "spatial" components.\\

	 \item We recast the "Energy Conditions" relating to the energy-momentum tensor in the setting of $\p$-SPFST, extending the recent work (\cite{Maeda_2020}), based on the analysis of Hawking and Ellis in \cite{HawkingEllis}. Some of them will play the role of assumptions in part of the results, while, in some other situations, their violation will be necessarily needed to gain rigidity of the structure.\\

	\item We introduce the notion of $\p$-curvatures, extending the standard curvature tensors; these new tensors will be an invaluable tool to merge the geometry of $(M,g)$ with that of the smooth map $\p:(M,g)\ra(N,h)$. We shall also prove some relevant formulas.\\
	
	\item We consider a generalized version of system \eqref{Gianny1} and we show, in some special cases, that it can be deduced from  well-known
	mathematical settings; indeed, beside the variational derivation of the Euler-Lagrange equations from an action functional involving a non-linear field, we also highlight how the system can be considered as a "soliton" structure coming from the coupled "Ricci-harmonic" flow, introduced by List in \cite{List2008EvolutionOA}. A further derivation is obtained from a warped product manifold carrying a  harmonic-Einstein structure  with potential; finally, we obtain the system for $U\equiv 0$, by a conformal deformation of a harmonic-Einstein manifold.\\

\end{itemize}

\section{Deduction of the System}\label{Subse: ded of system}
A \textit{Lorentzian manifold} $(\hat{M},\hat{g})$ is a smooth $(m+1)$-dimensional manifold with a non-degenerate $(0,2)$-symmetric tensor $\hat{g}$ with signature $(-,+,...,+)$. We fix the indexes range $0 \leq \alpha, \beta, \ldots \leq m$: then we can find a local  coframe $\set{\omega^\alpha}_{\alpha=0}^m$ on an open set $U\subseteq \hat{M}$ with the property that
\begin{align}\label{hatgcof}
	\hat{g}=-\omega^0\otimes\omega^0+\omega^1\otimes\omega^1+...+\omega^m\otimes\omega^m;
\end{align}
we can write \eqref{hatgcof} as
\begin{align*}
	\hat{g}=g_{\alpha\beta}\omega^\alpha\otimes\omega^\beta,
\end{align*}
where $g_{\alpha\beta}$ are the entries of the diagonal matrix with $g_{00}=-1$, $g_{ii}=1$ for $i=1,...,m$ (no sum over the index $i$); we denote by $\pa{g^{-1}}^{\alpha\beta}=g^{\alpha\beta} \pa{=g_{\alpha\beta}}$ the entries of the inverse matrix of $g_{\alpha\beta}$ and by $\set{e_\alpha}_{\alpha=0}^m$ the frame dual to $\set{\omega^\alpha}_{\alpha=0}^m$. The \emph{Levi-Civita connection} is defined \emph{via} the formula
\begin{align*}
	\nabla {e_\alpha}=\omega^\beta_\alpha\otimes e_\beta,
\end{align*}
where the \emph{Levi-Civita connection forms} $\set{\omega^\alpha_\beta}$ are uniquely defined by the requirements

\begin{align}\label{equazione: first struct eq}
	\begin{cases}
		d\omega^\alpha=-\omega^\alpha_\beta\wedge\omega^\beta \quad \text{(first structure equations)},\\
		\omega_{\alpha\beta}+\omega_{\beta\alpha}=0,
	\end{cases}
\end{align}
with $\omega_{\alpha\beta}= g_{\alpha\gamma}\omega^\gamma_\beta$.

The second structure equations read as
\begin{align}\label{equazione: secnd structure eq}
	d\omega^\alpha_\beta=-\omega^\alpha_{\gamma}\wedge\omega^\gamma_\beta+\Omega^\alpha_\beta,
\end{align}
where the $2$-forms $\set{\Omega^\alpha_\beta}$ are the \textit{curvature forms} associated to the orthonormal coframe $\set{\omega^\alpha}_{\alpha=0}^m$.
Note that, having defined
\[
\Omega_{\alpha\beta} =  g_{\alpha\gamma}\Omega^\gamma_\beta
\]
(and thus  $\Omega^{\alpha}_\beta =  g^{\alpha\gamma}\Omega_{\gamma_\beta}$),  we  have
\begin{align*}
	\Omega_{\alpha\beta}+\Omega_{\beta\alpha}=0,
\end{align*}
and the components $\hat{R}_{\alpha\beta\gamma\delta}$ of the $(0,4)$-version of the Riemann curvature tensor are given by
\begin{align}\label{components of riemann on Lor}
	\Omega_{\alpha\beta}=\frac{1}{2}\hat{R}_{\alpha\beta\gamma\delta}\omega^\gamma\wedge\omega^\delta,
\end{align}
while the components $\hat{R}^{\alpha}_{\beta\gamma\delta}$ of the $(1,3)$-version are given by
\[
\hat{R}^{\alpha}_{\beta\gamma\delta} = g^{\alpha\eta}\hat{R}_{\eta\beta\gamma\delta}.
\]

Observe that $\hat{R}_{\alpha\beta\gamma\delta}$ satisfy the usual symmetries
\begin{align*}
	\hat{R}_{\alpha\beta\gamma\delta}=\hat{R}_{\gamma\delta\alpha\beta}=-\hat{R}_{\delta\gamma\alpha\beta}=\hat{R}_{\delta\gamma\beta\alpha}
\end{align*}
and so on.


Suppose now that $(\hat{M}, \hat{g}) = (\erre\times M, \hat{g})$, where $\hat{g}$ is  the warped product metric
\begin{align}\label{eq_5.3}
	\hat{g}=-e^{-2f}dt^2+g,
\end{align}
where $(M, g)$ is an $m$-dimensional Riemannian manifold, $f\in C^{\infty}(M)$ and $t$ is the standard coordinate on $\erre$. To deduce system \eqref{Gianny1}, we explicitly calculate the relation between $\hat{\ric}$ and $\ric$, $\tau(\hat{\p})$ and $\tau(\p)$. We fix the further index range $1\leq i, j, \ldots, m$ and we let
 $\set{\theta^i}_{i=1}^m$ be a local orthonormal coframe for $g$ with Levi-Civita connection forms $\set{\theta^i_j}$ and curvature forms $\set{\Theta^i_j}$, $i,j=1,...,m$. Note that
\begin{align}
	\Theta^i_j=\frac{1}{2}R^i_{jkl}\theta^k\wedge\theta^l,
\end{align}
 where  $R^i_{jkl}$ are the components of the curvature tensor of $(M,g)$; therefore,
 \begin{equation}\label{eq_5.10}
   \Omega^\alpha_\beta = \frac{1}{2}\hat{R}^{\alpha}_{\beta\gamma\delta}\omega^\gamma \wedge \omega^\delta.
 \end{equation}
 We set
 \begin{equation}\label{eq_5.11}
   \omega^i=\theta^i, \quad i=1,...,m, \quad  \omega^0=e^{-f}dt;
 \end{equation} in this way,  the forms $\set{\omega^\alpha}_{\alpha=0}^m$ so defined give a local orthonormal coframe for the Lorentzian metric $\hat{g}$ in \eqref{eq_5.3}. The validity of the first structure equation for $g$, \eqref{eq_5.11} together with  $df=f_k\theta^k$ yield
\begin{align*}
	d\omega^i&=d\theta^i=-\theta^i_j\wedge\theta^j=-\omega^i_j\wedge\omega^j,\:\: \\
	d\omega^0&=-e^{-f}df\wedge dt=-e^{-f}f_k\theta^k\wedge dt=-f_k\omega^k\wedge\omega^0.
\end{align*}
Thus, defining
\begin{align}\label{le omega alpha beta per warped}
	\omega^0_i&=-\omega^i_0=-f_i\omega^0=-f_ie^{-f}dt,\:\:\\
	\omega^i_j&=\theta^i_j,\:\:\notag\\
	\omega^0_0&=0,\notag
\end{align}
we have that  \eqref{equazione: first struct eq} are satisfied. It follows that
\begin{align*}
	\Omega^i_j&=\Theta^i_j,\:\:\\
	\Omega^0_j&=(f_{jk}-f_jf_k)\omega^0\wedge\omega^k;
\end{align*}
as a consequence, we have
\begin{align*}
	\frac{1}{2}\hat{R}^i_{j\sigma\mu}\omega^\sigma\wedge\omega^\mu =\frac{1}{2}R^i_{jkl}\theta^k\wedge\theta^l
\end{align*}
and
\begin{align*}
\frac{1}{2}\hat{R}^0_{j\sigma\mu}\omega^\sigma\wedge\omega^\mu =-(f_{jk}-f_jf_k)\omega^k\wedge\omega^0.
\end{align*}
We let
\[
\tilde{a}_{j\sigma} = \begin{cases}
  0, & \text{ if }\, \sigma=0; \\ -(f_{jk}-f_jf_k), & \text{ if }\, \sigma=k.
\end{cases}
\]
Then
\begin{align*}
  \frac{1}{2}\hat{R}^0_{j\sigma\mu}\omega^\sigma\wedge\omega^\mu &= \tilde{a}_{j\sigma}\omega^\sigma\wedge\omega^0 \\ &= \tilde{a}_{j\sigma}\delta_{0\mu}\omega^\sigma\wedge\omega^\mu \\ &=\frac{1}{2}\pa{\tilde{a}_{j\sigma}\delta_{0\mu}-\tilde{a}_{j\mu}\delta_{0\sigma}}\omega^\sigma\wedge\omega^\mu,
\end{align*}
so that
\begin{equation}\label{Eq5.13}
  \hat{R}^0_{j\sigma\mu} = \tilde{a}_{j\sigma}\delta_{0\mu}-\tilde{a}_{j\mu}\delta_{0\sigma}.
\end{equation}
In particular, we have
\[
\hat{R}^0_{j0\mu} = \tilde{a}_{j0}\delta_{0\mu}-\tilde{a}_{j\mu}= -\tilde{a}_{j\mu}
\]
and
\begin{equation}\label{Eq5.14}
  \hat{R}^0_{j0k} = -\tilde{a}_{jk} = f_{jk}-f_jf_k.
\end{equation}
Similarly, let
\[
\tilde{R}^i_{j\sigma l} = \begin{cases}
  0, & \text{ if }\, \sigma=0, \\ {R}^i_{jk l}, & \text{ if }\, \sigma=k;
\end{cases}
\]
then
\begin{align*}
  \frac{1}{2}\hat{R}^i_{j\sigma\mu}\omega^\sigma\wedge\omega^\mu  &= \frac{1}{2}\hat{R}^i_{j\sigma l}\omega^\sigma\wedge\omega^l \\&=\frac{1}{2}\hat{R}^i_{j\sigma l}\delta_{l\mu}\omega^\sigma\wedge\omega^\mu \\ &=\frac{1}{4}\pa{\hat{R}^i_{j\sigma l}\delta_{l\mu}-\hat{R}^i_{j\mu l}\delta_{l\sigma}}\omega^\sigma\wedge\omega^\mu,
\end{align*}
so that
\begin{equation}\label{Eq5.15}
  \hat{R}^i_{j\sigma\mu} = \frac{1}{2}\pa{\hat{R}^i_{j\sigma l}\delta_{l\mu}-\hat{R}^i_{j\mu l}\delta_{l\sigma}}
\end{equation}
and, in particular,
\begin{equation}\label{Eq5.16}
  \hat{R}^i_{jkl} = {R}^i_{jkl}, \qquad \hat{R}^i_{j0l} = 0 = {R}^i_{jl0}.
\end{equation}

As far as the Ricci tensor is concerned, from equations \eqref{Eq5.14} and \eqref{Eq5.16} we have, for $\hat{R}_{\mu\nu}=\hat{R}^\alpha_{\mu\alpha\nu}$,

\begin{align*}
	\hat{R}_{ij}&=f_{ij}-f_if_j+R_{ij}=\hat{R}_{ji},\\
	\hat{R}_{00}&=-\Delta f+\abs{\nabla f}^2,\\
	\hat{R}_{j0}&=0=\hat{R}_{0j}
\end{align*}
(here $\abs{\nabla f}^2 = \abs{\nabla f}_g^2$).
Since
\begin{align*}
	\hat{\ric}=\hat{R}_{00}\omega^0\otimes\omega^0+\hat{R}_{ij}\omega^i\otimes\omega^j,
\end{align*}
we deduce
\begin{align}\label{riccihat e ricci}
	\hat{\ric}=\ric+\hs(f)-df\otimes df-\pa{\Delta f-\abs{\nabla f}^2}e^{-2f}dt\otimes dt.
\end{align}

Let us now consider a Riemannian manifold $(N, h)$ of dimension $n$, and let $\p:(M, g)\ra (N, h)$ be a smooth map; let $\hat{\p}:=\p\circ\pi_M : (\hat{M},\hat{g})\ra (N, h)$, where $\pi_M$ is the projection of $\hat{M}$ on $M$. Fix a local orthonormal coframe $\set{\eta^a}_{a=1}^n$ on $N$, with dual frame $\set{E_a}$, and let $\set{\eta^a_b}$, $a,b=1,...,n$, be the corresponding Levi-Civita connection forms. We have
\begin{equation*}
   d{\p} = {\p}_* = {\p}^a_i\theta^i\otimes E_a \qquad \text{ and } \qquad d\hat{\p}=\hat{\p}_*= \hat{\p}^a_\nu\omega^\nu\otimes E_a,
\end{equation*}
but since
\begin{equation*}
  d\hat{\p} =  {\p}_* \circ \pa{\pi_M}_* = d\p \circ \pa{\pi_M}_*,
\end{equation*}
a simple computation shows that
\begin{equation}\label{phi_phihat}
  \hat{\p}^a_0 = 0, \quad \hat{\p}^a_i = {\p}^a_i,
\end{equation}
that is
\begin{equation*}
   d\hat{\p}=d{\p}.
\end{equation*}
We recall that, by definition (and omitting the pullback notation for simplicity),

\begin{align}\label{equ per phi hat mu nu}
	\hat{\p}^a_{\mu\nu}\omega^\nu=d\hat{\p}^a_{\mu}-\hat{\p}^a_\nu\omega^\nu_\mu+\hat{\p}^b_\mu\eta^a_b;
\end{align}
therefore, using \eqref{phi_phihat} and simplifying we obtain
\[
\hat{\p}^a_{i\nu}\omega^{\nu}=\p^a_{ij}\omega^j = \p^a_{ij}\theta^j,
\]
which implies
\begin{equation}\label{eq_5.18}
  \hat{\p}^a_{ij} = {\p}^a_{ij} \quad \text{ and }\quad \hat{\p}^a_{i0}=0.
\end{equation}
To compute the coefficients $\hat{\p}^a_{00}$, note that, again from \eqref{equ per phi hat mu nu}, we have
\begin{align*}
	\hat{\p}^a_{0\nu}\omega^\nu=d\hat{\p}^a_{0}-\hat{\p}^a_\nu\omega^\nu_0+\hat{\p}^b_0\eta^a_b = -\p^a_k\omega^k_0 = \p^a_kf_k\omega^0,
\end{align*}
and thus
\begin{equation}\label{eq_5.19}
  \hat{\p}^a_{00}=f_k\p^a_k,
\end{equation}
which immediately implies
\begin{align}\label{tauhat}
	\tau(\hat{\p})=\tau(\p)-d\p(\nabla f).
\end{align}
Note that, when $f=-\log u$ for some $u\in C^\infty(M)$, $u>0$, equations \eqref{riccihat e ricci} and \eqref{tauhat} become
\begin{align}\label{riccihat u}
	\hat{\ric}=\ric-\frac{\hs(u)}{u}+u{\Delta u} \,dt\otimes dt
\end{align}
and
\begin{align}\label{tauhat u}
	\tau(\hat{\p})=\tau(\p)+\frac{1}{u}d\p(\nabla u).
\end{align}
Moreover, contracting \eqref{riccihat u}, we get
\begin{align}\label{scalhat}
	\hat{S}=S-2\frac{\Delta u}{u}.
\end{align}
\noindent

%

A $\p$-\textbf{static perfect fluid space-time}, $\p$-SPFST for short, is characterized as a solution of the Einstein equations
\begin{align}\label{einsteineq}
	\hat{\ric}-\frac{\hat{S}}{2}\hat{g}=\hat{T},
\end{align}
where $\hat{T}$ is the stress-energy tensor defined as the sum of the stress-energy tensor   of a static perfect fluid $\hat{T}_F$ with energy density $\mu\in C^\infty(M)$ and pressure $p\in C^\infty(M)$ with the stress-energy tensor $\hat{T}_{\hat{\p}}$ relative to the static ``non-linear field'' $\hat{\p}: (\hat{M}, \hat{g})\to (N,h)$, interacting with the potential $U\in C^\infty(N)$.
To define $\hat{T}_F$, we let $v=\frac{1}{u}\frac{\partial}{\partial t}$ and set $v^{\flat}$ for its dual form; note that $v$ is a unit, future directed timelike vector field. Then, $\hat{T}_F$ is phenomenologically defined by
\begin{align}\label{stess-energy SPFST}
	\hat{T}_{F}=[(\mu+p)\circ \pi_M] v^{\flat}\otimes v^{\flat}+ [p \circ \pi_M] \hat{g}
\end{align}
(see  \cite{CDPR2023},\cite{CDLR2020},\cite{HawkingEllis},\cite{KO},\cite{Maeda_2020}),
while $\hat{T}_{\hat{\p}}$ is given by the prescription
\begin{align}\label{stress-energy phi}
	\hat{T}_{\hat{\p}}=\alpha\hat{\p}^*h-\pa{U(\hat{\p})+\alpha\frac{\abs{d\hat{\p}}^2}{2}}\hat{g}
\end{align}
(see \cite{Maeda_2020}).

Hence, combining the two expressions, observing that $\hat{\p}^*h=\p^*h$ and that $\abs{d\hat{\p}}^2=\abs{d\p}^2$ and omitting the compositions with $\pi_M$ to simplify the writing, we obtain	
\begin{align}\label{T}
	\hat{T}
	&=\pa{\mu+\alpha\frac{\abs{d\p}^2}{2}+U(\p)}v^{\flat}\otimes v^{\flat}+\pa{p-\alpha\frac{\abs{d\p}^2}{2}-U(\p)}g+\alpha\p^*h.
\end{align}

\begin{rem}
	The stress-energy tensor of a SPFST and that of a non-linear field can be both obtained variationally. However, the approach used in the two cases is quite different: indeed, as remarked in \cite{HawkingEllis}, in the first case one has to take the variation of an appropriate functional depending on $\mu$ with respect to the flow lines of the tangent vector field $v$, while in the second case it is sufficient to compute the variation of a functional depending on $\hat{\p}$ and the metric $\hat{g}$, with respect to $\hat{\p}$ and the components of the metric (see \cite{Maeda_2020} for more details).
\end{rem}
We are now ready to show how   \eqref{Gianny1} is deduced; for the convenience of the reader, we recall here the system:
\begin{align}
	\begin{cases}
		i)\, \hs(u)-u\set{\ric^\p-\frac{1}{m-1}\pa{\frac{S^\p}{2}-p+U(\p)}g}=0,\\
		ii)\,\Delta u=\frac{u}{m-1}\sq{mp-mU(\p)+\frac{m-2}{2}S^\p},\\
		iii)\,u\tau(\p)=-d\p(\nabla u)+\frac{u}{\alpha}(\nabla U)(\p),\\
		iv)\,\mu+U(\p)=\frac{1}{2}S^\p,\\
		v)\,(\mu+p)\nabla u=-u\nabla p.
	\end{cases}
\end{align}
First, note that the splitting of $\hat{T}$ in \eqref{T} is with respect to the time and the space  components. Inserting \eqref{riccihat u}, \eqref{scalhat} and $\hat{g}=g-v^\flat\otimes v^\flat$ into the definition of $\hat{T}$, we have
\begin{align}\label{T con ricci e S}
	\hat{T}&=\hat{\ric}-\frac{\hat{S}}{2}\hat{g}\notag\\
	&=\ric-\frac{\hs(u)}{u}-\pa{\frac{S}{2}-\frac{\Delta u}{u}}g+\frac{S}{2}v^\flat\otimes v^\flat\notag\\
	&=\ric-\frac{\hs(u)}{u}-\pa{\frac{S^\p}{2}+\alpha\frac{\abs{d\p}^2}{2}-\frac{\Delta u}{u}}g+\pa{\frac{S^\p}{2}+\alpha\frac{\abs{d\p}^2}{2}}v^\flat\otimes v^\flat.
\end{align}
Comparing space and time components of
\eqref{T con ricci e S} and \eqref{T}, we infer equation \eqref{Gianny1} iv), that is,
\begin{align*}
	\mu+U(\p)=\frac{S^\p}{2}
\end{align*}
and
\begin{align}\label{ho finito la fantasia}
	\ric-\alpha\p^*h-\frac{\hs(u)}{u}-\frac{S^\p}{2}g+\frac{\Delta u}{u}g=(p-U(\p))g.
\end{align}
Taking the trace of \eqref{ho finito la fantasia} yields
\begin{align}\label{p e U per ii}
	m(p-U(\p))=(2-m)\frac{S^\p}{2}+(m-1)\frac{\Delta u}{u},
\end{align}
that immediately gives equation  \eqref{Gianny1} ii). Then, replacing \eqref{p e U per ii} into \eqref{ho finito la fantasia}, we obtain
\begin{align}\label{prima di i}
	\ric-\alpha\p^*h-\frac{\hs(u)}{u}=\frac{1}{m}\pa{S^\p-\frac{\Delta u}{u}}g.
\end{align}
Substituting the expression for $\Delta u$, obtained in \eqref{Gianny1} ii), into \eqref{prima di i} we easily get equation i) of \eqref{Gianny1}. Equation iii) and v) of system \eqref{Gianny1} are a consequence of the \emph{energy-momentum conservation}, $\mathrm{div}_{\hat{g}}\hat{T}=0$: indeed, a computation using the fact that
\begin{align*}
	\frac{\partial \mu}{\partial t}=\frac{\partial p}{\partial t}=0
\end{align*}
shows the validity of the \emph{equations of motion}
\begin{align}\label{divTapp}
	(\mu+p)\frac{\nabla u}{u}+\nabla\sq{p-U(\p)}+\alpha \sq{h\pa{\tau(\hat{\p}),d\hat{\p}}}^\sharp=0.
\end{align}

Since $\tau\pa{\hat{\p}}$ is related to the tension field  $\tau\pa{\p}$ by \eqref{tauhat u},
from \eqref{divTapp} we have
\begin{align}\label{dadivTapp}
	\alpha \sq{h\pa{ u\tau(\p)+d\p(\nabla u),d\p}}^\sharp+(\mu+p)\nabla u=u\nabla\sq{U(\p)-p}.
\end{align}
\begin{rem}
    Note that equation \eqref{dadivTapp} is weaker than the pair of equations \ref{Gianny1} iii) and \ref{Gianny1} v), that is,
\begin{align}\label{weakereq}
	u\tau(\p)=-d\p(\nabla u)+\frac{u}{\alpha}(\nabla U)(\p),&&(\mu+p)\nabla u=-u\nabla p.
\end{align}
\end{rem}

We can now give the following
\begin{defi}\label{Defi_Ch2_phiSPFST}
	We say that $(\hat{M},\hat{g})$ is a $\p$-\textbf{static perfect fluid space-time}, or $\p$-SPFST, if $(M, g)$ is an $m$-dimensional Riemannian manifold with a smooth solution $u$ of the system \eqref{Gianny1}, that is
	\begin{align}\label{Gianny}
		\begin{cases}
			i)\, \hs(u)-u\set{\ric-\alpha\p^*h-\frac{1}{m-1}\pa{\frac{S-\abs{d\p}^2}{2}-p+U(\p)}g}=0,\\
			ii)\,\Delta u=\frac{u}{m-1}\sq{-mU(\p)+mp+\frac{m-2}{2}(S-\alpha\abs{d\p}^2)},\\
			iii)\,u\tau(\p)=-d\p(\nabla u)+\frac{u}{\alpha}(\nabla U)(\p),\\
			iv)\,\mu+U(\p)=\frac{S-\alpha\abs{d\p}^2}{2},\\
			v)\,(\mu+p)\nabla u=-u\nabla p.
		\end{cases}
	\end{align}
	We also require $u>0$ on $\mathrm{int}(M)$ and $\partial M=u^{-1}(\set{0})$, in case $\partial M\neq \emptyset$.
\end{defi}
\noindent
Throughout this paper, we will denote the spatial factor of a $\p$-SPFST simply as $\p$-SPFST for simplicity, although this is a slight  abuse of terminology.\\
Observe also that when $\p$ is constant and $U(\p)\equiv 0$, system \eqref{Gianny1} reduces to \eqref{SPFSTsystem}, that is the one characterizing a static perfect fluid space time. Moreover, as we shall see in Lemma \ref{lemma mu e rho}, equation \eqref{Gianny} v) does not add further information to the other equations of the system.
\noindent

\section{Energy Conditions}\label{Subsection: EC}
As we pointed out in the introduction, some of our assumptions are expressed in terms of energy conditions. Some of the latter are discussed in  detail in \cite{HawkingEllis},
for the case $\dim\hat{M}=4$, and have been recently extended to any dimension in \cite{Maeda_2020}.
However, the General Relativity model cases they consider are different from our, so that we feel obliged to give a short treatment to justify our assumptions.

It is a general fact that in many theories of gravity the distribution of the sources of the gravitational field is encoded in the stress-energy tensor $\hat{T}$; natural assumptions on these sources that are expected to be satisfied by any reasonable physical model of space-time are related to the general principle of ``positivity of the energy'' and they translate into requests on $\hat{T}$ called \textit{energy conditions}.
We first recall that, in Lorentzian Geometry, with respect to a metric $\hat{g}$ a vector $w$ is said to be:
\begin{itemize}
  \item a \emph{null} vector, if $\hat{g}(w,w)=0$;
  \item a \emph{time-like} vector, if $\hat{g}(w,w)<0$;
  \item a \emph{space-like} vector, if $\hat{g}(w,w)>0$.
\end{itemize}

If $(\hat{M},\hat{g})$ is an $(m+1)$-dimensional Lorentzian manifold, the following is a list of the most standard conditions in the classical literature:
\begin{enumerate}
	\item[1.] the \textit{Null Energy Condition} (\textbf{NEC}) is satisfied if, for any null vector $w$,  we have
	\begin{align*}
		\hat{T}(w,w)\geq 0.
	\end{align*}
	\item[2.] The \textit{Weak Energy Condition} (\textbf{WEC}) is satisfied if, for any time-like vector $w$, we have
	\begin{align*}
		\hat{T}(w,w)\geq 0.
	\end{align*}
	\item[3.] The \textit{Strong Energy Condition} (\textbf{SEC}) is satisfied if, for any time-like vector $w$, we have
	\begin{align*}
		\hat{T}(w,w)\geq \frac{1}{m-1}(\mathrm{tr}_{\hat{g}} \hat{T})\hat{g}(w,w),
	\end{align*}
	where
	\begin{align*}
		\mathrm{tr}_{\hat{g}}\hat{T}=g^{\alpha\beta}\hat{T}_{\alpha\beta}.
	\end{align*}
	\item[4.] The \textit{Flux Energy Condition} (\textbf{FEC}) is satisfied if, for any time-like vector $w$, the flux vector
	\begin{align*}
		J_w:=-\hat{T}(w, \cdot)^\#
	\end{align*}
	satisfies
	\begin{align*}
		\hat{g}(J_w, J_w)\leq 0.
	\end{align*}
	\item[5.] The \textit{Dominant Energy Condition} (\textbf{DEC}) is satisfied if, for any time-like vector $w$, we have
	\begin{align*}
		\hat{T}(w,w)\geq 0 \text{  \ \ \ \ \      and   \ \ \ \ \      } \hat{g}(J_w, J_w)\leq 0.
	\end{align*}
\end{enumerate}
For some physical and geometrical interpretations of the above energy conditions see, for instance, \cite{Curiel_2017}.
Here we only observe that:
\begin{itemize}
	\item[i)] The WEC is saying that any time-like observer measures a non-negative energy density.
	\item[ii)] The NEC states the same fact for null-observers.
	\item[iii)] The SEC is more enlightening if we assume the validity of Einstein field equations
	\begin{align*}
		\hat{\ric}-\frac{1}{2}\hat{S}\hat{g}=\hat{T};
	\end{align*}
	indeed, in this case the SEC is equivalent to
	\begin{align*}
		\hat{\ric}(w,w)\geq 0
	\end{align*}
	for each time-like vector $w$.
	This latter fact is interpreted as gravity being ``essentially'' an attractive force.
	\item[iv)] The FEC states that every flux vector does not propagate faster than the speed of light.
\end{itemize}
In the next Proposition we analyze the above energy conditions for the stress-energy tensor $\hat{T}$ of a $\p$-SPFST, given, as we showed before, by
\begin{align}\label{prelimiaries: def: energia impulso}
	\hat{T}=(\mu+p)u^2dt\otimes dt+p\hat{g}+\alpha\hat{\p}^*h-\frac{1}{2}(\alpha |d\hat \p|^2+2U(\hat{\p}))\hat g,
\end{align}
on the Lorentzian warped product $\hat{M}=\erre\prescript{}{u}{\times} M$ with metric
\begin{align}\label{energy conditions: metrica lorenziana}
	\hat{g}=-u^2dt\otimes dt+g,
\end{align}
$u>0$ on $M$.

\begin{proposition}\label{prop: preliminaries: energy conditions}
	For the stress-energy tensor $\hat{T}$ as in (\ref{prelimiaries: def: energia impulso}), with $\alpha>0$, we have the following properties:
	\begin{enumerate}
		\item[1.] the condition
		\begin{align*}
			p+\mu\geq 0
		\end{align*}
		implies the validity of the NEC;
		\item[2.] the conditions
		\begin{align*}
			p+\mu \geq 0 \text{\ \ \ \ and \ \ \ \ } \mu+\frac{\alpha}{2}|d\p|^2+U(\p)\geq 0
		\end{align*}
		imply the validity of the WEC;
		\item[3.] the conditions
		\begin{align*}
			p+\mu\geq 0 \text{\ \ \ \ and \ \ \ \ } (m-2)\mu+mp\geq 2U(\p)
		\end{align*}
		imply the validity of the SEC;
		\item[4.] the conditions
		\begin{align*}
			U(\p)\geq p \text{\ \ \ \ and \ \ \ \ } \left(\mu+\frac{\alpha}{2}|d\p|^2+U(\p)\right)^2\geq\left(p-\frac{\alpha}{2}|d\p|^2-U(\p)\right)^2
		\end{align*}
		imply the validity of the FEC;
		\item[5.] the conditions
		\begin{align*}
			U(\p)\geq p \ , \quad \mu+p\geq 0 \quad  \text{and} \quad \mu+\frac{\alpha}{2}|d\p|^2+U(\p)\geq 0
		\end{align*}
		imply the validity of the DEC.
	\end{enumerate}
\end{proposition}
\begin{rem}\label{remark: energy conditions: necessarietà di alcune ipotesi}
	In  Proposition \ref{prop: preliminaries: energy conditions} we claim  that certain hypotheses are sufficient to guarantee the validity of the corresponding energy conditions. For the case of SPFSTs, which is obtained taking $\p$ constant and $U\equiv 0$, the hypotheses we give reduce to the classical ones and in this case we know that not only they are sufficient conditions, but also necessary.
	In our setting, we can prove that
	\begin{align*}
		\mu+\frac{\alpha}{2}|d\p|^2+U(\p)\geq 0
	\end{align*}
	is necessary for the validity of the DEC and the WEC, while
	\begin{align*}
		(m-2)\mu+mp\geq 2U(\p)
	\end{align*}
	is necessary for the validity of the SEC.
\end{rem}
The proof of Proposition \ref{prop: preliminaries: energy conditions}
will be divided in several lemmas.
To set the stage, recall that, using the expression for $\hat{g}$ given in (\ref{energy conditions: metrica lorenziana}) (and simplifying again the notation), we can rewrite $\hat{T}$ as
\begin{align}\label{energy conditions: tensore energia impulso 2}
	\hat{T}=&\left(\mu+\frac{\alpha}{2}|d\p|^2+U(\p)\right)u^2dt\otimes dt+\left(p-\frac{\alpha}{2}|d\p|^2-U(\p)\right)g\\ \nonumber
	&+\alpha\p^*h.
\end{align}
Let $\{e_i\}_{i=1}^m$ be a local orthonormal frame for $g$ and let $e_0=\frac{1}{u}\frac{\partial}{\partial t}$.
For a time-like vector $w$ we have that, up to a multiplication for a scalar, we can write $w$ as
\begin{align}\label{energy conditions: espressione per v di tipo tempo}
	w=e_0+w^ie_i
\end{align}
(sum over $i$, from $1$ to $m$), for some coefficients $w^i$'s satisfying
\begin{align*}
	\sum_{i=1}^m(w^i)^2<1.
\end{align*}
In the next lemma we collect some easy but useful calculations.
\begin{lemma}\label{lemma: energy condition: calcoli per WEC e SEC}
	With the notations above, we have
	\begin{align}\label{energy conditions: calcolo per WEC}
		\hat{T}(w,w)=&\pa{1-\sum_{i=1}^m(w^i)^2}\left[\mu +\frac{\alpha}{2}|d\p|^2+U(\p)\right]+\left(\sum_{i=1}^m(w^i)^2\right)(\mu+p)\\
		&+\alpha(\p^*h)(w,w), \nonumber
	\end{align}
	and
	\begin{align}\label{energy conditions: calcolo per SEC}
		\left(\hat{T}-\frac{\mathrm{tr}_{\hat{g}}\hat{T}}{m-1}\hat{g}\right)(w,w)=&\left(1-\sum_{i=1}^m(w^i)^2\right)\left[\pa{\frac{m-2}{m-1}}\mu+\frac{m}{m-1}p-\frac{2}{m-1}U(\p)\right]\\
		&+\left(\sum_{i=1}^m (w^i)^2\right)(\mu+p)+\alpha(\p^*h)(w,w). \nonumber
	\end{align}
\end{lemma}
\begin{proof}
	Equation (\ref{energy conditions: calcolo per WEC}) is a direct consequence of (\ref{energy conditions: tensore energia impulso 2}) and (\ref{energy conditions: espressione per v di tipo tempo}); indeed,
	\begin{align*}
		\hat{T}(w,w)=&\mu+\frac{\alpha}{2}|d\p|^2+U(\p)+\left(\sum_{i=1}^m(w^i)^2\right)\left[p-\frac{\alpha}{2}|d\p|^2-U(\p)\right]\\
		&+\alpha(\p^*h)(w,w)\\
		=&\left(1-\sum_{i=1}^m(w^i)^2\right)\left[\mu+\frac{\alpha}{2}\abs{d\p}^2+U(\p)\right]+\left(\sum_{i=1}^m(w^i)^2\right)(p+\mu)\\
		&+\alpha(\p^*h)(w,w).
	\end{align*}
	
	To prove (\ref{energy conditions: calcolo per SEC}), we first show the validity of the relation
	\begin{align}\label{calcolo intermedio per il calcolo per SEC}
		\hat{T}-\frac{\mathrm{tr}_{\hat{g}}\hat{T}}{m-1}\hat{g}=&\frac{u^2}{m-1}\left[(m-2)\mu+mp-2U(\p)\right]dt\otimes dt\\
		&+\frac{1}{m-1}\left[-p+\mu+2U(\p)\right]g+\alpha\p^*h. \nonumber
	\end{align}
	To prove it, observe that from (\ref{prelimiaries: def: energia impulso}) it follows
	\begin{align*}
		\mathrm{tr}_{\hat{g}}\hat{T}= -\mu+mp-\alpha\frac{m-1}{2}\abs{d\p}^2-(m+1)U(\p),
	\end{align*}
	and thus
	\begin{align*}
		\hat{T}-\frac{\mathrm{tr}_{\hat{g}}\hat{T}}{m-1}\hat{g}=&\frac{u^2}{m-1}\left[(m-1)\mu+\alpha\frac{m-1}{2}|d\p|^2+(m-1)U(\p)-\mu+mp\right.\\
		&\left.-\alpha\frac{m-1}{2}|d\p|^2-(m+1)U(\p)\vphantom{\frac{m-1}{2}}\right]dt\otimes dt\\
		&+\frac{1}{m-1}\left[(m-1)p-\alpha\frac{m-1}{2}|d\p|^2-(m-1)U(\p)+\mu-mp\right.\\
		&\left.+\alpha\frac{m-1}{2}|d\p|^2+(m+1)U(\p)\vphantom{\frac{m-1}{2}}\right]g+\alpha\p^*h\\
		=&\frac{u^2}{m-1}{[(m-2)\mu+mp-2U(\p)]}dt\otimes dt+\frac{1}{m-1}[-p+\mu+2U(\p)]g\\
		&+\alpha\p^*h,
	\end{align*}
	which gives (\ref{calcolo intermedio per il calcolo per SEC}). Equation (\ref{energy conditions: calcolo per SEC}) is now an immediate consequence of (\ref{calcolo intermedio per il calcolo per SEC}).
\end{proof}

For a time-like vector $w$, a direct computation shows that the flux vector $J_w$ satisfies
\begin{align*}
	J_w=&-\hat{T}(w, \cdot)^{\#}=\left(\mu+\frac{\alpha}{2}|d\p|^2+U(\p)\right)e_0\\
	&-\left(p-\frac{\alpha}{2}|d\p|^2-U(\p)\right)\pa{\sum_{i=1}^mw^ie_i}-\alpha(\p^*h)(w, \cdot)^{\#}.
\end{align*}
Moreover, since $\hat{\p}^*h = \p^*h$,
\begin{align*}
	(\p^*h)(w, \cdot)^{\#}=\sum_{i,j=1}^mw^j\p^a_j\p^a_ie_i,
\end{align*}
thus
\begin{align}\label{eqnergy condition: calcolo del vettore flusso}
	J_w=&\left(\mu+\frac{\alpha}{2}|d\p|^2+U(\p)\right)e_0\\
	&-\left(p-\frac{\alpha}{2}|d\p|^2-U(\p)\right)\sum_{i=1}^mw^ie_i-\alpha\sum_{i,j=1}^mw^j\p^a_j\p^a_ie_i. \nonumber
\end{align}
\begin{lemma}\label{lemma: energy conditions: calcolo della norma del vettore flusso}
	With the previous notations we have
	\begin{align}\label{energy conditions: calcolo della norma del vettore flusso}
		\hat{g}(J_w, J_w)=&-\pa{1-\sum_{i=1}^m(w^i)^2}\pa{\mu+\frac{\alpha}{2}|d\p|^2+U(\p)}^2\\
		\nonumber
		&+\pa{\sum_{i=1}^m (w^i)^2}\sq{\pa{p-\frac{\alpha}{2}|d\p|^2-U(\p)}^2-\pa{\mu+\frac{\alpha}{2}|d\p|^2+U(\p)}^2}\\ \nonumber
		&+\alpha^2\sum_{i,j,k=1}^mw^i\p^a_i\p^a_j\p^b_j\p^b_ka_k\\
		&+2\alpha\pa{p-\frac{\alpha}{2}|d\p|^2-U(\p)}\sum_{i,j=1}^mw^i\p^a_i\p^a_ja_j. \nonumber
	\end{align}
\end{lemma}
\begin{proof}
	The proof is a direct consequence of equation (\ref{eqnergy condition: calcolo del vettore flusso}).
\end{proof}
We will also need the following elementary lemma.
\begin{lemma}\label{lemma: energy conditions: stima algebrica elementare}
	Let $A$ be a $n\times n$ real, symmetric, positive semi-definite matrix. Then
	\begin{align*}
		\left<Au,u\right>\leq \pa{\mathrm{tr} A}\left<u,u\right>,
	\end{align*}
	where $u\in \erre^n$ and $\left<\cdot, \cdot\right>$ is the Euclidean inner product.
\end{lemma}
\begin{proof}
	Since $A$ is positive semi-definite we have
	\begin{align*}
		\pa{\sum_{i=1}^n \lambda_i}^2\geq \sum_{i=1}^n \lambda_i^2
	\end{align*}
	where the $\lambda_i$'s are the eigenvalues of $A$. If we let $\abs{\abs{A}}$ denote the operator norm of $A$, we deduce
	\begin{align*}
		\mathrm{tr} A\geq \abs{\abs{A}}=\sqrt{\sum_{i=1}^n \lambda_i^2}.
	\end{align*}
	Since
	\begin{align*}
		\left<Au,u\right>\leq \abs{\abs{A}}\left<u,u\right>,
	\end{align*}
	the conclusion follows.
\end{proof}
We are now ready for the proof of Proposition \ref{prop: preliminaries: energy conditions}.
\begin{proof}[Proof (of Proposition \ref{prop: preliminaries: energy conditions})]
	The WEC follows from equation (\ref{energy conditions: calcolo per WEC}) once we assume, as we are doing, that $\alpha>0$. Indeed, $\p^*h$ is positive semi-definite, so  that from (\ref{energy conditions: calcolo per WEC}) we deduce
	\begin{align*}
		\hat{T}(w,w)\geq \pa{1-\sum_{i=1}^m a_i^2}\sq{\mu+\frac{\alpha}{2}|d\p|^2+U(\p)}+\pa{\sum_{i=1}^ma_i^2}(p+\mu).
	\end{align*}
	If in this expression we allow $\sum_{i=1}^m (w^i)^2=1$, then $w$ becomes a null vector, and we also obtain the validity of the NEC.
	Reasoning as above, the SEC is implied by equation (\ref{energy conditions: calcolo per SEC}).
	To deduce the FEC from Lemma \ref{lemma: energy conditions: calcolo della norma del vettore flusso}, we apply Lemma \ref{lemma: energy conditions: stima algebrica elementare} to the matrix $A$ of entries
	\begin{align*}
		A_{ab}=\sum_{i=1}^m\p^a_i\p^b_i
	\end{align*}
	and the vector $u$ of components
	\begin{align*}
		u^a=\sum_{i=1}^mw^i\p^a_i
	\end{align*}
	to conclude, since $\mathrm{tr}A=|d\p|^2$,
	\begin{align*}
		\sum_{i,j,k=1}^mw^i\p^a_i\p^a_j\p^b_j\p^b_ka_k\leq |d\p|^2\sum_{i,j=1}w^i\p^a_i\p^a_ja_j.
	\end{align*}
	Inserting this information into (\ref{energy conditions: calcolo della norma del vettore flusso}) we obtain
	\begin{align*}
		\hat{g}(J_w,J_w)\leq&-\pa{1-\sum_{i=1}^m(w^i)^2}\pa{\mu+\frac{\alpha}{2}|d\p|^2+U(\p)}^2\\
		\nonumber
		&+\pa{\sum_{i=1}^m (w^i)^2}\sq{\pa{p-\frac{\alpha}{2}|d\p|^2-U(\p)}^2-\pa{\mu+\frac{\alpha}{2}|d\p|^2+U(\p)}^2}\\ \nonumber
		&+2\alpha\pa{p-U(\p)}\sum_{i,j=1}^mw^i\p^a_i\p^a_ja_j. \nonumber
	\end{align*}
	which gives the FEC.
	Now note that DEC is satisfied exactly when the FEC and the WEC are satisfied simultaneously; from what we previously proved we deduce that DEC is implied by the validity of the following four conditions: \begin{enumerate}
		\item[I)] $\mu+p\geq 0$;
		\item[II)] $\mu+\frac{\alpha}{2}|d\p|^2+U(\p)\geq 0$;
		\item[III)] $U(\p)\geq p$;
		\item[IV)] $\pa{\mu+\frac{\alpha}{2}|d\p|^2+U(\p)}^2\geq \pa{p-\frac{\alpha}{2}|d\p|^2-U(\p)}^2$.
	\end{enumerate}
	Then, to conclude, it is sufficient to prove that condition IV) is redundant. Indeed, computing the squares and simplifying we get
	\begin{align*}
		\mu^2+\alpha\mu |d\p|^2+2\mu U(\p)\geq p^2-\alpha p|d\p|^2-2pU(\p)
	\end{align*}
	which is equivalent to
	\begin{align*}
		0\leq&\mu^2-p^2+(\mu+p)\sq{\alpha|d\p|^2+2U(\p)}\\
		=&(\mu+p)(\mu-p)+(\mu+p)\sq{\alpha|d\p|^2+2U(\p)}\\
		=&(\mu+p)\sq{\mu-p+\alpha|d\p|^2+2U(\p)}.
	\end{align*}
	Therefore, from I), we need to show that
	\begin{align*}
		\mu-p+\alpha|d\p|^2+2U(\p)\geq 0;
	\end{align*}
	 since $U(\p)\geq p$ from III), we only need $\mu+\alpha|d\p|^2+U(p)\geq 0$, which is implied by II).
\end{proof}
\begin{rem}
	As observed in Remark \ref{remark: energy conditions: necessarietà di alcune ipotesi}, we prove here the necessity of
	\begin{align*}
		\mu+\frac{\alpha}{2}|d\p|^2+U(\p)\geq 0
	\end{align*}
	and
	\begin{align*}
		(m-2)\mu+mp\geq 2U(\p)
	\end{align*}
	for the respective energy conditions.
	Indeed, taking $w^i=0,\  \forall i=1,...,m$ in the expression (\ref{energy conditions: espressione per v di tipo tempo}) of $w$, we deduce
	\begin{align*}
		\p^*h(w,w)=\p^a_0\p^a_0=0,
	\end{align*}
	so that, choosing $w^i=0,\  \forall i=1,..,m$ in (\ref{energy conditions: calcolo per WEC}) and (\ref{energy conditions: calcolo per SEC}), we conclude.
\end{rem}

\section{$\varphi$-Curvatures}\label{Sect: phi-curvatures}

As we have seen in Section \ref{Subse: ded of system} of this chapter, Einstein field equations transform, in our setting, into system \eqref{Gianny}.
As we shall see, it is worth to encode some information on the non-linear field $\p$ into the curvature tensors of $(M,g)$, in order to see at the same time both the combined action of $\p$ and that of the Riemannian metric $g$ of $M$. With these motivations (others can be found in \cite{ACR}, \cite{CMR2022}, \cite{MR}) we introduce modified curvature tensors depending on the map $\p:(M,g)\to (N,h)$.

The first step in this direction, that is the definition of the $\p$-Ricci tensor, is due to B. List, that merged the Ricci flow with the harmonic map flow (for more details and background see \cite{List2008EvolutionOA}).\\
\noindent
For some fixed coupling constant $\alpha\neq 0$, we set
\begin{align}
	\ric^{\p}=\ric-\alpha \p^*h
\end{align}
for the $\p$-\emph{Ricci tensor}, and  the $\p$-\emph{scalar curvature} will be its contraction with the metric $g$, that is
\begin{align}
	S^{\p}=S-\alpha|d\p|^2.
\end{align}
In components we have
\[
R^{\p}_{ij} = R_{ij} - \alpha\p^a_i\p^a_j
\]
(note that List uses the notation $S_{ij}$ instead of $R^{\p}_{ij}$).
The next formula, the $\p$-\emph{Schur's identity} will be repeatedly used in the sequel:
\begin{align*}
	R^\p_{ij,i}=\frac{1}{2}S^\p_j-\alpha\p^a_{tt}\p^a_j
\end{align*}
(see \cite{ACR} for a simple proof). Observe that it is immediate from here to show that, for $m\geq 3$, $\Lambda$ as in \eqref{1.10} of Definition \ref{defi: Harmonic Einstein} is indeed constant. In fact, for
\begin{align*}
	\ric^\p=\Lambda g,
\end{align*}
with the usual procedure we obtain
\begin{align*}
	(m-2)\nabla \Lambda=2\alpha h\pa{\tau(\p),d\p}
\end{align*}
so that, for $m\geq 3$, $\Lambda$ is constant any time $\p$ is conservative.\\
By analogy with the classical case, we define the $\p$-\emph{Schouten tensor} by setting
\begin{align}\label{def of ohi schouten}
	A^{\p}=\ric^{\p}-\frac{S^{\p}}{2(m-1)}g,
\end{align}
while the $\p$-\emph{Cotton tensor} $C^{\p}$ is just the obstruction to $A^{\p}$ being Codazzi (see equation \eqref{phi Cotton} below). The aforementioned tensors have been used in the statements of the results reported in the Introduction; note that, with the above definitions, system \eqref{Gianny} can be written in the form (\ref{Gianny1}).
We are now going to introduce some more $\p$-curvature tensors that will be used later on and whose formal properties will speed up computations.

The $\p$-\emph{Weyl tensor} $W^{\p}$ is defined to formally respect the usual decomposition of the Riemann curvature tensor, that is,
\begin{align}\label{weyl phi}
	W^\p:=&\mathrm{Riem}-\frac{1}{m-2}A^\p\KN g\notag\\
	=&\mathrm{Riem}-\frac{1}{m-2}\ric^\p\KN g+\frac{S^\p}{2(m-1)(m-2)}g\KN g.
\end{align}
Although $W^\p$ has the same \emph{algebraic} symmetries of $\mathrm{Riem}$, it is not totally trace-free. Indeed, a computation in local orthonormal coframes (on $M$ and on $N$) shows that
\begin{align}\label{traccia di weyl phi}
	W^\p_{kikj}=\alpha(\p^*h)_{ij}=\alpha \p^a_i\p^a_j=W^\p_{ikjk},
\end{align}
(here, and in the rest of the paper, we fix the indices ranges $1\leq i, j, \ldots \leq m = \operatorname{dim}M$ and $1\leq a, b, \ldots \leq n = \operatorname{dim}N$),
while the remaining traces are identically zero.\\
Note that, in terms of the classical counterparts and of $\p$, we have
\begin{align*}
	A^\p&=A-\alpha A(\p^*h),\\
	W^\p&=W+\frac{\alpha}{m-2}A(\p^*h)\KN g,
\end{align*}
where
\begin{align*}
	A(\p^*h)=\p^*h-\frac{\abs{d\p}^2}{2(m-1)}g
\end{align*}
is the ``Schouten tensor'' of the symmetric 2-covariant tensor $\p^*h$.
For the $\p$-Cotton tensor $C^\p$, since
\begin{align}\label{phi Cotton}
	C^\p_{ijk}=A^\p_{ij,k}-A^\p_{ik,j},
\end{align}
we deduce
\begin{align*}
	C^\p_{ijk}=C_{ijk}-\alpha\sq{\p^a_{ik}\p^a_j-\p^a_{ij}\p^a_k-\frac{\p^a_t}{m-1}\pa{\p^a_{tk}\delta_{ij}-\p^a_{tj}\delta_{ik}}}.
\end{align*}
$C^\p$ has the same symmetries of  $C$: indeed,
\begin{align*}
	C^\p_{ijk}=-C^\p_{ikj} \quad\text{ and therefore }C^\p_{ijj}=0;
\end{align*}
however, it is not totally trace-free, since
\begin{align*}
	C^\p_{kki}=\alpha\p^a_{kk}\p^a_i.
\end{align*}
Note that $C^{\p}$ also satisfies the ``Bianchi identity''
\begin{align*}
	C^{\p}_{ijk}+C^{\p}_{kij}+C^{\p}_{jki}=0,
\end{align*}
as one immediately verifies.

The following alternative definition of the $\p$-Cotton tensor, for $m\geq 4$,  points out at  deep differences between the classical and the $\p$-curvatures (see Proposition 2.64 of \cite{ACR}) :
\begin{align}\label{phi weyl e phi cotton}
	-\pa{\frac{m-3}{m-2}}C^\p_{jkt}=W^\p_{sjkt,s}-\alpha\pa{\p^a_{jk}\p^a_t-\p^a_{jt}\p^a_k}-\frac{\alpha}{m-2}\p^a_{ss}\pa{\p^a_k\delta_{jt}-\p^a_t\delta_{jk}}.
\end{align}
Observe that the above formula reduces to the classical one, that is,
\begin{align*}
	W_{tijk,t}=-\pa{\frac{m-3}{m-2}}C_{ijk}
\end{align*}
in case $\p$ is constant and to
\begin{align*}
	W^{\p}_{tijk,t}=-\pa{\frac{m-3}{m-2}}C^{\p}_{ijk}
\end{align*}
for $\p$ conservative (see the discussion after Definition \ref{definition of G harmonic}) and $\p^*h$ Codazzi. In both cases, it follows immediately the validity of
\begin{align*}
	W^{\p}_{tijk,tkji}=-\pa{\frac{m-3}{m-2}}C^{\p}_{ijk,kji}.
\end{align*}
In the general case we have the following
\begin{proposition}\label{proposition su cotton tensor relative to a tensor}
	Let $(M,g)$ be a Riemannian manifold of dimension $m\geq2$ and $\p:(M,g)\ra(N,h)$ be a smooth map. Then
	\begin{align}\label{div 3 di div weyl}
		W^\p_{tijk,tkji}=-\pa{\frac{m-3}{m-2}}C^\p_{ijk,kji}+\alpha\set{R_{tijk}\p^a_j\p^a_{tk}}_i+\frac{1}{m-2}\set{\pa{R^\p_{si}+\alpha\p^b_s\p^b_i}C^\p_{tts}}_i.
	\end{align}
\end{proposition}

\begin{proof}
	We start from equation \eqref{phi weyl e phi cotton}. We set
	\begin{align*}
		E_{ijk}=\alpha \p^a_{ij}\p^a_k-\alpha \p^a_{ik}\p^a_j+\frac{\alpha}{m-2}\p^a_{tt}\pa{\p^a_j\delta_{ik}-\p^a_k\delta_{ij}};
	\end{align*}
note that $E_{ijk}=-E_{ikj}$.
	
	Then, equation (\ref{phi weyl e phi cotton}) becomes
	\begin{align*}
		W^{\p}_{tijk,t}=-\frac{m-3}{m-2}C^{\p}_{ijk}+E_{ijk}.
	\end{align*}
	Taking the divergence of the previous equation  with respect to $k$ and then with respect to $j$ we get
	\begin{align}\label{div totale di weyle: step 1}
		W^{\p}_{tijk,tkj}=-\pa{\frac{m-3}{m-2}}C^{\p}_{ijk,kj}+E_{ijk,kj}.
	\end{align}
		Using the skew symmetry of $E_{ijk}$ and the Ricci commutation relations we deduce
	\begin{align*}
		E_{ijk,kj}&=\frac{1}{2}\pa{E_{ijk,kj}-E_{ikj,kj}}=\frac{1}{2}\pa{E_{ijk,kj}-E_{ijk,jk}}\\
		&=\frac{1}{2}R_{pikj}E_{pjk}+\frac{1}{2}R_{pjkj}E_{ipk}+\frac{1}{2}R_{pkkj}E_{ijp}\\
		&=\frac{1}{2}R_{pikj}E_{pjk}+\frac{1}{2}R_{pk}E_{ipk}-\frac{1}{2}R_{pj}E_{ijp}\\
		&=\frac{1}{2}R_{pikj}E_{pjk};	
	\end{align*}
	from the definition of $E$ and the symmetries of the Riemann curvature tensor, a simple computation gives
	\begin{align*}
		E_{ijk,kj}=\alpha R_{pikj}\p^a_{pj}\p^a_k+\frac{\alpha}{m-2}\p^a_{tt}R_{ij}\p^a_j.
	\end{align*}
	Using the definition of $\ric^{\p}$ and formula
	\begin{align}
		C^{\p}_{ttk}=\alpha \p^a_{tt}\p^a_k
	\end{align}
	we obtain
	\begin{align*}
		E_{ijk,kj}&=\alpha R_{pikj}\p^a_{pj}\p^a_k+\frac{\alpha}{m-2}\p^a_{tt}\pa{R^{\p}_{ij}\p^a_j+\alpha \p^b_i\p^b_j\p^a_j}\\
		&=\alpha R_{pikj}\p^a_{pj}\p^a_k+\frac{C^{\p}_{ppj}}{m-2}\p^a_{tt}\pa{R^{\p}_{ij}+\alpha \p^a_i\p^a_j}.
	\end{align*}
	Substituting into (\ref{div totale di weyle: step 1}) we get
	\begin{align*}
		W^\p_{tijk,tkj}=-\pa{\frac{m-3}{m-2}}C^\p_{ijk,kj}+\alpha R_{tijk}\p^a_j\p^a_{tk}+\frac{1}{m-2}\pa{R^\p_{si}+\alpha\p^b_s\p^b_i}C^\p_{tts}.
	\end{align*}
	Computing the divergence with respect to the index $i$  we obtain (\ref{div 3 di div weyl}).

\end{proof}

Other $\p$-curvatures,  for instance the $\p$-\emph{Bach tensor}, are \emph{not} defined in analogy with the classical ones; indeed, for $m\geq 3$, we have
\begin{align}\label{1.20Bachphi}
	(m-2)B_{ij}^\p=&C^\p_{ijk,k}+R^\p_{tk}W^\p_{tikj}-\alpha R^\p_{tk}\p^a_t\p^a_i\delta_{jk}\notag\\
	&+\alpha\pa{\p^a_{ij}\p^a_{kk}-\p^a_{kkj}\p^a_i-\frac{\abs{\tau(\p)}^2}{m-2}\delta_{ij}}.
\end{align}
Note that $B^\p$ is a symmetric tensor, see \cite{ACR}, but it is not trace free; indeed,
\begin{align}\label{1.21Bachphitraccia}
	B^\p_{ii}=\alpha\frac{(m-4)}{(m-2)^2}\abs{\tau(\p)}^2.
\end{align}	
We refer to \cite{ACR} for a more detailed discussion.
The next result will be needed   in Chapter \ref{Other rigidity results} for the proof of Theorem \ref{other rigidity results: the boundary case: analogo thm 1.44}.

\begin{lemma}\label{other rigidity results: the boundary case: lemma con div bach}
	Let $(M,g)$ be a manifold of dimension $m\geq 3$ and $\p:(M,g)\ra(N,h)$  a smooth map. Then
	\begin{align}\label{other rigidity results: the boundary case: divergenza di bach}
		(m-2)B^\p_{ik,k}=&\pa{\frac{m-4}{m-2}}\sq{R^\p_{jk}C^\p_{jki}+\alpha\p^a_{tt}\pa{\p^a_{kki}+R^\p_{ij}\p^a_j}}\\
		&+\alpha\p^a_i\sq{\frac{m}{(m-1)(m-2)}\p^a_{tt}S^\p+2\alpha\p^b_{tt}\p^b_j\p^a_j}\notag\\
		&-\frac{\alpha}{2}\pa{\frac{m-2}{m-1}}\p^a_i\p^a_jS^\p_j-2\alpha\p^a_{jk}R^\p_{jk}\p^a_i\notag\\
		&-\alpha\p^a_i\tau_2^a(\p)\notag,
	\end{align}
	where
	\begin{align*}
		\tau_2^a(\p)=\p^a_{ttss}-{}^N\!R^a_{bcd}\p^b_s\p^c_s\p^d_{tt}
	\end{align*}
	are the components of the \emph{bi-tension field} of the map $\p$ and ${}^N\!R^a_{bcd}$ are the components of the curvature tensor of $N$.
\end{lemma}
\begin{rem}
	We recall that
	\begin{align*}
		\tau_2(\p)=0
	\end{align*}
	is the Euler-Lagrange equation of the \emph{bi-energy functional} $E_{\tau}^{\p}$ given on a relatively compact domain $\Omega$ in $M$ by the prescription
	\begin{align*}
		E^{\p}_{\tau} (\Omega)=\frac{1}{2}\int_{\Omega}|\tau (\p)|^2.
	\end{align*}
\end{rem}
\begin{rem}
	Lemma \ref{other rigidity results: the boundary case: lemma con div bach} was first proved in \cite{Anselli_2021}, but note that the two formulas are slightly different due to some minor typos that are present in \cite{Anselli_2021}.
	The proof that we give here for completeness is essentially the same as the one presented there, with some slight modifications.
\end{rem}
\begin{proof}
	Define the tensors of components
	\begin{align*}
		&\mathcal{L}_{ij}=\p^a_{tt}\p^a_{ij}-\frac{1}{m-2}|\tau (\p)|^2\delta_{ij},\\
		&\mathcal{M}_{ij}=R^{\p}_{tk}W^{\p}_{tikj},\\
		&\mathcal{N}_{ij}=C^{\p}_{ijk,k}-\alpha \p^a_i\pa{\p^a_{kkj}+R^{\p}_{jk}\p^a_k}.
	\end{align*}
	Then, using (\ref{1.20Bachphi}), it is immediate to see that
	\begin{align*}
		(m-2)B^{\p}_{ij}=\alpha \mathcal{L}_{ij}+\mathcal{M}_{ij}+\mathcal{N}_{ij}.
	\end{align*}
	We compute separately the divergences of these three tensors.
	First, recall the validity of the following commutation formula (see Section 1.7 of \cite{AMR} for a proof):
	\begin{align*}
		\p^a_{ijk}=\p^a_{ikj}+\p^a_tR^t_{ijk}-{}^{N}\!R^a_{bcd}\p^b_i\p^c_j\p^d_k.
	\end{align*}
	We have
	\begin{align*}
		\mathcal{L}_{ij,j}&=\p^a_{ttj}\p^a_{ij}+\p^a_{tt}\p^a_{ijj}-\frac{2}{m-2}\p^a_{tti}\p^a_{pp}\\
		&=\p^a_{ttj}\p^a_{ij}+\pa{\frac{m-4}{m-2}}\p^a_{tt}\p^a_{jji}+\p^a_{pp}\p^a_tR_{ti}-{}^N\!R^a_{bcd}\p^b_j\p^c_i\p^d_j\p^a_{tt}.
	\end{align*}
	Using the definition of $\ric^{\p}$ we deduce
	\begin{align}\label{diver di bach: tensore L}
		\mathcal{L}_{ij,j}=&\p^a_{ttj}\p^a_{ij}+\pa{\frac{m-4}{m-2}}\p^a_{tt}\p^a_{jji}+\p^a_{pp}\p^a_tR^{\p}_{ti}+\alpha\p^a_{pp}\p^a_t\p^b_{t}\p^b_i\\
		&-{}^N\!R^a_{bcd}\p^b_j\p^c_i\p^d_j\p^a_{tt}. \nonumber
	\end{align}
	Next, for the tensor $\mathcal{M}$, we use the symmetries of $W^{\p}$ and formula (\ref{phi weyl e phi cotton}) to get
	\begin{align*}
		\mathcal{M}_{ij,j}=&R^{\p}_{tk,j}W^{\p}_{tikj}+R^{\p}_{tk}W^{\p}_{tikj,j}\\
		=&\frac{1}{2}\pa{R^{\p}_{tk,j}-R^{\p}_{tj,k}}W^{\p}_{tikj}+\pa{\frac{m-3}{m-2}}R^{\p}_{tk}C^{\p}_{tki}+\alpha R^{\p}_{tk}\p^a_{ki}\p^a_t\\
		&-\alpha R^{\p}_{tk}\p^a_{tk}\p^a_i+\frac{\alpha}{m-2}\p^a_{tt}\pa{S^{\p}\p^a_i-R^{\p}_{pi}\p^a_p}.
	\end{align*}
	Using the definition of $C^{\p}$ and equation (\ref{traccia di weyl phi}) we obtain
	\begin{align}\label{diver di bach: tensore M}
		M_{ij,j}=&\frac{1}{2}C^{\p}_{tkj}W^{\p}_{tikj}+\frac{\alpha}{2(m-1)}S^{\p}_j\p^a_j\p^a_i+\pa{\frac{m-3}{m-2}}R^{\p}_{tk}C^{\p}_{tki}\\
		&+\alpha R^{\p}_{tk}\p^a_{ki}\p^a_t
		-\alpha R^{\p}_{tk}\p^a_{tk}\p^a_i+\frac{\alpha}{m-2}\p^a_{tt}\pa{S^{\p}\p^a_i-R^{\p}_{pi}\p^a_p}. \nonumber
	\end{align}
	For the tensor $\mathcal{N}$, we use the $\p$-Schur identity to compute
	\begin{align}\label{diver di bach: tensore N parte 1}
		\mathcal{N}_{ij,j}=&C^{\p}_{ijk,kj}-\alpha \p^a_{ij}\pa{\p^a_{ttj}+R^{\p}_{jk}\p^a_k}\\
		&-\alpha \p^a_i\pa{\p^a_{ttjj}+\frac{1}{2}S^{\p}_t\p^a_t-\alpha\p^b_{tt}\p^b_j\p^a_j+R^{\p}_{jk}\p^a_{jk}}. \nonumber
	\end{align}
	From the Ricci commutation relations and the symmetries of $C^{\p}$ we get
	\begin{align*}
		C^{\p}_{ijk,kj}&=\frac{1}{2}\pa{C^{\p}_{ijk,kj}-C^{\p}_{ikj,kj}}=\frac{1}{2}\pa{C^{\p}_{ijk,kj}-C^{\p}_{ijk,jk}}\\
		&=\frac{1}{2}\pa{R_{pikj}C^{\p}_{pjk}+R_{pjkj}C^{\p}_{ipk}+R_{pkkj}C^{\p}_{ijp}}\\
		&=\frac{1}{2}\pa{R_{pikj}C^{\p}_{pjk}+R_{pk}C^{\p}_{ipk}-R_{pj}C^{\p}_{ijp}}\\
		&=\frac{1}{2}R_{pikj}C^{\p}_{pjk}.
	\end{align*}
	Using the definition of $W^{\p}$ we deduce
	\begin{align*}
		C^{\p}_{ijk,kj}=&\frac{1}{2}W^{\p}_{pikj}C^{\p}_{pjk}+\frac{1}{m-2}R^{\p}_{pk}C^{\p}_{pik}-\frac{\alpha}{m-2}R^{\p}_{ij}\p^a_{tt}\p^a_j\\
		&+\frac{\alpha S^{\p}}{(m-1)(m-2)}\p^a_{tt}\p^a_i.
	\end{align*}
	Inserting this information into (\ref{diver di bach: tensore N parte 1}) we get
	\begin{align}\label{diver di bach: tensore N parte 2}
		\mathcal{N}_{ij,j}=&\frac{1}{2}W^{\p}_{pikj}C^{\p}_{pjk}+\frac{1}{m-2}R^{\p}_{pk}C^{\p}_{pik}-\frac{\alpha}{m-2}R^{\p}_{ij}\p^a_{tt}\p^a_j\\ \nonumber
		&+\frac{\alpha S^{\p}}{(m-1)(m-2)}\p^a_{tt}\p^a_i-\alpha \p^a_{ij}\pa{\p^a_{ttj}+R^{\p}_{jk}\p^a_k}\\ \nonumber
		&-\alpha \p^a_i\pa{\p^a_{ttjj}+\frac{1}{2}S^{\p}_t\p^a_t-\alpha\p^b_{tt}\p^b_j\p^a_j+R^{\p}_{jk}\p^a_{jk}}. \nonumber
	\end{align}
	Putting together equations (\ref{diver di bach: tensore L}), (\ref{diver di bach: tensore M}) and (\ref{diver di bach: tensore N parte 2}) we obtain equation \eqref{other rigidity results: the boundary case: divergenza di bach}.

\end{proof}

\section{Perspectives on a more general System}
Through this section, we are going to study system \eqref{Gianny} in a more general version that, in some special cases, can be derived from different mathematical settings;
in particular, we focus on a generalization of the system
\begin{align}\label{system_sec2.4}
	\begin{cases}
		i)\, u\ric^\p-\hs(u)=\frac{1}{m}\pa{uS^\p-\Delta u}g,\\
		ii)\,u\tau(\p)=-d\p(\nabla u)+\frac{1}{\alpha}(\nabla U)(\p)u;
	\end{cases}
\end{align}
note that the first equation of \eqref{system_sec2.4} is \eqref{prima di i}, while the second is the third equation of \eqref{Gianny}; we consider system \eqref{system_sec2.4} on a Riemannian manifold $(M,g)$ with empty boundary, so that $u>0$ on $M$. Setting
\begin{align*}
	f=-\log u,
\end{align*}
the system reduces to
\begin{align}\label{system nostro sezione su sistema generale}
	\begin{cases}
		i)\,\ric^\p+\hs(f)-df\otimes df=\lambda g,\\
		ii)\,\tau(\p)=d\p(\nabla f)+\frac{1}{\alpha}(\nabla U)(\p),
	\end{cases}
\end{align}
where $m\lambda=S^{\p}+\Delta f-|\nabla f|^2$. As we have already seen in Section \ref{Subse: ded of system}, these equations can be obtained decomposing the stress-energy tensor $\hat{T}$ into its spatial and temporal parts, in addition to the energy-momentum conservation. Introducing a parameter $\eta\in\erre$, we consider
\begin{align}\label{sistema generale in sezione sistema genrale}
	\begin{cases}
		i)\,\ric^\p+\hs(f)-\eta df\otimes df=\lambda g,\\
		ii)\, \tau(\p)=d\p(\nabla f)+\frac{1}{\alpha}(\nabla U)(\p),
	\end{cases}
\end{align}
where  $\lambda\in\cinf$. In the following subsections, we will focus on the different contexts in which some special versions of this generalized system emerge.
\subsection{An Euler-Lagrange equation}
In this subsection we show how a particular case of system \eqref{sistema generale in sezione sistema genrale} (i.e., the one with $\eta=1$) appears as Euler-Lagrange equations  of the functional
\begin{align}\label{funzionale di azione}
	\mathcal{F}=\mathcal{F}(\hat{g},\hat{\p})=\int_{\Omega}\pa{\hat{S}^{\hat{\p}}-2U(\hat{\p})} dV_{\hat{g}},
\end{align}
where $\Omega \subset \hat{M}$ is a relatively compact domain  and $\hat{S}^{\hat{\p}}$ is the $\hat{\p}$-scalar curvature with respect to $\hat{g}$.
We now compute the variation of $\mathcal{F}$, first with respect to the metric $\hat{g}$ and then with respect to $\hat{\p}$.
For a symmetric 2-covariant tensor $\psi$ on $\hat{M}$ and $t\in (-\eps, \eps)$, with $\eps>0$ sufficiently small, define
\begin{align*}
	\hat{g}_t=\hat{g}+t\psi.
\end{align*}
Then we have
\begin{align*}
	\frac{d}{dt}\mathcal{F}(\hat{g}_t,\hat{\p}){\Mid_{t=0}}=\int_{\Omega} \frac{d}{dt}\pa{\hat{S}^{\hat{\p}}_t-2U(\hat{\p})}{\Mid_{t=0}}dV_{\hat{g}}+\int_{\Omega}\sq{\pa{\hat{S}^{\hat{\p}}_t-2U(\hat{\p})}\frac{d}{dt}(dV_{\hat{g}_t}){\Mid_{t=0}}},
\end{align*}
where $\hat{S}^{\hat{\p}}_t$ is the $\hat{\p}$-Scalar curvature with respect to $\hat{g}_t$.
Recall that
\begin{align*}
	\hat{S}^{\hat{\p}}=\hat{S}-\alpha |d\hat{\p}|^2_{\hat{g}},
\end{align*}
where
\begin{align*}
	|d\hat{\p}|^2_{\hat{g}}=g^{\alpha \beta}\hat{\p}^a_{\alpha}\hat{\p}^a_{\beta}
\end{align*}
so that
\begin{align*}
	\frac{d}{dt}\pa{|d\hat{\p}|^2_{\hat{g}_t}}{\Mid_{t=0}}=-\psi_{\alpha \beta}\hat{\p}^a_{\mu} \hat{\p}^a_{\nu}g^{\alpha\mu}g^{\beta\nu}
\end{align*}
(for a proof see e.g. \cite{CMBook}, Chapter 2, and \cite{Choquet-Bruhat:2009xil});
moreover, we have the validity of the following formulas:
\begin{align*}
	\frac{d}{dt} (dV_{\hat{g}_t}){\Mid_{t=0}}=\frac{1}{2}\Psi dV_{\hat{g}},
\end{align*}
where
\begin{align*}
	\Psi=g^{\alpha\beta}k_{\alpha\beta}
\end{align*}
and
\begin{align*}
	\frac{d}{dt}(\hat{S}_t){\Mid_{t=0}}
	&=-k\psi_{\mu\beta,\nu\gamma}g^{\mu\beta}g^{\nu\gamma}+\psi_{\mu\nu,\beta\gamma}g^{\mu\beta}g^{\nu\gamma}-\hat{R}_{\mu\eta}g^{\mu \gamma}g^{\eta \beta}\psi_{\gamma \beta}.
\end{align*}
Therefore we have
\begin{align*}
	\frac{d}{dt}\mathcal{F}(\hat{g}_t,\hat{\p}){\Mid_{t=0}}&=\int_{\Omega}\Big(-\psi_{\mu\beta,\nu\gamma}g^{\mu\beta}g^{\nu\gamma}+\psi_{\mu\nu,\beta\gamma}g^{\mu\beta}g^{\nu\gamma}-\hat{R}_{\mu\eta}g^{\mu \gamma}g^{\eta \beta}\psi_{\gamma \beta}\\
	&\quad\quad+\alpha \psi_{\mu\nu}\hat{\p}^a_{\beta}\hat{\p}^a_{\gamma}g^{\mu\beta}g^{\nu\gamma}\Big)dV_{\hat{g}}\\
	&+\int_{\Omega}\pa{\frac{1}{2}\hat{S}\Psi-U(\hat{\p})\Psi-\frac{1}{2}\alpha |d\hat{\p}|^2_{\hat g}\Psi}dV_{\hat{g}},
\end{align*}
and using the divergence theorem we get
\begin{align*}
	\frac{d}{dt}\mathcal{F}(\hat{g}_t,\hat{\p}){\Mid_{t=0}}=\int_{\Omega}\pa{-\hat{R}^{\hat{\p}}_{\mu\eta}+\frac{1}{2}\hat{S}^{\hat{\p}}g_{\mu \eta}-U(\hat{\p})g_{\mu\eta}}g^{\mu \gamma}g^{\eta \beta}\psi_{\gamma\beta}dV_{\hat{g}}.
\end{align*}
Hence, the first Euler-Lagrange equation of the functional $\mathcal{F}$ is
\begin{align*}
	\hat{\ric}^{\hat{\p}}-\frac{1}{2}\hat{S}^{\hat{\p}}\hat{g}=-U(\hat{\p})\hat{g}.
\end{align*}

Now we take the variation of  $\mathcal{F}$ with respect to $\hat{\p}$. Following \cite{ES1964},
let $V$ be a  vector field along $\hat{\p}$, that is $\pi\pa{V(q)} = \hat{\p}(q)$ for all $q\in \hat{M}$ (where $\pi : TN \ra N$ is the canonical projection), and define $\hat{\p}_t(x)=\exp_{\hat{\p}(x)} (tV)$ for $t\in (-\eps, \eps)$, with $\eps>0$ sufficiently small. 
Then the variation of $\mathcal{\p}$ with respect to $V$ is given by
\begin{align*}
	\frac{d}{dt}\mathcal{F}(\hat{g}, \hat{\p}_t){\Mid_{t=0}}=\int_{\Omega}\frac{d}{dt}\pa{-\alpha |d\hat{\p}_t|^2_{\hat{g}}-2U(\hat{\p}_t)}{\Mid_{t=0}}dV_{\hat{g}}.
\end{align*}
By Lemma B of section 2 of \cite{ES1964}, which holds also in Lorentzian signature, as one can easily verify,
we get
\begin{align*}
	\int_{\Omega} -\alpha \frac{d}{dt}(|d\hat{\p}_t|^2_{\hat{g}}){\Mid_{t=0}}dV_{\hat{g}}=2\alpha \int_{\Omega} \tau(\hat{\p})^aV^a dV_{\hat g},
\end{align*}
while a simple computation shows that
\begin{align*}
	\frac{d}{dt} U(\hat{\p}_t)_{|_{t=0}}=(U^a\circ \hat{\p})V^a.
\end{align*}
Therefore we get
\begin{align*}
	\frac{d}{dt}\mathcal{F}(\hat{g},\hat{\p}_t){\Mid_{t=0}}=2\int_{\Omega} \pa{\alpha\tau(\hat{\p})^a-(U^a\circ \hat{\p})}V^a dV_{\hat{g}}.
\end{align*}
Hence, the second Euler-Lagrange equation of the functional $\mathcal{F}$ is given by
\begin{align*}
	\tau(\hat{\p})=\frac{1}{\alpha} \nabla U (\hat{\p}),
\end{align*}
and therefore the critical points of $\mathcal{F}$ satisfy the system
\begin{align}\label{variazione di F: punti critic}
	\begin{cases}
		&-\hat{\ric}+\frac{1}{2}\hat{S}\hat{g}-U(\hat{\p})\hat{g}+\alpha \hat{\p}^*h-\frac{\alpha}{2}|d\hat \p|_{\hat{g}}\hat{g}=0,\\
		&\alpha \tau(\hat{\p})=\nabla U (\hat{\p}).
	\end{cases}
\end{align}
Let us now consider manifolds $\hat{M}$ of the form $\hat{M}=M\times_{e^{-2f}} \erre$, with a warped product metric
\begin{align*}
	\hat{g}=g-e^{-2f}dt\otimes dt,
\end{align*}
where $g$ is the lifting on $\hat{M}$ of a Riemannian metric on $M$, $f\in C^{\infty} (M)$ and $t:\hat{M}\to \erre$ denotes the projection.
To perform computations, we will consider, at a fixed point $p\in \hat{M}$, a local orthonormal frame $\{e_i\}_{i=0,..,m}$ such that
\begin{align*}
	e_0=e^{f}\frac{\partial}{\partial t},
\end{align*}
while $e_1,...,e_m$ span the tangent space of $M$.
Moreover, we will assume that $\hat{\p}$ is the lift on $\hat{M}$ of a  map $\p:M\to N$, so that, as we have seen in previous computations,
\begin{align*}
	\hat{\p}^a_0=0;
\end{align*}
note that this also implies, as we have noted before, that
\begin{align*}
	|d\hat{\p}|^2_{\hat{g}}=|d\p|^2_g.
\end{align*}

If we let $\ric$ and $S$ be the lifts on $\hat{M}$ of the Ricci tensor and the scalar curvature of $(M,g)$, respectively, then we have the following expressions:
\begin{align}\label{varizione: ricci hat e ricci}
	\hat{\ric}=\ric+\hess(f)- df\otimes df-\pa{\Delta f- |\nabla f|^2}e^{-2f}dt\otimes dt
\end{align}
and
\begin{align*}
	\hat{S}=S+2\Delta f-2|\nabla f|^2.
\end{align*}
Moreover,
\begin{align*}
	\hat{\p}^a_{00}=\p^a_if_i, \ \ \ \hat{\p}^a_{0i}=0, \ \ \ \hat{\p}^a_{ij}=\p^a_{ij}.
\end{align*}

The second equation of (\ref{variazione di F: punti critic}) becomes
\begin{align*}
	0=\alpha \tau(\hat{\p})-\nabla U (\hat{\p})=\alpha \tau(\p)-\alpha d\p(\nabla f)-\nabla U(\p),
\end{align*}
and therefore
\begin{align*}
	\tau(\p)=d\p(\nabla f)+\frac{1}{\alpha } \nabla U (\p).
\end{align*}
Furthermore, contracting the first equation of (\ref{variazione di F: punti critic}) with $\hat{g}$ we get
\begin{align*}
	-\hat{S}+\frac{1}{2}(m+1)\hat{S}-(m+1)U(\hat{\p})+\alpha |d\hat{\p}|^2-\frac{\alpha}{2}(m+1)|d\hat{\p}|^2=0,
\end{align*}
so that
\begin{align*}
	\frac{1}{2}(m-1)\hat{S}=(m+1)U(\hat{\p})+\frac{\alpha}{2}\pa{m-1}|d\hat{\p}|^2.
\end{align*}
Inserting this information into the first equation of (\ref{variazione di F: punti critic}) we have
\begin{align}\label{variazione:1st EL}
	-\hat{\ric}+\frac{m+1}{m-1}U(\hat{\p})\hat{g}-U(\hat{\p})\hat{g}+\frac{\alpha}{2}|d\hat{\p}|^2\hat{g}+\alpha \p^*h-\frac{\alpha}{2}|d\hat{\p}|^2\hat{g}=0;
\end{align}
rearranging terms and simplifying, we get
\begin{align*}
	\hat{\ric}^{\p}=\frac{2}{m-1}U(\hat{\p})\hat{g}.
\end{align*}
Using equation \eqref{varizione: ricci hat e ricci} into \eqref{variazione:1st EL}, we deduce
\begin{align*}
	-\ric^\p-\hess(f)+&df\otimes df+(\Delta f- |\nabla f|^2)e^{-2f}dt\otimes dt+\frac{2}{m-1}U(\hat{\p})g\\
	&-\frac{2}{m-1}U(\hat{\p})e^{-2f}dt\otimes dt=0
\end{align*}
Splitting the spatial and temporal part we deduce
\begin{align*}
	\begin{cases}
		&\ric^{\p}+\hess(f)- df\otimes df=\frac{2}{m-1}U(\p)g,\\
		&\Delta f-|\nabla f|^2=\frac{2}{m-1}U(\p).
	\end{cases}
\end{align*}
In conclusion, we have derived the system
\begin{align*}
	\begin{cases}
		&\ric^{\p}+\hess(f)- df\otimes df=\frac{2}{m-1}U(\p)g,\\
		&\Delta f-|\nabla f|^2=\frac{2}{m-1}U(\p),\\
		&\tau(\p)=d\p(\nabla f)+\frac{1}{\alpha}(\nabla U)(\p).
	\end{cases}
\end{align*}

\subsection{On  some solutions of Ricci-Harmonic flow}
We now recall the definition of the \emph{Ricci-Harmonic flow}, introduced by List in \cite{List2008EvolutionOA}.
This is given by a solution $(g(t),\p(t))_{t\in[0,\eps)}$ of the evolution system
\begin{align}
	\begin{cases}\label{Ricci-Harmonic flow}
		&\frac{d}{dt} g(t)=-2\ric_{g(t)}+2\alpha\p(t)^*h,\\
		& \frac{d}{dt} \p(t)=\tau_{g(t)}(\p(t)),
	\end{cases}
\end{align}
for a family of Riemannian metrics $\{g(t)\}$ on $M$, smooth maps $\p(t): (M,g)\to (N,h)$ and a constant $\alpha$.
The self-similar solutions of this flow
are  solutions whose metric evolves by diffeomorphisms and rescaling, while the map $\p$ evolves by diffeomorphisms.
In other  terms,  there exists a one-parameter family $\{F_t\}_{t\in [0,\eps)}$ of diffeomorphisms of $M$ such that $F_0=\id_M$ and a positive function $c:[0,\eps)\to \erre $ such that $c(0)=1$ for which we have
\begin{align}
	\begin{cases}
		& g(t)=c(t)F_t^* g(0),\\
		& \p(t)=F_t^*\p(0).
	\end{cases}
\end{align}


In the next Proposition we will show how solutions of \eqref{sistema generale in sezione sistema genrale} with $\eta=0$, $\lambda$ constant and non-constant $U$ can also be characterized in terms of special solutions of (\ref{Ricci-Harmonic flow}).
\begin{proposition}
	For a smooth function $U\in C^{\infty}(N)$, let $\{G_t\}_{t\in [0,\eps)}$ be a one-parameter family of diffeomorphisms of $N$ such that $G_0=\id_N$ and ${\frac{d}{dt} G_{t}}{\mid_{t=0}} = \frac{1}{\alpha}\nabla U $. 
	Let $\{F_t\}_{t\in[0,\eps)}$ be a one-parameter family of diffeomorphisms of $M$ such that $F_0=\id_M$ and $\frac{d}{dt} F_{t}{\mid_{t=0}}= \nabla f$. Consider a solution of (\ref{Ricci-Harmonic flow}) of the form
	\begin{align}\label{Ricci harmonic solitons generalizzati caso dinamico}
		\begin{cases}
			& g(t)=c(t)F_t^* g(0),\\
			& \p(t)=G_t(F_t^* \p(0)),
		\end{cases}
	\end{align}
	where $c:[0,\eps)\to \erre$ is a positive function satisfying $c(0)=1$.
	Then, $(M,g(0),\p(0))$ is a solution of \eqref{sistema generale in sezione sistema genrale} with $\eta=0$ and $\lambda$ constant.
\end{proposition}
\begin{proof}
	Given a solution of (\ref{Ricci-Harmonic flow}) satisfying (\ref{Ricci harmonic solitons generalizzati caso dinamico}) we deduce
	\begin{align*}
		-2\ric_{g(0)}+2\alpha\p(0)^*h&= \frac{d}{dt} g(t){\Mid_{t=0}}\\
		&=\frac{d}{dt}\pa{c(t)F_t^*g(0)}{\Mid_{t=0}}=\dot{c}(0)g(0)+c(0)\mathcal{L}_{\nabla f} g(0)\\
		&=\dot{c}(0)g(0)+2c(0)\hess(f).
	\end{align*}
	Setting $\dot{c}(0)=-2\lambda$ we obtain the first equation of \eqref{sistema generale in sezione sistema genrale} with $\eta=0$.
	For the second equation, we compute
	\begin{align*}
		\tau_{g(0)}(\p(0))&=\frac{d}{dt}\p(t)_{|_{t=0}}\\
		&=\frac{d}{dt}\pa{G_t(F_t^*\p(0))}{\Mid_{t=0}}\\
		&=\pa{\frac{d}{dt} G_t}_{|_{t=0}} \pa{F_0^*\p(0)}+G_0\pa{\frac{d}{dt}\pa{F_t^*\p(0)}{\Mid_{t=0}}}\\
		&=\frac{\pa{\nabla U}}{\alpha}\pa{\p(0)}+d\p(0)(\nabla f).
	\end{align*}
\end{proof}
\begin{rem}
	If in (\ref{Ricci harmonic solitons generalizzati caso dinamico}) we let $c(t)$ be a function $c(t,x)$ of both time and space, we obtain that at time $t=0$ the corresponding solution of \eqref{sistema generale in sezione sistema genrale} does not need to have $\lambda$ constant. This has been first observed in the case of Ricci solitons by Gomes, Wang and Xia in \cite{gomes2015halmostriccisoliton}.
\end{rem}
\begin{rem}
	Unfortunately, to the best of our knowledge, it is not known if the converse statement of the above Proposition is true: in other words, a solution of \eqref{sistema generale in sezione sistema genrale} with $\eta=0$ and $\lambda$ constant might not be enough to construct a solution of (\ref{Ricci harmonic solitons generalizzati caso dinamico}).
\end{rem}
\subsection{Harmonic-Einstein warped products}
Under suitable assumptions, Riemannian manifolds satisfying system \eqref{sistema generale in sezione sistema genrale}, that is
\begin{align*}
	\begin{cases}
		i)\,\ric^\p+\hs(f)-\eta df \otimes df=\lambda g,\\
		ii)\,\tau(\p)=d\p(\nabla f)+\frac{1}{\alpha}(\nabla U)(\p),
	\end{cases}
\end{align*}
arise as warping factor of a Riemaniann manifold that is $\frac{1}{\alpha} U$-harmonic Einstein, where $\alpha\in \erre\setminus\set{0}$.\\
Let $(M,g)$ and $(F,g_F)$ be two Riemannian manifolds of dimensions $m$ and $d$, respectively, and let $f\in C^\infty(M)$; we denote by $\ol{M}=M\times_{e^{-f/d}}F$ the warped product manifold endowed with the metric
\begin{align*}
	\ol{g}=g+e^{-2f/d}g_F.
\end{align*}
Then, we have the validity of the following
\begin{proposition}\cite[Corollary 4.13]{Anselli_2021}
	In the notation above, let $\ol{\ric}$ and $\ol{S}$ be the Ricci tensor and the scalar curvature of $(\ol{M},\ol{g})$, respectively; let $\ric$, ${}^F\ric$  and $S$, ${}^FS$ denote the lifts to $\ol{M}$ of the Ricci tensors and the scalar curvatures of $(M,g)$ and $(F,g_F)$, respectively. Then, the non-vanishing components of $\ol{\ric}$ are given by
	\begin{align}\label{components of ricci M barrato}
		\ol{R}_{ij}=R_{ij}+f_{ij}-\frac{1}{d}f_if_j, &&\ol{R}_{(A+m)(B+m)}=\frac{\Delta_ff}{d}(g_F)_{AB}+e^{2f/d}\,{}^FR_{AB},
	\end{align}
	where $i,j=1,...,m$, $A,B=1,...,d$ and $\Delta_f f$ is the $f$-laplacian of $f$ in the metric $g$, that is,
	\begin{align*}
		\Delta_f f=\Delta f-\abs{\nabla f}^2.
	\end{align*}	
\end{proposition}
Let $\p:(M,g)\ra (N,h)$ be a smooth map form $M$ to a second Riemannian manifold of dimension $n$. We denote $\ol{\p}:(\ol{M},\ol{g})\ra (N,h)$ the smooth map defined as
\begin{align*}
	\ol{\p}:=\p\circ \pi_M,
\end{align*}
where $\pi_M:\ol{M}\ra M$ is the projection on the first factor. Then (see \cite[Proposition 4.17]{Anselli_2021} for a proof) we have
\begin{align}\label{components of phi barrato}
	\ol{\p}^a_i=\p^a_i, \quad\ol{\p}^a_{A+m}=0,\quad\tau(\ol{\p})=\tau(\p)-d\p(\nabla f),
\end{align}
where $\tau(\ol{\p})$ is the tension field of $\ol{\p}$.

\begin{theorem}
	Let $(M,g)$ and $(F,g_F)$ be Riemannian manifolds of dimension $m$ and $d$ respectively, with $d\geq 3$. Let $f\in C^{\infty}(M)$,  $\p:(M,g)\ra(N,h)$ and $U:(N,h)\ra \erre$ be smooth maps. Consider the warped product manifold $(\ol{M},\ol{g})=(M\times_{e^{-f/d}}F, g+e^{-2f/d}g_F)$ and let $\ol{\p}:(\ol{M},\ol{g})\ra (N,h)$ be as above. Then $(\ol{M},\ol{g})$ is $\frac{1}{\alpha}U$-harmonic-Einstein, $\alpha\in\erre\setminus\set{0}$, i.e. there exists $\lambda\in \erre$ such that
	\begin{align}\label{M barrato HE with pot}
		\begin{cases}
			\ol{\ric}^{\ol{\p}}=\lambda \ol{g},\\
			\tau(\ol{\p})=\frac{1}{\alpha}(\nabla U)(\p),
		\end{cases}
	\end{align}
	if and only if $(M,g)$ satisfies \eqref{sistema generale in sezione sistema genrale} with $\eta=\frac{1}{d}$ and $(F,g_F)$ is Einstein, with Einstein constant $\Lambda$ satisfying
	\begin{align}\label{def of E const per F}
		\Lambda-e^{-2f/d}\lambda+e^{-2f/d}\frac{\Delta_ff}{d}=0.
	\end{align}
\end{theorem}
\begin{proof}
	Assume that $(\ol{M},\ol{g})$ satisfies \eqref{M barrato HE with pot}: then, by \eqref{components of ricci M barrato} and \eqref{components of phi barrato} we get
	\begin{align*}
		\lambda(g+e^{-2f/d}g_F)&=\ol{\ric}^\p\\
		&=\ric+\hs(f)-\frac{1}{d}df\otimes df+\frac{\Delta_f f}{d}e^{-2f/d}g_F+\,{}^F\ric-\alpha\p^*h,
	\end{align*}
	from which we deduce
	\begin{align*}
		&\ric^\p+\hs(f)-\frac{1}{d}df\otimes df=\lambda g,\\
		&{}^F\ric=e^{-2f/d}\pa{\lambda -\frac{\Delta_ff}{d}}g_F=\Lambda g_F
	\end{align*}
	and by \eqref{components of phi barrato} we have
	\begin{align*}
		\frac{1}{\alpha}(\nabla U)(\p)=\tau(\ol{\p})=\tau(\p)-d\p(\nabla f).
	\end{align*}
	Conversely, let us assume that $(M,g)$ satisfies \eqref{sistema generale in sezione sistema genrale} and that $(F,g_F)$ is Einstein, with Einstein constant $\Lambda$ satisfying \eqref{def of E const per F}. Then, by \eqref{components of ricci M barrato} we deduce
	\begin{align}\label{ricci M barrato da eq. sist e hp}
		\ol{\ric}=\ric+\hs(f)-\frac{1}{d}df\otimes df+\frac{\Delta_f f}{d}g_F+{}^F\ric.
	\end{align}
	Since \eqref{sistema generale in sezione sistema genrale} holds, we have
	\begin{align}\label{ricci since eq gen holds}
		\ric=\alpha\p^*h-\hs(f)+\frac{1}{d}df\otimes df+\lambda g;
	\end{align}
	moreover, by the definition of $\Lambda$, we get
	\begin{align}\label{ricci F}
		{}^F\ric=\Lambda g_F=e^{-2f/d}\pa{\frac{\Delta_f f}{d}+\lambda}g_F.
	\end{align}
	Inserting \eqref{ricci since eq gen holds} and \eqref{ricci F} into \eqref{ricci M barrato da eq. sist e hp}, we obtain
	\begin{align*}
		\ol{\ric}=\alpha\p^*h+\lambda(g+e^{-2f/d}g_F);
	\end{align*}
	hence, \eqref{M barrato HE with pot} follows by \eqref{components of phi barrato}.
\end{proof}

\subsection{Conformally harmonic-Einsten}

When $\eta=-\frac{1}{m-2}$ and $U\equiv 0$, the structure \eqref{sistema generale in sezione sistema genrale} can be obtained \textit{via} a conformal deformation of a harmonic-Einstein structure; we recall that, as we have seen in the Introduction, a Riemannian manifold $(M,g)$ is said to be harmonic-Einstein if it carries a solution of the system
\begin{align*}
	\begin{cases}
		\ric^\p=\Lambda g,\\
		\tau(\p)=0,
	\end{cases}
\end{align*}
where $\Lambda\in C^{\infty}(M)$. Note that, since $\tau(\p)=0$, the $\p$-Schur identity rewrites as the usual one, that is
\begin{align*}
	R^{\p}_{ij,i}=\frac{S^{\p}_j}{2}-\p^a_{tt}\p^a_j=\frac{S^{\p}_j}{2},
\end{align*}
and therefore $\Lambda$ is constant.\\
\noindent
We say that a Riemannian manifold $(M,g)$ of dimension $m\geq 3$ is  \textit{conformally harmonic-Einstein} if there exists a smooth positive function $\psi\in C^{\infty}(M)$ such that, given the conformal change
\begin{align*}
	\tilde{g}=\psi^2 g,
\end{align*}
the manifold $(M,\tilde{g})$ is harmonic-Einstein.\\
\noindent
Note that, when $(M,g)$ carries a solution of system \eqref{sistema generale in sezione sistema genrale} that is constant and $U\equiv 0$, then $(M,g)$ is harmonic-Einstein. More in general, when $f$ is not constant and $U\equiv 0$, we have the validity of the following
\begin{proposition}\label{prop su cambi conf per ricavare sist}
	Let $(M,g)$ be a Riemannian manifold of dimension $m\geq3$,  $\p:(M,g)\ra(N,h)$ be a smooth map and let $\alpha\in \erre\setminus\set{0}$. Then $(M,g)$ carries a solution of the system
	\begin{align}\label{sistema generale con U zero}
		\begin{cases}
			\ric+\hs(f)+\frac{1}{m-2}df\otimes df=\lambda g\\
			\tau(\p)=d\p(\nabla f),
		\end{cases}
	\end{align}
	
	if and only if, given the conformal change of metric
	\begin{align*}
		\tilde{g}=e^{-\frac{2f}{m-2}}g,
	\end{align*}
	there exists $\Lambda\in C^\infty(M)$ such that
	\begin{align}\label{conformally harm einst}
		\begin{cases}
			\tilde{\ric}^{\tilde{\p}}=\tilde{\ric}-\alpha\tilde{\p}^*h=\Lambda \tilde{g},\\
			\tau(\tilde{\p})=0,
		\end{cases}
	\end{align}
	where $\tilde{\ric}$ and $\tilde{\ric}^{\tilde{\p}}$ are the Ricci and the $\p$-Ricci tensor of $(M,\tilde{g})$, respectively and $\tilde{\p}$ denotes the map $\p$ from $(M,\tilde{g})$ to $(N,h)$. Moreover, $\Lambda$ satisfies
	\begin{align*}
		\Lambda=\frac{1}{m-2}e^{\frac{2f}{m-2}}\pa{\Delta f-\abs{\nabla f}+(m-2)\lambda}.
	\end{align*}
\end{proposition}
\begin{rem}
	Note that $\Lambda$ is constant by the $\p$-Schur identity.
\end{rem}
\begin{proof}
	To prove Proposition \ref{prop su cambi conf per ricavare sist} we recall the transformation laws of the Ricci tensor $\ric$ (see e.g. \cite{CMBook}) and of the tension field $\tau(\p)$ (see e.g. \cite{EellsFerreira}), that are, respectively,
	\begin{align}
		&\tilde{\ric}=\ric+\hs(f)+\frac{1}{m-2}df\otimes df+\frac{1}{m-2}\pa{\Delta f-\abs{\nabla f}^2}g,\label{ricci conf transf}\\
		&\tau(\tilde{\p})=e^{-\frac{2f}{m-2}}\pa{\tau(\p)-d\p(\nabla f)}\label{tau conf transf}.
	\end{align}
	First note that,
	by \eqref{tau conf transf}, we deduce
	\begin{align*}
		\tau(\tilde{\p})=0
	\end{align*}
	if and only if
	\begin{align*}
		\tau(\p)=d\p(\nabla f).
	\end{align*}
	Suppose that \eqref{sistema generale con U zero} holds for some $f,\lambda\in C^{\infty}(M)$: by \eqref{ricci conf transf} and the first equation of \eqref{sistema generale con U zero}, we obtain
	\begin{align*}
		\tilde{\ric}^{\tilde{\p}}=&\tilde{\ric}-\alpha\tilde{\p}^*h\\
		=&\ric+\hs(f)+\frac{1}{m-2}df\otimes df+\frac{1}{m-2}\pa{\Delta f-\abs{\nabla f}^2}g-\alpha\p^*h\\
		=&\ric^\p+\hs(f)+\frac{1}{m-2}df\otimes df+\frac{1}{m-2}\pa{\Delta f-\abs{\nabla f}^2}g\\
		=&\lambda g+\frac{1}{m-2}\pa{\Delta f-\abs{\nabla f}^2}g\\
		=&e^{\frac{2f}{m-2}}\frac{1}{m-2}\pa{\Delta f-\abs{\nabla f}^2+(m-2)\lambda}\tilde{g}\\
		=&\Lambda \tilde{g}.
	\end{align*}
	Conversely, assume that \eqref{conformally harm einst} holds for some $\Lambda\in \erre$. By the transformation law  \eqref{ricci conf transf}, we deduce
	\begin{align*}
		\Lambda\tilde{g}=&\tilde{\ric}^{\tilde{\p}}\\
		=&\tilde{\ric}-\alpha\tilde{\p}^*h\\
		=&\ric-\alpha\p^*h+\hs(f)+\frac{1}{m-2}df\otimes df+\frac{1}{m-2}\pa{\Delta f-\abs{\nabla f}^2}g,
	\end{align*}
	that is
	\begin{align*}
		\ric-\alpha\p^*h+\hs(f)+\frac{1}{m-2}df\otimes df=\sq{e^{-\frac{2f}{m-2}}\Lambda+\frac{1}{m-2}\pa{\Delta f-\abs{\nabla f}^2}}g.
	\end{align*}
\end{proof}
\begin{rem}
	When we perform the conformal change
	\begin{align*}
		\tilde{g}=e^{-\frac{2f}{m-2}}g,
	\end{align*}
	we deduce, taking the trace of \eqref{ricci conf transf},
	\begin{align}\label{conf change per S}
		e^{-\frac{2f}{m-2}}\tilde{S}^{\tilde{\p}}=S^\p+\frac{m-1}{m-2}\pa{2\Delta f-\abs{\nabla f}^2}.
	\end{align}
	Moreover, by  the definition of the $\p$-Schouten tensor \eqref{def of ohi schouten} and the transformation laws \eqref{ricci conf transf} and \eqref{conf change per S} we have
	\begin{align*}
		\tilde{A}^{\tilde{\p}}=A^\p+\hs(f)+\frac{1}{m-2}\pa{df \otimes df - \frac{\abs{\nabla f}^2}{2}g}.
	\end{align*}
	As a consequence, using the definition of the covariant derivative of the Levi-Civita connection associated to the metric $\tilde{g}$ and the definition of $\tilde{C}^{\tilde{\p}}$, a simple but tedious computation shows that the conformal change for the $\p$-Cotton tensor is given by
	\begin{align}\label{conf law for cotton}
		\tilde{C}^{\tilde{\p}}=C^\p+W^\p(\nabla f,\cdot,\cdot,\cdot).
	\end{align}
\end{rem}

	\chapter{Elementary Considerations on System (\ref{Gianny1})}\label{Sect_SufficientConditions}

\section{Some Observations}
The aim of this chapter is to analyze the structure of system \eqref{Gianny1}, that we recall here  for the sake of readability, and some consequences which can be deduced by only considering some parts of it: 
\begin{align*}
	\begin{cases}
		i)\, \hs(u)-u\set{\ric^\p-\frac{1}{m-1}\pa{\frac{S^\p}{2}-p+U(\p)}g}=0,\\
		ii)\,\Delta u=\frac{u}{m-1}\sq{mp-mU(\p)+\frac{m-2}{2}S^\p},\\
		iii)\,u\tau(\p)=-d\p(\nabla u)+\frac{u}{\alpha}(\nabla U)(\p),\\
		iv)\,\mu+U(\p)=\frac{1}{2}S^\p,\\
		v)\,(\mu+p)\nabla u=-u\nabla p.
	\end{cases}
\end{align*}

The next observations will be also useful in the proofs of some of the main theorems: for instance, in Theorem \ref{thm A}  we rely on
Proposition \ref{tot geod} to deduce that $i : \partial M \to M$ is totally geodesic,
which allows us to apply a result of Reilly (\cite{R}) and conclude the validity of
the statement. In what follows, we shall always assume, unless otherwise stated,
\[
u>0 \quad \text{ on }\,\,\, \operatorname{int}(M) \qquad \text{ and }\,\,\, u^{-1}\pa{\set{0}}=\partial M \,\,\,\l \text{if }\,\,\, \partial M \neq \emptyset.
\]
In case  $\partial M \neq \emptyset$,  the boundary will also be assumed to be \textbf{connected}.

\begin{proposition}\label{tot geod}
	Let $(M,g)$ be a smooth manifold of dimension $m\geq 2$, with smooth boundary $\partial M \neq \emptyset$. Let $u$ be a solution on $M$ of
	\begin{equation}\label{Eq2.2}
		\hess(u) - u\set{\ric^\p-\Lambda(x)g} = 0,
	\end{equation}
	for some $\Lambda \in C^2(M)$. Then $\abs{\nabla u}$ is a positive constant on $\partial M$ and $i:\partial M \hookrightarrow M$ is totally geodesic.
\end{proposition}
\begin{rem}
	Obviously, \eqref{Eq2.2} can be replaced by \eqref{Gianny1} i).
\end{rem}
\begin{proof}
	Let $e_1, e_2, \cdots e_{m-1}, e_m$ be a Darboux frame along $i:\partial M \hookrightarrow M$, with $e_m=\nu$, the inward unit normal to $\partial M$. From \eqref{Eq2.2} we obtain
	\begin{align*}
		\abs{\nabla u}_j^2=2u_{jk}u_k=2u\pa{R^{\p}_{jk}-\Lambda(x)\delta_{jk}}u_k,
	\end{align*}
	and therefore, since $u\equiv 0$ on $\partial M$, it follows that $\abs{\nabla u}$ is constant on  $\partial M$. \\
	To show that $\abs{\nabla u}^2$ is not zero we reason by contradiction: fix $p\in \partial M$ and, for some $\eps>0$ sufficiently small, let $\gamma: [0,\eps)\ra M$ be the geodesic such that $\gamma(0)=p$, $\dot{\gamma}(0)=\nu$. We define
	\begin{align*}
		v(t):=\pa{u\circ \gamma}(t);
	\end{align*}
	then, since $\gamma$ is a geodesic, we have
	\begin{align*}
		\begin{cases}
			v''(t)=\hs(u)(\dot{\gamma},\dot{\gamma})(t)=v(t)\sq{\ric^\p\pa{\dot{\gamma}, \dot{\gamma}} - \Lambda\abs{\dot{\gamma}}^2},\\
			v'(0)= g\pa{\nabla u(p), \dot{\gamma}(0)},\\
			v(0)=u(p)=0.
		\end{cases}
	\end{align*}
	Therefore,  if $\nabla u(p)=0$, then $v'(0)=v(0)=0$ and $v\equiv0$ on $[0, \eps')$ for some $0<\eps'\leq \eps$. This is a contradiction, since $u>0$ and $\gamma((0,\eps'))\subseteq \mathrm{int}M$.\\
	It follows that
	\begin{align*}
		\nu=\frac{\nabla u}{\abs{\nabla u}}\quad \text{ on }\, \partial M
	\end{align*}
	and the second fundamental form $\se$ in the direction of $\nu$ is
	\begin{align*}
		\se = -\frac{\hs(u)\vert_{\mathfrak{X}(\partial M)\times \mathfrak{X}(\partial M)}}{\abs{\nabla u}}.
	\end{align*}
	From \eqref{Eq2.2}, since $u=0$ on $\partial M$, we deduce $\se=0$, that is, $i:\partial M \hookrightarrow M$ is totally geodesic.
\end{proof}

We next  show another relevant fact, that is, equation  \eqref{Gianny1} v), which is physically motivated by the preservation of the energy momentum, can be deduced from the other equations of the system. To do this, note that \eqref{Gianny1} v)  is equivalent to
\begin{align}\label{Eq2.4app}
	u\nabla\mu=\nabla\sq{u\pa{\mu +p}}\quad \text{on } M,
\end{align}
and this latter  can be obtained as follows:
\begin{proposition}\label{lemma mu e rho}
	Let $(M, g)$ be a manifold of dimension $m\geq 2$ and $u$ a solution of
	\begin{equation}\label{Eq2.6}
		\begin{cases}
			i)\, \hs(u)-u\set{\ric^\p-\frac{1}{m-1}\pa{\frac{S^\p}{2}-p+U(\p)}g}=0,\\
			ii)\,\Delta u=\frac{u}{m-1}\sq{mp-mU(\p)+\frac{m-2}{2}S^\p},\\
			iii')\,  h\pa{\alpha u\tau(\p)+\alpha d\p(\nabla u)-u(\nabla U)(\p), d\p}=0,\\
			iv)\,\mu+U(\p)=\frac{1}{2}S^\p\\
		\end{cases}
	\end{equation}
	on $M$.  Then \eqref{Eq2.4app} holds.
	
\end{proposition}

\begin{proof}
	We insert \eqref{Eq2.6} iv) into \eqref{Eq2.6} ii) and we take covariant derivative  to obtain
	\begin{align}\label{derivata cov del laplaciano}
		u_{iit}=\frac{u_t}{m-1}\sq{mp-2U(\p)+\pa{m-2}\mu}+\frac{u}{m-1}\pa{mp_t-2U^a\p^a_t+\pa{m-2}\mu_t};
	\end{align}
	next we compute the covariant derivative of \eqref{Eq2.6} i):
	\begin{align*}
		0=&u_{itj}-u_jR^\p_{it}+\frac{u_j}{m-1}\pa{\frac{S^\p}{2}-p+U(\p)}\delta_{it}\\
		&+\frac{u}{m-1}\pa{\frac{S^\p_j}{2}-p_j+U^a\p^a_j}\delta_{it}-uR^\p_{it, j};
	\end{align*}
	contracting with respect to $i$ and $j$, using the $\p$-Schur's identity, Ricci commutation relations, \eqref{Eq2.6} iii'), \eqref{Eq2.6} ii) and \eqref{derivata cov del laplaciano} we obtain
	\begin{align*}
		0=&u_{iit}+\alpha u_s\p^a_s\p^a_t-u\pa{\frac{S^\p_t}{2}-\alpha\p^a_{ii}\p^a_t}+\frac{u}{m-1}\pa{\frac{S^\p_t}{2}-p_t+U^a\p^a_t}\\
		&+\frac{u_t}{m-1}\pa{\frac{S^\p}{2}-p+U(\p)}\\
		=&\frac{u_t}{m-1}\pa{mp-2U(\p)+\pa{m-2}\mu}+\frac{u}{m-1}\pa{mp_t-2U^a\p^a_t+\pa{m-2}\mu_t}\\
		&-\alpha u \p^a_{ii}\p^a_t +uU^a\p^a_t-u\frac{S^\p_t}{2}+\alpha u \p^a_{ii}\p^a_t
		+\frac{u}{m-1}\pa{\frac{S^\p_t}{2}-p_t+U^a\p^a_t}\\
		&+\frac{u_t}{m-1}\pa{\frac{S^\p}{2}-p+U(\p)}\\
		=&\frac{u}{m-1}\sq{(m-1)p_t -U^a\p^a_t +\pa{m-1}U^a\p^a_t+ (m-2)\mu_t+\frac{S^\p_t}{2}-(m-1)\frac{S^\p_t}{2}} \\&+\frac{u_t}{m-1}\sq{mp-2U(\p)+(m-2)\mu+\frac{S^\p}{2}-p+U(\p)} \\=&-u\frac{S^\p_t}{2} +uU^a\p^a_t + \frac{1}{m-1}\sq{u\pa{(m-1)p-U(\p)+(m-2)\mu+\frac{S^\p}{2}}}_t \, ,
	\end{align*}
	that is,
	\begin{align}\label{equiv}
		u\nabla\pa{\frac{S^\p}{2}-U(\p)}=\frac{1}{m-1}\nabla\sq{u\pa{(m-1)p-U(\p)+(m-2)\mu+\frac{S^\p}{2}}}.
	\end{align}
	Using \eqref{Eq2.6} iv) we infer \eqref{Eq2.4app}.
\end{proof}
\begin{rem}
	Note that, if in
	\eqref{Eq2.6} iii') $d\p$ is a submersion at each point of $M$, it gives the validity of \eqref{Gianny1} iii); however, in general the latter is stronger.
\end{rem}

\begin{proposition}\label{Prop2.12}
	In the assumptions of Proposition \ref{lemma mu e rho}, suppose that $\partial M\neq \emptyset$. Then $\mu$ is constant if and only if $p=-\mu$.
\end{proposition}
\begin{proof}
	By Proposition \ref{lemma mu e rho}, if $\mu$ is constant on $M$, we have
	\begin{align}\label{prop 2.12 equazione}
		\nabla\sq{u\pa{\mu+p}}=u\nabla\mu=0,
	\end{align}
	that is, $u(\mu+p)$ is constant on $M$. Since $u>0$ on $\mathrm{int}(M)$ and $u=0$ on $\partial M$,  we deduce
	\begin{align*}
		u(\mu+p)=0 \quad \text{on } M,
	\end{align*}
	and thus
	\begin{align*}
		\mu+p=0 \quad \text{on }M.
	\end{align*}
	Conversely, if we assume $p+\mu=0$, Proposition \ref{lemma mu e rho} implies
	\begin{align*}
		u\nabla \mu=0;
	\end{align*}
	since $u>0$ on $\mathrm{int}(M)$, we deduce that $\mu$ is constant on $\mathrm{int}(M)$ and therefore on $M$.
\end{proof}

If we assume $p=-\mu$, using \eqref{Gianny1} i), ii) and iv) we get
\[
\hs(u) - u\pa{\ric^\p - \frac{S^\p}{m-1}g}=0;
\]
 adding the assumption $U$ constant, from \eqref{Gianny1} iii) we infer
\begin{align}\label{Eq2.13}
	\begin{cases}
		i)\:\hs(u)-u\pa{\ric^\p-\frac{S^\p}{m-1}g}=0,\\
		ii)\:\Delta u=-\frac{S^\p}{m-1}u,\\
		iii)\:u\tau(\p)+d\p(\nabla u)=0,\\
	\end{cases}
\end{align}
where the second equation is the trace of the first. In other words, the $\p$-static space system \eqref{Eq2.13} can be deduced from \eqref{Gianny1} i), ii), iii) and iv) in the assumptions $p=-\mu$ and $U$ constant (i.e., assuming lowest level of the Null Energy Condition and constant scalar potential).

%

It is well-known, at least for $\p$ constant, that \eqref{Eq2.13} implies $S^\p$ constant; as a matter of fact, we have

\begin{proposition}\label{Prop2.20}
	Let $(M,g)$ be a manifold of dimension $m\geq 2$ and $u\geq 0$, $u\not\equiv 0$ a solution on $M$ of the system
	\begin{align}\label{sistema con h U(phi) cost}
		\begin{cases}
			i)\,\hs(u)-u\pa{\ric^\p-\frac{S^\p}{m-1}g}=0,\\
			ii)\, h\pa{\alpha u \tau(\p)+\alpha d\p(\nabla u)-{u}(\nabla U)(\p), d\p}=0.
		\end{cases}
	\end{align}
	Then
	\begin{align*}
		\frac{1}{2}u\nabla S^\p=u\nabla\pa{U(\p)}.
	\end{align*}
	In particular, if $U(\p)$ is constant, then $S^{\p}$ is also constant.
\end{proposition}
\begin{proof}
	Tracing \eqref{sistema con h U(phi) cost} i) we obtain
	\begin{equation}\label{Eq2.15.1}
		\Delta u = - \frac{S^\p}{m-1}u,
	\end{equation}
	which enables us to rewrite \eqref{sistema con h U(phi) cost} i) in the form
	\begin{equation}\label{Eq2.16}
		\hs(u) -u \ric^\p -\Delta u g =0.
	\end{equation}
	From \eqref{sistema con h U(phi) cost} ii), that is,
	\begin{equation}\label{Eq2.17}
		\alpha u \p^a_{tt}\p^a_j = -\alpha\p^a_i u_i \p^a_j+uU^a\p^a_j
	\end{equation}
	and from the $\p$-Schur's identity,
	\begin{align*}
		\frac{1}{2}uS^\p_j&=\pa{uR_{ij}^\p}_i-u_iR^\p_{ij}+\alpha u\p^a_{tt}\p^a_j;
	\end{align*}
	we now use \eqref{Eq2.15.1}, \eqref{Eq2.16} and \eqref{Eq2.17} into the above to obtain
	\begin{align*}
		\frac{1}{2}uS^\p_j&=\pa{uR_{ij}^\p}_i-u_iR^\p_{ij}-\alpha u_i\p^a_i\p^a_j+uU^a\p^a_j\\
		&=\pa{u_{ij}-\Delta u \delta_{ij}}_i-u_i\pa{R^\p_{ij}+\alpha\p^a_i\p^a_j}+uU^a\p^a_j\\
		&=u_{iji}-u_{ttj}-u_iR_{ij}+uU^a\p^a_j=uU^a\p^a_j,
	\end{align*}
	where in the last equality we have used the Ricci commutation relations.
	In other words
	\begin{align}\label{eq nabla S phi} 
		\frac{1}{2}u\nabla S^\p=uU^a\p^a_j,
	\end{align}
	which gives the first part of the statement.
	For the second part, we set
	\begin{align*}
		\Sigma_0:=\set{x\in M\,:\, u(x)=0};
	\end{align*}
	then $S^\p$ is constant on each component of $M\setminus \Sigma_0$. If $\mathrm{int}(\Sigma_0)=\emptyset$, then, by continuity, $S^\p$ is constant on $M$. If $\mathrm{int}(\Sigma_0)\neq\emptyset$, then $u\equiv 0$ on an open set of $M$, and it satisfies \eqref{Eq2.15.1}; hence, by the unique continuation property (see Appendix A of \cite{PRS} and also \cite{Kazdan1988UniqueCI}), we have $u\equiv 0$ on $M$, which is a contradiction.
\end{proof}
\noindent
\section{On the Constancy of the Map $\p$}
As mentioned above, it is worth to try to determine sufficient conditions for the constancy of the map $\p$; towards this aim, we recall the classical Bochner-Weitzenb\"{o}ck formula for the energy density of a smooth map $\p : (M, g) \ra (N, h)$ (for more details, see for instance \cite{EL} and Proposition 1.5 of \cite{AMR}).
We have
\begin{align}\label{bochner per phi}
	\frac{1}{2}\Delta\abs{d\p}^2=\abs{\nabla d\p}^2+\p^a_i\p^a_{kki}+{}^N\!R^a_{bcd}\p^a_i\p^b_k\p^c_k\p^d_i+R_{ti}\p^a_t\p^a_i,
\end{align}
where ${}^N\!R^a_{bcd}$ denote the components of the Riemann curvature tensor of $(N,h)$ with respect to a local orthonormal coframe on $N$.\\
We consider the system
\begin{align}\label{equazione prima della 2.7.31} 
	\begin{cases}
		i)\,\hs(u)-u\set{\ric^\p-\frac{1}{m-1}\pa{S^\p-(p+\mu)}g}=0,\\
		ii)\,u\tau(\p)=-d\p(\nabla u)+\frac{u}{\alpha}(\nabla U)(\p),
	\end{cases}
\end{align}
with $u>0$ on $\mathrm{int}(M)$.
Note that \eqref{equazione prima della 2.7.31} is obtained from \eqref{Gianny1} in the following way: with the aid of \eqref{Gianny1} ii), we rewrite \eqref{Gianny1} i) in the form \begin{equation}\label{Eq2.25}
	\hs(u) -u\set{\ric^\p-\frac{1}{m}\pa{S^\p-\frac{\Delta u}{u}}}=0.
\end{equation}
Next, using \eqref{Gianny1} iv) into \eqref{Gianny1} ii) we obtain
\begin{equation}\label{Eq2.26}
	\frac{\Delta u}{u} = \frac{1}{m-1}\sq{m\pa{\mu+p}-S^\p};
\end{equation}
inserting \eqref{Eq2.26} into \eqref{Eq2.25} yields \eqref{equazione prima della 2.7.31} i). The second equation, that is \eqref{equazione prima della 2.7.31} ii), is simply \eqref{Gianny1} iii).\\
We proceed by introducing a definition that will be useful to prove the next result; we refer to it also in the statement of Theorem \ref{thm A}.
\begin{defi}\label{definition: weakly convex}
	Let $(M,g)$ be a Riemannian manifold of dimension $m$ and let $F:M\ra \erre$ be a smooth function; then, we say that $F$ is \emph{weakly convex} if the Hessian of $F$, $\hs(F)$, is positive semi-definite.	
\end{defi}
Now let
\begin{align}\label{def di f, 2.7.30} 
	f:=-\log u \quad \text{on }\mathrm{int}(M);
\end{align}
we rewrite system \eqref{equazione prima della 2.7.31} in the form
\begin{align}\label{2.7.31}
	\begin{cases}
		i)\: \ric^\p+\hs(f)-df\otimes df=\frac{1}{m-1}\pa{{S^\p}-(p+\mu)}g,\\
		ii)\: \tau(\p)=d\p(\nabla f)+\frac{1}{\alpha}\pa{\nabla U}(\p).\\
	\end{cases}
\end{align}
We contract \eqref{2.7.31} i) by $\p^a_i\p^a_j$ and we use \eqref{2.7.31} ii) to obtain
\begin{align}\label{2.39}
	R^\p_{ij}\p^a_i\p^a_j=&\frac{1}{m-1}\pa{S^\p-(p+\mu)}\abs{d\p}^2-f_{ij}\p^a_i\p^a_j+  \abs{\tau(\p)}^2 + \frac{1}{\alpha^2}\abs{\nabla U (\p)}^2  \\
	&-\frac{2}{\alpha} h\pa{\tau(\p), \nabla U(\p)}. \notag
\end{align}
Inserting \eqref{2.39} into \eqref{bochner per phi} and using \eqref{2.7.31} ii) we get
\begin{align}\label{2.7.34}
	\frac{1}{2}\Delta\abs{d\p}^2 =&\abs{\nabla d\p}^2+\p^a_i\p^a_{kki}+{}^N\!R^a_{bcd}\p^a_i\p^b_k\p^c_k\p^d_i+\alpha \p^a_i\p^a_t\p^b_i\p^b_t \\&+\frac{1}{m-1}\pa{S^\p-(p+\mu)}\abs{d\p}^2 -f_{ij}\p^a_i\p^a_j+\abs{\tau(\p)}^2+\frac{1}{\alpha^2}\abs{\nabla U (\p)}^2\notag\\
	&-\frac{2}{\alpha} h\pa{\tau(\p), \nabla U(\p)}\notag\\		=&\abs{\nabla d\p}^2+\p^a_{si}\p^a_if_s+ \p^a_i\p^a_sf_{si} + \frac{1}{\alpha}U^{ab}\p^a_i\p^b_i+{}^N\!R^a_{bcd}\p^a_i\p^b_k\p^c_k\p^d_i\notag \\&+\alpha\p^a_i\p^a_t\p^b_i\p^b_t   + \frac{1}{m-1}\pa{S^\p-(p+\mu)}\abs{d\p}^2 -f_{ij}\p^a_i\p^a_j+\abs{\tau(\p)}^2\notag\\ &+\frac{1}{\alpha^2}\abs{\nabla U (\p)}^2-\frac{2}{\alpha} h\pa{\tau(\p), \nabla U(\p)}\notag \\		=&\abs{\nabla d\p}^2 + \frac{1}{2} g\pa{\nabla\abs{d\p}^2, \nabla f} + \frac{1}{\alpha}U^{ab}\p^a_i\p^b_i + \alpha \p^a_i\p^a_t\p^b_i\p^b_t + \abs{\tau(\p)}^2 \notag \\&+ \frac{1}{m-1}\pa{S^\p-(p+\mu)}\abs{d\p}^2+\frac{1}{\alpha^2}\abs{\nabla U (\p)}^2+{}^N\!R^a_{bcd}\p^a_i\p^b_k\p^c_k\p^d_i \notag \\ &-\frac{2}{\alpha} h\pa{\tau(\p), \nabla U(\p)}.\notag
\end{align}
We are thus ready to prove
\begin{lemma}\label{Lemma2.31}
	Let $(M,g)$ be a manifold of dimension $m\geq 2$ satisfying \eqref{2.7.31} with $\alpha>0$. Assume
	\begin{align}\label{2.7.36}
		\mathrm{Sect}(N)\leq A,
	\end{align}
	where $\mathrm{Sect}(N)$ is the sectional curvature of $N$ and $A$ is a real constant. Then, on $\mathrm{int}(M)$,
	\begin{align}\label{2.7.37}
		\frac{1}{2}\Delta_{f}\abs{d\p}^2\geq&\abs{\nabla d\p}^2 +\pa{\frac{\alpha}{m}-A}\abs{d\p}^4+ \abs{d\p(\nabla f)}^2  \\ &+\frac{1}{m-1}\pa{S^\p-(p+\mu)}\abs{d\p}^2+\frac{1}{\alpha}\mathrm{tr}\pa{\hs(U)(d\p,d\p)}\notag,
	\end{align}
	where  $\Delta_f = \Delta - g\pa{\nabla f, \cdot}$ is the $f$-Laplacian.
\end{lemma}
\begin{proof}
	Observe that, using \eqref{2.7.31} ii),
	\[
	\frac{1}{\alpha^2}\abs{\nabla U (\p)}^2-\frac{2}{\alpha} h\pa{\tau(\p), \nabla U(\p)} + \abs{\tau(\p)}^2=\abs{\frac{1}{\alpha}\nabla U - \tau(\p)}^2 = \abs{d\p(\nabla f)}^2.
	\]
	Furthermore, since $\alpha>0$,
	\begin{align*}
		\alpha\p^a_i\p^a_t\p^b_i\p^b_t\geq\alpha\frac{\abs{d\p}^4}{m},
	\end{align*}
	and \eqref{2.7.36} implies
	\begin{align*}
		{}^N\!R^a_{bcd}\p^a_i\p^b_k\p^c_k\p^d_i\geq -A\abs{d\p}^4.
	\end{align*}
	Inserting this information  into \eqref{2.7.34} immediately yields the validity of \eqref{2.7.37}.
\end{proof}

Now observe that, for $v = \abs{d\p}^2$, $\frac{\alpha}{m}>A$, $U$ weakly convex  and supposing $\alpha>0$ and $\sigma := \inf_M\pa{S^\p-(p+\mu)}>-\infty$, \eqref{2.7.37} gives
\begin{equation}\label{Eq2.34}
	\frac{1}{2}\Delta_fv \geq \pa{\frac{\alpha}{m}-A}v^2 + \frac{\sigma}{m-1}v.
\end{equation}
We recall that $\partial M=u^{-1}(\set{0})$ if not empty and therefore,
 since $f=-\log u$, the operator $\Delta_f$ is not well defined on $\partial M$. However, when $\partial M\neq \emptyset$ we require  $d\p\equiv 0$ on $\partial M$; it follows that, when $v\not\equiv 0$, if $\gamma>0$ and near to $v^*=\sup_M v>0$ (possibly infinite), equation \eqref{Eq2.34} is well defined on the super-level set
\[
\Omega_\gamma = \set{x\in M : v(x)>\gamma},
\]
since the latter has empty intersection with $\partial M$.
This enables us to apply Theorems 4.2 and 4.1 of \cite{AMR} to $\Omega_{\gamma}$ and conclude \textit{via} \eqref{Eq2.34}
 that $v^* < +\infty$ and
\begin{equation}\label{Eq2.35}
	v^* \leq - \frac{\sigma}{(m-1)\pa{\frac{\alpha}{m}-A}},
\end{equation}
provided the growth condition
\begin{equation}\label{Eq2.36}
	\liminf_{r\ra +\infty} \frac{1}{r^2} \log\int_{B_r}e^{-f} < +\infty,
\end{equation}
where $B_r$ is a geodesic ball of radius $r$ centered at some origin $o\in M$.
Thus, assume that $(M,g)$, with $\partial M\neq \emptyset$, is complete in the sense of Cauchy; in order to guarantee \eqref{Eq2.36}, we suppose that $\partial M$ is compact.\\
To continue, we introduce the \emph{topological double} of $M$, defined gluing two copies of $M$ along their boundaries:
\begin{align*}
	\mathcal{D}(M)=\frac{M_0\cup M_1}{\sim}=\frac{M\times \set{0}\cup M\times \set{1}}{\sim}
\end{align*}
where, for each $x\in \partial M$,
$\sim$ identifies the points $\pa{x,0}$ and $\pa{x,1}$ (see \cite{Munkres} and \cite{PRS} for more details).
To define a differentiable structure on $\mathcal{D}(M)$, we consider two open sets $W_i$, $i=0,1$, such that $\partial M\subset W_i\subset M_i$ for each copy of $M$, and two diffeomorphisms 
\begin{align*}
	h_0:W_0\to \partial M\times [0,1), \quad \quad h_1:W_1\to \partial M\times (-1,0]
\end{align*}
such that
\begin{align*}
	h_0(x)=\pa{x,0}=h_1(x) \quad \forall x\in \partial M.
\end{align*}
 Let $W=W_0\cup W_1$ be the union of $W_0$ and $W_1$ in $\mathcal{D}(M)$; then there is a homeomorphism
\begin{align*}
	h:W\to \partial M\times \pa{-1,1}
\end{align*}
such that $h{|_{W_i}}=h_i$, $i=0,1$; moreover, consider the inclusions
\begin{align*}
	i_0:\partial M\times \set{0}\xhookrightarrow{} \mathcal{D}(M),\\
	i_1:\partial M\times \set{1}\xhookrightarrow{} \mathcal{D}(M).
\end{align*}
A differentiable structure on $\mathcal{D}(M)$ is obtained imposing that $h$ is a diffeomorphism and $i_0,i_1$ are smooth embeddings.

%

The Riemannian metric $g$ on $M$ induces Riemannian metrics $g_0, g_1$ on $M_0$ and $M_1$, respectively. Moreover, let
\begin{align*}
	h:W\to \partial M\times \pa{-1,1}
\end{align*}
be  a diffeomorphism such that $h{|_{W_i}}=h_i$, $i=0,1$; then we define
\begin{align*}
	g_W:=h^*\pa{g_{|_{\partial M}}+ dt\otimes dt}.
\end{align*}
Given $V\subset W$ open and relatively compact such that $\partial M\subset V$, we can define a Riemannian metric on $\mathcal{D}(M)$,
\begin{align*}
	g^{\mathcal{D}}=\psi_0g_0+\psi_1g_1+\psi_Vg_W
\end{align*}
where $\set{\psi_0,\psi_1\psi_V}$ is a partition of unity relative to the open cover
\begin{align*}
	\set{\text{int}(M_0), \text{int}(M_1), V}.
\end{align*}
Note that Cauchy sequences for the metric $g^{\mathcal{D}}$ are convergent in $\mathcal{D}(M)$ and, since $\partial \mathcal{D}(M)=\emptyset$, we have that $\mathcal{D}(M)$ is geodesically complete by the Hopf-Rinow Theorem.
Moreover, by the definition of $g^{\mathcal{D}}$ and since $\partial M$ is compact, the volume of $B_r(o)\subset M$ is bounded above by the volume of a ball in $\mathcal{D}(M)$ centered at $o$ and of radius possibly dilated by a positive constant $a$. Thus, to control
\begin{align*}
	\int_{B_r(o)}e^{-f}
\end{align*}
from above, it is enough to bound
\begin{align*}
	\int_{B_{ar}(0)\backslash U}e^{-f}
\end{align*}
for $U$ open, relatively compact such that $\partial M\subset U$ and with $f$ extended on $\mathcal{D}(M)$ outside of $U$.

Now we know, in the case we are interested in, that $f$ satisfies \eqref{2.7.31} i), so that, for $\alpha>0$ and
\begin{equation}\label{Eq2.37}
	\sigma = \inf_M\pa{S^\p-(p+\mu)},
\end{equation}
 assuming $\sigma>-\infty$, we have
\begin{equation}\label{Eq2.38}
	\ric +  \hs(f) = \alpha\p^*h + df\otimes df +\frac{1}{m-1}\pa{S^\p-(p+\mu)}g\geq \sigma g.
\end{equation}
Thus, inequality (8.115) of \cite{AMR} gives
\begin{equation}\label{Eq2.39}
	\int_{B_{ar}(o)\setminus U}e^{-f} \leq D + \int_0^{ar} e^{Ct-\frac{\sigma}{2(m-1)}t^2}\,dt
\end{equation}
for some sufficiently large constants $C, D>0$.  It follows that
\begin{align}\label{Eq2.40}
	\int_{B_{r}(o)}u \quad \text{ grows at most like }\quad
	\begin{cases}
		\frac{e^{\frac{\abs{\sigma}}{2}a^2r^2}}{ar} \quad \text{if }\sigma \neq 0,\\
		\frac{e^{Car}}{2} \qquad \,\, \text{if } \sigma=0.
	\end{cases}
\end{align}
when $r\ra +\infty$. In particular, \eqref{Eq2.36} is satisfied.

We are now ready to state the next
\begin{theorem}\label{theorem 2.7.48}
	Let $(M,g)$ and $(N,h)$ be two Riemannian manifolds, such that $(M,g)$ is complete with compact boundary $\partial M$ and dimension $m\geq 2$. Let $u$ be a solution of \eqref{equazione prima della 2.7.31}, where $\p : (M, g) \ra (N, h)$ is a smooth map, $U : N \ra \erre$, $p, \mu : M \ra \erre$ are smooth functions, $U$ is weakly convex, $\alpha>0$, and assume
	\begin{align*}
		&d\p\equiv 0\ \ \ \ \ \text{on}\ \partial M, \\
		&\sigma = \inf_M \set{S^\p-(p+\mu)}>-\infty,\\
		&\mathrm{Sect}(N) \leq A, \qquad A<\frac{\alpha}{m}.
	\end{align*}
	Then, either $\p$ is constant or
	\begin{align}\label{2.7.49}
		\abs{d\p}^2\leq-\frac{\sigma}{(m-1)(\frac{\alpha}{m}-A)} \quad\text{on }M.
	\end{align}
	In particular, if $S^\p-(p+\mu)\geq 0$, then $\p$ is constant.
\end{theorem}

\begin{rem}
	If we assume the validity of \eqref{equazione prima della 2.7.31} up to the boundary, that is, also on $\partial M$ if not empty, then it is not hard to see that for $\alpha>0$ and $\mathrm{Sect}(N) \leq A$ we obtain the validity of the following formula:
	\begin{align}\label{Eq2.47}
		\frac12\operatorname{div}\pa{u\nabla\abs{d\p}^2} \geq& u\abs{\nabla d\p}^2 + \pa{\frac{\alpha}{m}-A}\abs{d\p}^4u + \abs{\frac{1}{\alpha}\nabla U(\p)-\tau(\p)}^2u \\ &+\frac{1}{m-1}\pa{S^\p-(p+\mu)}\abs{d\p}^2u+\frac{u}{\alpha}\operatorname{tr}\pa{\hs(U)}\pa{d\p, d\p}. \notag
	\end{align}
\end{rem}
Thus, recalling that $u>0$ on $\mathrm{int}(M)$ and $u=0$ on $\partial M$, we have the following
\begin{proposition}\label{PR2.48}
	Let $(M, g)$ be a compact manifold of dimension $m\geq 2$, with $\partial M \neq \emptyset$. Suppose that $u$ is a solution of \eqref{equazione prima della 2.7.31} up to the boundary with $\alpha>0$ and let $U$ be weakly convex, $\mathrm{Sect}(N)\leq A$, $A<\frac{\alpha}{m}$. Assume
	\[
	S^\p - (p+\mu) \geq 0 \quad \text{ on }\, M.
	\]
	Then $\p$ is constant.
\end{proposition}
\begin{proof}
	Simply integrate \eqref{Eq2.47} on $M$, use $u\equiv 0$ on $\partial M$ and the remaining assumptions of the proposition, to deduce that
	\[
	u\pa{\frac{\alpha}{m}-A}\abs{d\p}^2 \equiv 0 \quad \text{ on }\, \mathrm{int}(M),
	\]
	and thus $\abs{d\p}^2 \equiv 0$ on $M$.
\end{proof}

	\chapter{A Kobayashi-Obata type Theorem}\label{Sect_KO}
In 1980 Kobayashi and Obata (see \cite{KO}) proved that a conformally flat manifold $(M,g)$ that admits a non-constant positive solution $u$ of the equation
\begin{align}\label{KO: eq ko}
	\hs(u)-u\pa{\ric-\frac{1}{m}\pa{uS-\Delta u}g}=0
\end{align}
is such that, for any regular value $c$ of $u$, the hypersurface $\Sigma=u^{-1}(\set{c})$ has constant curvature  and the metric $g$ splits locally as a warped product $g=dt^2+\rho(t)^2g_{\Sigma}$, where
\begin{align*}
	dt=\frac{du}{\abs{du}}.
\end{align*}
Needless to say, a prominent role in the proof of the result is played by the assumption $W\equiv 0$.\\
Note that, if we make the change of variable
\begin{align*}
	f=-\log u,
\end{align*}
\eqref{KO: eq ko} transforms into the system
\begin{align}\label{KO: eq ko2}
	\ric+\hs(f)-df\otimes df=\lambda g,
\end{align}
with
\begin{align*}
	\lambda=\frac{1}{m}\pa{S+\Delta f-\abs{\nabla f}^2} \in C^\infty(M).
\end{align*}
We want to interpret the assumption $W\equiv 0$ in an alternative way. Towards this aim, we consider the pointwise conformal change of metric $g$ of $M$ given by
\begin{align}\label{KO: conf change intro}
	\tilde{g}=e^{-\frac{2}{m-2}f}g,
\end{align}
where $m=\mathrm{dim} M\geq 3$. It is well known (see for instance \cite{Besse}) that the Cotton tensor $C$ of $g$ changes accordingly to the formula
\begin{align*}
	\tilde{C}=C+W(\nabla f, \cdot,\cdot,\cdot),
\end{align*}
where $\tilde{C}$ is the Cotton tensor with respect to $\tilde{g}$. Thus, the condition $W\equiv 0$ implies that the conformal change of metric $\tilde{g}$ gives rise to a Cotton-flat metric on $M$. This suggests two distinct aims: the first is that an appropriate formulation of $\tilde{C}$ is related to the geometry of the level set hypersurfaces of $f$; the second is that we can probably weaken the assumption $W\equiv 0$ to get the result.\\
In our setting, we are considering a more general system than \eqref{KO: eq ko2}, precisely
\begin{align}\label{systemKO f}
	\begin{cases}
		i)\,\ric^\p+\hs(f)-\eta df\otimes df=\lambda g,\\
		ii)\,\tau(\p) =d\p(\nabla f)+\frac{1}{\alpha}(\nabla U)(\p),
	\end{cases}
\end{align}
$\lambda \in C^{\infty}(M)$.
Let $\tilde{\p}=\p:(M,\tilde{g})\ra(N,h)$, then the conformal change \eqref{KO: conf change intro} gives the transformation \eqref{conf law for cotton}, that we recall here,
\begin{align*}
	\tilde{C}^{\tilde{\p}}=C^\p+W^\p(\nabla f,\cdot,\cdot,\cdot),
\end{align*}
but information on $U$ and $\nabla U$ is missing and cannot depend on \eqref{KO: conf change intro}. The latter fact is, somehow, an indication that we have to consider an ``intrinsic'' point of view. The right idea seems to be suggested by the introduction of a modification of the well-known tensor $D$, first introduced by Cao and Chen (\cite{Cao_2013}) to study gradient Ricci solitons. In particular, they derive the so-called integrability condition for a Ricci soliton,
\begin{align*}
	D=C+W(\nabla f,\cdot,\cdot,\cdot).
\end{align*}
Therefore, we introduce a suitable modification $\ol{D}^\p$ of $D$ and then we determine the first (and the second) integrability conditions for system \eqref{systemKO f}; then, we study the geometry of a regular level set hypersurface of $f$ under the assumption
\begin{align*}
	\ol{D}^\p\equiv 0.
\end{align*}

\section{First and Second Integrability Conditions}

As promised in the Introduction, we study system (\ref{systemKO f}), putting a special emphasis on the role of the tensor
\begin{align}\label{hat D phi}	\overline{D}^\p_{ijk}=&\frac{1}{m-2}\sq{R^\p_{ij}f_k-R^\p_{ik}f_j+\frac{1}{m-1}f_t\pa{R^\p_{tk}\delta_{ij}-R^\p_{tj}\delta_{ik}}-\frac{S^\p}{m-1}\pa{f_k\delta_{ij}-f_j\delta_{ik}}},
\end{align}
which is well defined for $m\geq 3$.
Note that we have the validity of the  symmetries
\begin{align*}
	\overline{D}^\p_{ijk}=-\overline{D}^\p_{ikj}, \qquad \overline{D}^\p_{iik}=0
\end{align*}
and
\begin{align*}
	\overline{D}^{\p}_{ijk}+\overline{D}^{\p}_{kij}+\overline{D}^{\p}_{jki}=0;
\end{align*}
moreover, it is an easy matter to check that when \eqref{systemKO f} is satisfied, we can express $\overline{D}^\p$ in the form
\begin{align}\label{D hat phi when system KO holds}
	\overline{D}^\p_{ijk}=\frac{1}{m-2}\sq{f_{ik}f_j-f_{ij}f_k+\frac{1}{m-1}f_t\pa{f_{tj}\delta_{ik}-f_{tk}\delta_{ij}}-\frac{\Delta f}{m-1}\pa{f_j\delta_{ik}-f_k\delta_{ij}}}.
\end{align}
Note that $\lambda$ does not appear in \eqref{D hat phi when system KO holds}.
\begin{proposition}[First and Second integrability conditions]\label{KO: 1 st and 2 nd integrability}
	Let $(M,g)$ be a manifold of dimension $m\geq 3$. Let $\p:(M,g)\ra (N,h)$, $U:(N,h)\ra \erre$, $\lambda:(M,g)\ra \erre$ be smooth maps, $\alpha\in \erre\setminus \set{0}$ and let $f\in C^{\infty}(M)$ be a solution of \eqref{systemKO f} on $\mathrm{int}(M)$. We then have
	\begin{align}\label{1st int cond}
		\sq{1+\eta(m-2)}\overline{D}^\p_{ijk}=C^\p_{ijk}+f_tW^\p_{tijk}+\frac{U^a}{m-1}\pa{\p^a_j\delta_{ik}-\p^a_k\delta_{ij}}
	\end{align}
	and
	\begin{align}\label{2nd int cond}
		(m-2)B^\p_{ij}=&[1+\eta(m-2)]\set{\overline{D}^\p_{ijk,k}-\frac{\alpha}{m-2}\p^a_{ss}\p^a_if_j}+\frac{m-3}{m-2}f_kC^\p_{jik}-\eta W^\p_{tijk}f_tf_k\\ \nonumber
		&+\frac{U^a}{m-1}\sq{(m-2)\p^a_{ij}-\frac{1}{m-2}\p^a_{ss}\delta_{ij}}\\ \nonumber
		&+\frac{U^{ab}}{m-1}\sq{\p^a_k\p^b_k\delta_{ij}-m\p^a_i\p^b_j}\\ \nonumber
		&+\eta f_jU^a\p^a_i.
	\end{align}
\end{proposition}
\begin{rem}
	When $U$ is constant, equations (\ref{1st int cond}) and (\ref{2nd int cond}) reduce to the first and second integrability conditions of a Riemannian manifold endowed with a non-trivial Einstein-type structure (see \cite{ACR}).
\end{rem}
\begin{rem}\label{KO: remark su 1st int cond e deformazione conforme}
	Formula (\ref{1st int cond}) is strictly related to the conformal change of metric given, for $m\geq 3$, by
	\begin{align*}
		\tilde{g}=e^{-\frac{2}{m-2}f}g.
	\end{align*}
	In view of the conformal transformation law \eqref{conf law for cotton}, that is
	\begin{align*}
		\tilde{C}^{\tilde{\p}}=C^{\p}+W^{\p}(\nabla f,\cdot,\cdot,\cdot),
	\end{align*}
	equation (\ref{1st int cond}) takes the form
	\begin{align*}
		[1+\eta(m-2)]\overline{D}^{\p}=\tilde{C}^{\tilde{\p}}+\frac{1}{2(m-1)}\diver_1\pa{U(\p)g\KN g}.
	\end{align*}
	Note that, when
	\begin{align*}
		2(m-1)\tilde{C}^{\tilde{\p}}=-\diver_1\pa{U(\p)g\KN g},
	\end{align*}
	then $\overline{D}^{\p}=0$.
\end{rem}

\begin{rem}
	Note that, since $B^\p_{ij}$  (see \cite[Proposition 2.38]{ACR}) and the terms on the second and third lines of  \eqref{2nd int cond} are symmetric, then
	\begin{align*}
		[1+\eta(m-2)]\set{\overline{D}^\p_{ijk,k}-\frac{\alpha}{m-2}\p^a_{ss}\p^a_if_j}+\frac{m-3}{m-2}f_kC^\p_{ijk}-\eta W^\p_{tijk}f_tf_k+\eta f_jU^a\p^a_i
	\end{align*}
	has to be symmetric with respect to the indices $i$ and $j$.
\end{rem}
\begin{proof}[Proof (of Proposition \ref{KO: 1 st and 2 nd integrability})]
	To obtain equation (\ref{1st int cond}), we need to find a ``good'' expression for $C^{\p}$. We use the first equation of system (\ref{systemKO f}) and the Ricci commutation relations to deduce
	\begin{align*}
		C^{\p}_{ijk}=&R^{\p}_{ij,k}-R^{\p}_{ik,j}-\frac{1}{2(m-1)}\pa{S^{\p}_k\delta_{ij}-S^{\p}_j\delta_{ik}}\nonumber \\
		=&-f_{ijk}+f_{ikj}+\eta\pa{f_{ik}f_j+f_if_{jk}-f_{ij}f_k-f_{i}f_{jk}}\nonumber\\
		&+\lambda_k\delta_{ij}-\lambda_j\delta_{ik}-\frac{1}{2(m-1)}\pa{S^{\p}_k\delta_{ij}-S^{\p}_j\delta_{ik}} \nonumber \\
		=&-f_pR_{pijk}+\eta(f_{ik}f_j-f_{ij}f_k)+\pa{\lambda_k-\frac{1}{2(m-1)}S^{\p}_k}\delta_{ij}\\ \nonumber
		&-\pa{\lambda_j-\frac{1}{2(m-1)}S^{\p}_j}\delta_{ik},
	\end{align*}
	that is,
	\begin{align}\label{KO: 1 int: prima equazione}
		C^{\p}_{ijk}=&-f_pR_{pijk}+\eta(f_{ik}f_j-f_{ij}f_k)+\pa{\lambda_k-\frac{1}{2(m-1)}S^{\p}_k}\delta_{ij}\\
		&-\pa{\lambda_j-\frac{1}{2(m-1)}S^{\p}_j}\delta_{ik}. \nonumber
	\end{align}
	Now we contract (\ref{KO: 1 int: prima equazione}) with respect to $i$ and $j$ and recall that $C^{\p}_{ttk}=\alpha \p^a_{tt}\p^a_k$, to obtain
	\begin{align*}
		\alpha \p^a_{tt}\p^a_k=f_tR_{tk}+\eta(f_{tk}f_t-\Delta f f_k)+(m-1)\lambda_k-\frac{1}{2}S^{\p}_k.
	\end{align*}
	Using the definition of $\ric^{\p}$ and rearranging the terms we get
	\begin{align*}
		\frac{1}{2}S^{\p}_k-(m-1)\lambda_k=f_tR^{\p}_{tk}+\eta(f_{tk}f_t-\Delta ff_k)-\alpha \p^a_k(\p^a_{tt}-\p^a_tf_t).
	\end{align*}
	Inserting the second equation of (\ref{systemKO f}) into the above formula we deduce
	\begin{align}\label{KO: 1 int: seconda equazione}
		\frac{1}{2}S^{\p}_k-(m-1)\lambda_k=f_tR^{\p}_{tk}+\eta(f_{tk}f_t-\Delta ff_k)-U^a\p^a_k,
	\end{align}
	while using (\ref{KO: 1 int: seconda equazione}) into (\ref{KO: 1 int: prima equazione}) we obtain
	\begin{align}\label{KO: 1 int: terza equazione}
		C^{\p}_{ijk}=&-f_pR_{pijk}+\eta(f_{ik}f_j-f_{ij}f_k)\\
		&-\frac{\delta_{ij}}{m-1}\sq{f_tR^{\p}_{tk}+\eta \pa{f_{tk}f_t-\Delta ff_k}-U^a\p^a_k}\nonumber\\
		&+\frac{\delta_{ik}}{m-1}\sq{f_tR^{\p}_{tj}+\eta \pa{f_{tj}f_t-\Delta ff_j}-U^a\p^a_j}. \nonumber
	\end{align}
	We now multiply the first equation of system (\ref{systemKO f}) by $\eta f_k$, that is
	\begin{align*}
		\eta f_{ij}f_k=-\eta R^\p_{ij}f_k+\eta f_if_jf_k+\eta\lambda f_{k}\delta_{ij},
	\end{align*}
	and skew-symmetrize the latter equation, in order to deduce
	\begin{align}\label{KO: 1 int: quarta equazione}
		\eta(f_{ik}f_j-f_{ij}f_k)=\eta(R^{\p}_{ij}f_k-R^{\p}_{ik}f_j+\lambda f_j\delta_{ik}-\lambda f_k\delta_{ij});
	\end{align}
	contracting \eqref{KO: 1 int: quarta equazione} with respect to $i$ and $j$ we also have
	\begin{align}\label{KO: 1 int: quinta equazione}
		\eta(f_{tk}f_t-\Delta f f_k)=\eta(S^{\p}f_k-f_tR^{\p}_{tk}-(m-1)\lambda f_k).
	\end{align}
Using (\ref{KO: 1 int: quarta equazione}) and (\ref{KO: 1 int: quinta equazione}) into (\ref{KO: 1 int: terza equazione}) we obtain
	\begin{align*}
		C^{\p}_{ijk}=&-f_pR_{pijk}+\eta(R^{\p}_{ij}f_k-R^{\p}_{ik}f_j+\lambda f_j\delta_{ik}-\lambda f_k\delta_{ij})\\
		&-\frac{\delta_{ij}}{m-1}\sq{\pa{1-\eta}f_tR^{\p}_{tk}+\eta S^{\p}f_k-(m-1)\eta\lambda f_k-U^a\p^a_k}\\
		&+\frac{\delta_{ik}}{m-1}\sq{\pa{1-\eta}f_tR^{\p}_{tj}+\eta S^{\p}f_j-(m-1)\eta\lambda f_j-U^a\p^a_j},
	\end{align*}
	which can be written as
	\begin{align}\label{KO: 1 int: sesta equazione}
		C^{\p}_{ijk}=&-f_pR_{pijk}+\eta(R^{\p}_{ij}f_k-R^{\p}_{ik}f_j)\\ \nonumber
		&-\frac{\delta_{ij}}{m-1}\sq{\pa{1-\eta}f_tR^{\p}_{tk}+\eta S^{\p}f_k-U^a\p^a_k}\\ \nonumber
		&+\frac{\delta_{ik}}{m-1}\sq{\pa{1-\eta}f_tR^{\p}_{tj}+\eta S^{\p}f_j-U^a\p^a_j}.
	\end{align}
	Now we use the definition of $W^{\p}$ to get
	\begin{align*}
		f_pR_{pijk}=&f_pW^{\p}_{pijk}+\frac{f_j}{m-2}R^{\p}_{ik}+\frac{f_pR^{\p}_{pj}}{m-2}\delta_{ik}-\frac{f_k}{m-2}R^{\p}_{ij}\\
		&-\frac{f_pR^{\p}_{pk}}{m-2}\delta_{ij}-\frac{S^{\p}}{(m-1)(m-2)}\pa{f_j\delta_{ik}-f_k\delta_{ij}}.
	\end{align*}
	Inserting the previous relation into (\ref{KO: 1 int: sesta equazione}) we obtain, after some simplifications,
	\begin{align*}
		C^{\p}_{ijk}=&\frac{1+\eta(m-2)}{m-2}\sq{R^{\p}_{ij}f_k-R^{\p}_{ik}f_j-\frac{f_t}{m-1}\pa{R^{\p}_{tj}\delta_{ik}-R^{\p}_{tk}\delta_{ij}}}\\
		&+\frac{1+\eta(m-2)}{(m-1)(m-2)}S^{\p}\pa{f_j\delta_{ik}-f_k\delta_{ij}}-f_pW^{\p}_{pijk}+\frac{U^a}{m-1}\p^a_k\delta_{ij}\\
		&-\frac{U^a}{m-1}\p^a_j\delta_{ik}.
	\end{align*}
	Using the definition (\ref{hat D phi}) of $\overline{D}^{\p}$ we get
	\begin{align*}
		C^{\p}_{ijk}=\sq{1+\eta(m-2)}D^{\p}_{ijk}-f_pW^{\p}_{pijk}+\frac{U^a}{m-1}\p^a_k\delta_{ij}-\frac{U^a}{m-1}\p^a_j\delta_{ik},
	\end{align*}
	and therefore (\ref{1st int cond}). \\
	To obtain \eqref{2nd int cond}, we compute $\diver_3$ of \eqref{1st int cond}:
	\begin{align}\label{3.23 2nd}
		[1+\eta(m-2)]\overline{D}^\p_{ijk,k}=&C^\p_{ijk,k}+f_{tk}W^\p_{tijk}+f_tW^\p_{tijk,k}+\frac{U^a}{m-1}\pa{\p^a_{ij}-\p^a_{kk}\delta_{ij}}\\
		&+\frac{U^{ab}}{m-1}\p^b_k(\p^a_j\delta_{ik}-\p^a_k\delta_{ij}). \nonumber
	\end{align}
	From \eqref{phi weyl e phi cotton}, we have
	\begin{align*}
		W^\p_{tijk,k}=\frac{m-3}{m-2}C^\p_{jti}+\alpha(\p^a_{ij}\p^a_t-\p^a_{jt}\p^a_i)+\frac{\alpha}{m-2}\p^a_{ss}(\p^a_i\delta_{jt}-\p^a_t\delta_{ij}).
	\end{align*}
	Inserting the previous relation into \eqref{3.23 2nd} and using \eqref{systemKO f} we get
	\begin{align*}
		&[1+\eta(m-2)]\overline{D}^\p_{ijk,k}=C^\p_{ijk,k}+W^\p_{tijk}\pa{-R^\p_{tk}+\eta f_tf_k+\lambda\delta_{tk}}\\
		&\qquad+f_t\sq{\pa{\frac{m-3}{m-2}}C^\p_{jti}+\alpha\pa{\p^a_{ij}\p^a_t-\p^a_{jt}\p^a_i}+\frac{\alpha}{m-2}\p^a_{ss}\pa{\p^a_i\delta_{jt}-\p^a_t\delta_{ij}}}\\
		&\qquad+\frac{U^a}{m-1}\pa{\p^a_{ij}-\p^a_{kk}\delta_{ij}}+\frac{U^{ab}}{m-1}\p^b_k(\p^a_j\delta_{ik}-\p^a_k\delta_{ij})\\
		&\quad=C^\p_{ijk,k}-W^\p_{tijk}R^\p_{tk}+\eta W^\p_{tijk}f_tf_k+\lambda W^\p_{tijk}\delta_{tk}+\pa{\frac{m-3}{m-2}}f_kC^\p_{jki}\\
		&\qquad+\alpha\pa{\p^a_{ij}f_t\p^a_t-\p^a_{jk}f_k\p^a_i}+\frac{\alpha}{m-2}\p^a_{ss}\pa{\p^a_if_j-\p^a_kf_k\delta_{ij}}\\
		&\qquad+\frac{U^a}{m-1}\pa{\p^a_{ij}-\p^a_{kk}\delta_{ij}}+\frac{U^{ab}}{m-1}\p^b_k(\p^a_j\delta_{ik}-\p^a_k\delta_{ij}).
	\end{align*}
	Using the definition of the $\p$-Bach tensor (see \eqref{1.20Bachphi}) we deduce
	\begin{align}\label{3.24 2nd}
		&[1+\eta(m-2)]\overline{D}^\p_{ijk,k}=(m-2)B^\p_{ij}+\alpha R^\p_{tj}\p^a_t\p^a_i-\alpha\pa{\p^a_{ij}\p^a_{kk}-\p^a_{kkj}\p^a_i}\\ \notag
		&\quad+\frac{\alpha}{m-2}\abs{\tau(\p)}^2\delta_{ij}
		+\eta W^\p_{tijk}f_tf_k+\lambda W^\p_{tijk}\delta_{tk}+\frac{m-3}{m-2}f_kC^\p_{jki}\\
		&\quad+\alpha\pa{\p^a_{ij}f_t\p^a_t-\p^a_{jk}f_k\p^a_i}\notag
		+\frac{\alpha}{m-2}\p^a_{ss}\pa{\p^a_if_j-\p^a_kf_k\delta_{ij}}\\
		&\quad+\frac{U^a}{m-1}\pa{\p^a_{ij}-\p^a_{kk}\delta_{ij}}
		+\frac{U^{ab}}{m-1}\pa{\p^a_i\p^b_j-\p^a_k\p^b_k\delta_{ij}}. \notag
	\end{align}
	Observing that, by \eqref{systemKO f},
	\begin{align*}
		R^\p_{tj}\p^a_t=-f_{tj}\p^a_t+\eta \p^a_{ss}f_j-\frac{\eta}{\alpha}U^af_j+\lambda \p^a_j,
	\end{align*}
	and using \eqref{systemKO f} ii) and equation \eqref{traccia di weyl phi}, \eqref{3.24 2nd} yields
	\begin{align*}
		[1+\eta(m-2)]\overline{D}^\p_{ijk,k}&=(m-2)B^\p_{ij}-\alpha f_{tj}\p^a_t\p^a_i+\alpha\eta f_j\p^a_{ss}\p^a_i- \eta f_jU^a\p^a_i\\
		&+\alpha \lambda \p^a_i\p^a_j-\alpha\pa{\p^a_{ij}\p^a_{kk}-\p^a_{kkj}\p^a_i-\frac{\abs{\tau(\p)}^2}{m-2}\delta_{ij}}+\eta f_tf_kW^\p_{tijk}\\
		&-\alpha\lambda\p^a_i\p^a_j+f_t\pa{\frac{m-3}{m-2}}C^\p_{jti}+\alpha\pa{\p^a_{ij}\p^a_{ss}-\frac{1}{\alpha}\p^a_{ij}U^a-f_t\p^a_{jt}\p^a_i}\\
		&+\frac{\alpha}{m-2}\p^a_{ss}\pa{f_j\p^a_i-\p^a_{tt}\delta_{ij}+\frac{1}{\alpha}U^a\delta_{ij}}+\frac{U^a}{m-1}\pa{\p^a_{ij}-\p^a_{kk}\delta_{ij}}\\
		&+\frac{U^{ab}}{m-1}\pa{\p^a_j\p^b_i-\p^a_k\p^b_k\delta_{ij}}\\
		&=(m-2)B^\p_{ij}-\alpha f_{tj}\p^a_t\p^a_i+\alpha\frac{1+\eta(m-2)}{m-2}f_j\p^a_{ss}\p^a_i\\&-\eta f_jU^a\p^a_i+\alpha\p^a_{kkj}\p^a_i
		+\eta f_tf_kW^\p_{tijk}+f_t\pa{\frac{m-3}{m-2}}C^\p_{jti}\\&-\alpha f_t\p^a_{jt}\p^a_i+\frac{1}{(m-1)(m-2)}U^a\p^a_{kk}\delta_{ij}\\
		&-\pa{\frac{m-2}{m-1}}U^a\p^a_{ij}+\frac{U^{ab}}{m-1}\p^a_j\p^b_i-\frac{U^{ab}}{m-1}\p^a_k\p^b_k\delta_{ij},
	\end{align*}
	that is,
	\begin{align}\label{quasi 2nd int cond}
		&(m-2)B^\p_{ij}=[1+\eta(m-2)]\overline{D}^\p_{ijk,k}+\alpha f_{tj}\p^a_t\p^a_i-\alpha\frac{1+\eta(m-2)}{m-2}f_j\p^a_{ss}\p^a_i\\ \notag
		&\quad+\eta f_jU^a\p^a_i-\alpha\p^a_{kkj}\p^a_i
		-\eta f_tf_kW^\p_{tijk}-f_t\pa{\frac{m-3}{m-2}}C^\p_{jti}+\alpha f_t\p^a_{jt}\p^a_i\\ \notag
		&\quad-\frac{1}{(m-1)(m-2)}U^a\p^a_{kk}\delta_{ij}
		+\pa{\frac{m-2}{m-1}}U^a\p^a_{ij}-\frac{U^{ab}}{m-1}\p^a_j\p^b_i\\
		&\quad+\frac{U^{ab}}{m-1}\p^a_k\p^b_k\delta_{ij}. \notag
	\end{align}
	Taking the covariant derivative of \eqref{systemKO f} ii), we obtain
	\begin{align*}
		\alpha\p^a_i\p^a_{kkj}=\p^b_j\p^a_iU^{ab}+\alpha\p^a_i\p^a_tf_{tj}+\alpha\p^a_i\p^a_{tj}f_t.
	\end{align*}
	Using the equation above we rewrite \eqref{quasi 2nd int cond} as
	\begin{align*}
		(m-2)B^\p_{ij}=&[1+\eta(m-2)]\overline{D}^\p_{ijk,k}+\alpha f_{tj}\p^a_t\p^a_i-\alpha\frac{1+\eta(m-2)}{m-2}f_j\p^a_{ss}\p^a_i\\
		&+\eta f_jU^a\p^a_i-\p^a_j\p^b_iU^{ab}-\alpha \p^a_i\p^a_tf_{jt}-\alpha\p^a_i\p^a_{tj}f_t-\eta f_tf_kW^\p_{tijk}\\
		&-f_t\pa{\frac{m-3}{m-2}}C^\p_{jti}
		+\alpha f_t\p^a_{jt}\p^a_i-\frac{1}{(m-1)(m-2)}U^a\p^a_{kk}\delta_{ij}\\
		&+\frac{m-2}{m-1}U^a\p^a_{ij}-\frac{U^{ab}}{m-1}\p^a_j\p^a_i+\frac{U^{ab}}{m-1}\p^a_k\p^b_k\delta_{ij}\\
		=&[1+\eta(m-2)]\overline{D}^\p_{ijk,k}-\alpha\frac{1+\eta(m-2)}{m-2}f_j\p^a_{ss}\p^a_i+\eta f_jU^a\p^a_i\\
		&-\frac{m}{m-1}\p^a_j\p^b_iU^{ab}-\eta f_kf_t W^\p_{tijk}-f_t\pa{\frac{m-3}{m-2}}C^\p_{jti}\\
		&-\frac{1}{(m-1)(m-2)}U^{a}\p^a_{kk}\delta_{ij}+\frac{m-2}{m-1}U^a\p^a_{ij}+\frac{U^{ab}}{m-1}\p^a_k\p^b_k\delta_{ij}.
	\end{align*}
\end{proof}
\begin{rem}
	It is interesting to compare \eqref{1st int cond} to an equation appearing in the context of \emph{Cotton Gravity}, a gravitational theory that, as we already discussed in the Introduction, has been recently introduced by Harada (\cite{Harada}).
	One can prove, after some computations, that, on a static space-time $\hat{M}=M\times_u I$ of metric
	\begin{align*}
		\hat{g}=-u^2dt\otimes dt+g
	\end{align*}
	for some Riemannian metric $g$ on $M$ and $u\in C^{\infty}(M)$,
	 the field equations of Cotton Gravity for the stress-energy tensor \eqref{T}
	of the Introduction with conservation laws
	\begin{align*}
		\begin{cases}
			\pa{\mu+p}\nabla u=-u\nabla p,\\
			\alpha h\pa{\tau(\hat{\p}),d\hat{\p}}=h\pa{\nabla U,d\hat{\p}}
		\end{cases}
	\end{align*}
	 reduce to
	\begin{align}\label{Cotton phi perfect fluids}
		\begin{cases}
			C^\p_{ijk}+f_tW^\p_{tijk}-\frac{U^a\p^a_k}{m-1}\delta_{ij}+\frac{U^a\p^a_j}{m-1}\delta_{ik}-D^A_{ijk}-(m-2)D^B_{ijk}&=0,\\
			-\frac{m-1}{m}f_{tti}+\frac{m-2}{m}f_{ti}f_t+\Delta f f_i+\frac{1}{2m}S^\p_i-R^\p_{it}f_t+\frac{U^a\p^a_i}{m}-\frac{m-1}{m}\mu_i&=0,
		\end{cases}
	\end{align}
	where $f=-\log (u)$, $D^A$ is given by the right hand side of \eqref{hat D phi} and $D^B$ is given by the right hand side of \eqref{D hat phi when system KO holds}. Note that, for a $\p$-SPFST, equation \eqref{systemKO f} i) with $\eta=1$ is satisfied so that, as we already discussed, $D^A=D^B=\ol{D}^\p$ and \eqref{Cotton phi perfect fluids} i) reduces to \eqref{1st int cond}.\\
	It is also interesting to observe that, for the stress energy tensor
	\begin{align*}
		\hat{T}=\alpha\hat{\p}^*h-\pa{\alpha \frac{\abs{d\hat{\p}}^2}{2}+U(\hat{\p})}\hat{g},
	\end{align*}
	 the field equations of Cotton Gravity become, for a general $m$-dimensional Lorentzian manifold $(\hat{M},\hat{g})$,
	 \begin{align*}
	 	\hat{C}^{\hat{\p}}=-\frac{1}{2(m-1)}\diver_1\pa{U(\hat{\p})\hat{g}\KN \hat{g}}.
	 \end{align*}
	 The above equation, in Riemannian signature, will appear several times in the rest of this paper, starting from Proposition \ref{Cotton phi prop} .
\end{rem}
\section{The Geometry of the Level Sets of $f$}
Let $(M,g)$ and $(N,h)$ be two Riemannian manifolds, let $\p:(M,g)\ra (N,h)$ be a smooth map, $f,\lambda\in C^{\infty}(M)$ and $\eta\in \erre$ such that
\begin{align}\label{A.1}
	\ric^\p+\hs(f)- \eta df\otimes df=\lambda g.
\end{align}
The aim of this subsection is to understand how the vanishing of the tensor $\overline{D}^{\p}$ affects the geometry of the level sets of $f$.
From now on, indices in lower case letters $i,j,k,...$ will run between $1$ and $m$, while indices in upper case letters $A,B,...$ will run between $1$ and $m-1$.
\begin{proposition}\label{Pr4.6frame}
	Let $(M,g)$ be a Riemannian manifold, $f\in C^{\infty}(M)$ satisfying \eqref{A.1},  $c$ be a regular value of $f$ and   $\Sigma$ the corresponding level set. For each $p\in\Sigma$, let $\set{e_i}_{i=1}^m$ be a local frame such that $e_1,...,e_{m-1}$ are tangent to $\Sigma$ and
	\begin{align*}
		e_m=\frac{\nabla f}{\abs{\nabla f}}
	\end{align*}
	denotes a unit normal with respect to $\Sigma$.
	Then, at $p\in \Sigma$, we have the validity of
	\begin{align}\label{A.2}
		\frac{(m-2)^2}{2\abs{\nabla f}^2}\abs{\overline{D}^\p}^2=&\abs{\mathring{\se}}^2\abs{\nabla f}^2+\pa{\frac{m-2}{m-1}}\frac{1}{|\nabla f|^2}g\pa{\ric^\p\pa{\nabla f, \cdot}^{\sharp},\ric^\p\pa{\nabla f, \cdot}^{\sharp}}\\
		&-\pa{\frac{m-2}{m-1}}\frac{1}{|\nabla f|^4}\pa{\ric^{\p}\pa{\nabla f,\nabla f}}^2, \nonumber
	\end{align}
	where $\mathring{\se}$ is the traceless part of the second fundamental form of $\Sigma$ with respect to the inward unit normal.
\end{proposition}
\begin{proof}
	For the sake of simplicity, we divide the proof in two steps.\\
	\textbf{Step 1}. We start by proving the validity of
	\begin{align}\label{A.3}
		\frac{(m-2)^2}{2\abs{\nabla f}^2}\abs{\overline{D}^\p}^2=&\abs{\ric^\p}^2-\frac{m}{m-1}R^\p_{Am}R^\p_{Am}-\frac{m}{m-1}(R^\p_{mm})^2\\
		&+\frac{2}{m-1}S^\p R^\p_{mm}-\frac{1}{m-1}(S^\p)^2. \nonumber
	\end{align}
	Using \eqref{hat D phi} together with the fact that $\overline{D}^\p$ is totally trace free and skew symmetric in the last two indices, we immediately get
	\begin{align*}
		\abs{\overline{D}^\p}^2=&\overline{D}^\p_{ijk}\ol{D}^\p_{ijk}\\
				=&\frac{2}{m-2}\ol{D}^\p_{ijk}R^\p_{ij}f_k,
	\end{align*}
	that is,
	\begin{align}\label{A.5}
		\frac{m-2}{2}\abs{\overline{D}^\p}^2=\ol{D}^\p_{ijk}R^\p_{ij}f_k.
	\end{align}
	Then, we insert \eqref{hat D phi} into \eqref{A.5} to get
	\begin{align*}
		\frac{m-2}{2}\abs{\overline{D}^\p}^2=&\ol{D}^\p_{ijk}R^\p_{ij}f_k\\
		=&\frac{R^\p_{ij}f_k}{m-2}\bigg[R^\p_{ij}f_k-R^\p_{ik}f_j+\frac{1}{m-1}f_t\pa{R^\p_{tk}\delta_{ij}-R^\p_{tj}\delta_{ik}}\\
		&-\frac{S^\p}{m-1}\pa{f_k\delta_{ij}-f_j\delta_{ik}}\bigg]\\
		=&\frac{1}{m-2}\bigg[\abs{\ric^\p}^2\abs{\nabla f}^2-\abs{\nabla f}^2R^\p_{im}R^\p_{im}+\frac{S^\p}{m-1}R^\p_{mm}\abs{\nabla f}^2\\
		&-\frac{1}{m-1}\abs{\nabla f}^2R^\p_{im}R^\p_{im}-\frac{(S^\p)^2}{m-1}\abs{\nabla f}^2+\frac{S^\p}{m-1}\abs{\nabla f}^2R^\p_{mm}\bigg]\\
		=&\frac{\abs{\nabla f}^2}{m-2}\sq{\abs{\ric^\p}^2-\frac{m}{m-1}R^\p_{im}R^\p_{im}+\frac{2}{m-1}S^\p R^\p_{mm}-\frac{(S^\p)^2}{m-1}}\\
		=&\frac{\abs{\nabla f}^2}{m-2}\bigg[\abs{\ric^\p}^2-\frac{m}{m-1}R^\p_{Am}R^\p_{Am}-\frac{m}{m-1}(R^\p_{mm})^2\\
		&+\frac{2}{m-1}S^\p R^\p_{mm}-\frac{(S^\p)^2}{m-1}\bigg],
	\end{align*}
	that is, \eqref{A.3}.\\
	\textbf{Step 2} With respect to the given frame $\set{e_i}_{i=1}^m$ we have
	\begin{align*}
		f_A=0,&&f_m=\abs{\nabla f},
	\end{align*}
	where $A=1,...,m-1$. Let $\se_{AB}$ be the second fundamental form of $\Sigma$; then, by \cite[Proposition 8.1]{CMMRGet}, we have
	\begin{align*}
		\se_{AB}=-\theta^m_A(e_B)=-\frac{f_{AB}}{\abs{\nabla f}}.
	\end{align*}
	Moreover, by \eqref{A.1}, we get
	\begin{align}\label{A.7}
		\se_{AB}&=\frac{1}{\abs{\nabla f}}\pa{R^\p_{AB}- \eta f_Af_B-\lambda\delta_{AB}}\\
		&=\frac{1}{\abs{\nabla f}}\pa{R^\p_{AB}-\lambda\delta_{AB}};\nonumber
	\end{align}
	furthermore, the (normalized) mean curvature of $\Sigma$ is
	\begin{align*}
		H=\frac{\se_{AA}}{m-1},
	\end{align*}
	so that, by \eqref{A.7}, we deduce
	\begin{align}\label{A.7.1}
		H=\frac{1}{\abs{\nabla f}}\pa{\frac{S^\p-R^\p_{mm}}{m-1}-\lambda}
	\end{align}
	and
	\begin{align}\label{A.7.2}
		\mathring{\se}_{AB}&=\se_{AB}-H\delta_{AB}\notag\\
		&=\frac{1}{\abs{\nabla f}}\pa{R^\p_{AB}-\frac{S^\p-R^\p_{mm}}{m-1}\delta_{AB}}.
	\end{align}
	Computing the norm we have
	\begin{align}\label{A.6}
		\abs{\nabla f}^2\abs{\mathring{\se}}^2=&R^\p_{AB}R^\p_{AB}+\frac{1}{m-1}\pa{S^\p-R^\p_{mm}}^2-\frac{2}{m-1}R^\p_{AA}\pa{S^\p-R^\p_{mm}}\\ \notag
		=&\abs{\ric^\p}^2-2R^\p_{Am}R^\p_{Am}-R^\p_{mm}R^\p_{mm}+\frac{1}{m-1}\pa{S^\p-R^\p_{mm}}^2\\ \notag
		&-\frac{2}{m-1}\pa{S^\p-R^\p_{mm}}^2\notag\\
		=&\abs{\ric^\p}^2-2R^\p_{Am}R^\p_{Am}-\frac{m}{m-1}R^\p_{mm}R^\p_{mm}-\frac{(S^\p)^2}{m-1}+\frac{2}{m-1}S^\p R^\p_{mm}. \notag
	\end{align}
	Hence, inserting \eqref{A.6} into \eqref{A.3} we obtain \eqref{A.2}.
\end{proof}
\begin{cor}\label{cor: equivalence betw autov e rett}
	Let $(M,g)$ be a Riemannian manifold and let $f$ be a solution of \eqref{A.1}. Then the tensor $\overline{D}^{\p}$ vanishes identically on $M$ if and only if the following two conditions are simultaneously satisfied:
	\begin{itemize}
		\item[i)] any regular level set of $f$ is totally umbilical;
		\item[ii)] for any regular point $p$ of  $f$, $\nabla f_p$ is an eigenvector of $\ric^\p_p$.
	\end{itemize}
\end{cor}
\begin{proof}
	From \eqref{A.2}, we only need to prove that $R^\p_{Am}=0$, $A=1,...,m-1$, if and only if $\nabla f$ is an eigenvector of $\ric^\p$; this fact follows by the choice of the frame: indeed,
	\begin{align*}
		\ric^\p(\nabla f,\cdot)^{\sharp}=\abs{\nabla f}R^\p_{im}e_i.
	\end{align*}
	Therefore, $R^\p_{Am}=0$ for each $A$ if and only if $\ric^\p(\nabla f,\cdot)^{\sharp}$ is proportional to $e_m=\frac{\nabla f}{\abs{\nabla f}}$, that is, if and only if $\nabla f$ is an eigenvector of $\ric^\p$. Moreover, \eqref{A.2} reads, in components,
	\begin{align*}
	\frac{(m-2)^2}{2\abs{\nabla f}^2}\abs{\overline{D}^\p}^2=\abs{\mathring{\se}}^2\abs{\nabla f}^2+\frac{m-2}{m-1}R^\p_{Am}R^\p_{Am}	
	\end{align*}
	 so that $\overline{D}^\p$ vanishes if and only if $\mathring{\se}$ and $R^\p_{Am}$ are zero.
\end{proof}
\begin{rem}
	From   Corollary \ref{cor: equivalence betw autov e rett} we deduce that condition ii) is necessary to the validity of $\overline{D}^{\p}\equiv 0$. This condition, that will be assumed in Theorem \ref{teo: KO: Main results: bach piatto}, has some interesting geometric consequences that we explore in the next
\end{rem}
\begin{lemma}\label{KO: lemma: condizioni equivalenti a grad f autovettore}
	Let $(M,g)$ be a Riemannian manifold and  $f$  a solution of (\ref{systemKO f}). Then the following conditions are equivalent:
	\begin{itemize}
		\item[i)] at any regular point $p$ of $f$, we have that $\nabla f_p$ is an eigenvector of $\ric^{\p}_p$;
		\item[ii)] $|\nabla f|$ is constant on any regular level set of $f$;
		\item[iii)] for any fixed regular level set $\Sigma$ and $p\in \Sigma$, there exists an open neighbourhood $U$ of $p$ in $M$ and a distance function $r:U\to \erre$ such that $f$ depends only on $r$;
		\item[iv)] we have the validity of
		\begin{align*}
		\overline{D}^{\p}\pa{\nabla f, \cdot,\cdot}=0;
		\end{align*}
		\item[v)] we have the validity of
		\begin{align*}
			C^{\p}\pa{\nabla f,\cdot,\cdot}+\frac{1}{2(m-1)}\diver_1\pa{U(\p)g\KN g}\pa{\nabla f,\cdot,\cdot}=0.
		\end{align*}
	\end{itemize}
\end{lemma}
\begin{rem}
	Petersen and Wylie, in \cite{petersen2008gradientriccisolitonssymmetry}, studied Ricci solitons with a potential function satisfying condition $iii)$. There, such a function is said to be \textit{rectifiable} on $U$.
\end{rem}
\begin{proof}[Proof (of Lemma \ref{KO: lemma: condizioni equivalenti a grad f autovettore})]
	To prove the equivalence of conditions $i)$ and $ii)$ it is
	sufficient to compute, in the same frame introduced in Proposition \ref{Pr4.6frame},
	\begin{align*}
		\frac{1}{2}\abs{\nabla f}^2_A&=f_{Ai}f_i\\
		&=\abs{\nabla f}f_{Am}\\
		&=\abs{\nabla f}\pa{-R^\p_{Am}+\eta f_Af_m+\lambda \delta_{Am}}\\
		&=-\abs{\nabla f}R^\p_{Am},
	\end{align*}
	where we have used  the first equation of (\ref{systemKO f}).
	It follows that $|\nabla f|$ is constant on $\Sigma$ if and only if $R^{\p}_{Am}=0$, $\forall A\in\{1,..,m-1\}$.
	\\
It is clear that $iii)$ implies $ii)$; to see the converse, consider an open neighbourhood $U$ of $p$ in $M$ such that every point of $U$ is regular for $f$. We want to show that every integral curve $\gamma$ of $\frac{\nabla f}{\abs{\nabla f}}$ is a unit speed geodesic. Consider a local orthonormal frame $\set{e_i}_{i=1}^m$ on $U$ such that $e_m=\frac{\nabla f}{\abs{\nabla f}}$ and its dual coframe $\set{\theta^i}_{i=1}^m$. By assumptions,
\begin{align*}
	e_A(f)=0\  \text{ and }\ e_A \pa{\abs{\nabla f}}=0, \ \ \  \forall A\in \set{1,...,m-1}.
\end{align*}
The first structure equations give
\begin{align*}
	d\pa{\frac{df}{\abs{\nabla f}}}=-\theta^m_j\wedge\theta^j,
\end{align*}
so that
\begin{align*}
	d\pa{\frac{1}{\abs{\nabla f}}}\wedge df=-\theta^m_j\wedge\theta^j.
\end{align*}
Contracting the above with the local tensor field $e_A\otimes e_m$ gives
\begin{align*}
	e_A\pa{\frac{1}{\abs{\nabla f}}}e_m\pa{f}-e_m\pa{\frac{1}{\abs{\nabla f}}}e_A\pa{f}=-\theta^m_m(e_A)+\theta^m_A(e_m).
\end{align*}
Since $e_A\pa{\abs{\nabla f}}=0=e_A(f)$ and $\theta^m_m=0$ we get
\begin{align*}
	\theta^m_A(e_m)=0.
\end{align*}
By definition of covariant derivative, this implies
\begin{align*}
	\nabla_{e_m}e_m=0,
\end{align*}
so that $\gamma$ is a geodesic.
Therefore, choosing $r:U\to \erre$ to be the signed distance function from $U\cap \Sigma$, oriented accordingly to the chosen unit normal,  we obtain
\begin{align*}
	\nabla r=\frac{\nabla f}{\abs{\nabla f}}
\end{align*}
at every point of $U$,
so that $f$ only depends on $r$.
	\\
	Next we prove the equivalence of $i)$ and $iv)$.
	From (\ref{hat D phi}) we get
	\begin{align*}
		(m-2)f_i\overline{D}^{\p}_{ijk}&=f_i\sq{R^{\p}_{ij}f_k-R^{\p}_{ik}f_j-\frac{f_t}{m-1}\pa{R^{\p}_{jt}\delta_{ik}-R^{\p}_{kt}\delta_{ij}}+\frac{S^{\p}}{m-1}\pa{f_j\delta_{ik}-f_k\delta_{ij}}}\\
		&=f_iR^{\p}_{ij}f_k-f_iR^{\p}_{ik}f_j-\frac{f_t}{m-1}\pa{R^{\p}_{jt}f_k-R^{\p}_{kt}f_j}+\frac{S^{\p}}{m-1}\pa{f_jf_k-f_kf_j}\\
		&=\frac{m-2}{m-1}\pa{f_iR^{\p}_{ij}f_k-f_iR^{\p}_{ik}f_j}.
	\end{align*}
	Therefore, at any regular point of $f$,
	\begin{align}\label{insomma...}
		(m-1)\overline{D}^{\p}_{mjk}=R^{\p}_{mj}f_k-R^{\p}_{mk}f_j.
	\end{align}
	Since we have
	\begin{align*}
		f_m=|\nabla f|, \ \ f_A=0, \ \forall A\in\{1,..,m-1\},
	\end{align*}
	we deduce that the right hand side, and hence also the left hand side, of \eqref{insomma...} vanishes if $j=m=k$ or if $j,k\in\{1,..,m-1\}$, so that we are only left with
	\begin{align*}
		(m-1)\overline{D}^{\p}_{mAm}=|\nabla f|R^{\p}_{mA}.
	\end{align*}
	Therefore, $f_i\overline{D}^{\p}_{ijk}$ vanishes identically on $M$ if and only if $R^{\p}_{Am}=0$ for any regular point.
	\\
	To conclude, we prove the equivalence of $iv)$ and $v)$. This is a simple application of the first integrability condition of system (\ref{systemKO f}). Indeed, contracting (\ref{1st int cond}) with $f_i$ we deduce
	\begin{align*}
		[1+\eta(m-2)]f_i\overline{D}^{\p}_{ijk}&=f_if_tW^{\p}_{tijk}+f_i\sq{C^{\p}_{ijk}-\frac{U^a\p^a_k}{m-1}\delta_{ij}+\frac{U^a\p^a_j}{m-1}\delta_{ik}}\\
		&=f_i\sq{C^{\p}_{ijk}-\frac{U^a\p^a_k}{m-1}\delta_{ij}+\frac{U^a\p^a_j}{m-1}\delta_{ik}},
	\end{align*}
	and we conclude.
\end{proof}

\begin{proposition}\label{prop B}
	Let $(M,g)$ and $(N,h)$ be two Riemannian manifolds and let $\p:(M,g)\ra(N,h)$ be a smooth map; assume that $\overline{D}^\p\equiv 0$ and let $f$ be a solution of \eqref{systemKO f},
	where $U\in C^\infty(N)$, $\lambda\in C^\infty(M)$ and $\alpha \in \erre\setminus\set{0}$. Assume that $\p$ is $\frac{U}{\alpha}$-harmonic, i.e.
	\begin{align}\label{A.9}
		\tau(\p)=\frac{1}{\alpha}(\nabla U)(\p).
	\end{align}
	Then, for any regular value $c$ of $f$, $\abs{\nabla f}$ and $H$ are constant on every connected component of $\Sigma=f^{-1}(c)$ and $\p|_{\Sigma}$ is $\frac{U}{\alpha}$-harmonic.
	Moreover, $S^\p-(m-1)\lambda$ and $S^{\p_{|_{\Sigma}}}-(m-1)\lambda|_{\Sigma}$ are also constant on every connected component of $\Sigma$.
\end{proposition}
\begin{rem}
	Note that the constancy of $S^\p-(m-1)\lambda$ and $S^{\p_{|_{\Sigma}}}-(m-1)\lambda|_{\Sigma}$ is not needed in the proof of the Kobayashi-Obata type theorem; however, we give a proof for the sake of completeness.
\end{rem}
\begin{proof}
	By Lemma \ref{KO: lemma: condizioni equivalenti a grad f autovettore} we have that $\abs{\nabla f}$ is constant on $\Sigma$. Using Codazzi equations, $R^\p_{Am}=0$ and
	\begin{align*}
		\se_{AB}=H\delta_{AB},
	\end{align*}
	we obtain
	\begin{align*}
		(m-1)H_B&=\se_{AA,B}=\se_{AB,A}-R_{mAAB}=\se_{AB,A}+R_{mB}=H_B+R^\p_{mB}+\alpha\p^a_m\p^a_B\\
		&=H_B+\alpha\p^a_m\p^a_B;
	\end{align*}
	therefore,
	\begin{align*}
		(m-2)H_B=\alpha\p^a_m\p^a_B.
	\end{align*}
	Since
	\begin{align*}
		0=\p^a_{tt}-\frac{1}{\alpha}U^a=\p^a_if_i=\abs{\nabla f}\p^a_m,
	\end{align*}
	we have that, on a regular level set, $\p^a_m=0$ and hence
	\begin{align*}
		H_B=0.
	\end{align*}
	Since $i:\Sigma\hookrightarrow M$ is an isometric immersion, a computation shows that
	\begin{align*}
		\tau^a(\p|_{\Sigma})=\tau^a(\p)-\p^a_{mm}+(m-1)H\p^a_m,
	\end{align*}
	that is, since $\p^a_m=0$,
	\begin{align}\label{A.10}
		\tau^a(\p|_{\Sigma})=\tau^a(\p)-\p^a_{mm}.
	\end{align}
	Taking the covariant derivative of
	\begin{align*}
		0=\p^a_if_i
	\end{align*}
	we obtain
	\begin{align*}
		0=\p^a_{ij}f_i+\p^a_if_{ij}=\p^a_{mj}\abs{\nabla f}+\p^a_if_{ij}.
	\end{align*}
	When $j=m$ we have
	\begin{align}
		\p^a_{mm}\abs{\nabla f}=-\p^a_if_{im};
	\end{align}
	by \eqref{systemKO f}, and using the vanishing of $\p^a_m, f_A$ and $R^\p_{Am}$ we get
	\begin{align*}
		\p^a_{mm}\abs{\nabla f}&=-\p^a_if_{im}=\p^a_iR^\p_{im}- f_if_m\p^a_i+\lambda\p^a_m=0.		
	\end{align*}
	Since $c$ is a regular value, $\abs{\nabla f}\neq 0$, thus $\p^a_{mm}=0$ and \eqref{A.10} implies
	\begin{align*}
		\tau(\p_{|_{\Sigma}})=\frac{1}{\alpha}(\nabla U)(\p_{|_{\Sigma}}).
	\end{align*}
	To prove the constancy of $S^\p-(m-1)\lambda$, we first show the validity of
	\begin{align}\label{A.11}
		\frac{1}{2}S^\p_k=R^\p_{ik}f_i+\eta (f_{ik}f_i-\Delta f f_k)+(m-1)\lambda_k-U^a\p^a_k.
	\end{align}
	Indeed, taking the covariant derivative of \eqref{systemKO f} i), we get
	\begin{align*}
		R^\p_{ij,k}+f_{ijk}- \eta f_{ik}f_j- \eta f_if_{jk}=\lambda_k\delta_{ij};
	\end{align*}
	exchanging the role of $i$ and $j$ in the above equation
	\begin{align*}
		R^\p_{ik,j}+f_{ikj}- \eta f_{ij}f_k- \eta f_if_{kj}=\lambda_j\delta_{ik}
	\end{align*}
	and subtracting the second equation to the first one, we deduce
	\begin{align*}
		R^\p_{ij,k}-R^\p_{ik,j}=f_{ikj}-f_{ijk}-\eta (f_{ij}f_k-f_{ik}f_j)+\lambda_k\delta_{ij}-\lambda_j\delta_{ik}.
	\end{align*}
	Using the Ricci commutation relations, the previous equation reduces to
	\begin{align*}
		R^\p_{ij,k}-R^\p_{ik,j}=f_tR_{tikj}-\eta (f_{ij}f_k-f_{ik}f_j)+\lambda_k\delta_{ij}-\lambda_j\delta_{ik}.
	\end{align*}
	Then, taking the trace with respect to $i$ and $j$ and using the $\p$-Schur's identity, we get
	\begin{align*}
		\frac{1}{2}S^\p_k+\alpha\p^a_{tt}\p^a_k=f_tR_{tk}+\eta (f_{tk}f_t-\Delta ff_k)+(m-1)\lambda_k;
	\end{align*}
	now we use the definition of $\ric^\p$ and \eqref{systemKO f} ii) to obtain
	\begin{align*}
		\frac{1}{2}S^\p_k&=f_tR^\p_{tk}+\eta (f_{tk}f_t-\Delta ff_k)+(m-1)\lambda_k+\alpha(f_t\p^a_t\p^a_k-\p^a_{tt}\p^a_k)\\
		&=f_tR^\p_{tk}+\eta (f_{tk}f_t-\Delta ff_k)+(m-1)\lambda_k-U^a\p^a_k,
	\end{align*}
	that is, \eqref{A.11}.\\
	\noindent
	Taking $k=B$ in \eqref{A.11} we have
	\begin{align}\label{constancy of Sp + roba}
		\frac{1}{2}S^\p_B=&R^\p_{Bj}f_j+\eta(f_{Bj}f_j-\Delta ff_B)+(m-1)\lambda_B-U^a\p^a_B\notag\\
		=&(m-1)\lambda_B-U^a\p^a_B,
	\end{align}
	where, in the last equality, we have used the fact that $R^\p_{Bm}=0,\,f_B=0$ and
	\begin{align*}
		f_{mB}=-R^\p_{mB}+ \eta f_Bf_m+\lambda\delta_{mB}=0,
	\end{align*}
	which is a consequence of the first equation of \eqref{systemKO f}. Therefore, by \eqref{constancy of Sp + roba}, we have that
	$S^\p-(m-1)\lambda+U(\p)$ is constant on every connected component of $\Sigma$. By \eqref{A.7.1} we have
	\begin{align*}
		(m-1)\abs{\nabla f}^2H&=S^\p-(m-1)\lambda-R^\p_{mm}\\
		&=\pa{\frac{S^\p}{2}-(m-1)\lambda +U(\p)}+\frac{S^\p}{2}-R^\p_{mm}-U(\p).
	\end{align*}
	Thus, $ $$\frac{S^\p}{2}-R^\p_{mm}-U(\p)$ is constant on every connected component of $\Sigma$. Since $\overline{D}^\p\equiv 0$, we have that $\mathring{\se}\equiv 0$; then, by \eqref{A.7.2}, we have
	\begin{align}\label{A.13}
		R^\p_{AB}=\frac{1}{m-1}(S^\p-R^\p_{mm})\delta_{AB}
	\end{align}
	that, together with \eqref{systemKO f} ii), gives
	\begin{align*}
		f_{AA}=-R^\p_{AA}+ \eta f_Af_A+(m-1)\lambda=-S^\p+R^\p_{mm}+(m-1)\lambda=-(m-1)\abs{\nabla f}H.
	\end{align*}
	Therefore, using \eqref{A.11} with $k=m$, $\p^a_m=0$, $R^\p_{Am}=0$ and
	\begin{align*}
		\se_{AB}=-\frac{f_{AB}}{\abs{\nabla f}},
	\end{align*}
	we deduce
	\begin{align*}
		\frac{1}{2}S^\p_m&=R^\p_{mm}+\eta (f_{mm}\abs{\nabla f}-\Delta f \abs{\nabla f})+(m-1)\lambda_m-U^a\p^a_m\\
		&=R^\p_{mm}\abs{\nabla f}-\eta \abs{\nabla f}f_{AA}+(m-1)\lambda_m\\
		&=R^\p_{mm}\abs{\nabla f}+(m-1)H\abs{\nabla f}^2+(m-1)\lambda_m.
	\end{align*}
	Since we have already proved the constancy of $\abs{\nabla f}$ and $H$, we conclude that, since $(U(\p))_m=0$,
	\begin{align*}
		\frac{1}{2}S^\p_m-R^\p_{mm}\abs{\nabla f}-(m-1)\lambda_m+\pa{U\pa{\p}}_m
	\end{align*}
	is constant on the connected components of $\Sigma$.
	Hence,
	\begin{align*}
		0&=\set{\frac{1}{2}S^\p_m-R^\p_{mm}\abs{\nabla f}-(m-1)\lambda_m+(U(\p))_m}_A\\
		&=\frac{1}{2}S^\p_{mA}-R^\p_{mm,A}\abs{\nabla f}-R^\p_{mm}\abs{\nabla f}_A-(m-1)\lambda_{mA}+(U(\p))_{mA}\\
		&=\set{\sq{\frac{1}{2}S^\p-(m-1)\lambda+U(\p)}_A}_m-R^\p_{mm,A}\abs{\nabla f}\\
		&=-R^\p_{mm,A}\abs{\nabla f},
	\end{align*}
	since $\abs{\nabla f}$ and $\frac{1}{2}S^\p-(m-1)\lambda+U(\p)$ are constant on the connected components of $\Sigma.$ It follows that, since $\abs{\nabla f}\neq 0$, $R^\p_{mm,A}= 0$ and $R^\p_{mm}$ is constant on the connected components of $\Sigma$. By the constancy of $\frac{1}{2}S^\p+R^\p_{mm}-U(\p)$ and $R^\p_{mm}$ we have that
	\begin{align*}
		\frac{1}{2}S^\p-U(\p)
	\end{align*}
	is constant on every connected component of $\Sigma$, then
	\begin{align}\label{Sp-(m-1)lambda}
		S^\p-(m-1)\lambda=\frac{1}{2}S^\p-(m-1)\lambda+U(\p)+\frac{1}{2}S^\p-U(\p)
	\end{align}
	is also constant.\\
	Since $i:\Sigma\hookrightarrow M$ is totally umbilical, Gauss equation yields
	\begin{align}\label{S sigma ed S}
		{}^{\Sigma}S=S-2R_{mm}+(m-1)(m-2)H^2,
	\end{align}
	where ${}^{\Sigma}S$ is the scalar curvature of $\Sigma$.
	By the definition of $S^{\p_{|_{\Sigma}}}$ and the fact that $\p^a_m=0$, we get
	\begin{align*}
		S^{\p_{|_{\Sigma}}}&={}^{\Sigma}S-\alpha\p^a_A\p^a_A\\
		&=S-2R_{mm}+(m-1)(m-2)H^2-\alpha\abs{d\p}^2+\alpha\p^a_m\p^a_m\\
		&=S^\p-2R_{mm}^\p+(m-1)(m-2)H^2\\
		&=S^\p-2R_{mm}^\p+(m-1)(m-2)H^2
	\end{align*}
	By the constancy of \eqref{Sp-(m-1)lambda} we conclude that
	\begin{align*}
		S^{\p_{|_{\Sigma}}}-(m-1)\lambda|_{\Sigma}
	\end{align*}
	is constant on the connected components of $\Sigma$.
\end{proof}	

\begin{proposition}\label{Cotton phi prop}
	Let $(M,g)$ be a manifold of dimension $m\geq 3$. Assume $\overline{D}^\p\equiv 0$ on $M$ and that \eqref{systemKO f} is satisfied by $f\in C^\infty(M)$. Then, at any regular point $p$ of $f$, we have
	\begin{align}\label{100.1}
		C^\p=-\frac{1}{2(m-1)}\diver_1\pa{U(\p)g\KN g}.
	\end{align}
\end{proposition}
\begin{proof}
	By \eqref{1st int cond} and $D^\p\equiv0$, we have
	\begin{align}\label{100.2}
		C^\p_{ijk}=-f_tW^\p_{tijk}-\frac{U^a}{m-1}\pa{\p^a_j\delta_{ik}-\p^a_k\delta_{ij}}.
	\end{align}
	Note that
	\eqref{100.1} is automatically satisfied when $f$ is constant. When $f$ is not constant, we need to show the validity of \eqref{100.1}
	on $\set{x\in M\,:\,\nabla f\neq 0}$, hence, if $c$ is a regular value of $f$ it is sufficient to show \eqref{100.1} on the corresponding level set $\Sigma$.
	\\
	 Fix $p\in \Sigma$, let us choose a Darboux frame as in Proposition \ref{prop B}. By \eqref{100.2} and the symmetries of the $\p$-Weyl tensor, we have
	\begin{align*}
		0=&f_if_tW^\p_{tijk}=\frac{U^a}{m-1}\pa{f_k\p^a_j-f_j\p^a_k}+f_iC^\p_{ijk}\\
		=&\frac{U^a}{m-1}\pa{f_k\p^a_j-f_j\p^a_k}+\abs{\nabla f}C^\p_{mjk};
	\end{align*}
	therefore,
	\begin{align*}
		C^\p_{mjk}=-\frac{1}{\abs{\nabla f}}\frac{U^a}{m-1}\pa{f_k\p^a_j-f_j\p^a_k} &&\text{on }\Sigma.
	\end{align*}
	Thus, \eqref{100.1} holds for $C^\p_{mjk}$. Since $i:\Sigma\hookrightarrow M$ is totally umbilical with constant mean curvature $H$, we have
	\begin{align*}
		\se_{AB,C}=0,
	\end{align*}
	where $\se_{AB}$, $A,B=1,...,m-1$, is the second fundamental form of $i$. Thus, by Codazzi equations (see e.g. \cite{AMR}), we deduce
	\begin{align*}
		R_{mABC}=0.
	\end{align*}
	Therefore, by the definition of $W^\p$ and using $R^\p_{Am}=0$, we obtain
	\begin{align}\label{100.4}
		0=R_{mABC}=&W^\p_{mABC}+\frac{1}{m-2}\pa{R^\p_{mB}\delta_{AC}+R^\p_{AC}\delta_{mB}-R^\p_{mC}\delta_{AB}-R^\p_{AB}\delta_{mC}}\notag\\
		&-\frac{S^\p}{(m-1)(m-2)}(\delta_{mB}\delta_{AC}-\delta_{mc}\delta_{AB})\notag\\
		=&W^\p_{mABC}+\frac{1}{m-2}\pa{R^\p_{mB}\delta_{AC}-R^\p_{mC}\delta_{AB}}\notag\\
		=&	W^\p_{mABC}.
	\end{align}
	Hence, by \eqref{100.2} and \eqref{100.4}
	\begin{align*}
		C^\p_{ABC}=&-f_tW^\p_{tABC}-\frac{U^a}{m-1}\pa{\p^a_B\delta_{AC}-\p^a_C\delta_{AB}}\\
		=&-f_mW^\p_{mABC}-\frac{U^a}{m-1}\pa{\p^a_B\delta_{AC}-\p^a_C\delta_{AB}}\\
		=&-\frac{U^a}{m-1}\pa{\p^a_B\delta_{AC}-\p^a_C\delta_{AB}}.
	\end{align*}
	Note that, when $j,k=m$, we have $C^\p_{imm}=0$ and
	\begin{align*}
		\frac{U^a}{m-1}\pa{\p^a_m\delta_{mi}-\p^a_m\delta_{mi}}=0
	\end{align*}
	therefore, \eqref{100.1} holds. Hence we only need to prove the validity of \eqref{100.1} for $C^\p_{mmB}$ and $C^\p_{AmB}$. By \eqref{100.2}
	\begin{align*}
		C^\p_{mmB}=&-\frac{U^a}{m-1}\pa{\p_m^a\delta_{mB}-\p^a_B}-f_tW^\p_{tmmB}\\
		=&\frac{U^a}{m-1}\p^a_B-f_AW^\p_{AmmB}-f_mW^\p_{mmmB}\\
		=&\frac{U^a}{m-1}\p^a_B,
	\end{align*}
	so that \eqref{100.1} is verified in this case too.
	Let us show the last case. Since $R^\p_{Am}=0$, we deduce
	\begin{align*}				0=dR^\p_{Am}=&R^\p_{Am,k}\theta^k+R^\p_{km}\theta^k_A+R^\p_{Ak}\theta^k_m\\
		=&R^\p_{Am,k}\theta^k+R^\p_{Bm}\theta^B_A+R^\p_{mm}\theta^m_A+R^\p_{AB}\theta^B_m+R_{Am}\theta^m_m\\
		=&R^\p_{Am,k}\theta^k+R^\p_{mm}\theta^m_A+R^\p_{AB}\theta^B_m
	\end{align*}
	(note that we are not taking the sum over $m$). Therefore, using \eqref{A.13} in the above equation, we obtain
	\begin{align}\label{100.5}
		R^\p_{Am,k}\theta^k=&-R^\p_{mm}\theta^m_A-\frac{1}{m-1}\pa{S^\p-R^\p_{mm}}\theta^A_m\notag\\
		=&\frac{1}{m-1}\pa{S^\p-mR^\p_{mm}}\theta^m_A.
	\end{align}
	By the definition of the $\p$-Cotton tensor and \eqref{A.13}, we have
	\begin{align}\label{a}
		C^\p_{ABm}=&R^\p_{AB,m}-R^{\p}_{Am,B}-\frac{1}{2(m-1)}\pa{S^\p_m\delta_{AB}-S^\p_B\delta_{Am}}\notag\\
		=&R^\p_{AB,m}-R^{\p}_{Am,B}-\frac{S^\p_m}{2(m-1)}\delta_{AB}\notag\\
		=&\frac{1}{m-1}\sq{\pa{S^\p-R^\p_{mm}}\delta_{AB}}_m-R^\p_{Am,B}-\frac{S^\p_m}{2(m-1)}\delta_{AB}.
	\end{align}		
	Then, from \eqref{100.5} we obtain, on $\Sigma$,
	\begin{align*}
		R^\p_{Am,B}=&\frac{1}{m-1}\pa{S^\p-mR^\p_{mm}}\theta^m_A(e_B)\\
		=&\frac{1}{m-1}\pa{S^\p-mR^\p_{mm}}\se_{AB}\\
		=&-\frac{1}{m-1}\pa{S^\p-mR^\p_{mm}}\frac{f_{AB}}{\abs{\nabla f}};
	\end{align*}
	therefore, on $\Sigma$ we have
	\begin{align}\label{b}
		R^\p_{Am,B}=&-\frac{1}{m-1}\pa{S^\p-mR^\p_{mm}}\frac{f_{AB}}{\abs{\nabla f}}.
	\end{align}
	Inserting \eqref{b} into \eqref{a} we get
	\begin{align}\label{c}
		C^\p_{ABm}=&\frac{1}{m-1}\sq{\pa{S^\p_m-R^\p_{mm,m}}\delta_{AB}}+\frac{1}{m-1}\pa{S^\p-mR^\p_{mm}}\frac{f_{AB}}{\abs{\nabla f}}-\frac{S^\p_m}{2(m-1)}\delta_{AB}\notag\\
		=&\frac{S^\p_m}{2(m-1)}\delta_{AB}-\frac{1}{m-1}R^\p_{mm,m}\delta_{AB}+\frac{1}{m-1}\pa{S^\p-mR^\p_{mm}}\frac{f_{AB}}{\abs{\nabla f}}.
	\end{align}
	By the $\p$-Schur's identity
	\begin{align*}
		\frac{1}{2}S^\p_m=&\alpha\p^a_{tt}\p^a_m+R^\p_{im,i}\\
		=&\alpha\p^a_{tt}\p^a_m+R^\p_{Am,A}+R^\p_{mm,m}.
	\end{align*}
	Using that $\p^a_{tt}=\frac{1}{\alpha}U^a$, we obtain
	\begin{align*}
		\frac{S^\p_m}{2(m-1)}=&\frac{U^a}{m-1}\p^a_m+\frac{1}{m-1}R^\p_{Am,A}+\frac{1}{m-1}R^\p_{mm,m}.
	\end{align*}
	Hence,
	\begin{align}\label{d}
		C^\p_{ABm}=&\frac{U^a}{m-1}\p^a_m\delta_{AB}+\frac{1}{m-1}R^\p_{Cm,C}\delta_{AB}+\frac{1}{m-1}\pa{S^\p-mR^\p_{mm}}\frac{f_{AB}}{\abs{\nabla f}}.
	\end{align}
	Taking the trace of \eqref{b}, we have
	\begin{align}\label{e}
		R^\p_{Bm,B}=-\frac{1}{m-1}\pa{S^\p-mR^\p_{mm}}\frac{f_{BB}}{\abs{\nabla f}}.
	\end{align}
	Moreover, by \eqref{A.1},
	\begin{align}\label{f}
		f_{AB}=&-R^\p_{AB}+ \eta f_Af_B+\lambda\delta_{AB}\notag\\
		=&-\frac{1}{m-1}\pa{S^\p-R^\p_{mm}}\delta_{AB}+\lambda\delta_{AB}\notag\\
		=&-\frac{1}{m-1}\pa{S^\p-R^\p_{mm}-(m-1)\lambda}\delta_{AB};
	\end{align}
	thus, taking the trace, we get
	\begin{align*}
		f_{BB}=-S^\p+R^\p_{mm}+(m-1)\lambda.
	\end{align*}
	From \eqref{e} we deduce
	\begin{align}\label{g}
		R^\p_{Cm,C}=&\frac{1}{\abs{\nabla f}}\frac{1}{m-1}\pa{S^\p-R^\p_{mm}-(m-1)\lambda}\pa{S^\p-mR^\p_{mm}};
	\end{align}
	on the other hand, from \eqref{f}, we have
	\begin{align}\label{h}
		\frac{1}{\abs{\nabla f}}\frac{\pa{S^\p-mR^\p_{mm}}}{m-1}f_{AB}=-\frac{1}{\abs{\nabla f}}\frac{\pa{S^\p-mR^\p_{mm}}}{(m-1)^2}\pa{S^\p-R^\p_{mm}-(m-1)\lambda}\delta_{AB}.
	\end{align}	
	Inserting \eqref{g}	and \eqref{h} into \eqref{d}, we get
	\begin{align*}
		C^\p_{ABm}=&\frac{1}{m-1}U^a\p^a_m\delta_{AB}\\
		&+\frac{1}{(m-1)^2}\frac{1}{\abs{\nabla f}}\pa{S^\p-R^\p_{mm}-(m-1)\lambda}\pa{S^\p-mR^\p_{mm}}\delta_{AB}\\
		&-\frac{1}{(m-1)^2}\frac{1}{\abs{\nabla f}}\pa{S^\p-R^\p_{mm}-(m-1)\lambda}\pa{S^\p-mR^\p_{mm}}\delta_{AB}\\
		=&\frac{1}{m-1}U^a\p^a_m\delta_{AB},
	\end{align*}
	that is \eqref{100.1}.
	
\end{proof}	

In the next Theorem we will prove, among other things, that, under the assumptions of Proposition \ref{prop B}, for any regular level set $\Sigma$ of $f$ and for any point $p\in\Sigma$ there exists an open neighbourhood $A$ of $p$ such that  $f_{|_A}$ only depends on the signed distance function $r:A\to \erre$ from $A\cap \Sigma$. With a slight  abuse of notation, we will write
\begin{align*}
	f=f(r),
\end{align*}
identifying $f$ with a function $f:\erre\to \erre$.
\begin{theorem}\label{almost KO}
	Let $(M,g)$, $(N,h)$ be two Riemannian manifolds of dimension $m\geq 3$ and $n$, respectively. Let $\p:(M,g)\ra(N,h)$ be a smooth map and $f$  a solution of \eqref{systemKO f}, where $U\in C^{\infty}(N)$, $\alpha\in \erre\setminus \set{0}$ and
	\begin{align*}
		\lambda=\frac{1}{m}\pa{S^\p+\Delta f-\eta\abs{\nabla f}^2}.
	\end{align*}
	Assume that $\ol{D}^\p\equiv0$ and that $\p$ is $\frac{1}{\alpha}U$-harmonic; let $c$ be a regular value of $f$ and $\Sigma$ the corresponding level set. Then, for each $p\in \Sigma$, there exists  an open set $A$, with $p\in A\subseteq M$, such that $g|_A$ is a warped product $(I\times_\rho (\Sigma\cap A), dr^2+\rho^2g_{\Sigma})$, with $I$ an open interval. When $M$ is compact we can choose $A$ such that $\Sigma\subset A$.\\
	Let $r$ denote the signed distance function from $\Sigma$: then
	\begin{enumerate}
		\item[1.]  $f_{|_A}$ only depends on $r$.
		\item[2.]  $S^{\p_{|_{\Sigma}}}$ is constant and it holds
		\begin{align}\label{Sp sigma eq}
			&0=\frac{S^{\p_{|_{\Sigma}}}}{\rho^2(r)}+(m-1)(m-2)\sq{\frac{\rho''(r)}{\rho(r)}-\pa{\frac{\rho'(r)}{\rho(r)}}^2}-(m-1)f''(r)\\
			&\quad+\eta( m-1)(f'(r))^2+(m-1)f'(r)\frac{\rho'(r)}{\rho(r)}. \nonumber
		\end{align}
		\item[3.] $U(\p)$ is constant on the connected components of $\Sigma$ and $(\Sigma,g_{\Sigma})$ satisfies
		\begin{align}\label{system sigma}
			\begin{cases}
				i)\,\ric^{\p_{|_{\Sigma}}}=\frac{S^{\p_{|_{\Sigma}}}}{m-1}g_{\Sigma},\\
				ii)\, h\pa{d\p_{|_{\Sigma}}, \tau(\p_{|_{\Sigma}})}=0.
			\end{cases}
		\end{align}
	\end{enumerate}
	
\end{theorem}


\begin{proof} We divide the proof into three steps.\\
	\textbf{Step 1}. We start by showing that the metric $g$ locally splits as a warped product around $p$; to do this we follow the proof in \cite{CatinoMantegazzaMazzieri}.\\
	Let $r$ be the signed distance function from $\Sigma=\Sigma_c=f^{-1}(c)$, chosen so that its
	  gradient  is 
	\begin{align*}
		\nabla r=\frac{\nabla f}{\abs{\nabla f}}.
	\end{align*}
	 On an open neighborhood $A\subset M$ of $p$, there exist \textit{adapted} coordinates, namely \textit{Fermi coordinates}, $\set{x^i}_{i=1}^m$ (see \cite[Corollary 6.42]{Lee} for more details), such that $x_m=r$ and the tangent space  of $\Sigma\cap A$ is spanned by
	\begin{align*}
		\set{\frac{\partial}{\partial x^1},...,\frac{\partial}{\partial x^{m-1}}}.
	\end{align*}
	By restricting $A$ if necessary, we can assume that all the points of $A$ are regular for $f$.
	For the ease of notations, let
	\begin{align*}
		\underline{x}=\pa{x^1,...,x^{m-1}}.
	\end{align*}
	In these coordinates the metric $g$ takes the form
	\begin{align}\label{ko: metrica da warpare}
		g=dr\otimes dr+g_{AB}(\underline{x},r)dx^A\otimes d x^B,
	\end{align}
	where $1\leq A,B \leq m-1$ and
	\begin{align}\label{ko: metrica da warpare per r=0}
		g_{AB}(\underline{x},0)dx^A\otimes dx^B=g_{\Sigma}.
	\end{align}
	Furthermore, the Christoffel symbols of the Levi-Civita connection satisfy
	\begin{align}\label{christoffel adapted}
		\begin{cases}
			\Gamma^m_{mm}=\Gamma^A_{mm}=\Gamma^m_{Am}=0,\\
			\Gamma^m_{AB}=-\frac{1}{2}\partial_rg_{AB},\\
			\Gamma^A_{Bm}=\frac{1}{2}g^{AC}\partial_rg_{BC}.
		\end{cases}
	\end{align}
	Since, by Proposition \ref{prop B}, $\abs{\nabla f}$ is a positive constant on any regular level set of $f$, it follows that $f$ only depends on $r$, if we remain sufficiently close to $\Sigma$. In particular, near $\Sigma$,
	\begin{align}
		df&=f'dr,\\
		\nabla f&=f'\frac{\partial}{\partial r} ,\label{grad f e r}\\
		\hs(f)&=f''dr\otimes dr+f'\hs(r)\nonumber\\
		&=f''dr\otimes dr+\frac{f'}{2}\partial_rg_{AB}dx^A\otimes dx^B.\label{hess f e r}
	\end{align}
	 From \eqref{grad f e r} and our choice of $r$ we get
	\begin{align*}
		\nabla f=\frac{f'}{\abs{\nabla f}}\nabla f
	\end{align*}
	so that $f'>0$. Therefore, $f:\erre\to \erre$ is monotone increasing and hence injective, so that there exists a one to one correspondence between the regular level sets of $r$ and those of $f$ that meets $A$.
	Equation \eqref{hess f e r} implies
	\begin{align}\label{ko: hess f 2}
		f_{AB}=\frac{f'}{2}\partial_rg_{AB}.
	\end{align}
	Since, by Corollary \ref{cor: equivalence betw autov e rett}, $i:\Sigma_{\ol{c}}\hookrightarrow M$
	 is totally umbilical, with $\ol{c}$ a regular value sufficiently close to $c$, 
	\begin{align}\label{hess f}
		\frac{f_{AB}}{f'}=-\se_{AB}=-Hg_{AB}.
	\end{align}
	Note that $c=f(0)$ and $\ol{c}=f(\ol{r})$ for some $\ol{r}$ close to $0$; moreover,  $H$ is constant on the connected components of $\Sigma_{\ol{c}}$, as we have proved in Proposition \ref{prop B}.
	Comparing \eqref{ko: hess f 2} and \eqref{hess f}, we deduce
	\begin{align*}
		-f'Hg_{AB}=\frac{f'}{2}\partial_r g_{AB};
	\end{align*}
	since $f'\neq 0$ for $0\leq r<<1$, we have
	\begin{align*}
		-Hg_{AB}=\frac{\partial_r g_{AB}}{2}
	\end{align*}
	and integrating the above expression with respect to $r$, we get
	\begin{align*}
		g_{AB}(\underline{x},r)=e^{-2\int_0^rH(t)dt}g_{AB}(\underline{x},0).
	\end{align*}
	Inserting the above into \eqref{ko: metrica da warpare}, we obtain that $g$ is a warped product metric.

	
	\textbf{Step 2.} In this step we prove the validity of \eqref{Sp sigma eq}, i.e.,
	\begin{align*}
		&0=\frac{S^{\p_{|_{\Sigma}}}}{\rho^2(r)}+(m-1)(m-2)\sq{\frac{\rho''(r)}{\rho(r)}-\pa{\frac{\rho'(r)}{\rho(r)}}^2}-(m-1)f''(r)\\
		&\quad+\eta( m-1)(f'(r))^2+(m-1)f'(r)\frac{\rho'(r)}{\rho(r)}.
	\end{align*}
	Assume that $g$ locally splits as in Step 1 in a neighborhood of $p$,
	that is, there is an open neighborhood $A$ of $p$ such that
	\begin{align*}
		A\simeq(-\eps,\eps)\times(\Sigma\cap A),
	\end{align*}
	where we identify $\Sigma\cap A$ with $\set{0}\times (\Sigma\cap A)$ and
	\begin{align*}
		g=dr \otimes dr +\rho^2g_{\Sigma};
	\end{align*}
	 moreover, by \eqref{ko: metrica da warpare per r=0}, $\rho$ satisfies $\rho(0)=1$, $\rho>0$.
	Consider a local orthonormal coframe $\set{\theta^i}_{i=1}^m$ such that $\theta^m=dr$ and $\set{\theta^i}_{i=1}^{m-1}$ is a local orthonormal coframe for $g_{\Sigma}$; then we have
	\begin{align}
		&f_{mm}=f'',\label{A.14}\\
		&\Delta f=f''+(m-1)f'\frac{\rho'}{\rho},\label{A.15}\\
		&S=\frac{{}^{\Sigma}S}{\rho^2}-(m-1)(m-2)\pa{\frac{\rho'}{\rho}}^2-2(m-1)\frac{\rho''}{\rho},\label{A.16}\\
		&R_{mm}=-(m-1)\frac{\rho''}{\rho},\label{A.17}
	\end{align}
	where, for the sake of simplicity, we have omitted the dependence on $r$ of $f=f(r)$ and $\rho$ (see e.g. \cite{KO}, and page 50 of \cite{AMR} for more details on the previous formulas). Observe that, since $\p$ is $\frac{1}{\alpha}U$-harmonic by assumption and \eqref{systemKO f} ii) holds, we have
	\begin{align*}
		\p^a_m=0;
	\end{align*}
	  thus
	\begin{align*}
		\abs{d \p}_{g}^2=\abs{d \p}_{\rho^2g_{\Sigma}}^2=\frac{1}{\rho^2}\abs{d \p}_{g_{\Sigma}}^2,
	\end{align*}
	so that
	\begin{align}\label{A.19}
		S^\p=&S-\alpha\abs{d\p}_{g}^2\notag\\
		=&\frac{{}^{\Sigma}S}{\rho^2}-\alpha\abs{d\p}_{g}^2-(m-1)(m-2)\pa{\frac{\rho'}{\rho}}^2-2(m-1)\frac{\rho''}{\rho}\notag\\
		=&\frac{S^{\p_{|_{\Sigma}}}}{\rho^2}-(m-1)(m-2)\pa{\frac{\rho'}{\rho}}^2-2(m-1)\frac{\rho''}{\rho}.
	\end{align}
	From \eqref{systemKO f} i) and the definition of $\lambda$, we get
	\begin{align}\label{A.8 2 explicit}
		R^\p_{mm}+f_{mm}-\eta \abs{\nabla f}^2=\frac{1}{m}\pa{S^\p+\Delta f-\eta\abs{\nabla f}^2}.
	\end{align}
	Inserting \eqref{A.14}, $\eqref{A.15}$ and \eqref{A.17} into \eqref{A.8 2 explicit}, we obtain
	\begin{align*}
		S^\p+m(m-1)\frac{\rho''}{\rho}-(m-1)f''+\eta ( m-1)(f')^2+(m-1)f'\frac{\rho'}{\rho}=0;
	\end{align*}
	then, using \eqref{A.19}, we conclude the validity of \eqref{Sp sigma eq}, i.e.
	\begin{align*}
		0=&\frac{S^{\p_{|_{\Sigma}}}}{\rho^2(r)}+(m-1)(m-2)\sq{\frac{\rho''(r)}{\rho(r)}-\pa{\frac{\rho'(r)}{\rho(r)}}^2}-(m-1)f''(r)\\
		&+\eta( m-1)(f'(r))^2+(m-1)f'(r)\frac{\rho'(r)}{\rho(r)},
	\end{align*}
	where we have expressed the dependence on $r$ in the above formula; from this we deduce that $S^{\p_{|_{\Sigma}}}$ is constant on $\Sigma$.\\
	\textbf{Step 3.} In this step we prove the final part of the statement.\\
	By the assumption $\ol{D}^\p\equiv0$, the first integrability condition \eqref{1st int cond} rewrites as
	\begin{align*}
		0&=C^\p_{AmB}+\abs{\nabla f}W^\p_{mAmB}-\frac{U^a}{m-1}\p^a_B\delta_{Am}+\frac{U^a}{m-1}\p^a_m\delta_{AB};
	\end{align*}
	by Proposition \ref{Cotton phi prop}
	\begin{align*}
		C^\p_{AmB}=-C^\p_{ABm}=-\frac{U^a}{m-1}\p^a_m\delta_{AB},
	\end{align*}
	hence	
	$$W^\p_{mAmB}=0$$
	on $\Sigma$. By the definition of the $\p$-Weyl tensor we deduce
	\begin{align}\label{riem mAmB}
		R_{mAmB}=&\frac{1}{m-2}\pa{R^\p_{mm}\delta_{AB}+R^\p_{AB}-\frac{S^\p}{m-1}\delta_{AB}}
	\end{align}
	on $\Sigma$. Using \eqref{A.13} of Proposition \ref{prop B}, i.e.
	\begin{align*}
		R^{\p}_{AB}=\frac{1}{m-1}\pa{S^\p-R^{\p}_{mm}}\delta_{AB},
	\end{align*}
	 into \eqref{riem mAmB}, we obtain
	\begin{align}\label{A.20}
		R_{mAmB}=&\frac{1}{m-2}\pa{\frac{S^\p}{m-1}\delta_{AB}-\frac{1}{m-1}R^\p_{mm}\delta_{AB}+R^\p_{mm}\delta_{AB}-\frac{S^\p}{m-1}\delta_{AB}}\notag\\
		=&\frac{1}{m-1}R^\p_{mm}\delta_{AB}.
	\end{align}
	From the Gauss equation for the Ricci tensor of $i:\Sigma\hookrightarrow M$, ${}^{\Sigma}\ric$, and since $i$ is totally umbilical we have
	\begin{align*}
		{}^{\Sigma}R_{AC}=R_{AC}-R_{AmCm}+(m-2)H^2\delta_{AC};
	\end{align*}
	 from the definition of $\ric^{\p_{|_{\Sigma}}}$ we deduce
	\begin{align}\label{A.qualcosa}
		R_{AC}^{\p_{|_{\Sigma}}}={}^{\Sigma}R_{AC}-\alpha\p^a_A\p^a_C=R_{AC}^\p-R_{AmCm}+(m-2)H^2\delta_{AC}.
	\end{align}
	Using \eqref{A.13} and \eqref{A.20} into \eqref{A.qualcosa}, we obtain
	\begin{align*}
		R^{\p_{|_{\Sigma}}}_{AC}=&\frac{S^\p-R^\p_{mm}}{m-1}\delta_{AC}-\frac{1}{m-1}R^\p_{mm}\delta_{AC}+(m-2)H^2\delta_{AC}\\
		=&\sq{\frac{1}{m-1}(S^\p-2R^\p_{mm})+(m-2)H^2}\delta_{AC}.
	\end{align*}
	Contracting with respect to the indexes $A$ and $C$, we get
	\begin{align*}
		S^{\p_{|_{\Sigma}}}=S^\p-2R_{mm}+(m-1)(m-2)H^2.
	\end{align*}
	Therefore, the above can be rewritten as
	\begin{align*}
		R^{\p_{|_{\Sigma}}}_{AC}=\frac{S^{\p_{|_{\Sigma}}}}{m-1}\delta_{AC},
	\end{align*}
	that is, we have the validity of \eqref{system sigma} i).
	Taking the divergence of \eqref{system sigma} i) and using the $\p$-Schur's identity we deduce
	\begin{align*}
		R^{\p_{|_{\Sigma}}}_{BA,B}=\frac{1}{2}S^{\p_{|_{\Sigma}}}_A-\alpha\tau^a(\p|_{\Sigma})\p^a_A=\frac{1}{m-1}S^{\p_{|_{\Sigma}}}_A;
	\end{align*}
	by the constancy of $S^{\p_{|_{\Sigma}}}$ we have
	\begin{align*}
		\tau^a(\p|_{\Sigma})\p^a_A=0,
	\end{align*}
	from which we deduce
	\begin{align*}
		0=\tau^a(\p|_{\Sigma})\p^a_A=\frac{1}{\alpha}U^a\p^a_A=\frac{1}{\alpha}(U(\p))_A,
	\end{align*}
	where we have used that, according to Proposition \ref{prop B}, $\p_{|_{\Sigma}}$ is $\frac{U}{\alpha}$-harmonic.
	This proves that $U(\p)$ is constant on the connected components of $\Sigma$.
\end{proof}
\section{Main Results}


In this subsection we are finally ready to give a proof of Theorem \ref{thm 018_KO}, that we recall here for the ease of readability.
\begin{theorem}\label{KO: Teorema KO ma con la f e la eta}
	Let $(M,g)$ be a Riemannian manifold of dimension $m\geq 3$ with $\partial M=\emptyset$. Let $\p:(M,g)\ra(N,h)$, $U:(N,h)\ra \erre$ be smooth maps, $\alpha\in \erre\setminus\set{0}$ and let $f\in C^2(M)$ be a solution  on $M$ of the system
	\begin{align}
		\begin{cases}
			i)\,\ric^\p+\hs(f)-\eta df\otimes df=\lambda g,\\
			ii)\,\tau(\p)=d\p(\nabla f)+\frac{1}{\alpha}(\nabla U)(\p)
		\end{cases}
	\end{align}
	where  $\lambda=\frac{1}{m}\pa{S^\p+\Delta f-\eta\abs{\nabla f}^2}$ and $\eta\neq -\frac{1}{m-2}$. Consider the conformal change of metric
	\begin{align*}
		\tilde{g}=e^{-\frac{2}{m-2} f}g;
	\end{align*}
	suppose that $\p$ is $\frac{1}{\alpha} U$-harmonic and that
	\begin{align*}
		2(m-1)\tilde{C}^{\tilde{\p}}=-\diver_1(U(\p)g\KN g).
	\end{align*}

	Then, for each $\Sigma$ regular level set of $f$ and $p\in \Sigma$, there exists $A \subseteq M$ open such that $p\in A$ and $g|_A$ is a warped product metric. Moreover $U(\p)$ is constant on $M$, $S^{\p_{|_{\Sigma}}}$ is the constant in \eqref{Sp sigma eq} and $(\Sigma,g_{\Sigma})$ satisfies
	\begin{align}\label{KO: quasi U-Harmonic Einstein}
		\begin{cases}
			i)\,\ric^{\p_{|_{\Sigma}}}=\frac{S^{\p_{|_{\Sigma}}}}{m-1}g_{\Sigma},\\
			ii)\,h(\tau(\p_{|_{\Sigma}}),d\p_{|_{\Sigma}})=0.
		\end{cases}
	\end{align}
	
\end{theorem}

\begin{proof}
	Since
	\begin{align*}
		2(m-1)\tilde{C}^{\tilde{\p}}=-\diver_1(U(\p)g\KN g),
	\end{align*}
	by \eqref{1st int cond} and Remark \ref{KO: remark su 1st int cond e deformazione conforme}, we have $\overline{D}^\p\equiv 0$. Therefore the claim follows by Theorem \ref{almost KO}.
\end{proof}

\noindent
Assume now, again, the validity of system \eqref{systemKO f}, that is
\begin{align*}
	\begin{cases}
		i)\,\ric^\p+\hs(f)-\eta df\otimes df=\lambda g,\\
		ii)\, \tau(\p)=d\p(\nabla f)+\frac{1}{\alpha}(\nabla U)(\p),
	\end{cases}
\end{align*}
where $\lambda\in C^{\infty}(M)$, and define on $\mathrm{int}(M)$ the vector field $Y$ of (local) components
\begin{align}\label{3.25 def of Y}
	Y_k=[1+\eta(m-2)]\overline{D}^\p_{ijk}f_if_j-\frac{U^a}{m-1}\pa{\p^a_jf_jf_k-\p^a_k\abs{\nabla f}^2}.
\end{align}
Note that
\begin{align}\label{3.26 Y nabla f =0}
	g(Y,\nabla f)=Y_kf_k\equiv 0,
\end{align}
so that on the level set of a regular value of $f$
\begin{align}\label{Y nu =0}
	g(Y,\nu)=g\pa{Y,\frac{\nabla f}{\abs{\nabla f}}}=0,
\end{align}
where $\nu=\frac{\nabla f}{\abs{\nabla f}}$ is a unit normal vector field. 

\begin{lemma}
	Let $(M,g)$ be a Riemannian manifold of dimension $m\geq3$, let $\p:(M,g)\ra(N,h)$ be a smooth map to a second Riemannian manifold, $U:(N,h)\ra \erre$ be a smooth function and let $Y$ be as in \eqref{3.25 def of Y}. Under the validity of \eqref{systemKO f} we have
	\begin{align}\label{ko: diver Y}
		\diver Y=&\frac{1}{2}[1+\eta(m-2)](m-2)\abs{\overline{D}^{\p}}^2+(m-2)B^{\p}\pa{\nabla f,\nabla f}-h\pa{\nabla U,\nabla d\p\pa{\nabla f,\nabla f}}\\
		&+\hess(U)\pa{d\p(\nabla f),d\p(\nabla f)}+\frac{1}{2(m-1)}h\pa{\nabla U,d\p\pa{\nabla \abs{\nabla f}^2}}\notag\\
		&-\frac{\Delta f}{m-1}h\pa{\nabla U,d\p(\nabla f)}+\frac{\alpha}{m-2}|\nabla f|^2|\tau (\p)|^2+\alpha \eta |\nabla f|^2\abs{\tau(\p)-\frac{1}{\alpha}\nabla U}^2.\notag
	\end{align}
\end{lemma}
\begin{proof}
	Recall the validity of equation (\ref{A.5}), that is,
	\begin{align}\label{2.29 Y}
		\frac{2}{m-2}\overline{D}^\p_{ijk}R^\p_{ij}f_k=\abs{\overline{D}^\p}^2;
	\end{align}
	from the first equation of (\ref{systemKO f}) and the symmetries of $\overline{D}^{\p}$ we get
	\begin{align}\label{2.29 Y bis}
		-\frac{2}{m-2}\overline{D}^\p_{ijk}f_{ij}f_k=\abs{\overline{D}^\p}^2.
	\end{align}
	This identity will be used later on.
	We compute the divergence of $Y$
	\begin{align*}
		\diver Y=&[1+\eta(m-2)]\pa{f_{ik}f_j\overline{D}^{\p}_{ijk}+f_if_j\overline{D}^{\p}_{ijk,k}}\\
		&+\frac{|\nabla f|^2}{m-1}U^{ab}\p^a_t\p^b_t+\frac{|\nabla f|^2}{m-1}U^{a}\p^a_{tt}+\frac{2}{m-1}f_{tk}f_tU^a\p^a_k\\
		&-\frac{1}{m-1}U^{ab}\p^a_tf_t\p^b_kf_k-\frac{1}{m-1}U^{a}\p^a_{tk}f_tf_k-\frac{1}{m-1}U^a\p^a_t f_{tk}f_k\\
		&-\frac{\Delta f}{m-1}U^a\p^a_jf_j.
	\end{align*}
	Using the second integrability condition (\ref{2nd int cond}) and equation (\ref{2.29 Y bis}) we get
	\begin{align*}
		\diver Y=&\frac{1}{2}[1+\eta (m-2)](m-2)\abs{\overline{D}^{\p}}^2+(m-2)B^{\p}_{ij}f_if_j-\frac{m-2}{m-1}U^a\p^a_{ij}f_if_j\\
		&+\alpha\frac{1+\eta(m-2)}{m-2}|\nabla f|^2\p^a_{tt}\p^a_kf_k+\frac{|\nabla f|^2}{(m-1)(m-2)}U^{a}\p^a_{tt}\\
		&-\frac{|\nabla f|^2}{m-1}U^{ab}\p^a_t\p^a_t+\frac{m}{m-1}U^{ab}\p^a_if_i\p^a_jf_j-\eta |\nabla f|^2 U^a\p^a_tf_t\\
		&+\frac{|\nabla f|^2}{m-1}U^{ab}\p^a_t\p^b_t+\frac{|\nabla f|^2}{m-1}U^a\p^a_{tt}+\frac{2}{m-1}U^a\p^a_k f_{tk}f_t\\
		&-\frac{1}{m-1}U^{ab}\p^a_kf_k\p^b_tf_t-\frac{1}{m-1}U^a\p^a_{tk}f_tf_k-\frac{1}{m-1}U^a\p^a_t f_{tk}f_k\\
		&-\frac{\Delta f}{m-1}U^a\p^a_tf_t.
	\end{align*}
	After some simplifications, the previous relation becomes
	\begin{align}\label{KO: div Y: fomula quasi finale 2}
		\diver Y=&\frac{1}{2}[1+\eta(m-2)](m-2)\abs{\overline{D}^{\p}}^2+(m-2)B^{\p}_{ij}f_if_j-U^a\p^a_{ij}f_if_j\\
		&+\alpha\frac{1+\eta(m-2)}{m-2}|\nabla f|^2\p^a_{tt}\p^a_kf_k+\frac{|\nabla f|^2}{m-2}U^a\p^a_{tt}+U^{ab}\p^a_tf_t\p^b_kf_k \nonumber\\
		&-\eta |\nabla f|^2U^a\p^a_tf_t+\frac{1}{m-1}U^a\p^a_tf_{tk}f_k-\frac{\Delta f}{m-1}U^a\p^a_kf_k. \nonumber
	\end{align}
	Using the second equation of \eqref{systemKO f}, that is
	\begin{align*}
		\p^a_tf_t=\p^a_{tt}-\frac{1}{\alpha} U^a,
	\end{align*}
	we compute
	\begin{align*}
		&\alpha \frac{1+\eta(m-2)}{m-2}|\nabla f|^2\p^a_{tt}\p^a_kf_k+\frac{|\nabla f|^2}{m-2}U^a\p^a_{tt}-\eta|\nabla f|^2U^a\p^a_tf_t\\
		&\quad=\alpha \frac{1+\eta(m-2)}{m-2}|\nabla f|^2\abs{\tau(\p)}^2-\frac{1+\eta(m-2)}{m-2}|\nabla f|^2U^a\p^a_{tt}+\frac{|\nabla f|^2}{m-2}U^a\p^a_{tt}\\
		&\qquad+\frac{\eta}{\alpha} |\nabla f|^2|\nabla U|^2-\eta |\nabla f|^2U^a\p^a_{tt}\\
		&\quad=\frac{\alpha}{m-2}|\nabla f|^2\abs{\tau(\p)}^2+\alpha \eta |\nabla f|^2|\tau (\p)|^2+\frac{\eta}{\alpha}|\nabla U|^2-2\eta |\nabla f|^2U^a\p^a_{tt}\\
		&\quad=\frac{\alpha}{m-2}|\nabla f|^2|\tau (\p)|^2+\alpha \eta |\nabla f|^2\abs{\tau(\p)-\frac{1}{\alpha}\nabla U}^2,
	\end{align*}
	that is
	\begin{align*}
		&\alpha \frac{1+\eta(m-2)}{m-2}|\nabla f|^2\p^a_{tt}\p^a_kf_k+\frac{|\nabla f|^2}{m-2}U^a\p^a_{tt}-\eta|\nabla f|^2U^a\p^a_tf_t\\
		&\quad=\frac{\alpha}{m-2}|\nabla f|^2|\tau (\p)|^2+\alpha \eta |\nabla f|^2\abs{\tau(\p)-\frac{1}{\alpha}\nabla U}^2.
	\end{align*}
	Inserting the above formula into \eqref{KO: div Y: fomula quasi finale 2}, we obtain
	\begin{align*}
		\diver Y=&\frac{1}{2}[1+\eta(m-2)](m-2)\abs{\overline{D}^{\p}}^2+(m-2)B^{\p}_{ij}f_if_j-U^a\p^a_{ij}f_if_j\\
		&+U^{ab}\p^a_tf_t\p^b_kf_k+\frac{1}{m-1}U^a\p^a_tf_{tk}f_k-\frac{\Delta f}{m-1}U^a\p^a_kf_k\\
		&+\frac{\alpha}{m-2}|\nabla f|^2|\tau (\p)|^2+\alpha \eta |\nabla f|^2\abs{\tau(\p)-\frac{1}{\alpha}\nabla U}^2,
	\end{align*}
	that is, \eqref{ko: diver Y}.
\end{proof}

Note that, when $\p$ is $\frac{1}{\alpha} U$-harmonic, that is,
\begin{align}\label{dp(nabla f)=0}
	0=\tau(\p)-\frac{1}{\alpha}(\nabla U)(\p)=d\p(\nabla f),
\end{align}
then the divergence of $Y$ rewrites as
\begin{align}\label{divergenza Y U harm}
	\diver Y=&\frac{1}{2}[1+\eta(m-2)](m-2)\abs{\overline{D}^{\p}}^2+(m-2)B^{\p}_{ij}f_if_j-U^a\p^a_{ij}f_if_j\\
	&+\frac{1}{m-1}U^a\p^a_tf_{tk}f_k+\frac{\alpha}{m-2}|\nabla f|^2|\tau (\p)|^2.\notag
\end{align}
Moreover, by \eqref{dp(nabla f)=0}, we have
\begin{align}\label{diver dp nabla f}
	0=(\p^a_if_i)_j=\p^a_{ij}f_i+\p^a_jf_{ij},
\end{align}
thus, \eqref{divergenza Y U harm} rewrites as
\begin{align*}
	\diver Y=&\frac{1}{2}[1+\eta(m-2)](m-2)\abs{\overline{D}^{\p}}^2+(m-2)B^{\p}_{ij}f_if_j\\
	&+\frac{m}{m-1}U^a\p^a_tf_{tk}f_k+\frac{\alpha}{m-2}|\nabla f|^2|\tau (\p)|^2
\end{align*}
and, since $\p^a_kf_k=0$, by the first equation of system \eqref{systemKO f} we deduce
\begin{align*}
	f_{tk}f_k\p^a_t=&-R^\p_{tk}f_k\p^a_t+\eta\abs{\nabla f}^2f_t\p^a_t+\lambda\p^a_tf_t\\
	=&-R^\p_{tk}f_k\p^a_t.
\end{align*}
Therefore,
\begin{align}\label{B.0}
	\diver Y=&(m-2)B^\p(\nabla f,\nabla f)-\frac{m}{m-1}U^a\p^a_jf_kR^\p_{jk}\notag\\
	&+\frac{\alpha}{m-2}\abs{\tau(\p)}^2\abs{\nabla f}^2+\frac{1}{2}(1+\eta(m-2))(m-2)\abs{\overline{D}^\p}^2.
\end{align}

Note that \eqref{divergenza Y U harm} simplifies under the assumption $B^{\varphi}(\nabla f,\nabla f)=
0$: this observation, together with the divergence theorem, is exploited in the proof of the following theorem.

\begin{theorem}\label{teo: KO: Main results: bach piatto}
	Let $(M,g)$ be a complete manifold of dimension $m\geq 3$ and with $\partial M=\emptyset$. Let $\p:(M,g)\to (N,h)$, $U:N\to \erre$ be smooth maps, $\lambda\in C^{\infty}(M)$ and let $f\in C^{\infty}(M)$ be a solution of \eqref{systemKO f}. Let $\Sigma$ be a regular level set of $f$. Assume that:
	\begin{itemize}
		\item[1.] $f$ is proper;
		\item[2.] either $\alpha>0$ and $\eta>-\frac{1}{m-2}$ or $\alpha<0$ and $\eta<-\frac{1}{m-2}$;
		\item[3.] we have the validity of
		\begin{align}\label{B.2}
			B^{\p}(\nabla f,\nabla f)=0;
		\end{align}
		\item[4.] for each regular $p\in M$, $\nabla f_p$ is an eigenvector of $\ric^{\p}_p$;
		\item[5.] $\p$ is $\frac{1}{\alpha} U$-harmonic
	\end{itemize}
	Then, for each $p\in \Sigma$, there exists $p\in A\subset M$, $A$ open in $M$ such that $g_{|_A}$ is a warped product metric. Moreover, $U(\p)$ is constant on $\Sigma$ and $(\Sigma, g_{\Sigma})$ satisfies
	\begin{align*}
		\begin{cases}
			&\ric^{\p_{|_{\Sigma}}}=\frac{S^{\p_{|_{\Sigma}}}}{m-1}g_{\Sigma}\\
			&\tau(\p_{|_{\Sigma}})=0
		\end{cases}
	\end{align*}
	and $S^{\p_{|_{\Sigma}}}$ is constant.
\end{theorem}

\begin{rem}
	Note that, by Corollary \ref{cor: equivalence betw autov e rett}, the request of $\nabla f$ being an eigenvector of $\ric^\p$ is necessary for the validity of $\overline{D}^\p\equiv 0$, which is a fundamental tool in the proof of the theorem.
\end{rem}
\begin{proof}
	By hypothesis we have that $\nabla f$ is an eigenvector of $\ric^\p$ at any regular point of $f$, that is, at such a point,
	\begin{align}\label{nabla f eigenvector}
		R^\p_{ik}f_k=\Lambda f_i
	\end{align}
	for some $\Lambda \in \erre$. Therefore, since $\p$ is $\frac{1}{\alpha} U$-harmonic
	\begin{align}\label{pezzo ricci=0}
		\frac{m}{m-1}U^a\p^a_jf_kR^\p_{jk}=\Lambda\frac{m}{m-1}U^a\p^a_jf_j=0
	\end{align}
	and by assumption \eqref{B.2}, \eqref{divergenza Y U harm} rewrites as
	\begin{align}\label{B.5}
		\diver Y=\frac{\alpha}{m-2}\abs{\tau(\p)}^2\abs{\nabla f}^2+\frac{1}{2}(1+\eta(m-2))(m-2)\abs{\overline{D}^\p}^2.
	\end{align}
	For two regular values $\delta, \sigma$ of $f$, with $\sigma<\delta$ , consider the  set
	\begin{align*}
		\Omega_{\delta,\sigma}=\set{x\in M\,:\,\sigma\leq f(x)\leq \delta};
	\end{align*}
	since the map $f$ is proper we can integrate \eqref{B.5} over $\Omega_{\delta,\sigma}$ and we can apply the divergence theorem to obtain
	\begin{align*}
		&\frac{1}{2}(1+\eta(m-2))(m-2)\int_{\Omega_{\delta,\sigma}}\abs{\overline{D}^\p}^2+\frac{\alpha}{m-2}\int_{\Omega_{\delta,\sigma}}\abs{\tau(\p)}^2\abs{\nabla f}^2\\
		&\quad=\int_{\Omega_{\delta,\sigma}}\diver Y
		=\int_{\partial\Omega_{\delta,\sigma}}\!-g(Y, \nu)=0,
	\end{align*}
	where the last equality is implied by \eqref{Y nu =0}. Therefore,
	\begin{align*}
		\frac{1}{2}(1+\eta(m-2))(m-2)\int_{\Omega_{\delta,\sigma}}\abs{\overline{D}^\p}^2+\frac{\alpha}{m-2}\int_{\Omega_{\delta,\sigma}}\abs{\tau(\p)}^2\abs{\nabla f}^2=0
	\end{align*}
	and letting $\delta\ra +\infty,\sigma\ra -\infty$, we have
	\begin{align*}
		\frac{1}{2}(1+\eta(m-2))(m-2)\int_{M}\abs{\overline{D}^\p}^2+\frac{\alpha}{m-2}\int_{M}\abs{\tau(\p)}^2\abs{\nabla f}^2=0.
	\end{align*}
	If $\alpha>0$ and $\eta>-\frac{1}{m-2}$ the left hand side of the above expression is non-negative, while if $\alpha<0$ and $\eta<-\frac{1}{m-2}$ it is non-positive. Either way, we get $\overline{D}^\p=0=\tau(\p)$ and the claim follows by Theorem \ref{almost KO}.

\end{proof}
	\chapter{Other rigidity and related results}\label{Other rigidity results}
\section{Other Rigidity Results}

In this Chapter we prove further rigidity results for a Riemannian manifold $(M,g)$ supporting a structure of the type
\begin{align}\label{Other rigidity: Einstein-type}
	\begin{cases}
		i)\,\ric^\p+\hs(f)-\eta df\otimes df=\lambda g,\\
		ii)\,\tau(\p)=d\p(\nabla f)+\frac{1}{\alpha}(\nabla U)(\p).
	\end{cases}
\end{align}
As in Chapter \ref{Sect_KO}, our main goal is to prove the vanishing of the tensor $\ol{D}^\p$ defined in \eqref{hat D phi}.
In Theorem \ref{KO: Teorema KO ma con la f e la eta}, this was established through the study of the first integrability condition
\begin{align}\label{Other rig: 1 cond di int}
	\sq{1+\eta(m-2)}\ol{D}^{\p}_{ijk}=C^{\p}_{ijk}+f_tW^{\p}_{tijk}-\frac{U^a}{m-1}\pa{\p^a_k\delta_{ij}-\p^a_j\delta_{ik}}
\end{align}
of system (\ref{Other rigidity: Einstein-type}). Using it, one can prove that  condition
\begin{align}\label{Other rig: U cotton conforme è zero}
	2(m-1)\tilde{C}^{\tilde{\p}}+\diver_1(U(\p)g\KN g)=0,
\end{align}
where, as before, tilded quantities are relative to the conformal metric $\tilde{g}=e^{-\frac{2}{m-2}f}g$, is equivalent to $\ol{D}^{\p}\equiv 0$.
Our aim  is now to relax assumption \eqref{Other rig: U cotton conforme è zero}: in particular, we study the consequences of the vanishing of the total divergence $\diver^3C^\p$ of the $\p$-Cotton tensor.
In components, this is defined by
\begin{align*}
	\diver^3 C^\p=C^{\p}_{ijk,kji}.
\end{align*}
In this Chapter we provide a proof of Theorem \ref{intro: teo: div tot cotton 1} and Theorem \ref{intro: teo: div tot cotton 2} of the Introduction, starting with the latter, whose assumptions are closer to those of Theorem \ref{teo: KO: Main results: bach piatto}.
\begin{theorem}\label{Other rigidity: Main Theorem}
	Let $(M,g)$ be a complete Riemannian manifold satisfying system (\ref{Other rigidity: Einstein-type}), where $f$ is a proper function and $\eta\neq -\frac{1}{m-2}$. If $M$ is non-compact, we  also require
	\begin{align*}
		f(x)\to +\infty \ \ \ \text{ as } x\to +\infty.
	\end{align*}
	Assume that
	\begin{align}
		&\diver^3 C^{\p}=0, \label{Other rigidity: Main Theorem: divergenza totale di C=0}\\
		&\p \text{ is }\frac{1}{\alpha}U\text{-harmonic}, \label{Other rigidity: Main theorem: phi costante lungo grad f}\\
		& \nabla f_p \text{ is an eigenvector of } \ric^{\p}_p \text{ for every regular point p of } f. \label{Other rigidity: main theorem: grad f autovettore}
	\end{align}
	Then we have two possibilities:
	\begin{itemize}
		\item[i)] if $\eta\neq 0$ and  we further assume
		 \begin{align}\label{Other rigidity: Main theorem: radial weyl flatness}
			W^{\p}\pa{\nabla f,\cdot,\cdot,\cdot}=0,
		\end{align}
		then we have
		\begin{align}\label{Other rigidity: Main Theorem: S=0}
			C^{\p}=-\frac{1}{2(m-1)}\diver_1\pa{U(\p)g\KN g}
		\end{align}
		and
		\begin{align}\label{Other rigidity: Main Theorem: D=0}
			\overline{D}^{\p}=0;
		\end{align}
		\item[ii)] if $\eta=0$ we have
		\begin{align*}
			C^{\p}=-\frac{1}{2(m-1)}\diver_1\pa{U(\p)g\KN g}
		\end{align*}
		and, if we further assume (\ref{Other rigidity: Main theorem: radial weyl flatness}), also
		\begin{align*}
			\overline{D}^{\p}=0.
		\end{align*}
	\end{itemize}
\end{theorem}
\begin{rem}
	Condition \eqref{Other rigidity: main theorem: grad f autovettore} is necessary for \eqref{Other rigidity: Main Theorem: S=0}, as observed in Lemma \ref{KO: lemma: condizioni equivalenti a grad f autovettore}.
\end{rem}
\begin{theorem}\label{Other rigidity: Main Theorem 2}
	Let $(M,g)$ be a complete Riemannian manifold satisfying system (\ref{Other rigidity: Einstein-type}), where $f$ is a proper function and $\eta\neq -\frac{1}{m-2}$. If $M$ is non-compact, we will also require
	\begin{align*}
		f(x)\to +\infty \ \ \ \text{ as } x\to +\infty.
	\end{align*}
	Assume that
	\begin{align}
		&\diver^3 C^{\p}=0, \label{Other rigidity: Main Theorem 2: divergenza totale di C=0}\\
		&h(\tau(\p), d\p)=0. \label{Other rigidity: Main theorem 2: phi conservativa}
	\end{align}
	Then we have two possibilities:
	\begin{itemize}
		\item[i)] if $\eta\neq 0$ and  we further assume
		 \begin{align}\label{Other rigidity: Main theorem 2: radial weyl flatness}
			W^{\p}\pa{\nabla f,\cdot,\cdot,\cdot}=0,
		\end{align}
		then we have
		\begin{align}\label{Other rigidity: Main Theorem 2: S=0}
			C^{\p}=-\frac{1}{2(m-1)}\diver_1\pa{U(\p)g\KN g}
		\end{align}
		and
		\begin{align}\label{Other rigidity: Main Theorem 2: D=0}
			\overline{D}^{\p}=0;
		\end{align}
		\item[ii)] if $\eta=0$ we have
		\begin{align*}
			C^{\p}=-\frac{1}{2(m-1)}\diver_1\pa{U(\p)g\KN g}
		\end{align*}
		and, if we further assume (\ref{Other rigidity: Main theorem 2: radial weyl flatness}), also
		\begin{align*}
			\overline{D}^{\p}=0.
		\end{align*}
	\end{itemize}
\end{theorem}
Although the assumptions of the two Theorems that we have presented are different, the techniques of their proofs are similar
and we get the  same conclusion.\\
To prove Theorem \ref{Other rigidity: Main Theorem} and Theorem \ref{Other rigidity: Main Theorem 2} we will need some preliminary results;\\
we begin with a very general lemma.
\begin{lemma}\label{lemma: divergenza titale id una 2-forma}
	If $\omega\in \Lambda^2(M)$ is a $2$-form, its \emph{total divergence} $\diver^2\omega$ vanishes. In components,
	\begin{align*}
		\omega_{tk,kt}=0.
	\end{align*}
\end{lemma}
\begin{proof}
	The proof is a simple application of the Ricci commutation formulas. Indeed
	\begin{align*}
		\omega_{tk,kt}&=\frac{1}{2}\pa{\omega_{tk,kt}-\omega_{kt,kt}}\\&=\frac{1}{2}\pa{\omega_{tk,kt}-\omega_{tk,tk}}\\
		&=\frac{1}{2}\pa{R_{ptkt}\omega_{pk}+R_{pkkt}\omega_{tp}}\\
		&=R_{pk}\omega_{pk}=0.
	\end{align*}
\end{proof}
The core of the proof of Theorem \ref{Other rigidity: Main Theorem} lies in the next Proposition.
\begin{proposition}\label{other rigidity: prop: equazione fonamentale}
	Let $(M,g)$ be a Riemannian manifold admitting a solution $f\in C^{\infty}(M)$ of
	\begin{align}\label{Other rigidity: prima equazione}
		\ric^{\p}+\hs(f)-\eta df\otimes df=\lambda g,
	\end{align}
	for $\eta\neq -\frac{1}{m-2}$. Then we have the validity of the following formula:
	\begin{align}\label{Other rigidity: dimostrazione 2: formula fondamentale}
		\frac{1}{2\sq{\pa{m-2}\eta+1}}&\abs{C^{\p}}^2=f_tC^{\p}_{tjk,jk}-\frac{\eta(m-2)}{2(m-2)\eta+2}f_tW^{\p}_{ptjk}C^{\p}_{pjk}\\ \nonumber
		&+\frac{1}{(m-2)\eta+1}\sq{\lambda_k-\frac{1}{2(m-1)}S^{\p}_k-\eta\lambda f_k}C^{\p}_{ttk}\\
		&-\frac{\eta}{(m-2)\eta+1}\sq{f_tR^{\p}_{tk}-\frac{S^{\p}}{(m-1)}f_k}C^{\p}_{ttk}. \nonumber
	\end{align}
\end{proposition}

\begin{proof}
	We apply Lemma \ref{lemma: divergenza titale id una 2-forma} to the 2-form of components
	\begin{align*}
		f_tC^{\p}_{tjk}
	\end{align*}
	to obtain
	\begin{align*}
		0&=\pa{f_tC^{\p}_{tjk}}_{jk}=f_{tjk}C^{\p}_{tjk}+f_{tj}C^{\p}_{tjk,k}+f_{tk}C^{\p}_{tjk,j}+f_{t}C^{\p}_{tjk,jk}\\
		&=f_{tjk}C^{\p}_{tjk}+f_tC^{\p}_{tjk,jk},
	\end{align*}
	that is,
	\begin{align}\label{Other rigidity: dimostrazione 2: prima equazione}
		0=f_{tjk}C^{\p}_{tjk}+f_tC^{\p}_{tjk,jk}.
	\end{align}
	We claim the  validity of the equation
	\begin{align}\label{Other rigidity: dimostrazioen 2: seconda equazione}
		0=&f_tC^{\p}_{tjk,jk}-\frac{1}{2}\abs{C^{\p}}^2+\eta f_kR^{\p}_{tj}C^{\p}_{tjk}\\
		&+C^{\p}_{ttk}\pa{\lambda_k-\frac{1}{2(m-1)}S^{\p}_k-\eta\lambda f_k}. \nonumber
	\end{align}
	To prove it, take covariant derivative of equation (\ref{Other rigidity: prima equazione}) and skew symmetrize with respect to the last two indexes to deduce
	\begin{align*}
		f_{tjk}-f_{tkj}&=-R^{\p}_{tj,k}+R^{\p}_{tk,j}+\eta f_{tk}f_j-\eta f_{tj}f_k+\lambda_k\delta_{tj}-\lambda_j\delta_{tk}.
	\end{align*}
	Using the definition of $C^{\p}$ we get
	\begin{align}\label{ko: other rig: da inserire nella formula sopra}
		f_{tjk}-f_{tkj}=&-C^{\p}_{tjk}-\frac{1}{2(m-1)}\pa{S^{\p}_k\delta_{tj}-S^{\p}_j\delta_{tk}}+\eta f_{tk}f_j-\eta f_{tj}f_k\\
		&+\lambda_k\delta_{tj}-\lambda_j\delta_{tk}.\notag
	\end{align}
	From the skew-symmetry of the last two indexes of $C^\p$ and
	equation (\ref{Other rigidity: dimostrazione 2: prima equazione}) we obtain
	\begin{align*}
		0=f_tC^{\p}_{tjk,jk}+\frac{1}{2}(f_{tjk}-f_{tkj})C^{\p}_{tjk}.
	\end{align*}
	Inserting \eqref{ko: other rig: da inserire nella formula sopra} into the above formula we have
	\begin{align*}
		0=&f_tC^{\p}_{tjk,jk}-\frac{1}{2}\abs{C^{\p}}^2-\frac{1}{2(m-1)}S^{\p}_kC^{\p}_{ttk}+\lambda_kC^{\p}_{ttk}\\
		&+\eta f_{tk}f_j C^{\p}_{tjk}.
	\end{align*}
	Using again equation (\ref{Other rigidity: prima equazione}), which we rewrite as
	\begin{align*}
		f_{tk}=-R^\p_{tk}+\eta f_tf_k+\lambda\delta_{tk},
	\end{align*}
	into the latter, we obtain (\ref{Other rigidity: dimostrazioen 2: seconda equazione}).\\
	The last formula that we will need to conclude is
	\begin{align}\label{Other rigidity: dimostrazione 2: terza equazione}
		&\frac{(m-2)\eta+1}{(m-2)}f_kR^{\p}_{pj}C^{\p}_{pjk}=\frac{1}{2}\abs{C^{\p}}^2-\frac{1}{2}f_tW^{\p}_{ptjk}C^{\p}_{pjk}\\ \nonumber
		&\quad-C^{\p}_{ttk}\pa{\lambda_k-\frac{1}{2(m-1)}S^{\p}_k-\eta\lambda f_k+\frac{1}{m-2}f_tR^{\p}_{tk}-\frac{S^{\p}}{(m-1)(m-2)}f_k}.
	\end{align}
	To deduce it, we apply the Ricci commutation rules to obtain
	\begin{align*}
		f_tC^{\p}_{tjk,jk}&=\frac{1}{2}f_t\pa{C^{\p}_{tjk,jk}-C^{\p}_{tjk,kj}}\\
		&=\frac{1}{2}f_tR_{ptjk}C^{\p}_{pjk}+\frac{1}{2}f_tR_{pjjk}C^{\p}_{tpk}+\frac{1}{2}f_tR_{pkjk}C^{\p}_{tjp}\\
		&=\frac{1}{2}f_tR_{ptjk}C^{\p}_{pjk}+f_tR^{\p}_{pk}C^{\p}_{tpk}\\
		&=\frac{1}{2}f_tR_{ptjk}C^{\p}_{pjk}.
	\end{align*}
	Using the definition of $W^{\p}$ into the above equation  we get
	\begin{align*}
		f_tC^{\p}_{tjk,jk}=&\frac{1}{2}f_tW^{\p}_{ptjk}C^{\p}_{pjk}+\frac{1}{2(m-2)}\pa{f_tR^{\p}_{tk}C^{\p}_{ppk}-f_tR^{\p}_{tj}C^{\p}_{pjp}}\\
		&+\frac{1}{2(m-2)}\pa{f_kR^{\p}_{pj}C^{\p}_{pjk}-f_jR^{\p}_{pk}C^{\p}_{pjk}}\\
		&-\frac{S^{\p}}{2(m-1)(m-2)}\pa{f_kC^{\p}_{ttk}-f_jC^{\p}_{tjt}}\\
		=&\frac{1}{2}f_tW^{\p}_{ptjk}C^{\p}_{pjk}+\frac{1}{(m-2)}f_tR^{\p}_{tk}C^{\p}_{ppk}\\
		&+\frac{1}{(m-2)}f_kR^{\p}_{pj}C^{\p}_{pjk}-\frac{S^{\p}}{(m-1)(m-2)}f_kC^{\p}_{ttk},
	\end{align*}
	that is,
	\begin{align}\label{Other rigidity: dimostrazioen 2: quasi finito}
		f_tC^{\p}_{tjk,jk}=&\frac{1}{2}f_tW^{\p}_{ptjk}C^{\p}_{pjk}+\frac{1}{(m-2)}f_tR^{\p}_{tk}C^{\p}_{ppk}\\
		&+\frac{1}{(m-2)}f_kR^{\p}_{pj}C^{\p}_{pjk}-\frac{S^{\p}}{(m-1)(m-2)}f_kC^{\p}_{ttk}. \nonumber
	\end{align}
	Inserting \eqref{Other rigidity: dimostrazioen 2: quasi finito} into (\ref{Other rigidity: dimostrazioen 2: seconda equazione}) and rearranging we get (\ref{Other rigidity: dimostrazione 2: terza equazione}). Using (\ref{Other rigidity: dimostrazione 2: terza equazione}) into (\ref{Other rigidity: dimostrazioen 2: seconda equazione}) we obtain
	\begin{align*}
		0=&f_tC^{\p}_{tjk,jk}-\frac{1}{2}\abs{C^{\p}}^2+C^{\p}_{ttk}\pa{\lambda_k-\frac{1}{2(m-1)}S^{\p}_k-\eta\lambda f_k}\\
		&+\frac{\eta(m-2)}{(m-2)\eta+1}\pa{\frac{1}{2}\abs{C^{\p}}^2-\frac{1}{2}f_tW^{\p}_{ptjk}C^{\p}_{pjk}}\\
		&-\frac{\eta(m-2)}{(m-2)\eta +1}C^{\p}_{ttk}\left(\lambda_k-\frac{1}{2(m-1)}S^{\p}_k\right.\\
		&\left.-\eta \lambda f_k+\frac{1}{m-2}f_tR^{\p}_{tk}-\frac{S^{\p}}{(m-1)(m-2)}f_k\right),
	\end{align*}
	which, after some simplifications, becomes (\ref{Other rigidity: dimostrazione 2: formula fondamentale}).
\end{proof}

\begin{proposition}
	Let $(M,g)$ be a Riemannian manifold admitting a solution of
	\begin{align*}
		\ric^{\p}+\hs(f)-\eta df\otimes df=\lambda g,
	\end{align*}
	for $\eta\neq -\frac{1}{m-2}$.
	Let $z:\erre\to \erre$ be a smooth function such that $z(f)$ is compactly supported on $M$.
	Then we have the validity of the following integral formula:
	\begin{align}\label{Other rigidity: dimostrazione 2: formula integrale fondamentale}
		&\frac{1}{2\sq{\pa{m-2}\eta+1}}\int_M\abs{C^{\p}}^2\sq{z(f)-z'(f)}e^{-f}=-\int_MC^{\p}_{tjk,kjt}z(f)e^{-f}\\ \nonumber
		&\quad-\frac{\eta(m-2)}{2(m-2)\eta+2}\int_M f_tW^{\p}_{ptjk}C^{\p}_{pjk}\sq{z(f)-z'(f)}e^{-f}\\ \nonumber
		&\quad+\frac{1}{(m-2)\eta+1}\int_M \pa{\lambda_k-\frac{1}{2(m-1)}S^{\p}_k-\eta\lambda f_k)}C^{\p}_{ttk}\sq{z(f)-z'(f)}e^{-f}\\
		&\quad-\frac{1}{(m-2)\eta+1}\pa{ f_tR^{\p}_{tk}-\frac{S^{\p}}{(m-1)}f_k}C^{\p}_{ttk}\sq{z(f)-z'(f)}e^{-f}. \nonumber
	\end{align}
\end{proposition}
\begin{proof}
	Integrating by parts, we have
	\begin{align*}
		\int_M f_tC^{\p}_{tik,ki}z(f)e^{-f}=&\int_M C^{\p}_{tik,kit}z(f)e^{-f}\\
		&+\int_M f_tC^{\p}_{tik,ki}z'(f)e^{-f},
	\end{align*}
	which implies
	\begin{align}\label{other rigidities: dimostrazione 2 formula integrale step 1}
		\int_M f_t C^{\p}_{tik,ki}\sq{z(f)-z'(f)}e^{-f}=\int_M \diver^3C^{\p}z(f) e^{-f}.
	\end{align}
	Multiplying (\ref{Other rigidity: dimostrazione 2: formula fondamentale}) by $[z(f)-z'(f)]e^{-f}$, integrating on $M$ and using (\ref{other rigidities: dimostrazione 2 formula integrale step 1}) we immediately obtain (\ref{Other rigidity: dimostrazione 2: formula integrale fondamentale}).
	
\end{proof}

As a computational short-hand, we define
\begin{align}\label{other rigidities: tensore F}
	F_{ijk}=C^{\p}_{ijk}-\frac{1}{m-1}U^a\p^a_k\delta_{ij}+\frac{1}{m-1}U^a\p^a_j\delta_{ik}.
\end{align}
Then, if we assume the validity of  (\ref{Other rigidity: Main theorem: phi costante lungo grad f}), we have
\begin{align*}
	\abs{F}^2&=\abs{C^{\p}}^2+\frac{2m}{(m-1)^2}\abs{\nabla (U(\p))}^2-\frac{2}{(m-1)^2}\abs{\nabla (U(\p))}^2-\frac{4}{m-1}C^{\p}_{ijk}\pa{U^a\p^a_k\delta_{ij}}\\
	&=\abs{C^{\p}}^2+\frac{2}{m-1}\abs{\nabla (U(\p))}^2-\frac{4\alpha}{m-1}\p^a_{tt}\p^a_kU^b\p^b_k\\
	&=\abs{C^{\p}}^2-\frac{2}{m-1}\abs{\nabla (U(\p))}^2
\end{align*}
where we have used
\begin{align*}
	C^{\p}_{ttk}=\alpha \p^a_{tt}\p^a_k
\end{align*}
and (\ref{Other rigidity: Main theorem: phi costante lungo grad f}).
Thus we have
\begin{align}\label{other rigidity: norma di F}
	\abs{F}^2=\abs{C^{\p}}^2-\frac{2}{m-1}\abs{\nabla (U(\p))}^2.
\end{align}
\begin{cor}
	Let $(M,g)$ satisfy system (\ref{Other rigidity: Einstein-type}) with $\eta\neq-\frac{1}{m-2}$. Assume that
	\begin{align}
		& \p \text{ is } \frac{1}{\alpha}U\text{-harmonic}, \label{ennesimo cor: U armonicità} \\
		& \nabla f_p \text{ is an eigenvector of } \ric^{\p}_p, \text{ for every regular point p of } f. \label{ennesimo cor: autovettore}
	\end{align}
	Moreover, when $\eta\neq 0$, also assume
	\begin{align*}
		W^{\p}\pa{\nabla f,\cdot,\cdot,\cdot}=0.
	\end{align*}
	Let $z:\erre\to \erre$ be a smooth function such that $z(f)$ is compactly supported on $M$.
	Then
	\begin{align}\label{Other rigidity : dimostrazione 2: corollario finale}
		\frac{1}{2\sq{1+(m-2)\eta}}\int_M\abs{F}^2[z(f)-z'(f)]e^{-f}=-\int_M\diver^3 C^\p z(f)e^{-f},
	\end{align}
	where the compomnents of $F$ have been defined in (\ref{other rigidities: tensore F}).
\end{cor}
\begin{proof}
	From the second equation of (\ref{Other rigidity: Einstein-type}) and since $\p$ is $\frac{1}{\alpha}U$-harmonic we get
	\begin{align*}
		\p^a_tf_t=0.
	\end{align*}
	Since
	\begin{align*}
		C^{\p}_{ttk}=\alpha \p^a_{tt}\p^a_k,
	\end{align*}
	we obtain that, from assumptions \eqref{ennesimo cor: U armonicità} and \eqref{ennesimo cor: autovettore}, equation (\ref{Other rigidity: dimostrazione 2: formula integrale fondamentale}) reduces to
	\begin{align*}
		&\frac{1}{2\sq{\pa{m-2}\eta+1}}\int_M\abs{C^{\p}}^2[z(f)-z'(f)]e^{-f}=-\int_M\diver^3 C^\p z(f)e^{-f}\\
		&\quad+\int_M\alpha\p^a_{tt}\p^a_k\sq{\frac{1}{(m-2)\eta+1}\pa{\lambda_k-\frac{1}{2(m-1)}S^{\p}_k}}[z(f)-z'(f)]e^{-f}.
	\end{align*}
	To deal with the last term, recall the validity of (\ref{A.11}), that is,
	\begin{align}\label{Other rigidity: A.11 rivisitata}
		\frac{1}{2}S^\p_k=R^\p_{ik}f_i+\eta(f_{ik}f_i-\Delta f f_k)+(m-1)\lambda_k-\alpha\p^a_k(\p^a_{tt}-\p^a_tf_t).
	\end{align}
	This, together with assumptions \eqref{ennesimo cor: U armonicità} and \eqref{ennesimo cor: autovettore}, implies
	\begin{align*}
		&\alpha\p^a_{tt}\p^a_k\sq{\frac{1}{(m-2)\eta+1}\pa{\lambda_k-\frac{1}{2(m-1)}S^{\p}_k}}\\
		&\qquad=\frac{\alpha^2}{(m-1)[(m-2)\eta+1]}\p^a_{tt}\p^a_k\p^b_{pp}\p^b_k\\
		&\qquad=\frac{1}{(m-1)[(m-2)\eta+1]}\abs{\nabla U(\p)}^2,
	\end{align*}
	and therefore
	\begin{align*}
		&\frac{1}{2\sq{\pa{m-2}\eta+1}}\int_M\abs{C^{\p}}^2[z(f)-z'(f)]e^{-f}=-\int_M \diver^3 C^\p z(f)e^{-f}\\
		&\quad+\frac{1}{(m-1)[(m-2)\eta+1]}\int_M\abs{\nabla U(\p)}^2[z(f)-z'(f)]e^{-f}.
	\end{align*}
	Rearranging and using equation (\ref{other rigidity: norma di F}) we conclude.
\end{proof}

We are ready for the proof of Theorem \ref{Other rigidity: Main Theorem}. We will use equation (\ref{Other rigidity : dimostrazione 2: corollario finale}) together with a trick taken from \cite{catino2016gradientriccisolitonsvanishing}.

\begin{proof}[Proof (of Theorem \ref{Other rigidity: Main Theorem})]
	From equation (\ref{Other rigidity : dimostrazione 2: corollario finale}) and assumption (\ref{Other rigidity: Main Theorem: divergenza totale di C=0}) we get
	\begin{align}\label{other rigidity: main theorem: formula fondamentale 4}
		\frac{1}{2[1+(m-2)\eta]}\int_M\abs{F}^2e^{-f}\pa{z(f)-z'(f)}=0.
	\end{align}
	If $M$ is compact, set $z(f)\equiv 1$ to obtain $F\equiv 0$ and therefore (\ref{Other rigidity: Main Theorem: S=0}). If $M$ is non-compact, let
	\begin{align*}
		z_k(t)=\begin{cases}
			&1  \ \ \ \ \ \ \ t\in[-k,k],\\
			&\frac{2k+t}{k} \ \ \ \  t\in [-2k,-k],\\
			&\frac{2k-t}{k} \ \ \ \ t\in [k,2k],\\
			&0 \ \ \ \ \ \ \ t\in (-\infty,-2k]\cup [2k,+\infty)
		\end{cases}
	\end{align*}
	for  a positive natural number $k$.
	Since $f(x)\to +\infty$ as $x\to +\infty$ by assumptions, there exists $A\in \erre$ such that
	\begin{align*}
		f(x)\geq -2A \ \ \ \text{ on $M$}.
	\end{align*}
	Let $k\geq 2A$. Then, for $f(x)\in [-2k,-k]$, we have
	\begin{align*}
		z_k(f(x))-z_k'(f(x))=2+\frac{f(x)-1}{k}\geq 2-\frac{2A+1}{k}\geq 1-\frac{1}{k}
	\end{align*}
	and, for $f(x)\in [k,2k]$, we have
	\begin{align*}
		z_k(f(x))-z_k'(f(x))=2-\frac{f(x)-1}{k}\geq \frac{1}{k}.
	\end{align*}
	Therefore, from (\ref{other rigidity: main theorem: formula fondamentale 4}) we deduce
	\begin{align*}
		0=\int_{\overline{\Omega}_{2k}}\abs{F}^2e^{-f}\pa{z_k(f)-z_k'(f)}\geq\frac{1}{k}\int_{\overline{\Omega}_{2k}}\abs{F}^2e^{-f}
	\end{align*}
	where
	\begin{align*}
		\overline{\Omega}_{2k}=\overline{\set{x\in M: f(x)\in [-2k,2k]}}.
	\end{align*}
	Letting $k\to +\infty$ we get $F\equiv 0$ and therefore equation (\ref{Other rigidity: Main Theorem: S=0}).
	To conclude, if $F\equiv 0$ and $f_tW^{\p}_{tijk}=0$, the first integrability condition (\ref{Other rig: 1 cond di int}) of system (\ref{Other rigidity: Einstein-type}) implies $\overline{D}^{\p}\equiv 0$.
\end{proof}
\begin{proof}[Proof (of Theorem \ref{Other rigidity: Main Theorem 2})]
	Assumption (\ref{Other rigidity: Main theorem 2: phi conservativa}) reads, in components,
	\begin{align*}
		\p^a_{tt}\p^a_k=0, \ \ \forall i=1,...,m.
	\end{align*}
	Since $C^\p_{ttk}=\alpha \p^a_{tt}\p^a_k$
	we deduce
	\begin{align*}
		C^{\p}_{ttk}=0.
	\end{align*}
	Therefore, when either $\eta=0$ or the zero radial Weyl curvature condition (\ref{Other rigidity: Main theorem 2: radial weyl flatness}) holds, using (\ref{Other rigidity: Main theorem 2: phi conservativa}) and (\ref{Other rigidity: dimostrazione 2: formula integrale fondamentale}), we deduce that, for any smooth, compactly supported function $z:\erre\to \erre$, we have
	\begin{align*}
		\frac{1}{2\sq{(m-2)\eta +1}}\int_M\abs{C^\p}^2\sq{z(f)-z'(f)}e^{-f}=-\int _M C^\p_{tjk,kjt}z(f)e^{-f}.
	\end{align*}
	The proof now continues as that of Theorem \ref{Other rigidity: Main Theorem}, replacing $F$ with $C^\p$, to give
	\begin{align*}
		C^\p\equiv 0.
	\end{align*}
	From the first integrability condition (\ref{Other rig: 1 cond di int}), we deduce
	\begin{align*}
		\sq{1+\eta(m-2)}\ol{D}^{\p}_{ijk}=f_tW^{\p}_{tijk}-\frac{U^a\p^a_k}{m-1}\delta_{ij}+\frac{U^a\p^a_j}{m-1}\delta_{ik}.
	\end{align*}
	If we are assuming (\ref{Other rigidity: Main theorem 2: radial weyl flatness}), then we have
	\begin{align}\label{other rigidities: Main Theorem 2: abbiamo finito}
		\sq{1+\eta(m-2)}\ol{D}^{\p}_{ijk}=-\frac{U^a\p^a_k}{m-1}\delta_{ij}+\frac{U^a\p^a_j}{m-1}\delta_{ik}.
	\end{align}
	Since $\ol{D}^\p$ is totally trace-free, tracing the above equation with respect to $i$ and $j$ gives
	\begin{align*}
		0=\sq{1+\eta(m-2)}\ol{D}^\p_{ttk}=-\frac{m}{m-1}U^a\p^a_k+\frac{1}{m-1}U^a\p^a_k=-U^a\p^a_k.
	\end{align*}
	From equation (\ref{other rigidities: Main Theorem 2: abbiamo finito}) and assumption $\eta\neq -\frac{1}{m-2}$ we deduce $\ol{D}^\p\equiv 0$, and this concludes the proof.
\end{proof}

\section{The Boundary Case}
In this subsection we prove a rigidity result for Riemannian manifolds with non-empty boundary.\\
Before focusing on manifolds with boundary, we highlight some facts concerning system \eqref{Other rigidity: Einstein-type}, that is,
\begin{align*}
	\begin{cases}
		i)\,\ric^\p+\hs(f)-\eta df\otimes df=\lambda g,\\
		ii)\,\tau(\p)=d\p(\nabla f)+\frac{1}{\alpha}(\nabla U)(\p),
	\end{cases}
\end{align*}
where
\begin{align*}
	m\lambda=S^\p+\Delta f-\eta\abs{\nabla f}^2.
\end{align*}
If we let
\begin{align}\label{other rigidity: the bounday case: change of var}
	u=e^{-\eta f},
\end{align}
then system \eqref{Other rigidity: Einstein-type} transforms into
\begin{align}\label{Other rigidity results: the boundary case: einstein type con u}
	\begin{cases}
		i)\,u\eta\ric^\p-\hs(u)=\frac{1}{m}\pa{\eta uS^\p-\Delta u}g\\
		ii)\,\eta u\tau(\p)=-d\p(\nabla u)+\frac{\eta u}{\alpha}(\nabla U)(\p).
	\end{cases}
\end{align}
Note that looking for solutions of system \eqref{Other rigidity: Einstein-type} is equivalent to finding positive solutions of system \eqref{Other rigidity results: the boundary case: einstein type con u}; however, the latter system can be considered an extension of \eqref{Other rigidity: Einstein-type} if one looks for possibly changing-sign solutions $u$. In our setting, system \eqref{Other rigidity results: the boundary case: einstein type con u} will play a fundamental role in the study of manifolds with boundary, $\partial M=u^{-1}(\set{0})$.\\

\noindent
From now on, let $(M,g)$ be a connected Riemannian manifold of dimension $m\geq 3$ with non-empty boundary; let $\p:(M,g)\ra (N,h)$, where $(N,h)$ is a Riemannian manifold of dimension $n$, be a smooth map, let $U:(N,h)\ra \erre$ be a smooth function and let $u\in C^{\infty}(M)$ be a solution of \eqref{Other rigidity results: the boundary case: einstein type con u} such that $u>0$ on $\mathrm{int}(M)$ and $\partial M=u^{-1}(\set{0})$. In this setting, we have the following
\begin{theorem}\label{other rigidity results: the boundary case: analogo thm 1.44}
	Let $(M,g)$ be a connected Riemannian manifold of dimension $m\geq 3$ and let $u\in C^{\infty}(M)$ be a solution of system \eqref{Other rigidity results: the boundary case: einstein type con u}, with $\alpha\in \erre\setminus \set{0}$ and $\eta\neq- \frac{1}{m-2},\  0$. Assume that
	\begin{align}
		&\diver^3 C^{\p}=0;\label{other rigidity: the boundary case: div3 cotton}\\
		&\p \text{ is } \frac{1}{\alpha}U\text{-harmonic};\label{other rigidity: the boundary case: gradient}\\
		&\nabla f_p \text{ is an eigenvector of } \ric^\p_p, \text{ for every regular point } p \text{ of }f;\label{other rigidity: the boundary case: autov}\\
		&W^\p\pa{\nabla f,\cdot,\cdot,\cdot}=0\label{other rigidity: the boundary case: weyl}.
	\end{align}
	Then
	\begin{align}\label{other rigidity: the boundary case: cotton e U}
		C^\p=-\frac{1}{2(m-1)}\diver_1\pa{U(\p)g\KN g}
	\end{align}
	and
	\begin{align*}
		\ol{D}^\p\equiv 0.
	\end{align*}
\end{theorem}
For the proof of Theorem \ref{other rigidity results: the boundary case: analogo thm 1.44} we need some preliminary results. First we introduce the $(0,3)$-tensor $D^\p$, whose components are
\begin{align}\label{other rigidity: the boundary case: D}
	(m-2)D^\p_{ijk}:=u_{ik}u_j-u_{ij}u_k+\frac{u_t}{m-1}(u_{tj}\delta_{ik}-u_{tk}\delta_{ij})+\frac{\Delta u}{m-1}(u_k\delta_{ij}-u_j\delta_{ik}).
\end{align}

\begin{lemma}
	Let $(M,g)$ be a connected Riemannian manifold of dimension $m\geq 3$, let $u\in C^{\infty}(M)$ be a solution of system \eqref{Other rigidity results: the boundary case: einstein type con u}, with $\alpha\in \erre\setminus \set{0}$ and $\eta\neq- \frac{1}{m-2},0$ and let $D^\p$ be the tensor field on $M$ defined in \eqref{other rigidity: the boundary case: D}. Then,
	\begin{align}\label{other rigidity: the boundary case: 1st int}
		\frac{[1+\eta(m-2)]}{\eta u}D^\p=\eta u C^\p-W^\p\pa{\nabla u,\cdot,\cdot,\cdot}+\eta u\frac{1}{2(m-1)}\diver_1\pa{U(\p)g\KN g}.
	\end{align}
	on $\mathrm{int}(M)$.
\end{lemma}
\begin{proof}
	Note that, when we consider the change of variable \eqref{other rigidity: the bounday case: change of var}, the first integrability condition \eqref{other rigidity: the boundary case: 1st int} of system \eqref{Other rigidity results: the boundary case: einstein type con u} can be easily obtained by \eqref{Other rig: 1 cond di int}. Indeed, by \eqref{D hat phi when system KO holds}, we have
	\begin{align*}
		\ol{D}^\p_{ijk}=(u\eta)^{-2}D^{\p}_{ijk};
	\end{align*}
	hence, by \eqref{Other rig: 1 cond di int}, we deduce \eqref{other rigidity: the boundary case: 1st int}.
\end{proof}
\begin{lemma}\label{other rigidity results: the boundary case: lemma con Z}
	Let $(M,g)$ be a connected Riemannian manifold of dimension $m\geq 3$ and let $u\in C^{\infty}(M)$ be a solution of system \eqref{Other rigidity results: the boundary case: einstein type con u}, with $\alpha\in \erre\setminus \set{0}$ and $\eta\neq -\frac{1}{m-2},\ 0$. Define the vector field $Z$ on $M$ of components
	\begin{align}\label{other rigidity: the boundary case: def of Z}
		Z_i=&u\set{R^\p_{tk}W^\p_{tikj}}_j-\pa{\frac{m-4}{m-2}}uR^\p_{tk}C^\p_{tki}+2\alpha u\p^a_i\p^a_{jk}R^\p_{jk}-\alpha\set{u R^\p_{tk}\pa{\p^a_t\p^a_i}_k}\notag\\
		&-\alpha\set{u\pa{\frac{1}{2}S^\p_t\p^a_t\p^a_i-\alpha\p^b_{ss}\p^b_t\p^a_t\p^a_i}}\notag\\
		&+\sq{u\pa{\p^a_{ik}U^a-U^{ab}\p^b_k\p^a_i-\frac{\alpha}{m-2}\abs{\tau(\p)}^2\delta_{ik}}_{k}},
	\end{align}
	and assume that $\p$ is $\frac{1}{\alpha}U$-harmonic and $\nabla u$ is an eigenvector of $\ric^\p$, for every regular point of $u$ (that are assumptions \eqref{other rigidity: the boundary case: gradient} and \eqref{other rigidity: the boundary case: autov}, respectively). Then, on $\mathrm{int}(M)$, we have
	\begin{align}\label{other rigidity: the boundary case: div di Z}
		\diver Z=&-u_tW^\p_{tijk}\pa{\frac{1}{2(\eta u)^2}D^\p_{ijk}+R^\p_{ij,k}}-\frac{[1+\eta(m-2)]}{(\eta u)^3}\abs{D^\p}^2\\
		&-uC^\p_{ikj,jki}+\pa{\frac{m-4}{m-2}}u\sq{U^a\pa{\frac{1}{\alpha}U^{ab}\p^b_i+R^\p_{ij}\p^a_j}}_i\notag\\
		&+\bigg\{\p^a_i\pa{\frac{m}{(m-1)(m-2)}U^aS^\p+2\alpha U^b\p^b_j\p^a_i}-\frac{\alpha}{2}\pa{\frac{m-2}{m-1}}\p^a_i\p^a_jS^\p_j\notag \\
		&-2\alpha\p^a_{jk}R^\p_{jk}\p^a_i-\alpha\p^a_i\tau^a_2(\p)\bigg\}_iu.\notag
	\end{align}
	If we further assume
	$$u_tW^\p_{tijk}=0,$$
	equation \eqref{other rigidity: the boundary case: div di Z} becomes
	\begin{align}\label{other rigidity: the boundary case: div di Z con hp su weyl}
		\diver Z=&-\frac{[1+\eta(m-2)]}{(\eta u)^3}\abs{D^\p}^2-uC^\p_{ikj,jki}+\pa{\frac{m-4}{m-2}}u\sq{U^a\pa{U^{ab}\p^b_i+R^\p_{ij}\p^a_j}}_i\\
		&+\bigg\{\p^a_i\pa{\frac{m}{(m-1)(m-2)}U^aS^\p+2\alpha U^b\p^b_j\p^a_i}-\frac{\alpha}{2}\frac{m-2}{m-1}\p^a_i\p^a_jS^\p_j\notag\\
		&-2\alpha\p^a_{jk}R^\p_{jk}\p^a_i-\alpha\p^a_i\tau^a_2(\p)\bigg\}_iu.\notag
	\end{align}
\end{lemma}
\begin{proof}
	Let $X$ be the vector field of components
	\begin{align*}
		X_i=u\set{R^\p_{tk}W^\p_{tikj}}_k-\pa{\frac{m-4}{m-2}}uR^\p_{tk}C^\p_{tki}+2\alpha u\p^a_i\p^a_{jk}R^\p_{jk};
	\end{align*}
	computing its divergence and using $0=\tau(\p)-\frac{1}{\alpha}(\nabla U)(\p)$ and hence $d\p(\nabla u)=0$, we obtain
	\begin{align}\label{other rigidity: boundary: primo calcolo di div X}
		X_{ii}=&\set{u\sq{R^\p_{tk}W^\p_{tikj}}_k}_i-\pa{\frac{m-4}{m-2}}u_iR^\p_{tk}C^\p_{tki}-\pa{\frac{m-4}{m-2}}uR^\p_{tk,i}C^\p_{tki}-\pa{\frac{m-4}{m-2}}uR^\p_{tk}C^\p_{tki,i} \\
		&+2\alpha u_i\p^a_i\p^a_{jk}R^\p_{jk}+2\alpha u\p^a_{ii}\p^a_{jk}R^\p_{jk}+2\alpha u\p^a_i\p^a_{jki}R^\p_{jk}+2\alpha u\p^a_i\p^a_{jk}R^\p_{jk,i} \nonumber\\
		=&\set{u\sq{R^\p_{tk}W^\p_{tikj}}_k}_i-\pa{\frac{m-4}{m-2}}u_iR^\p_{tk}C^\p_{tki}-\pa{\frac{m-4}{m-2}}uR^\p_{tk,i}C^\p_{tki}-\pa{\frac{m-4}{m-2}}uR^ \p_{tk}C^\p_{tki,i}\nonumber \\
		&+2\alpha u\p^a_{ii}\p^a_{jk}R^\p_{jk}+2\alpha u\p^a_i\p^a_{jki}R^\p_{jk}+2\alpha u\p^a_i\p^a_{jk}R^\p_{jk,i}\nonumber \\
		=&\set{u\sq{R^\p_{tk}W^\p_{tikj}}_k}_i-\pa{\frac{m-4}{m-2}}u_iR^\p_{tk}C^\p_{tki}-\pa{\frac{m-4}{m-2}}u\set{R^\p_{tk}C^\p_{tki}}_i+2\alpha u\set{\p^a_i\p^a_{jk}R^\p_{jk}}_{\!i}\notag.
	\end{align}
	Since $\p$ is $\frac{1}{\alpha}U$-harmonic, equation \eqref{other rigidity results: the boundary case: divergenza di bach} rewrites as
	\begin{align}\label{other rigidity: the boundary case: div bach con U harm}
		(m-2)B^\p_{ik,k}=&\frac{m-4}{m-2}\sq{R^\p_{jk}C^\p_{jki}+U^a\left(\frac{1}{\alpha}U^{ab}\p^b_i+R^\p_{ij}\p^a_j\right)}\\ \nonumber
		&+\p^a_i\pa{\frac{m}{(m-1)(m-2)}U^aS^\p+2\alpha U^b\p^b_j\p^a_i}\\ \nonumber
		&-\frac{\alpha}{2}\frac{m-2}{m-1}\p^a_i\p^a_jS^\p_j-2\alpha\p^a_{jk}R^\p_{jk}\p^a_i-\alpha\p^a_i\tau^a_2(\p);
	\end{align}
	therefore, taking its divergence, we obtain
	\begin{align}\label{other rigidity: the boundary case: div della div di bach}
		(m-2)B^\p_{ik,ki}=&\pa{\frac{m-4}{m-2}}\pa{R^\p_{jk}C^\p_{jki}}_i+\pa{\frac{m-4}{m-2}}\pa{U^a(\frac{1}{\alpha}U^{ab}\p^b_i+R^\p_{ij}\p^a_j)}_i\\ \nonumber
		&+\Bigg[\p^a_i\pa{\frac{m}{(m-1)(m-2)}U^aS^\p+2\alpha U^b\p^b_j\p^a_i}\\ \nonumber
		&-\frac{\alpha}{2}\frac{m-2}{m-1}\p^a_i\p^a_jS^\p_j-2\alpha\p^a_{jk}R^\p_{jk}\p^a_i-\alpha\p^a_i\tau^a_2(\p)\Bigg]_i.
	\end{align}
	To simplify the writing we set
	\begin{align*}
		\mathcal{A}=&-\Bigg[\p^a_i\pa{\frac{m}{(m-1)(m-2)}U^aS^\p+2\alpha U^b\p^b_j\p^a_i}\\
		&-\frac{\alpha}{2}\frac{m-2}{m-1}\p^a_i\p^a_jS^\p_j-\alpha\p^a_i\tau^a_2(\p)\Bigg]_i.
	\end{align*}
	As a consequence, using \eqref{other rigidity: the boundary case: div della div di bach}, we have
	\begin{align}\label{other rigidity: the boundary case: pezzo con ricci, cotton ecc da mettere in div X}
		&u\frac{m-4}{m-2}(R^\p_{jk}C^\p_{jki})_i-2\alpha u(\p^a_{jk}R^\p_{jk}\p^a_i)_i\\
		&\quad=u(m-2)B^\p_{ik,ki}-\pa{\frac{m-4}{m-2}}u\pa{\frac{1}{\alpha}U^a(U^{ab}\p^b_i+R^\p_{ij}\p^a_j)}_i+\mathcal{A}u.\notag
	\end{align}
	Using the first equation of \eqref{Other rigidity results: the boundary case: einstein type con u} into \eqref{other rigidity: the boundary case: D}, we get
	\begin{align}\label{other rigidity: the boundary case: riscrittura di D}
		(m-2)D^\p_{ijk}=&\eta u R^\p_{ik}u_j-\frac{1}{m}(\eta u S^\p-\Delta u)\delta_{ik}u_j-\eta u R^\p_{ij}u_k+\frac{1}{m}(\eta u S^\p-\Delta u)\delta_{ij}u_k\\
		&+\frac{u_t}{m-1}\bigg[R^\p_{tj}\delta_{ik}\eta u-\frac{1}{m}(\eta u S^\p-\Delta u)\delta_{ik}\delta_{tj}-R^\p_{tk}\delta_{ij}\eta u \notag \\
		&+\frac{1}{m}(\eta u S^\p-\Delta u)\delta_{ij}\delta_{tk}\bigg]+\frac{1}{m-1}\Delta u (u_k\delta_{ij}-u_j\delta_{ik})\notag\\
		=&\eta u\left\{R^\p_{ik}u_j-R^\p_{ij}u_k+\frac{u_t}{m-1}\pa{R^\p_{tj}\delta_{ik}-R^\p_{tk}\delta_{ij}}\right.\notag\\
		&\left.+\frac{S^\p}{m-1}\pa{u_k\delta_{ij}-u_j\delta_{ik}}\right\}\notag.\notag
	\end{align}
	Hence, by \eqref{other rigidity: the boundary case: gradient}, \eqref{other rigidity: the boundary case: autov} and \eqref{other rigidity: the boundary case: riscrittura di D}, we deduce
	\begin{align}\label{other rigidity: cotton e ricci uguale d e cotton}
		(m-2)D^\p_{ijk}C^\p_{ijk}=-2\eta u u_kR^\p_{ij}C^\p_{ijk}.
	\end{align}
	Thus, inserting (\ref{other rigidity: cotton e ricci uguale d e cotton}) and (\ref{other rigidity: the boundary case: pezzo con ricci, cotton ecc da mettere in div X}) into (\ref{other rigidity: boundary: primo calcolo di div X}), we get the validity of
	\begin{align}\label{other rigidity: div of X non finita 1}
		X_{ii}=&\set{u\sq{R^\p_{tk}W^\p_{tikj}}_j}_i+\frac{m-4}{2\eta u}D^\p_{ijk}C^\p_{ijk}-u(m-2)B^\p_{ik,ki}\notag\\
		&+\pa{\frac{m-4}{m-2}}u\sq{U^a\pa{\frac{1}{\alpha}U^{ab}\p^b_i+R^\p_{ij}\p^a_j}}_i-\mathcal{A}u
	\end{align}
	on $\mathrm{int}(M)$.\\
	Next, using the definition of $B^\p$ and that $\p$ is $\frac{1}{\alpha}U$-harmonic, we obtain
	\begin{align*}
		u(m-2)B^\p_{ik,ki}=&uC^\p_{ikj,jki}+u(R^\p_{st}W^\p_{sitk})_{ki}-\alpha u(R_{tk}^\p\p^a_t\p^a_i)_{ki}\\
		&+u\pa{\p^a_{ij}U^a-U^{ab}\p^b_k\p^a_i-\frac{\alpha }{m-2}\abs{\tau(\p)}^2\delta_{ik}}_{ki}\\
		=&uC^\p_{ikj,jki}+\sq{u(R^\p_{st}W^\p_{sitk})_{k}}_i-u_i(R^\p_{st}W^\p_{sitk})_{k}\\
		&-\alpha \sq{u(R_{tk}^\p\p^a_t\p^a_i)_{k}}_i+\alpha u_i(R_{tk}^\p\p^a_t\p^a_i)_{k}\\
		&+\sq{u\pa{\p^a_{ik}U^a-U^{ab}\p^b_k\p^a_i-\frac{\alpha}{m-2}\abs{\tau(\p)}^2\delta_{ik}}_{k}}_i\\
		&-u_i\pa{\p^a_{ik}U^a-U^{ab}\p^b_k\p^a_i-\frac{\alpha}{m-2}\abs{\tau(\p)}^2\delta_{ik}}_{k}.
	\end{align*}
	Inserting the latter into \eqref{other rigidity: div of X non finita 1}, we obtain, on $\mathrm{int}(M)$,
	\begin{align}\label{other rigidity: div X non finita 2}
		X_{ii}=&\frac{m-4}{2\eta u}D^\p_{ijk}C^\p_{ijk}-uC^\p_{ikj,jki}+u_i\set{R^\p_{st}W^\p_{sitk}}_k+\alpha \sq{u(R_{tk}^\p\p^a_t\p^a_i)_{k}}_i\\
		&-\alpha u_i(R_{tk}^\p\p^a_t\p^a_i)_{k}-\sq{u\pa{\p^a_{ik}U^a-U^{ab}\p^b_k\p^a_i-\frac{\alpha}{m-2}\abs{\tau(\p)}^2\delta_{ik}}_{k}}_i\notag\\
		&+u_i\pa{\p^a_{ik}U^a-U^{ab}\p^b_k\p^a_i-\frac{\alpha}{m-2}\abs{\tau(\p)}^2\delta_{ik}}_{k}\notag\\
		&+\pa{\frac{m-4}{m-2}}u\sq{U^a\pa{\frac{1}{\alpha}U^{ab}\p^b_i+R^\p_{ij}\p^a_j}}_i-\mathcal{A}u\notag.
	\end{align}
	Observe that by \eqref{phi weyl e phi cotton}, since $\p$ is $\frac{1}{\alpha}U$-harmonic and $\nabla u$ is an eigenvector of $\ric^{\p}$ for every regular point of $u$, we get
	\begin{align*}
		u_i\set{R^\p_{st}W^\p_{sitk}}_k&=u_iR^\p_{st,k}W^\p_{sitk}+u_iR^{\p}_{st}W^\p_{sitk,k}\\
		&=u_iR^\p_{st,k}W^\p_{sitk}+\pa{\frac{m-3}{m-2}}u_iR^\p_{st}C^\p_{tsi}+\alpha u_iR^\p_{ts}\p^a_{ti}\p^a_s
	\end{align*}
 which	by \eqref{other rigidity: cotton e ricci uguale d e cotton}, on $\mathrm{int}(M)$, can be rewritten as
	\begin{align*}
		u_i\set{R^\p_{st}W^\p_{sitk}}_k&=u_iR^\p_{st,k}W^\p_{sitk}+\pa{\frac{m-3}{m-2}}u_iR^\p_{st}C^\p_{tsi}+\alpha u_iR^\p_{ts}(\p^a_i\p^a_s)_t\\
		&=u_iR^\p_{st,k}W^\p_{sitk}-\frac{m-3}{2\eta u}D^\p_{tsi}C^\p_{tsi}+\alpha u_iR^\p_{ts}(\p^a_i\p^a_s)_t.
	\end{align*}
	Therefore, on $\mathrm{int}(M)$, we have
	\begin{align*}
		X_{ii}=&-\frac{1}{2\eta u}D^\p_{ijk}C^\p_{ijk}-uC^\p_{ikj,jki}+u_iR^\p_{st,k}W^\p_{sitk}+\alpha u_iR^\p_{ts}(\p^a_i\p^a_s)_t\\
		&+\alpha \sq{u(R_{tk}^\p\p^a_t\p^a_i)_{k}}_i-\alpha u_i(R_{tk}^\p\p^a_t\p^a_i)_{k}\\
		&-\sq{u\pa{\p^a_{ik}U^a-U^{ab}\p^b_k\p^a_i-\frac{\alpha}{m-2}\abs{\tau(\p)}^2\delta_{ik}}_{k}}_i\\
		&+u_i\pa{\p^a_{ik}U^a-U^{ab}\p^b_k\p^a_i-\frac{\alpha}{m-2}\abs{\tau(\p)}^2\delta_{ik}}_{k}\\
		&+\pa{\frac{m-4}{m-2}}u\sq{U^a\pa{\frac{1}{\alpha}U^{ab}\p^b_i+R^\p_{ij}\p^a_j}}_i-\mathcal{A}u;
	\end{align*}
	using the $\p$-Schur's identity, we have
	\begin{align*}
		X_{ii}=&-\frac{1}{2\eta u}D^\p_{ijk}C^\p_{ijk}-uC^\p_{ikj,jki}+u_iR^\p_{st,k}W^\p_{sitk}+\alpha u_iR^\p_{ts}(\p^a_i\p^a_s)_t\\
		&+\alpha\set{u\pa{\frac{1}{2}S^\p_t\p^a_t\p^a_i-\alpha\p^b_{ss}\p^b_t\p^a_t\p^a_i}}_i+\alpha\set{u R^\p_{tk}\pa{\p^a_t\p^a_i}_k}_i\\
		&-\alpha u_i\pa{\frac{1}{2}S^\p_t\p^a_t\p^a_i-\alpha\p^b_{ss}\p^b_t\p^a_t\p^a_i+R^\p_{tk}\pa{\p^a_t\p^a_i}_k }\\
		&-\sq{u\pa{\p^a_{ik}U^a-U^{ab}\p^b_k\p^a_i-\frac{\alpha}{m-2}\abs{\tau(\p)}^2\delta_{ik}}_{k}}_i\\
		&+u_i\pa{\p^a_{ik}U^a-U^{ab}\p^b_k\p^a_i-\frac{\alpha}{m-2}\abs{\tau(\p)}^2\delta_{ik}}_{k}\\
		&+\pa{\frac{m-4}{m-2}}u\sq{U^a\pa{\frac{1}{\alpha}U^{ab}\p^b_i+R^\p_{ij}\p^a_j}}_i-\mathcal{A}u;
	\end{align*}
	since $\p^a_iu_i=0$ we deduce
	\begin{align}\label{other rigidity: diverg X non completa 3}
		X_{ii}=&-\frac{1}{2\eta u}D^\p_{ijk}C^\p_{ijk}-uC^\p_{ikj,jki}+u_iR^\p_{st,k}W^\p_{sitk}		+\alpha\set{u\pa{\frac{1}{2}S^\p_t\p^a_t\p^a_i-\alpha\p^b_{ss}\p^b_t\p^a_t\p^a_i}}_{i}\\
		&+\alpha\set{u R^\p_{tk}\pa{\p^a_t\p^a_i}_k}_i
		-\sq{u\pa{\p^a_{ik}U^a-U^{ab}\p^b_k\p^a_i-\frac{\alpha}{m-2}\abs{\tau(\p)}^2\delta_{ik}}_{k}}_i\notag\\
		&+u_i\pa{\p^a_{ik}U^a-U^{ab}\p^b_k\p^a_i-\frac{\alpha}{m-2}\abs{\tau(\p)}^2\delta_{ik}}_{k}\notag\\
		&+\frac{m-4}{m-2}u\sq{U^a\pa{\frac{1}{\alpha}U^{ab}\p^b_i+R^\p_{ij}\p^a_j}}_{i}-\mathcal{A}u\notag
	\end{align}
	on $\mathrm{int}(M)$ and
	\begin{align*}
		&u_i\pa{\p^a_{ik}U^a-U^{ab}\p^b_k\p^a_i-\frac{\alpha}{m-2}\abs{\tau(\p)}^2\delta_{ik}}_{k}\\
		&\quad =u_i\Bigg(\p^a_{ikk}U^a+\p^a_{ik}U^{ab}\p^b_k-U^{abc}\p^c_k\p^b_k\p^a_i-U^{ab}\p^b_{kk}\p^a_i\\
		&\quad\quad -U^{ab}\p^b_k\p^a_{ik}-\alpha\frac{2}{m-2}\p^a_{tti}\p^a_{ss}
		\Bigg)\\
		&\quad =u_i\Bigg(\p^a_{ikk}U^a-\alpha\frac{2}{m-2}\p^a_{tti}\p^a_{ss}
		\Bigg).
	\end{align*}
	Using the commutation rule (see Section 1.7 of \cite{AMR})
	\begin{align*}
		\p^a_{ijk}=\p^a_{ikj}+R_{tijk}\p^a_t-{}^N\!R^a_{bcd}\p^b_i\p^c_j\p^d_k,
	\end{align*}
	we get
	\begin{align*}
		&u_i\pa{\p^a_{ik}U^a-U^{ab}\p^b_k\p^a_i-\frac{\alpha}{m-2}\abs{\tau(\p)}^2\delta_{ik}}_{k}\\
		&\quad=u_i\sq{\pa{\frac{m-4}{m-2}}\p^a_{kki}U^a+U^aR_{si}\p^a_s-{}^N\!R^a_{bcd}\p^b_k\p^c_i\p^d_kU^a}\\
		&\quad=u_i\pa{\frac{m-4}{\alpha(m-2)}U^{ab}\p^b_iU^a+U^aR^\p_{si}\p^a_s-{}^N\!R^a_{bcd}\p^b_k\p^c_i\p^d_kU^a}\\
		&\quad=u_iU^aR^\p_{si}\p^a_s=0,
	\end{align*}
	where the last equality follows by \eqref{other rigidity: the boundary case: autov}. As a consequence, \eqref{other rigidity: diverg X non completa 3} becomes
	\begin{align*}
		X_{ii}=&-\frac{1}{2\eta u}D^\p_{ijk}C^\p_{ijk}-uC^\p_{ikj,jki}+u_iR^\p_{st,k}W^\p_{sitk}\\
		&+\alpha\set{u\pa{\frac{1}{2}S^\p_t\p^a_t\p^a_i-\alpha\p^b_{ss}\p^b_t\p^a_t\p^a_i}}_i+\alpha\set{u R^\p_{tk}\pa{\p^a_t\p^a_i}_k}_i\\
		&-\sq{u\pa{\p^a_{ik}U^a-U^{ab}\p^b_k\p^a_i-\frac{\alpha}{m-2}\abs{\tau(\p)}^2\delta_{ik}}_{k}}_i\\
		&+\frac{m-4}{m-2}u\sq{U^a\pa{U^{ab}\p^b_i+R^\p_{ij}\p^a_j}}_i-\mathcal{A}u.
	\end{align*}
	On $\mathrm{int}(M)$, since $u>0$, by \eqref{other rigidity: the boundary case: 1st int}, we obtain
	\begin{align}\label{other rigidity: the boundary case: da 1st integr}
		C^\p_{ijk}D^\p_{ijk}&=\frac{1}{\eta u}D^{\p}_{ijk}\set{u_tW^\p_{tijk}-\eta u \frac{U^a}{m-1}\pa{\p^a_j\delta_{ik}-\p^a_k\delta_{ik}}+\frac{[1+\eta(m-2)]}{\eta u}D^\p_{ijk}}\\
		&=\frac{1}{\eta u}u_tW^\p_{tijk}D^\p_{ijk}+\frac{[1+\eta(m-2)]}{(\eta u)^2}\abs{D^\p}^2,\notag
	\end{align}
	where the last equality is a consequence of the fact that $D^\p$ is totally trace-free (since $D^\p=(u\eta)^2\ol{D}^\p$). It follows that, on $\mathrm{int}(M)$, we have the validity of
	\begin{align*}
		X_{ii}=&-\frac{1}{2(\eta u)^2}u_tW^\p_{tijk}D^\p_{ijk}-\frac{[1+\eta(m-2)]}{(\eta u)^3}\abs{D^\p}^2-uC^\p_{ikj,jki}+u_iR^\p_{st,k}W^\p_{sitk}\\
		&+\alpha\set{u\pa{\frac{1}{2}S^\p_t\p^a_t\p^a_i-\alpha\p^b_{ss}\p^b_t\p^a_t\p^a_i}}_i+\alpha\set{u R^\p_{tk}\pa{\p^a_t\p^a_i}_k}_i\\
		&-\sq{u\pa{\p^a_{ik}U^a-U^{ab}\p^b_k\p^a_i-\frac{\alpha}{m-2}\abs{\tau(\p)}^2\delta_{ik}}_{k}}_i\\
		&+\frac{m-4}{m-2}u\sq{U^a\pa{\frac{1}{\alpha}U^{ab}\p^b_i+R^\p_{ij}\p^a_j}}_i-\mathcal{A}u.
	\end{align*}
	Therefore, we have
	\begin{align*}
		&X_{ii}-\alpha\set{u\pa{\frac{1}{2}S^\p_t\p^a_t\p^a_i-\alpha\p^b_{ss}\p^b_t\p^a_t\p^a_i}}_i-\alpha\set{u R^\p_{tk}\pa{\p^a_t\p^a_i}_k}_i\\
		&+\sq{u\pa{\p^a_{ik}U^a-U^{ab}\p^b_k\p^a_i-\frac{\alpha}{m-2}\abs{\tau(\p)}^2\delta_{ik}}_{k}}_i=-\frac{1}{2(\eta u)^2}u_tW^\p_{tijk}D^\p_{ijk}\\
		&-\frac{[1+\eta(m-2)]}{(\eta u)^3}\abs{D^\p}^2
		-uC^\p_{ikj,jki}+u_iR^\p_{st,k}W^\p_{sitk}\\
		&+\frac{m-4}{m-2}u\sq{U^a\pa{\frac{1}{\alpha}U^{ab}\p^b_i+R^\p_{ij}\p^a_j}}_i-\mathcal{A}u,
	\end{align*}
	and substituting the expression of $\mathcal{A}$ into the latter, we get \eqref{other rigidity: the boundary case: div di Z}.
\end{proof}
\begin{proof}[Proof (of Theorem \ref{other rigidity results: the boundary case: analogo thm 1.44})]
	Let $Z$ be the vector field defined in \eqref{other rigidity: the boundary case: def of Z}; from Lemma \ref{other rigidity results: the boundary case: lemma con Z} and assumptions \eqref{other rigidity: the boundary case: div3 cotton} and \eqref{other rigidity: the boundary case: weyl} we get
	\begin{align}\label{other rigidity: the boundary case: div Z per dim teorema}
		\diver Z=&-\frac{\sq{1+\eta(m-2)}}{(\eta u)^3}\abs{D^\p}^2+\pa{\frac{m-4}{m-2}}\sq{U^a\pa{\frac{1}{\alpha}U^{ab}\p^b_i-R^\p_{ij}\p^a_j}}_iu\\ \notag
		&+\bigg\{\p^a_i\pa{\frac{m}{(m-1)(m-2)}U^aS^\p+2\alpha U^b\p^b_j\p^a_i}-\frac{\alpha}{2}\pa{\frac{m-2}{m-1}}\p^a_i\p^a_jS^\p_j\\
		&-2\alpha\p^a_{jk}R^\p_{jk}\p^a_i-\alpha\p^a_i\tau^a_2(\p)\bigg\}_iu\notag.
	\end{align}
	Let
	\begin{align*}
		M_{\eps}:=\set{x\in M\,:\,u(x)\geq \eps}, \quad\partial M_{\eps}:=\set{x\in M\,:\,u(x)= \eps}.
	\end{align*}
	Then, using the divergence theorem, we deduce
	\begin{align*}
		\int_{\partial M_{\eps}} -g(Z,\nu)=\int_{M_{\eps}}\diver Z,
	\end{align*}
	where
	\begin{align*}
		\nu=\frac{\nabla u}{|\nabla u|}
	\end{align*}
	is the inward unit normal. 
	For the boundary part, we have, from the definition (\ref{other rigidity: the boundary case: def of Z}) of the vector field $Z$, that
	\begin{align*}
		\int_{\partial_{M_{\eps}}}g(Z,\nu)=&\eps \int_{\partial_{M_{\eps}}} \set{\pa{R^{\p}_{tk}W^{\p}_{tikj}}_j-\frac{m-4}{m-2}R^{\p}_{tk}C^{\p}_{tki}+2\alpha \p^a_i\p^a_{jk}R^{\p}_{jk}}\nu_i\\
		&-\eps\int_{\partial_{M_{\eps}}}\set{R^{\p}_{tk}\pa{\p^a_i\p^a_t}_k+\alpha \pa{\frac{1}{2}S^{\p}_t\p^a_t\p^a_i-\alpha \p^b_{tt}\p^b_s\p^a_s\p^a_i}}\nu_i\\
		&+\eps\int_{\partial_{M_{\eps}}}\set{\pa{\p^a_{ik}U^a-U^{ab}\p^b_k\p^a_i-\frac{\alpha}{m-2}|\tau (\p)|^2\delta_{ik}}_k}\nu_i
	\end{align*}
	and this term vanishes as $\eps\to 0^+$.
	For the left hand side, we integrate by parts  equation (\ref{other rigidity: the boundary case: div Z per dim teorema})  to obtain
	\begin{align*}
		\int_{M_{\eps}}\diver Z=&-[1+\eta(m-2)]\int_{M_{\eps}}\frac{1}{(\eta u)^3}|D^{\p}|^2\\
		&-\frac{m-4}{m-2}\int_{M_{\eps}}u_i\sq{U^a\pa{\frac{1}{\alpha}U^{ab}\p^b_i+R^{\p}_{ij}\p^a_j}}\\
		&-\frac{m-4}{m-2}\eps\int_{\partial M_{\eps}}\sq{U^a\pa{\frac{1}{\alpha}U^{ab}\p^b_i+R^{\p}_{ij}\p^a_j}}\nu_i\\
		&-\int_{M_{\eps}}u_i\pa{\frac{m}{(m-1)(m-2)}S^{\p}U^a+2\alpha U^b\p^b_j\p^a_j}\p^a_i\\
		&-\eps\int_{\partial M_{\eps}}\pa{\frac{m}{(m-1)(m-2)}S^{\p}U^a+2\alpha U^b\p^b_j\p^a_j}\p^a_i\nu_i\\
		&+\int_{M_{\eps}}u_i\pa{\frac{\alpha}{2}\frac{m-2}{m-1}S^{\p}_j\p^a_j\p^a_i+2\alpha \p^a_{jk}R^{\p}_{jk}\p^a_i}\\
		&+\eps\int_{\partial M_{\eps}}\pa{\frac{\alpha}{2}\frac{m-2}{m-1}S^{\p}_j\p^a_j\p^a_i+2\alpha \p^a_{jk}R^{\p}_{jk}\p^a_i}\nu_i\\
		&+\int_{M_{\eps}}\alpha \p^a_i\tau^a_2(\p)u_i+\eps\int_{\partial M_{\eps}}\alpha \p^a_i\tau^a_2(\p)\nu_i.
	\end{align*}
	Using (\ref{other rigidity: the boundary case: autov}) and (\ref{other rigidity: the boundary case: gradient}) we deduce
	\begin{align*}
		\int_{M_{\eps}}\diver Z=-[1+\eta(m-2)]\int_{M_{\eps}}\frac{1}{(\eta u)^3}|D^{\p}|^2
	\end{align*}
	and therefore, letting $\eps$ tend to zero, we conclude
	\begin{align*}
		D^{\p}\equiv 0.
	\end{align*}
\end{proof}

\section{Some Conditions on the Cotton Tensor}

For the ease of readability, we recall here system \eqref{Other rigidity: Einstein-type}:
\begin{align}\label{Other rigidity: Einstein-type 2}
	\begin{cases}
		i)\,\ric^\p+\hs(f)-\eta df\otimes df=\lambda g,\\
		ii)\,\tau(\p)=d\p(\nabla f)+\frac{1}{\alpha}(\nabla U)(\p).
	\end{cases}
\end{align}
As we have seen in Theorems \ref{Other rigidity: Main Theorem} and \ref{other rigidity results: the boundary case: analogo thm 1.44} and Proposition \ref{Cotton phi prop},
a condition that is often met for a system of the type (\ref{Other rigidity: Einstein-type 2}) is
\begin{align}\label{U phi Cotton}
	C^{\p}_{ijk}=\frac{U^a\p^a_k}{m-1}\delta_{ij}-\frac{U^a\p^a_j}{m-1}\delta_{ik}.
\end{align}

It is immediate to see that condition (\ref{U phi Cotton}) is satisfied if and only if the modification of  the $\p$-Schouten tensor given by
\begin{align}
	A^\p-\frac{U(\p)}{m-1}g
\end{align}
is a Codazzi tensor.
We will prove the following
\begin{theorem}\label{teo: U Cotton zero conseguenze}
	Let $(M,g)$ be compact with empty boundary and satisfy system (\ref{Other rigidity: Einstein-type 2}), for some non-constant function $f$ and $\alpha>0$. Assume that (\ref{U phi Cotton}) holds and that the (normalized) $k$-th elementary symmetric polynomial $\sigma_k$ in the eigenvalues of $A^\p-\frac{U(\p)}{m-1}g$ is  constant. If  $k\geq 2$, assume, moreover, that $\sigma_k$ is positive and that  there exists a point of $M$ at which all the eigenvalues of $A^\p-\frac{U(\p)}{m-1}g$ are positive. Then $\p$ is constant and $(M,g)$ is isometric to a Euclidean sphere.
\end{theorem}
Before proceeding, we  recall some results on the elementary symmetric polynomials, then we will prove an Obata-type result which will be instrumental in the proof of Theorem \ref{teo: U Cotton zero conseguenze}.
\subsection{On the elementary symmetric polynomials}
Here we collect some well-known facts about elementary symmetric polynomials and we fix the notation.
Given $m$ real numbers $\lambda_1,\lambda_2,...,\lambda_m$, we set
\begin{align*}
	S_k:=\sum_{i_1<...<i_k}\lambda_{i_1}\cdot ..\cdot \lambda_{i_k}
\end{align*}
for the $k$-th elementary symmetric polynomial and
\begin{align*}
	\sigma_k:=\begin{pmatrix}
		m\\
		k
	\end{pmatrix}^{-1}S_k
\end{align*}
for its normalization.
Then we have the validity of Newton's inequality,
\begin{align*}
	\sigma_{k-1}\sigma_{k+1}\leq \sigma_k^2
\end{align*}
where, provided $\sigma_{k-1}\neq 0$, the equality is obtained if and only if all the $\lambda_i$'s coincide (see \cite{Inequalities} for a proof).
Moreover, we have Gårding's inequalities (\cite{Garding})
\begin{align*}
	\sigma_1\geq \sigma_2^{\frac{1}{2}}\geq ...\geq \sigma_k^{\frac{1}{k}},
\end{align*}
which hold when all of the $\sigma_j$'s, $j=1,..,k$, are positive. Again, equality is obtained if and only if all the $\lambda_i$'s are equal.
Combining them, one can prove the validity of the next lemma, which is
Lemma 5.33 of \cite{ACR}.
\begin{lemma}\label{Lemma: disuguaglianza di Anselli}
	Let $A$ be a 2-covariant symmetric tensor on $(M,g)$ with $m=\dim M\geq 3$. Let $\sigma_k$ be the $k$-th normalized symmetric function in the eigenvalues of $A$; if $k\geq 2$, assume that it is positive on $M$ and assume that there exists a point $p\in M$ at which all the eigenvalues of $A$ are positive. Then
	\begin{align}\label{disuguaglianza di Anselli}
		\sigma_1\sigma_k-\sigma_{k+1}\geq 0,
	\end{align}
	with equality holding at a point $x\in M$ if and only if the eigenvalues of $A$ at $x$ coincide.
\end{lemma}


\begin{proof}[Proof of Lemma \ref{Lemma: disuguaglianza di Anselli}]
	For $k=1$, \eqref{disuguaglianza di Anselli} coincides with the Newton's inequality so that the conclusion immediately follows.\\
	For $k\geq 2$, note that, by assumptions, the set
	\begin{align*}
		\{(\lambda_1(x),\lambda_2(x),...,\lambda_m(x))\in \erre^m: x\in M\}
	\end{align*}
	is contained in the connected component of $\{x\in \erre^m:\sigma_k(x)>0\}$ that contains the positive orthant
	\begin{align*}
		\{(x_1,x_2,..,x_m)\in \erre^m: x_i>0, \forall i=1,..,m\}.
	\end{align*}
	As it is well-known, this connected component coincides with the set
	\begin{align*}
		\{x\in \erre^m: \sigma_i(x)>0, \forall i=1,..,k\}
	\end{align*}
	so that Gårding's inequalities can be applied.
	Moreover, since $\sigma_{k-1}>0$, we obtain from Newton's inequalities that, for $k\geq 2$,
	\begin{align*}
		\sigma_{k+1}=\frac{\sigma_{k+1}\sigma_{k-1}}{\sigma_{k-1}}\leq \frac{\sigma_k^2}{\sigma_{k-1}}=\sigma_k\frac{\sigma_k}{\sigma_{k-1}}.
	\end{align*}
	We claim
	\begin{align*}
		\frac{\sigma_k}{\sigma_{k-1}}\leq \sigma_1
	\end{align*}
	which, since $\sigma_k>0$, would imply (\ref{disuguaglianza di Anselli}).
	To prove the claim we use Gårding's inequalities and the positivity of $\sigma_1,\sigma_k$ and $\sigma_{k-1}$ to deduce
	\begin{align*}
		\sigma_k=\sigma_k^{1/k}\sigma_k^{(k-1)/k}\leq \sigma_1\sigma_k^{(k-1)/k}\leq \sigma_1\sigma_{k-1}.
	\end{align*}
   This proves the claim.
	Moreover, equality at a point $x\in M$ of (\ref{disuguaglianza di Anselli}) would imply equality in the Newton's inequalities and in Gårding's inequalities and this happens if and only if all of the eigenvalues of $A$ at $x$ coincide.
\end{proof}

In the following, we will fix a symmetric 2-covariant tensor $A$ on $M$. The $\sigma_i$'s will be intended as functions of the eigenvalues of $A$.
The \emph{Newton endomorphisms}
\begin{align*}
	P_k=P_k(A):\mathfrak{X}(M)\to \mathfrak{X}(M)
\end{align*}
associated to $A$, where $\mathfrak{X}(M)$ is the $C^{\infty}(M)$-algebra of smooth vector fields on $M$, are inductively defined by
\begin{align}\label{operatori di newton}
	P_0=\id, \ \ P_k=S_k \id-A\circ P_{k-1}, \qquad 1\leq k\leq m.
\end{align}
Here $A$ is identified with the corresponding endomorphism $A:\mathfrak{X}(M)\to \mathfrak{X}(M)$. Note that, by Cayley-Hamilton's Theorem, we have $P_m\equiv 0$.
It is well known that the $P_k$'s satisfy the following formulas (see \cite{MarquesBarbosa1997}):
\begin{align}
	& \mathrm{tr} P_k=(m-k)S_k \label{traccia di P}, \\
	&\mathrm{tr} (A\circ P_k)=(k+1)S_{k+1}. \label{traccia di A P}
\end{align}
Setting
\begin{align}
	c_k=(m-k)\begin{pmatrix}
		m\\
		k
	\end{pmatrix},
\end{align}
a simple computation shows that equations (\ref{traccia di P}) and (\ref{traccia di A P}) become
\begin{align}
	& \mathrm{tr} P_k=c_k\sigma_k, \label{traccia di P 2}\\
	& \mathrm{tr} (A\circ P_k)=c_k \sigma_{k+1}. \label{traccia di A P 2}
\end{align}
As a last well-known fact, we recall that if $A$ is a Codazzi tensor, then all of the $P_i$'s are divergence-free;
this is a consequence of the following
\begin{lemma}
	Let $A$ be a $2$-covariant symmetric tensor and let $P_k$ denote the $k$-th symmetric Newton operator given by (\ref{operatori di newton}). Then the following formula holds:
	\begin{align}\label{divergenza del k-esimo operatore di Newton}
		\diver (P_k)_i=-A_{ij}\diver (P_{k-1})_{j}-C(A)_{jit}(P_{k-1})_{jt},
	\end{align}
	where
	\begin{align*}
		C(A)_{jit}=A_{ji,t}-A_{jt,i}
	\end{align*}
	is the obstruction to $A$ being a Codazzi tensor.
\end{lemma}
\begin{proof}
	From (\ref{operatori di newton}) we deduce
	\begin{align*}
		(P_k)_{ij,j}=(S_k)_i-A_{ip,j}(P_{k-1})_{pj}-A_{ip}(P_{k-1})_{pj,j}.
	\end{align*}
	To obtain (\ref{divergenza del k-esimo operatore di Newton}) from it, we only need to prove the validity of
	\begin{align}\label{divergenza di P k: derivata covarainte di S_k}
		(S_k)_i=A_{pj,i}(P_{k-1})_{pj}.
	\end{align}
 To deduce it, we fix $p\in M$ such that each eigenvalue of $A$ is smooth at $p$. As it is well-known (see Paragraph 16.10 of \cite{Besse}), the set of points at which all the eigenvalues of $A$ are smooth is dense in $M$, so that it is enough to prove the validity of \eqref{divergenza di P k: derivata covarainte di S_k} at $p$.

	 We  choose a local reference frame in a neighbourhood of $p$ that diagonalizes $A$:
	 from the definition of $S_k$, we immediately obtain
	\begin{align*}
		(S_k)_i&=(\lambda_1)_i S_{k-1}{|_{\lambda_1=0}}+(\lambda_2)_i S_{k-1}{|_{\lambda_2=0}}+....+(\lambda_m)_i S_{k-1}{|_{\lambda_m=0}}
	\end{align*}
	where $S_{k-1}{|_{\lambda_j=0}}$ denotes the sum of those summands of $S_{k-1}$ in which $\lambda_j$ does not appear.
	With this notations in mind, a simple mathematical induction that uses only the definition (\ref{operatori di newton}) of $P_k$ gives that, in the chosen frame, $P_{k-1}$
	takes the form
	\begin{align*}
		P_{k-1}=\text{diag}\pa{S_{k-1}{|_{\lambda_1=0}}, S_{k-1}{|_{\lambda_2=0}},...,S_{k-1}{|_{\lambda_m=0}}}.
	\end{align*}
	Therefore we have
	\begin{align*}
		A_{pj,i}(P_{k-1})_{pj}=&(\lambda_1)_i S_{k-1}{|_{\lambda_1=0}}+(\lambda_2)_i S_{k-1}{|_{\lambda_2=0}}+....+(\lambda_m)_i S_{k-1}{|_{\lambda_m=0}}
	\end{align*}
	and (\ref{divergenza di P k: derivata covarainte di S_k}) follows.
\end{proof}
\begin{cor}\label{cor: Codazzi e divergence free}
	Let $A$ be a 2-covariant symmetric Codazzi tensor. Then all of the Newton operators in the eigenvalues of $A$ are divergence-free. Moreover, we also have
	\begin{align}\label{5.63.1}
		\mathring{\sq{\pa{P_{k-1}}\circ A}}_{ij,i}&=\frac{m-k}{m}(S_k)_j\\
		&=\frac{c_k}{m}\pa{\sigma_k}_j\notag
	\end{align}
	for $k=1,...,m-1$, where
	\begin{align*}
		\mathring{\sq{\pa{P_{k-1}}\circ A}}=\pa{P_{k-1}}\circ A-\frac{1}{m}\mathrm{tr}\sq{\pa{P_{k-1}}\circ A}g.
	\end{align*}
\end{cor}
\begin{proof}
	Since $A$ is Codazzi, we have
	\begin{align*}
		\diver (P_1)_i&=(S_1)_i-A_{ij,j}\\
		&=A_{jj,i}-A_{ij,j}=0,
	\end{align*}
	and therefore the first part of the statement follows by induction on $k$, using (\ref{divergenza del k-esimo operatore di Newton}).\\
	To prove \eqref{5.63.1}, we use \eqref{traccia di A P} to deduce
	\begin{align*}
			\mathring{\sq{\pa{P_{k-1}}\circ A}}=\pa{P_{k-1}}\circ A-\frac{k}{m}S_kg.
	\end{align*}
	Taking the divergence of the above expression we get
	\begin{align*}
		\mathring{\sq{\pa{P_{k-1}}\circ A}}_{ij,i}=\pa{P_{k-1}}_{it,i}A_{tj}+\pa{P_{k-1}}_{it}A_{tj,i}-\frac{k}{m}\pa{S_k}_j.
	\end{align*}
	Using $\diver \pa{P_k}=0$, the fact that $A$ is Codazzi and \eqref{divergenza di P k: derivata covarainte di S_k} we deduce
	\begin{align*}
		\mathring{\sq{\pa{P_{k-1}}_{it}A_{tj}}}_i=\frac{m-k}{k}\pa{S_k}_j,
	\end{align*}
	 that is, \eqref{5.63.1}.
\end{proof}
\begin{rem}
	Symmetrizing the last term of (\ref{divergenza del k-esimo operatore di Newton}) we get
	\begin{align*}
		C(A)_{jit}(P_{k-1})_{jt}=\frac{1}{2}(P_{k-1})_{jt}\sq{C(A)_{jit}+C(A)_{tij}}.
	\end{align*}
	Reasoning as in the first part of the proof of Corollary \ref{cor: Codazzi e divergence free}, we deduce that $P_k$ is divergence-free assuming only
	\begin{align*}
		C(A)_{jit}+C(A)_{tij}=0.
	\end{align*}
\end{rem}
The Newton endomorphisms give rise to a family of second order differential operators, $L_k=L_k(A)$, defined as follows: let $u\in C^2(M)$ and set $\hess(u)$ to denote both the 2-covariant tensor and the corresponding endomorphism; we set
\begin{align}\label{L k di u}
	L_k u=\tr \pa{P_k \circ \hess(u)},
\end{align}
that we also rewrite in the useful form
\begin{align*}
	L_k u=\sum_{i=0}^k (-1)^i S_{k-i}\tr \pa{A^i\circ \hess(u)},
\end{align*}
where $A^0=\id$ and, for $i\geq 1$, $A^i$ is the $i$-th iterated composition of $A$ with itself.
A computation shows that $L_k u$ can also be expressed in the form
\begin{align*}
	L_k u=\diver \pa{P_k(\nabla u)}- g(\diver P_k, \nabla u).
\end{align*}
From the above formula we see that $L_k$ is semielliptic whenever $P_k$ is positive semidefinite and it is in divergence form when $\diver P_k=0$. From Corollary \ref{cor: Codazzi e divergence free} we see that the second condition is satisfied when $A$ is a Codazzi tensor.

Suppose now that, for some
\begin{align*}
	p(x), q(x), l(x)\in C^{\infty}(M),
\end{align*}
we can express $\hess(u)$ in the form
\begin{align}\label{polinomi simmetrici: hess u forma generale}
	\hess(u)=p(x)g+q(x)du\otimes du-l(x)A.
\end{align}
Then from equations (\ref{traccia di A P 2}), (\ref{traccia di P 2}) and (\ref{L k di u}) we obtain, after some algebraic manipulations,
\begin{align}\label{polinomi simmetrici: L k u forma generale}
	L_k u=c_k\sq{p(x)\sigma_k-l(x)\sigma_{k+1}}+q(x)g(P_k(\nabla u), \nabla u)
\end{align}
for $1\leq k\leq m-1$.

\subsection{An Obata-type result}
Recall that a vector field $X$ on $(M,g)$ is said to be \textit{conformal} if
\begin{align}\label{definizione di campo conforme}
	\mathcal{L}_X g=\frac{2\diver X}{m}g,
\end{align}
where $\mathcal{L}_X g$ is the Lie derivative of the metric in the direction of $X$, while $X$ is \textit{Killing} when it is conformal and $\diver X\equiv 0$.
Moreover, $X$ is \textit{closed} when, in an orthonormal coframe,
\begin{align*}
	X_{ij}=X_{ji}.
\end{align*}
We now prove the next result, which is inspired by a classical theorem of Obata (\cite{Obata}).
\begin{proposition}\label{prop: risultato alla obata per U-Harmonic Einstein}
	Let $(M,g)$ be a closed, connected Riemannian manifold supporting the structure
	\begin{align}\label{risultato alla Obata: U-harmonic Einstein}
		\begin{cases}
			\ric^\p=\frac{S^\p}{m}g,\\
			\tau (\p)=\frac{1}{\alpha}\nabla U(\p).
		\end{cases}
	\end{align}
	 Assume that there exists a closed, conformal, non-Killing vector field $X$ on $(M,g)$ and assume
	\begin{align}\label{Definizione di Campo verticale}
		d\p (X)=0.
	\end{align}
	Then $\p$ is constant and $(M,g)$ is isometric to a Euclidean sphere.
\end{proposition}
When $U$ is constant, the above result has been proved, without assuming that $X$ is closed, in Lemma 5.2 of \cite{ACR}.
We first prove the following
\begin{lemma}
	Let $(M,g)$ be a Riemannian manifold of dimension $m\geq 3$ and
	$X$ a conformal vector field on $(M,g)$. Let $\p:(M,g)\to (N,h)$ be a smooth map and $\alpha \in \erre\backslash \{0\}$.
	Set
	\begin{align*}
		\gamma=\diver X.
	\end{align*}
	Then
	\begin{align}\label{hessiano di gamma}
		\hess(\gamma)=& \frac{S^\p}{(m-1)(m-2)}\gamma g+\frac{m}{2(m-1)(m-2)}g(\nabla S^\p, X)g-\frac{m}{m-2}\mathcal{L}_X \ric^{\p}\nonumber\\
		&-\frac{m}{m-2}\alpha \sq{\mathcal{L}_X\pa{\p^*h}-\frac{1}{2(m-1)}\mathrm{tr}\pa{\mathcal{L}_X\pa{\p^*h}}g}.
	\end{align}
	In particular,
	\begin{align}\label{laplaciano di gamma}
		\Delta \gamma=-\frac{S^\p}{m-1} \gamma-\frac{m}{2(m-1)}g(\nabla S^\p, X)-\frac{m}{2(m-1)}\alpha \mathrm{tr}\pa{\mathcal{L}_X\pa{\p^*h}}.
	\end{align}
\end{lemma}
\begin{proof}
	
	We first prove
	\begin{align}\label{hessiano di gamma-1}
		\hess(\gamma)=\frac{S}{(m-1)(m-2)}\gamma g+\frac{m}{2(m-1)(m-2)}g(\nabla S, X)g-\frac{m}{m-2}\mathcal{L}_X \ric.
	\end{align}
	We apply Ricci commutations rules twice to deduce
	\begin{align*}
		\gamma_{ij}&=X_{ttij}=X_{titj}+X_{pj}R_{ptti}+X_pR_{ptti,j}\\
		&=X_{tijt}+X_{pi}R_{pttj}+X_{tp}R_{pitj}-X_{pj}R_{pi}-X_pR_{pi,j}\\
		&=X_{tijt}-X_{pi}R_{pj}+X_{tp}R_{pitj}-X_{pj}R_{pi}-X_{p}R_{pi,j}.
	\end{align*}
	Since $X$ is conformal, we have
	\begin{align}\label{X conforme}
		X_{ij}+X_{ji}=\frac{2}{m}\gamma \delta_{ij}.
	\end{align}
	From (\ref{X conforme}) and the Ricci commutation rules we deduce
	\begin{align*}
		\gamma_{ij}=&-X_{itjt}+\frac{2}{m}X_{ttji}-X_{pi}R_{pj}+X_{tp}R_{pitj}-X_{pj}R_{pi}-X_{p}R_{pi,j}\\
		=&-X_{ijtt}-X_{pt}R_{pitj}-X_pR_{pitj,t}+\frac{2}{m}X_{ttji}-X_{pi}R_{pj}+X_{tp}R_{pitj}\\
		&-X_{pj}R_{pi}-X_{p}R_{pi,j}.
	\end{align*}
	Using the first and second Bianchi identities we obtain
	\begin{align*}
		\gamma_{ij}=&-X_{ijtt}+\frac{2}{m}X_{ttji}-X_{pt}R_{ptij}-X_pR_{ij,p}+X_pR_{pj,i}-X_pR_{pi,j}\\
		&-X_{pi}R_{pj}-X_{pj}R_{pi}.
	\end{align*}
	Recalling that
	\begin{align*}
		\pa{\mathcal{L}_X \ric}_{ij}=X_tR_{ij,t}+X_{ti}R_{tj}+X_{tj}R_{it}
	\end{align*}
	we obtain, rearranging some terms,
	\begin{align*}
		\gamma_{ij}=-X_{ijtt}+\frac{2}{m}X_{ttji}-X_{pt}R_{ptij}+X_t(R_{tj,i}-R_{ti,j})-\pa{\mathcal{L}_X\ric}_{ij}.
	\end{align*}
	Symmetrizing the above expression and using (\ref{X conforme}) we deduce
	\begin{align*}
		\gamma_{ij}=&-\frac{1}{2}(X_{ij}+X_{ji})_{tt}+\frac{2}{m}\gamma_{ij}-\pa{\mathcal{L}_X \ric}_{ij}\\
		=&-\frac{1}{m}\gamma_{tt}\delta_{ij}+\frac{2}{m}\gamma_{ij}-\pa{\mathcal{L}_X \ric}_{ij},
	\end{align*}
	which, after some simplifications, becomes, in global notation,
	\begin{align}\label{hessiano di gamma -2}
		\hess(\gamma)=-\frac{\Delta \gamma}{m-2}g-\frac{m}{m-2}\mathcal{L}_X \ric.
	\end{align}
	Tracing \eqref{hessiano di gamma -2} and simplifying we get
	\begin{align*}
		\Delta \gamma=-\frac{S}{m-1}\gamma-\frac{m}{2(m-1)}g(\nabla S, X).
	\end{align*}
	Substituting it into \eqref{hessiano di gamma -2} we obtain \eqref{hessiano di gamma-1}.
	From the definitions of $\ric^\p$ and $S^\p$ we immediately get
	\begin{align}\label{Lie di ricci e ricci phi}
		\mathcal{L}_X \ric=\mathcal{L}_X \ric^\p+\alpha \mathcal{L}_X \pa{\p^*h}
	\end{align}
	and
	\begin{align}\label{lie di S e S phi}
		S\gamma g+\frac{m}{2}g\pa{\nabla S,X}=&S^\p \gamma g+\frac{m}{2}g\pa{\nabla S^\p, X}g\\
		&+\alpha\frac{m}{2}\mathrm{tr}\pa{\mathcal{L}_X\pa{\p^*h}}g. \nonumber
	\end{align}
	Using \eqref{Lie di ricci e ricci phi} and \eqref{lie di S e S phi} into \eqref{hessiano di gamma-1} we obtain \eqref{hessiano di gamma}, while tracing \eqref{hessiano di gamma} we obtain \eqref{laplaciano di gamma}.
\end{proof}
\begin{rem}
	Assume now that $X\in \mathrm{Ker} \pa{d\p}$, that is,
	\begin{align*}
		\p^a_iX_i=0.
	\end{align*}
	Taking the covariant derivative of the above equation we get
	\begin{align*}
		\p^a_{ij}X_i+\p^a_iX_{ij}=0,
	\end{align*}
	which implies $\mathcal{L}_X \pa{\p^*h}=0$. Therefore, \eqref{hessiano di gamma} and \eqref{laplaciano di gamma} reduce, respectively, to
	\begin{align}\label{hessiano di gamma 2}
		\hess(\gamma)=& \frac{S^\p}{(m-1)(m-2)}\gamma g+\frac{m}{2(m-1)(m-2)}g(\nabla S^\p, X)g-\frac{m}{m-2}\mathcal{L}_X \ric^{\p}
	\end{align}
	and
	\begin{align}\label{laplaciano di gamma 2}
		\Delta \gamma=-\frac{S^\p}{m-1}-\frac{m}{2(m-1)}g\pa{X,\nabla S^\p}.
	\end{align}
	Suppose now
	\begin{align*}
		\begin{cases}
			\ric^\p=\frac{S^\p}{m}g\\
			X\in \mathrm{Ker}\pa{d\p}.
		\end{cases}
	\end{align*}
	Then, taking the divergence of the first of the above equations and using the $\p$-Schur's identity, we get
	\begin{align*}
		\frac{1}{2}S^\p_k-\alpha \p^a_{tt}\p^a_k=R^\p_{tk,t}=\frac{S^\p_k}{m}
	\end{align*}
	and
	\begin{align*}
		\frac{m-2}{2m}S^\p_k=\p^a_{tt}\p^a_k.
	\end{align*}
	Since $X\in \mathrm{Ker}\pa{d\p}$, it follows that
	\begin{align*}
		\frac{m-2}{2m}S^\p_kX_k=\p^a_{tt}\p^a_kX_k=0,
	\end{align*}
	that is, since $m\geq 3$,
	\begin{align}\label{risultato alla obata: S phi costante in direzione X}
		S^\p_kX_k=0.
	\end{align}
	Furthermore, using $\ric^\p=\frac{S^\p}{m}g$ and \eqref{risultato alla obata: S phi costante in direzione X},
	\begin{align*}
		\pa{\mathcal{L}_X \ric^\p}_{ij}=&X_tR^\p_{ij,k}+X_{ti}R^\p_{tj}+X_{tj}R^\p_{ti}\\
		=&\frac{1}{m}X_t S^\p_t \delta_{ij}+\frac{S^\p}{m}X_{ti}\delta_{tj}+\frac{S^\p}{m}X_{tj}\delta_{ti}\\
		=&\frac{1}{m}\pa{S^\p_tX_t\delta_{ij}+S^\p X_{ji}+S^\p X_{ij}}\\
		=&\frac{S^\p}{m}\pa{\mathcal{L}_X g}_{ij}=\frac{2}{m^2}S^\p\gamma \delta_{ij},
	\end{align*}
	that is,
	\begin{align}
		\pa{\mathcal{L}_X \ric^\p}_{ij}=\frac{2}{m^2}S^\p\gamma \delta_{ij}.
	\end{align}
	Hence, \eqref{hessiano di gamma 2} and \eqref{laplaciano di gamma 2} become, respectively,
	\begin{align}\label{risultato alla obata: hessiano di eta}
		\hess \pa{\gamma}=-\frac{1}{m(m-1)}S^\p \gamma g
	\end{align}
	and
	\begin{align}\label{risultato alla obata: laplaciano di eta}
		\Delta \gamma=-\frac{1}{m-1}S^\p\gamma.
	\end{align}
\end{rem}

\begin{rem}\label{remark: continuazione analitica}
	From equation (\ref{risultato alla obata: laplaciano di eta}) we deduce (see Appendix A of \cite{PRS} or \cite{Kazdan1988UniqueCI}) that $\gamma$ satisfies the \textit{unique continuation property}. In particular, it vanishes on some open subset of $M$ if and only if it vanishes on the connected components of $M$ containing it.
\end{rem}
\begin{rem}
	Since we do not know weather $S^\p$ is constant or not, it seems hard to prove Proposition \ref{prop: risultato alla obata per U-Harmonic Einstein} from equation (\ref{risultato alla obata: hessiano di eta}) using the original ideas of Obata, or those used in \cite{ACR}.
	Instead, to prove Proposition \ref{prop: risultato alla obata per U-Harmonic Einstein}, we will prove the constancy of $U$ under the further assumption that $X$ is closed and then apply Lemma 5.2 of \cite{ACR} to conclude. In this direction, the next Lemma will prove to be useful.
\end{rem}

\begin{lemma}\label{lemma: KO: campo conforme}
	Let $(M,g)$ be a Riemannian manifold, let $\p: (M,g)\to (N,h)$ be a smooth map with target another Riemannian manifold $(N,h)$ and let $X$ be a smooth, closed, conformal vector field on $M$.
	Assume that, for some $\eta\in \erre$, $(M,g)$ supports the structure
	\begin{align}\label{KO: lemma sui campi conformi: sistema}
		\begin{cases}
			\ric^{\p} +\frac{1}{2}\mathcal{L}_X g-\eta X^{\flat}\otimes X^{\flat}=\lambda g,\\
			\tau(\p)=\frac{1}{\alpha} \nabla U ( \p), \\
			d\p(X)=0,
		\end{cases}
	\end{align}
	where $\lambda=\frac{1}{m}\pa{S^\p+\diver X-\eta \abs{X}^2}$.
	Then $\tau (\p)$ vanishes identically on the set
	\begin{align*}
		\set{x\in M: (\diver X)(x) \neq 0}.
	\end{align*}	
\end{lemma}
\begin{proof}
	The third equation of (\ref{KO: lemma sui campi conformi: sistema}) reads, in components,
	\begin{align}\label{Definizione di Campo verticale A}
		\p^a_i X_i=0.
	\end{align}
	Setting
	\begin{align*}
		\gamma= \diver X,
	\end{align*}
	since $X$ is conformal and closed we have
	\begin{align}\label{Definizione di campo conforme e chiuso A}
		X_{ij}=\frac{\gamma}{m}\delta_{ij}.
	\end{align}
	Take the covariant derivative of (\ref{Definizione di Campo verticale A}) and use (\ref{Definizione di campo conforme e chiuso A}) to deduce
	\begin{align*}
		0=\p^a_{ij}X_i+\p^a_iX_{ij}=\p^a_{ij}X_i+\frac{\gamma}{m}\p^a_j.
	\end{align*}
	Compute the divergence of the expression above and use (\ref{Definizione di campo conforme e chiuso A}) to get
	\begin{align*}
		0=&\p^a_{ijj}X_i+\p^a_{ij}X_{ij}+\frac{\gamma}{m}\p^a_{jj}+\frac{1}{m}\p^a_j\gamma_j\\
		=&\p^a_{ijj}X_i+2\frac{\gamma}{m}\p^a_{jj}+\frac{1}{m}\p^a_j\gamma_j,
	\end{align*}
	that is,
	\begin{align}\label{risultato alla obata: A A}
		0=\p^a_{ijj}X_i+2\frac{\gamma}{m}\p^a_{jj}+\frac{1}{m}\p^a_j\gamma_j.
	\end{align}
	Since $X$ is conformal and $\lambda=\frac{1}{m}\pa{S^\p+\diver X-\eta |X|^2}$, the first equation of \eqref{KO: lemma sui campi conformi: sistema} becomes
	\begin{align}\label{KO: lemma sui campi conformi: sistema riscritto}
		\ric^\p-\eta X^\flat\otimes X^\flat=\frac{1}{m}\pa{S^\p-\eta \abs{X}^2}g.
	\end{align}
	Taking the divergence of (\ref{Definizione di campo conforme e chiuso A}) we have
	\begin{align*}
		X_{jij}=\frac{\gamma_i}{m}.
	\end{align*}
	Applying the Ricci commutation relations and using $d\p\pa{X}=0$  we deduce
	\begin{align*}
		\frac{\gamma_i}{m}=X_tR_{ti}+X_{jji}=X_t R^{\p}_{ti}+\gamma_i.
	\end{align*}
	Rearranging and using (\ref{KO: lemma sui campi conformi: sistema riscritto}) we get
	\begin{align}\label{risulltato alla Obata: derivata di eta A}
		\frac{m-1}{m}\gamma_i=-R^\p_{ti}X_t=-\frac{S^\p}{m}X_i-\eta\frac{m-1}{m}\abs{X}^2X_i.
	\end{align}
	From equations (\ref{risulltato alla Obata: derivata di eta A}) and (\ref{Definizione di Campo verticale A}) we deduce
	\begin{align}\label{risultato alla Obata: A B}
		\p^a_j\gamma_j=0.
	\end{align}
	We want to prove
	\begin{align}\label{risultato alla Obata: A C}
		\p^a_{ijj}X_i=0.
	\end{align}
	We exploit the commutation relation (see Section 1.7 of \cite{AMR})
	\begin{align*}
		\p^a_{ijk}=\p^a_{ikj}+\p^a_t R_{tijk}-{}^N\!R^a_{bcd}\p^b_i\p^c_j\p^d_k,
	\end{align*}
	which implies
	\begin{align*}
		\p^a_{ijj}=\p^a_{jij}=\p^a_{jji}+\p^a_t R_{ti}-{}^N\!R^a_{bcd}\p^b_j\p^c_i\p^d_j.
	\end{align*}
	Contracting the above with $X_i$ we get
	\begin{align}\label{risultato alla Obata: A C.1}
		\p^a_{ijj}X_i=\p^a_{jji}X_i+\p^a_t R_{ti}X_i-{}^NR^a_{bcd}\p^b_j\p^c_i\p^d_jX_i.
	\end{align}
	From (\ref{Definizione di Campo verticale A}) we deduce
	\begin{align*}
		{}^NR^a_{bcd}\p^b_j\p^c_i\p^d_jX_i=0,
	\end{align*}
	and using (\ref{Definizione di Campo verticale A}) and (\ref{KO: lemma sui campi conformi: sistema riscritto}) we obtain
	\begin{align*}
		\p^a_tR_{ti}X_i=\p^a_tR_{ti}^{\p}X_i=\frac{S^\p}{m}\p^a_tX_t+\eta\frac{m-1}{m}\abs{X}^2\p^a_tX_t=0.
	\end{align*}
	On the other hand, differentiating the second equation of (\ref{KO: lemma sui campi conformi: sistema}) we have
	\begin{align*}
		\p^a_{tti}=\frac{1}{\alpha}U^{ab}\p^b_i,
	\end{align*}
	and contracting with $X_i$ and using (\ref{Definizione di Campo verticale A}) we deduce
	\begin{align*}
		\p^a_{tti}X_i=\frac{1}{\alpha}U^{ab}\p^b_iX_i=0.
	\end{align*}
	Inserting  in (\ref{risultato alla Obata: A C.1}) we obtain
	\begin{align*}
		\p^a_{ijj}X_i=0,
	\end{align*}
	that is, equation (\ref{risultato alla Obata: A C}).
	Using (\ref{risultato alla Obata: A C}) and (\ref{risultato alla Obata: A B}) into (\ref{risultato alla obata: A A}) we deduce
	\begin{align}\label{risultato alla Obata: A D}
		0=2\frac{\gamma}{m}\p^a_{tt}.
	\end{align}
	Therefore, on the set where $\gamma=\diver X$ does not vanish, we have that $\tau(\p)$ vanishes.
	

\end{proof}

\begin{proof}[Proof (of Proposition \ref{prop: risultato alla obata per U-Harmonic Einstein})]
	Since $X$ is conformal and closed, we are in the assumptions of Lemma \ref{lemma: KO: campo conforme}, that is, we are in the case $\eta=0$ of that Lemma. Applying it, we deduce that $\tau(\p)$ vanishes on the set
	\begin{align*}
		A:=\{x\in M: \gamma (x)\neq 0\}.
	\end{align*}
	From Remark \ref{remark: continuazione analitica}, we deduce that $\gamma$ satisfies the unique continuation property; therefore, it vanishes on an open subset of $M$ if and only if it vanishes identically. Since $X$ is non-Killing by assumption, this cannot happen so that the zero set of $\gamma$ does not contain any open subset of  $M$. As a consequence, $A$ is dense in $M$. Since $\tau(\p)$ vanishes there, we deduce that, by continuity, $\tau(\p)=0$ on all of $M$. Therefore $(M,g)$ is a harmonic-Einstein manifold and we are in the conditions to apply Lemma 5.2 of \cite{ACR} to conclude.
	\end{proof}

\subsection{Proof of the Main Theorem}
From our assumptions, we know that
\begin{align*}
	A^\p-\frac{U(\p)}{m-1}g
\end{align*} is a Codazzi tensor; in the following, $P_k$ will be the $k$-th Newton operator ``in its eigenvalues'', which is divergence free by Corollary \ref{cor: Codazzi e divergence free}.
Therefore
\begin{align*}
	[(P_k)_{ij}f_i]_j=(P_k)_{ij}f_{ij},
\end{align*}
that is
\begin{align}\label{divergenza di P k f}
	\diver (P_k(\nabla f))=\mathrm{tr}(P_k\circ \hess(f)).
\end{align}
We first prove the validity of the following identity:
\begin{align}\label{polinomi simmetrici: identità fondamentale}
	\tr \set{\pa{P_k-\frac{c_k}{m}\sigma_k \id}\circ \hess(u)}&-g\pa{\nabla \log u^{1-\frac{\eta}{\beta}}, \pa{P_k-\frac{c_k}{m}\sigma_k \id}(\nabla u)}\\
	=&c_k\beta u(\sigma_{k+1}-\sigma_1\sigma_k), \notag
\end{align}
where $u=e^{-\beta f}$, for some $\beta\in \erre, \ \beta\neq 0$.
Note that, to prove it, we will only use the constancy of $\sigma_k$, equation (\ref{U phi Cotton}) and the first equation of (\ref{Other rigidity: Einstein-type 2}).\\
Rewriting \eqref{Other rigidity: Einstein-type 2} i)  we get
\begin{align}\label{Hess f e A U phi}
	f_{ij}=\eta f_if_j-A^{\p}_{ij}+\frac{U(\p)}{m-1}\delta_{ij}-\frac{1}{2(m-1)}S^{\p}\delta_{ij}-\frac{U(\p)}{m-1}\delta_{ij}+\lambda \delta_{ij}.
\end{align}
We perform the substitution $u=e^{-\beta f}$
to obtain
\begin{align}\label{hess u e A U phi}
	\hess(u)=&\beta u\pa{\frac{S^\p}{2(m-1)}-\lambda (x)+\frac{U(\p)}{m-1}}g+\pa{1-\frac{\eta}{\beta}}\frac{1}{u}du\otimes du\\
	&+\beta u \pa{A^{\p}-\frac{U(\p)}{m-1}g}, \nonumber
\end{align}
which is a structure of the form (\ref{polinomi simmetrici: hess u forma generale}) for the choices
\begin{align*}
	& p(x)=\beta u\pa{\frac{S^\p}{2(m-1)}-\lambda+\frac{U(\p)}{m-1}}(x), \ \ q(x)=\pa{1-\frac{\eta}{\beta }}\frac{1}{u}(x), \\
	& l(x)=-\beta u(x)
\end{align*}
and $A=A^\p-\frac{U(\p)}{m-1}g$.
From equation (\ref{polinomi simmetrici: L k u forma generale}) we deduce
\begin{align*}
	L_k u=&c_k \beta u\sq{\pa{\frac{S^\p}{2(m-1)}-\lambda (x)+\frac{U(\p)}{m-1}}\sigma_k+ \sigma_{k+1}}\\
	&+\pa{1-\frac{\eta}{\beta}}\frac{1}{u}g(P_k(\nabla u), \nabla u).
\end{align*}
A simple computation gives the validity of
\begin{align*}
	\frac{U(\p)}{m-1}=\frac{m-2}{2m(m-1)}S^{\p}-\sigma_1
\end{align*}
and, therefore,
\begin{align*}
	L_k u=&c_k \beta u\pa{-\sigma_1\sigma_k +\sigma_{k+1}}+g\pa{P_k(\nabla u), \nabla \log u^{1-\frac{\eta}{\beta}}}\\
	&+c_k\beta u \sigma_k \frac{S^\p}{m}-c_k\beta u \sigma_k \lambda.
\end{align*}
Tracing equation (\ref{hess u e A U phi}) we obtain
\begin{align*}
	\Delta u=&\frac{m}{2(m-1)}\beta uS^\p+\frac{m}{m-1}\beta u U(\p)-m\beta u\lambda(x)+\pa{1-\frac{\eta}{\beta}}\frac{1}{u}|\nabla u|^2\\
	&+\frac{m-2}{2(m-1)}\beta u S^\p-\frac{m}{m-1}\beta u U(\p)\\
	=&S^\p\beta u-m\beta u\lambda (x)+\pa{1-\frac{\eta}{\beta}}\frac{1}{u}|\nabla u|^2,
\end{align*}
so that
\begin{align*}
	L_k u&-\frac{c_k}{m}\sigma_k\Delta u-g\pa{P_k(\nabla u), \nabla \log u^{1-\frac{\eta}{\beta}}}+\frac{c_k}{m}\sigma_k g\pa{\nabla u, \nabla \log u^{1-\frac{\eta}{\beta}}}\\
	&=c_k \beta u(\sigma_{k+1}-\sigma_1\sigma_k);
\end{align*}
since $\sigma_k$ is constant, the above equation can be rewritten as
\begin{align*}
	\tr &\set{\pa{P_k-\frac{c_k}{m}\sigma_k \id}\circ \hess(u)}-g\pa{\nabla \log u^{1-\frac{\eta}{\beta}}, \pa{P_k-\frac{c_k}{m}\sigma_k \id}(\nabla u)}\\
	=&c_k\beta u(\sigma_{k+1}-\sigma_1\sigma_k),
\end{align*}
which is (\ref{polinomi simmetrici: identità fondamentale}).
The importance of (\ref{polinomi simmetrici: identità fondamentale}) stems from the fact that, when $P_k$ is divergence-free,  its left-hand side becomes a weighted divergence; indeed we have
\begin{align*}
	\tr &\set{\pa{P_k-\frac{c_k}{m}\sigma_k \id}\circ \hess(u)}-g\pa{\nabla \log u^{1-\frac{\eta}{\beta}}, \pa{P_k-\frac{c_k}{m}\sigma_k \id}(\nabla u)}\\
	=&u^{1-\frac{\eta}{\beta}}\diver\pa{u^{\frac{\eta}{\beta}-1}\pa{P_k-\frac{c_k}{m}\sigma_k \id}(\nabla u)}.
\end{align*}
%
Therefore, multiplying (\ref{polinomi simmetrici: identità fondamentale}) by $u^{\frac{\eta}{\beta}-1}$, integrating the result over $M$, using the divergence theorem and $\beta\neq 0$ we obtain
\begin{align*}
	\int_M u^{\frac{\eta}{\beta}}\pa{\sigma_{k+1}-\sigma_1\sigma_k}=0.
\end{align*}
From $u>0$ and (\ref{disuguaglianza di Anselli}) we get
\begin{align*}
	\sigma_{k+1}-\sigma_1\sigma_k=0,
\end{align*}
so that the equality case in (\ref{disuguaglianza di Anselli}) is satisfied. Therefore, all the eigenvalues of $A^\p-\frac{U(\p)}{m-1}g$ coincide so that $A^\p$ is proportional to the metric, and hence $\ric^\p$ is too.
We then have that
\begin{align*}
	\ric^{\p}=\gamma g
\end{align*}
holds, for some $\gamma \in C^{\infty}(M)$. Tracing the equation above, we deduce
\begin{align*}
	m\gamma=S^\p
\end{align*}
and
\begin{align*}
	\ric^\p=\frac{S^\p}{m}g
\end{align*}
so that the first equation of a harmonic-Einstein manifold is satisfied. From equation
\begin{align*}
	C^{\p}_{iik}=\alpha \p^a_k\p^a_{ii}
\end{align*}
and (\ref{U phi Cotton}) we deduce
\begin{align*}
	\alpha \p^a_k\p^a_{ii}=\frac{1}{\alpha}U^a\p^a_k.
\end{align*}
Contracting the above equation with $f_k$
and using the second equation of (\ref{Other rigidity: Einstein-type 2}) we deduce
\begin{align*}
	0=\alpha\p^a_k f_k\pa{\p^a_{tt}-\frac{1}{\alpha} U^a}=\alpha \abs{\tau (\p)-\frac{1}{\alpha} \pa{\nabla U}\pa{\p}}^2.
\end{align*}
In particular,  we have that $\p$ is $\frac{U}{\alpha}$-harmonic and, again from the second equation of (\ref{Other rigidity: Einstein-type 2}), $\nabla f\in \mathrm{Ker}(d\p)$. In other words, we have proved the validity of
\begin{align}\label{ancora U-harmonic Einstein}
	\begin{cases}
		\ric^\p=\frac{S^\p}{m}g\\
		\tau(\p)=\frac{1}{\alpha} \nabla U (\p)
	\end{cases}
\end{align}
and
\begin{align*}
	d\p(\nabla f)=0=d\p(\nabla u).
\end{align*}
Moreover, when $\eta=0$, since $A^{\p}$ is proportional to the metric, (\ref{Hess f e A U phi}) tells us that also $\hess(f)$ is, so that $\nabla f$ is a closed conformal vector field. When $\eta \neq 0$, choosing $\beta=\eta$ in (\ref{hess u e A U phi}) shows us that $\nabla u$ is a closed conformal vector field. In both cases, we are in the assumptions of Proposition \ref{prop: risultato alla obata per U-Harmonic Einstein}, and we can use it to conclude.

	\chapter[The Boucher-Gibbons-Horowitz conjecture]{The Boucher-Gibbons-Horowitz conjecture and related results}\label{Sect_proof1.3}

To prove Theorem \ref{thm A} we make a tricky use of the maximum principle, together with a \textit{Shen-type} identity for a suitable vector field. Throughout the proof we are going to use the equations of $\p$-SPFST many times; thus, for the sake of simplicity, we recall system \eqref{Gianny} here:
\begin{align}\label{sysyem per shen type}
	\begin{cases}
		i)\, \hs(u)-u\set{\ric^\p-\frac{1}{m-1}\pa{\frac{S^\p}{2}-p+U(\p)}g}=0,\\
		ii)\,\Delta u=\frac{u}{m-1}\sq{mp-mU(\p)+\frac{m-2}{2}S^\p},\\
		iii)\,u\tau(\p)=-d\p(\nabla u)+\frac{u}{\alpha}(\nabla U)(\p),\\
		iv)\,\mu+U(\p)=\frac{1}{2}S^\p,\\
		v)\,(\mu+p)\nabla u=-u\nabla p.
	\end{cases}
\end{align}
We now focus on the following result that, besides being interesting in its own, is instrumental for the proof of the theorem.

\begin{proposition}\label{prop divX}
	Let $(M,g)$ be a $\p$-SPFST of dimension $m\geq 2$ with $\alpha>0$.
	Consider the vector field
	\begin{align}\label{3.2 cioè X}
		Z:=\frac{1}{u}\nabla\set{\abs{\nabla u}^2-\frac{u^2}{m(m-1)}\sq{((m-2)\mu+mp)-2U(\p)}}=\frac{1}{u}\nabla\pa{\abs{\nabla u}^2-\frac{u}{m}\Delta u},
	\end{align}
	on $\mathrm{int}(M)$. Then
	\begin{align}\label{div X 3.3}
		\mathrm{div}Z=&\frac{2}{u}\set{\abs{\hs(u)}^2-\frac{(\Delta u)^2}{m}}+\frac{(2m-3)}{m(m-1)}g(\nabla ((m-2)\mu+mp),\nabla u)\\
		&+\frac{2u}{m(m-1)}\mathrm{tr}(\hs(U)(d\p,d\p))+\frac{2u}{\alpha m^2(m-1)}\abs{{\nabla U}}^2(\p)\notag\\
		&-\frac{u}{m(m-1)}\Delta ((m-2)\mu+mp)+u\abs{\sqrt{\frac{2}{\alpha}}\frac{1}{m}\pa{\nabla U}(\p)-\sqrt{2\alpha}d\p\pa{\frac{\nabla u}{u}}}^2\notag\\
		&+\frac{2}{u}\pa{\mu+p}\abs{\nabla u}^2.\notag
	\end{align}
\end{proposition}
\begin{rem}
	Note that $Z$ is a modification of an \emph{Obata-type vector field}, as defined in \cite{MRSYamabe}.
\end{rem}
\begin{proof}
	Inserting equation \eqref{sysyem per shen type} iv) into \eqref{sysyem per shen type} ii) we have
	\begin{align}\label{eq_lap with V and G}
		\Delta u=\frac{u}{m-1}\set{\sq{mp+(m-2)\mu}-2U(\p)};
	\end{align}
	next, we observe that the components of $Z$ are given by
	\begin{align*}
		Z_i=&\frac{1}{u}\left[2u_{ki}u_k-\frac{1}{m(m-1)}\pa{(m-2)\mu+mp}_iu^2-\frac{2u}{m(m-1)}\pa{(m-2)\mu+mp}u_i\right.\\
		&\left.+\frac{2}{m(m-1)}U^a\p^a_iu^2+\frac{4U(\p)}{m(m-1)}uu_i\right];
	\end{align*}
	computing the divergence of $Z$ we get
	\begin{align*}
		Z_{ii}=&\frac{1}{u}\left(2\abs{\hs(u)}^2+2u_{iki}u_k-\frac{2u}{m(m-1)}u_i\pa{(m-2)\mu+mp}_i\right.\\
		&\left.-\frac{u^2}{m(m-1)}\pa{(m-2)\mu+mp}_{ii}
		-\frac{2}{m(m-1)}\pa{(m-2)\mu+mp}\abs{\nabla u}^2\right.\\
		&\left.-\frac{2u}{m(m-1)}\pa{(m-2)\mu+mp}_iu_i-\frac{2u}{m(m-1)}\pa{(m-2)\mu+mp}\Delta u\right.\\
		&\left.+\frac{2}{m(m-1)}U^{ab}\p^a_i\p^b_iu^2+\frac{4u}{m(m-1)}U^a\p^a_iu_i+\frac{2}{m(m-1)}U^a\p^a_{ii}u^2\right.\\
		&\left.+\frac{4}{m(m-1)}U^a\p^a_iu_iu+\frac{4U(\p)}{m(m-1)}\abs{\nabla u}^2+\frac{4U(\p)}{m(m-1)}uu_{ii}
		\right)\\
		&-\frac{u_i}{u^2}\left[2u_{ki}u_k-\frac{1}{m(m-1)}\pa{(m-2)\mu+mp}_iu^2\right.\\
		&\left.-\frac{2u}{m(m-1)}\pa{(m-2)\mu+mp}u_i+\frac{2}{m(m-1)}U^a\p^a_iu^2+\frac{4U(\p)}{m(m-1)}uu_i\right].
	\end{align*}
	Next, using the commutation rules for the covariant derivative of a smooth function
	we obtain
	\begin{align*}
		\diver Z=&\frac{1}{u}\left[2\abs{\hs(u)}^2+2g(\nabla\Delta u,\nabla u)+2\ric(\nabla u,\nabla u)\right.\\
		&\left.-\frac{4u}{m(m-1)}g(\nabla u,\nabla \pa{(m-2)\mu+mp})-\frac{u^2}{m(m-1)}\Delta \pa{(m-2)\mu+mp}\right.\\
		&\left.-\frac{2}{m(m-1)}\pa{(m-2)\mu+mp}\abs{\nabla u}^2-\frac{2u}{m(m-1)}\pa{(m-2)\mu+mp}\Delta u\right.\\
		&\left.+\frac{2}{m(m-1)}\mathrm{tr}(\hs(U)(d\p,d\p))u^2+\frac{8u}{m(m-1)}g(\nabla(U(\p)),\nabla u)\right.\\
		&\left.+\frac{2}{m(m-1)}h(\nabla U,\tau(\p))u^2+\frac{4U(\p)}{m(m-1)}\abs{\nabla u}^2+\frac{4U(\p)u}{m(m-1)}\Delta u\right]\\
		&-\frac{2}{u^2}\hs(u)(\nabla u,\nabla u)+\frac{1}{m(m-1)}g(\nabla u,\nabla \pa{(m-2)\mu+mp})\\
		&+\frac{2}{u}\frac{\pa{(m-2)\mu+mp}}{m(m-1)}\abs{\nabla u}^2-\frac{2}{m(m-1)}g(\nabla(U(\p),\nabla u)-\frac{4}{u}\frac{U(\p)}{m(m-1)}\abs{\nabla u}^2.
	\end{align*}
	Using equation \eqref{eq_lap with V and G} to express $\Delta u$ we infer
	\begin{align*}
		\diver Z=&\frac{2}{u}\abs{\hs(u)}^2+\frac{2}{u(m-1)}g(\nabla\sq{\pa{(m-2)\mu+mp-2U(\p)}u},\nabla u)+\frac{2}{u}\ric(\nabla u,\nabla u)\\
		&-\frac{3}{m(m-1)}g(\nabla u,\nabla \pa{(m-2)\mu+mp})-\frac{u}{m(m-1)}\Delta \pa{(m-2)\mu+mp}\\
		&-\frac{2\pa{(m-2)\mu+mp}}{um(m-1)}\abs{\nabla u}^2-\frac{2}{mu}\pa{\Delta u}^2+\frac{2}{m(m-1)}\mathrm{tr}(\hs(U)(d\p,d\p))u\\
		&+\frac{6}{m(m-1)}g(\nabla u,\nabla(U(\p)))+\frac{2u}{m(m-1)}h(\nabla U,\tau(\p))\\
		&+\frac{4U(\p)}{um(m-1)}\abs{\nabla u}^2-\frac{2}{u^2}\hs(u)(\nabla u,\nabla u)\\
		&+\frac{2\pa{(m-2)\mu+mp}}{um(m-1)}\abs{\nabla u}^2-\frac{4U(\p)}{um(m-1)}\abs{\nabla u}^2.
	\end{align*}
	Writing
	\begin{align*}
		\ric=\ric^\p+\p^*h
	\end{align*}
	and using the first equation of \eqref{sysyem per shen type} we rewrite $\hs(u)$ as:
	\begin{align}\label{eq for hs in divX}
		\hs(u)&=u\set{\ric^\p-\frac{1}{m-1}\pa{\frac{S^\p}{2}-p+U(\p)}g}.
	\end{align}
	Inserting \eqref{eq for hs in divX} into the expression of $\diver Z$ we obtain
	\begin{align*}
		\diver Z=&\frac{2}{u}\abs{\hs(u)}^2+\frac{2}{u(m-1)}g(\nabla\sq{m\pa{(m-2)\mu+mp-2U(\p)}u},\nabla u)\\
		&+\frac{2}{u}\ric(\nabla u,\nabla u)-\frac{3}{m(m-1)}g(\nabla u,\nabla \pa{(m-2)\mu+mp})\\
		&-\frac{u}{m(m-1)}\Delta \pa{(m-2)\mu+mp}-\frac{2}{mu}(\Delta u)^2+\frac{2}{m(m-1)}\mathrm{tr}(\hs(U)(d\p,d\p))u\\
		&+\frac{6}{m(m-1)}g(\nabla u,\nabla (U(\p)))+\frac{2u}{m(m-1)}h(\nabla U,\tau(\p))\\
		&-\frac{2}{u}\set{\ric(\nabla u,\nabla u)-\alpha\abs{d\p(\nabla u)}^2-\frac{1}{m-1}\pa{\frac{S^\p}{2}-p+U(\p)}\abs{\nabla u}^2}.
	\end{align*}
	Then, using \eqref{sysyem per shen type} iii) we deduce
	\begin{align*}
		\diver Z=&\frac{2}{u}\abs{\hs(u)}^2+\frac{2}{m-1}g(\nabla \pa{(m-2)\mu+mp},\nabla u)\\
		&-\frac{4}{m-1}g(\nabla(U(\p)),\nabla u)+\frac{2}{u(m-1)}((m-2)\mu+mp-2U(\p))\abs{\nabla u}^2\\
		&-\frac{3}{m(m-1)}g(\nabla u,\nabla \pa{(m-2)\mu+mp})+\frac{6}{m(m-1)}g(\nabla u,\nabla(U(\p)))\\
		&+\frac{2u}{m(m-1)}\mathrm{tr}(\hs(U)(d\p,d\p))+\frac{2}{m(m-1)}\pa{\frac{1}{\alpha}\abs{\nabla U}^2(\p)u-g(\nabla(U(\p)),\nabla u)}\\
		&-\frac{2}{um}(\Delta u)^2+\frac{2\alpha}{u}\abs{d\p(\nabla u)}^2+\frac{2}{2(m-1)}\frac{2}{u}\pa{\frac{S^\p}{2}-p+U(\p)}\abs{\nabla u}^2\\
		&-\frac{u}{m(m-1)}\Delta \pa{(m-2)\mu+mp},
	\end{align*}
	that is
	\begin{align*}
		\diver Z
		=&\frac{2}{u}\set{\abs{\hs(u)}^2-\frac{(\Delta u)^2}{m}}+\frac{(2m-3)}{m(m-1)}g(\nabla ((m-2)\mu+mp),\nabla u)\\
		&-\frac{4}{m}g(\nabla (U(\p)),\nabla u)+\frac{2}{u}\frac{(m-2)}{(m-1)}\mu\abs{\nabla u}^2+\frac{2m}{u(m-1)}p\abs{\nabla u}^2-\frac{4U(\p)}{u(m-1)}\abs{\nabla u}^2\\
		&+\frac{2}{m(m-1)}\mathrm{tr}(\hs(U)(d\p,d\p))u+\frac{2\abs{\nabla U}^2(\p)u}{m(m-1)\alpha}+\frac{2\alpha}{u}\abs{d\p(\nabla u)}^2+\frac{2\abs{\nabla u}^2}{u(m-1)}\frac{S^\p}{2}\\
		&-\frac{2}{u(m-1)}p\abs{\nabla u}^2+\frac{2U(\p)}{u(m-1)}\abs{\nabla u}^2-\frac{u}{m(m-1)}\Delta \pa{(m-2)\mu+mp}\\
		=&\frac{2}{u}\set{\abs{\hs(u)}^2-\frac{(\Delta u)^2}{m}}+\frac{(2m-3)}{m(m-1)}g(\nabla \pa{(m-2)\mu+mp},\nabla u)\\
		&-\frac{4}{m}g(\nabla(U(\p)),\nabla u)+\frac{2}{u(m-1)}\pa{\frac{S^\p}{2}+(m-2)\mu+(m-1)p-U(\p)}\abs{\nabla u}^2\\
		&+\frac{2}{m(m-1)}\mathrm{tr}(\hs(U)(d\p,d\p))u+\frac{2}{m(m-1)\alpha}\abs{\nabla U}^2(\p)\\
		&+\frac{2\alpha}{u}\abs{d\p(\nabla u)}^2-\frac{u}{m(m-1)}\Delta \pa{(m-2)\mu+mp}.
	\end{align*}
	Rewriting
	\begin{align}\label{eq per nabla G in divX}
		\frac{2}{m(m-1)\alpha}u\abs{\nabla U}^2(\p)=\frac{2}{m^2\alpha}u\abs{\nabla U}^2(\p)+\frac{2}{m^2(m-1)\alpha}u\abs{\nabla U}^2(\p)
	\end{align}
	we obtain
	\begin{align*}
		\diver Z=&\frac{2}{u}\set{\abs{\hs(u)}^2-\frac{(\Delta u)^2}{m}}+\frac{(2m-3)}{m(m-1)}g(\nabla \pa{(m-2)\mu+mp},\nabla u)\\
		&+\frac{2u}{m(m-1)}\mathrm{tr}(\hs(U)(d\p,d\p))+\frac{2u}{m^2(m-1)\alpha}\abs{\nabla U}^2(\p)\\
		&-\frac{u}{m(m-1)}\Delta \pa{(m-2)\mu+mp}+u\abs{\frac{\sqrt{2}}{m\sqrt{\alpha}}\pa{\nabla U}(\p)-\sqrt{2\alpha}d\p\pa{\frac{\nabla u}{u}}}^2\\
		&+\frac{2}{u(m-1)}\pa{\frac{S^\p}{2}+(m-2)\mu+(m-1)p-U(\p)}\abs{\nabla u}^2.
	\end{align*}
	Now, using \eqref{sysyem per shen type} iv), that is $\frac{S^\p}{2}=\mu+U(\p)$, we obtain
	\begin{align*}
		\diver Z=&\frac{2}{u}\set{\abs{\hs(u)}^2-\frac{(\Delta u)^2}{m}}+\frac{(2m-3)}{m(m-1)}g(\nabla \pa{(m-2)\mu+mp},\nabla u)\notag\\
		&+\frac{2u}{m(m-1)}\mathrm{tr}(\hs(U)(d\p,d\p))+\frac{2u}{m^2(m-1)\alpha}\abs{\nabla U}^2(\p)\notag\\
		&-\frac{u}{m(m-1)}\Delta \pa{(m-2)\mu+mp}+u\abs{\frac{\sqrt{2}}{m\sqrt{\alpha}}\pa{\nabla U}(\p)-\sqrt{2\alpha}d\p\pa{\frac{\nabla u}{u}}}^2\notag\\
		&+\frac{2}{u}\pa{\mu+p}\abs{\nabla u}^2.
	\end{align*}
	that is, \eqref{div X 3.3}.
\end{proof}
We now prove a generalization of Theorem \ref{thm A}.
\begin{theorem}\label{thm A di nuovo}
	Let $(M,g)$ be an $m$-dimensional, $m\geq 3$,
	compact $\p$-SPFST with connected, non-empty boundary and $\alpha>0$. Assume that:
\begin{itemize}
  \item $U$ is weakly convex,
  \item $		\Delta_{\log u^{2m-3}}\pa{(m-2)\mu +mp}\leq 0$;,
\item $p+\mu\geq 0$
\end{itemize}
and that
 \begin{align}\label{ipotesi BM in proof of thm A}
		m(m-1)|\nabla u|^2_{|_{\partial M}}\leq \max_M\set{\sq{2U(\p)-\pa{(m-2)\mu+mp}}u^2}.
	\end{align}

	Then $\p, \mu, p$ and $S$ are constant on $M$, with $\mu$ and $S$ positive and $\mu=-p$; finally, $(M,g)$ is isometric to the hemisphere
	\begin{align}
		S^m_+\pa{\frac{S}{m(m-1)}}\subset \erre^{m+1}.
	\end{align}
\end{theorem}
\begin{proof}
	We start by showing that the function
\begin{align}\label{3.4}
	v:=\abs{\nabla u}^2-\frac{1}{m(m-1)}\pa{(m-2)\mu+mp-2U(\p)}u^2
\end{align}
is constant on $M$. Since $\abs{\nabla u}=c^2>0$ on $\partial M$ by Proposition \ref{tot geod}, we can fix $\delta>0$ sufficiently small such that the level set
\begin{align}\label{3.5}
	\partial M_{\eps}:=\set{x\in M\: :\: u(x)=\eps}
\end{align}
is a smooth hypersurface for each $0<\eps\leq \delta$, where
\begin{align}\label{3.6}
	M_{\eps}:=\set{x\in M\: :\: u(x)\geq \eps}.
\end{align}
Thus, $\frac{1}{u}$ is positive and bounded on $M_{\eps}$. Let $Z$ be as in \eqref{3.2 cioè X};
by Proposition \ref{prop divX} we have
\begin{align*}
	\diver Z=&\frac{2}{u}\set{\abs{\hs(u)}^2-\frac{\pa{\Delta u}^2}{m}}+\frac{2}{m(m-1)}u\mathrm{tr}\pa{\hs(U)(d\p,d\p)}\\
	&+\frac{2}{m^2(m-1)\alpha}\abs{\nabla U}^2(\p)
	+u\abs{\frac{2}{\sqrt{2\alpha}m}\pa{\nabla U}(\p)-\sqrt{2\alpha}d\p\pa{\frac{\nabla u}{u}}}^2\\
	&+\frac{2}{u}(\mu+p)|\nabla u|^2-\frac{u}{m(m-1)}\Delta_{\log u^{2m-3}}\pa{(m-2)\mu+mp}.
\end{align*}
Hence, since $U$ is weakly convex, $p+\mu\geq 0$ and $\Delta_{\log u^{2m-3}}\pa{(m-2)\mu+mp}\leq 0$, the above gives
\begin{align*}
	\diver Z\geq 0\quad\text{on }\mathrm{int}(M)
\end{align*}
and by definition \eqref{3.2 cioè X} of $Z$, that is, $Z=\frac{1}{u}\nabla v$, we deduce
\begin{align}\label{3.7}
	\Delta v-\frac{1}{u}g( \nabla u,\nabla v)\geq 0.
\end{align}
Thus, by the maximum principle
\begin{align}\label{3.8}
	\max_{M_{\eps}}v=\max_{\partial M_{\eps}}v \quad \text{for } 0<\eps\leq \delta.
\end{align}
Since $M_{\delta}\subseteq M_{\eps}$ for $0<\eps\leq \delta$, by \eqref{3.8} we obtain
\begin{align}\label{3.9}
	\max_{\partial M_{\delta}}v\leq \max_{\partial M_{\eps}}v.
\end{align}
Letting $\eps\ra0^+$, we get
\begin{align}\label{3.100}
	\lim_{\eps \ra 0^+}\max_{\partial M_{\eps}}v=\abs{\nabla u}^2|_{\partial M}
\end{align}
and therefore, using \eqref{3.9}, we have
\begin{align}\label{3.11}
	\max_{\partial M_{\delta}}v\leq \abs{\nabla u}^2|_{\partial M}.
\end{align}
We use the latter, the fact that $u>0$ on $\mathrm{int}(M)$ and assumption \eqref{ipotesi BM in proof of thm A} to infer
\begin{align}\label{3.12}
	\max_{\partial M_{\delta}} v\leq \max_{ M}\set{\frac{1}{m(m-1)}\pa{2U(\p)-(m-2)\mu-mp}u^2}.
\end{align}
On the other hand, by \eqref{3.8} with $\eps=\delta$ and \eqref{3.12}, we deduce
\begin{align}\label{3.13}
	\max_{ M_{\delta}} v\leq \max_{ M}\set{\frac{1}{m(m-1)}\pa{2U(\p)-(m-2)\mu-mp}u^2}.
\end{align}
Let $K$ be the set of points of $M$ at which the function $(2U(\p)-(m-2)\mu-mp)u^2$ realizes its absolute maximum;
%
$K$ is closed, so that, since $M$ is compact, it is compact and for each $x\in K$, $u(x)>0$. It follows that
\begin{align*}
	\min_Ku=2\eta
\end{align*}
for some $\eta>0$. Choosing $\delta\leq \eta$, we have
\begin{align*}
	K\subseteq M_{\delta}\setminus\partial M_{\delta}
\end{align*}
and using \eqref{3.13}, for each $p\in K$, it follows
\begin{align*}
	\max_{M}\Bigg\{\frac{1}{m(m-1)}&\pa{2U(\p)-(m-2)\mu-mp}u^2\Bigg\}\\
	&=\frac{1}{m(m-1)}\pa{2U(\p)-(m-2)\mu-mp}u^2(p)\\
	&\leq v(p)\leq \max_{M_{\delta}}v\\
	&\leq\max_{M}\set{\frac{1}{m(m-1)}\pa{2U(\p)-(m-2)\mu-mp}u^2}.
\end{align*}
Hence, for each $p\in K\subseteq M_{\delta}\setminus\partial M_{\delta}$, we get
\begin{align*}
	v(p)=\max_{M_{\delta}}v=\max_{M}\set{\frac{1}{m(m-1)}\pa{2U(\p)-(m-2)\mu-mp}u^2},
\end{align*}
i.e., $v$ assumes its absolute maximum at an interior point of $M_{\delta}$. Using \eqref{3.7} and again the maximum principle we deduce that
\begin{align}\label{v attains max val}
	v&=\abs{\nabla u}^2+\frac{1}{m(m-1)}\pa{2U(\p)-(m-2)\mu-mp}u^2\\
	&\equiv \max_{M}\set{\frac{1}{m(m-1)}\pa{2U(\p)-(m-2)\mu-mp}u^2} \quad \text{on }M_{\delta}. \notag
\end{align}
Letting $\delta\ra 0^+$, we deduce that \eqref{v attains max val} holds on $M$ and
$v$ is constant. From \eqref{div X 3.3} and the assumptions of the theorem, we obtain,on $\mathrm{int}M$,
\begin{align}
	&\hs(u)=\frac{\Delta u}{m}g;\label{3.15}\\
	&\frac{2}{\sqrt{2\alpha}m}(\nabla U)(\p)-\sqrt{2\alpha} d\p\pa{\frac{\nabla u}{u}}\equiv 0;\label{3.16} \\
	&\pa{\nabla U}(\p)\equiv 0;\label{3.17} \\
	&\frac{2}{u}(\mu+p)|\nabla u|^2\equiv 0.\label{mu +p a altra roba}
\end{align}

We want to show that  there is no open subset of $ M$ such that $\abs{\nabla u}$  vanishes identically on it: suppose by contradiction the existence of an open subset  $A\subset M$ where $u$ is constant. Using \eqref{sysyem per shen type}  ii) and iv) we get
\begin{align*}
	\Delta u=-\frac{1}{m(m-1)}\sq{2U(\p)-(m-2)\mu-mp}u,
\end{align*}
so that,
\begin{align*}
	0\equiv \frac{1}{m(m-1)}\pa{2U(\p)-(m-2)\mu-mp}u \quad \text{  on $A$.}
\end{align*}
Therefore,
\begin{align*}
	v:=\abs{\nabla u}^2-\frac{1}{m(m-1)}\pa{(m-2)\mu+mp-2U(\p)}u^2
\end{align*}
vanishes identically on $A$. Since $v$ is constant, as we just proved, we deduce $v\equiv 0$ on $M$. By Proposition \ref{prop: violazione della SEC}, there exists $p \in \mathrm{int}M$ such that
\begin{align*}
	2U(\p)-(m-2)\mu-mp>0 \quad \text{ at $p$}
\end{align*}
so that $v(p)>0$, a contradiction.
It follows by \eqref{mu +p a altra roba} and $u>0$ on $\mathrm{int}(M)$ that
	\begin{align*}
			\mu+p=0
		\end{align*}
	on a dense open subset of $M$ and therefore on $M$.
	From Proposition \ref{Prop2.12}, we get that $\mu$ and $p$ are therefore constant.\\
Moreover, \eqref{3.17} implies
\begin{align}\label{3.19}
	U^a\p^a_i\equiv 0 \quad \text{on }M.
\end{align}
It follows that $U(\p)$ is constant on $M$ and therefore
\begin{align*}
	2U(\p)-(m-2)\mu-mp
\end{align*}
is constant on $M$. Since $\mu=-p$, \eqref{sysyem per shen type} iv) implies
\begin{align}\label{Scalar = cose con mu e p}
	S^\p=2U(\p)-(m-2)\mu-mp,
\end{align}
so that $S^\p$ is a positive constant because of Proposition \ref{prop: violazione della SEC}.
From \eqref{sysyem per shen type} ii), iv), \eqref{3.15} and \eqref{Scalar = cose con mu e p} we deduce
\begin{align}\label{hess u e u}
	\hs(u)=-\frac{u}{m(m-1)}S^{\p} g.
\end{align}
Since $u$ is non-constant and $i: \partial M \hookrightarrow M$ is totally geodesic 
we can apply Lemma 3 of \cite{R} to conclude that $(M,g)$ is isometric to
\begin{align}\label{isometric to emisfrea}
	\mathbb{S}^m_+\pa{\frac{S^\p}{m(m-1)}}.
\end{align}
Next, by \eqref{3.16} and \eqref{3.17}, we have
\begin{align*}
	d\p(\nabla u)\equiv 0 \qquad \text{and }\quad \pa{\nabla U}(\p)\equiv 0;
\end{align*}
hence, by \eqref{sysyem per shen type} iii),
\begin{align}\label{3.20}
	\tau(\p)\equiv0,
\end{align}
that is $\p$ is harmonic. Moreover, using \eqref{hess u e u} and \eqref{sysyem per shen type} i), iv) we deduce
\begin{align}\label{3.21}
	\ric^\p&=\frac{S^\p}{m}g.
\end{align}
From \eqref{3.20} and \eqref{3.21} we get
\begin{align*}
	\begin{cases}
		&\ric^\p=\Lambda g\\
		&\tau(\p)=0
	\end{cases}
\end{align*}
for some constant $\Lambda$, that is, $(M,g,\p)$ is harmonic-Einstein; moreover, since $(M,g)$ is isometric to a Euclidean hemisphere, it is also Einstein, i.e.
\begin{align*}
	\ric=\zeta g,
\end{align*}
where $\zeta\in\erre$.
As a consequence $\p^*h$ is proportional to the metric $g$,
\begin{align*}
	\p^*h=\rho g
\end{align*}
for some constant $\rho\in \erre$, and by \eqref{3.16} we infer
\begin{align*}
	\p^a_iu_i=0,
\end{align*}
which implies
\begin{align*}
	0=\p^*h(\nabla u, \nabla u)=\rho\abs{\nabla u}^2.
\end{align*}
Thus, since $u$ is non-constant, we deduce $\rho=0$ and therefore the constancy of $\p$.
As a consequence, $S^\p=S$ and $\pa{M,g}$ is isometric to
\begin{align*}
	\mathbb{S}^m_+\pa{\frac{S}{m(m-1)}}.
\end{align*}
\end{proof}
\begin{rem}
	The validity of
	\begin{align*}
		\hs (u)=\frac{\Delta u}{m}g
	\end{align*}
	gives rise to a structure on $M$ which is usually called a \emph{conformal gradient soliton}. They were studied in the late '60 (see for instance \cite{tashiro}) and recently reconsidered by Petersen and Wylie, \cite{PetWylie}, where they sketched a classification which has been fully proved in Catino, Mantegazza and Mazzieri, \cite{CatinoMantegazzaMazzieri}.
\end{rem}
\begin{rem}
	We discuss here the sharpness of assumption \eqref{ipotesi BM in proof of thm A} in Theorem \ref{thm A di nuovo}. To do so, we present an example of $\p$-vacuum static space, which is based on an example of Costa, Diogenes, Pinheiro and Ribeiro (\cite{CDPR2023}) and reduces to it when $\alpha=0$.
	Let $M=\mathbb{S}^{n+1}_+\times \mathbb{S}^q$, with $q>1$,  and
	\begin{align*}
		g=dr^2+\sin^2(r)g_{\mathbb{S}^n}+\rho g_{\mathbb{S}^q}
	\end{align*}
	where $r$ is the height function of $\mathbb{S}^{n+1}_+$ and $\rho$ is a positive constant.
	In particular, $r\equiv \frac{\pi}{2}$ on $\partial \mathbb{S}^{n+1}_+$.
	Denoting with $m$ the dimension of $M$, we have $m=n+1+q$.
	Note that $g$ is just the product metric
	\begin{align*}
		g=g_{\mathbb{S}^{n+1}_+}+\rho g_{\mathbb{S}^q}.
	\end{align*}
	Choosing $u(r)=\cos(r)\in C^{\infty}(M)$, $\p:(M,g)\to (\mathbb{S}^q,\rho g_{\mathbb{S}^q})$ the projection, $U\equiv 0$ and $\alpha=\frac{q-1}{\rho}-(n+1)$, we have that $(M,g)$ becomes a $\p$-SPFST, as we are going to see.
	Note that the $(0,2)$-version of the Ricci tensor is invariant by a homothetic rescaling of the metric, so that
	\begin{align*}
		\ric_{\rho g_{\mathbb{S}^q}}=\ric_{g_{\mathbb{S}^q}}=(q-1)g_{\mathbb{S}^q}.
	\end{align*}
	Since $g$ is a product metric, its Ricci tensor, $\ric$, satisfies
	\begin{align}\label{ribeiro: tensore di ricci}
		\ric&=\ric_{g_{\mathbb{S}^{n+1}_+}}+\ric_{\rho g_{\mathbb{S}^q}} \nonumber\\
		&=ng_{\mathbb{S}^{n+1}_+}+(q-1)g_{\mathbb{S}^q}.
	\end{align}
	Taking the covariant derivative of $u$ we get
	\begin{align}\label{ribeiro: nabla u}
		\nabla u=-\sin(r)\nabla r
	\end{align}
	and
	\begin{align}\label{ribiero: norma nabla u}
		\abs{\nabla u}^2=\sin^2(r).
	\end{align}
		Therefore
	\begin{align}\label{ribeiro: nabla f}
		\abs{\nabla u}^2{\mid_{\partial M}}=\sin^2\pa{\frac{\pi}{2}}=1.
	\end{align}
	Since $g$ is a product metric and $u$ only depends on the first factor, we have
	\begin{align}\label{ribeiro: hessiano di u}
		\hess (u)=&-\cos(r)dr^2-\cos(r)\sin^2(r)g_{\mathbb{S}^2}\nonumber\\
		=&-ug_{\mathbb{S}^{n+1}_+}
	\end{align}
	and
	\begin{align}\label{ribeiro: delta u}
		\Delta u=-(n+1)u.
	\end{align}
	From our choice of $\p$, we have
	\begin{align*}
		\p^*h=\rho g_{\mathbb{S}^q}
	\end{align*} so that
	\begin{align}\label{ribeiro: ricci phi}
		\ric^\p=ng_{\mathbb S^{n+1}_+}+\sq{q-1-\alpha\rho} g_{\mathbb{S}^q}.
	\end{align}
	Tracing \eqref{ribeiro: ricci phi} with respect to $g$ we get, from our choice of $\alpha$,
	\begin{align*}
		S^\p=&n(n+1)+\frac{1}{\rho}\sq{q-1-\alpha\rho}q\\
		=&n(n+1)+q(n+1)=(n+1)(n+q)\\
		=&(m-1)(n+1),
	\end{align*}
	that is,
	\begin{align}\label{ribeiro: scalare}
		S^\p=(m-1)(n+1).
	\end{align}
	Using \eqref{ribeiro: hessiano di u}, \eqref{ribeiro: delta u} and \eqref{ribeiro: ricci phi} we compute
\begin{align*}
	-\Delta u g+&\hess(u)-u\ric^\p\\
	=&(n+1)ug-ug_{\mathbb{S}_{+}^{n+1}}-nug_{\mathbb{S}^{n+1}_+}-u\pa{
	q-1-\alpha \rho }g_{\mathbb{S}^q}\\
	=&\sq{(n+1)u-u-nu}g_{\mathbb{S}^{n+1}_+}+\set{(n+1)\rho-\sq{(q-1)-\alpha \rho}}ug_{\mathbb{S}^q}\\
	=&0
\end{align*}
where the last equality is due to our choice of $\alpha$.
Therefore, we have
\begin{align}\label{ribeiro: prima eq. di phi vac space}
	-\Delta u \,g+\hess(u)-u\ric^\p=0
\end{align}
and tracing \eqref{ribeiro: prima eq. di phi vac space} we also have
\begin{align}\label{ribeiro: seconda eq}
	\Delta u=-\frac{S^\p}{m-1}u.
\end{align}
Since $\p$ is the projection map, from \eqref{ribeiro: nabla u} we get $d\p(\nabla u)=0$. Setting $U\equiv 0$, since $\p$ is clearly harmonic, we get
\begin{align}\label{ribeiro: terza eq}
	u\tau(\p)=-d\p(\nabla u)+\frac{u}{\alpha}\pa{\nabla U}(\p).
\end{align}
Setting $\mu=\frac{1}{2}S^\p$ and $p=-\mu$ we deduce, from \eqref{ribeiro: prima eq. di phi vac space}, \eqref{ribeiro: seconda eq} and \eqref{ribeiro: terza eq}, that system \eqref{Gianny1} is satisfied so that $(M,g)$ is a $\p$-SPFST. Note that, when $\rho=\frac{q-1}{n+1}$, we get $\alpha=0$ and our example reduces to Example 2 of \cite{CDPR2023}.\\
		We show that inequality \eqref{ipotesi BM in proof of thm A}, that is,
		\begin{align}\label{ribeiro: stima borg mazz}
					m(m-1)|\nabla u|^2{\mid_{\partial M}}\leq \max_M\set{\sq{2U(\p)-\pa{(m-2)\mu+mp}}u^2},
		\end{align}
		does not hold on $(M,g)$.\\
		From \eqref{ribeiro: nabla f}, the left hand side of \eqref{ribeiro: stima borg mazz} is
		\begin{align}\label{ribeiro: LHS}
			m(m-1)|\nabla u|^2{\mid_{\partial M}}=m(m-1).
		\end{align}
	Since  $\mu=-p$ and $U\equiv 0$,  the right hand side of \eqref{ribeiro: stima borg mazz} becomes
	\begin{align}\label{ribeiro: stima preliminare}
		 \max_M\set{\sq{2U(\p)-\pa{(m-2)\mu+mp}}u^2}=2\max_M\set{\mu u^2}.
	\end{align}
	Using \eqref{sysyem per shen type} iv) in \eqref{ribeiro: stima preliminare} we have
	\begin{align*}
	\max_M\set{\sq{2U(\p)-\pa{(m-2)\mu+mp}}u^2}=\max_M \set{S^\p u^2}.
	\end{align*}
	From \eqref{ribeiro: nabla f} and \eqref{ribeiro: scalare} we get
	\begin{align}\label{ribeiro: RHS}
		\max_M\set{\sq{2U(\p)-\pa{(m-2)\mu+mp}}u^2}&=(m-1)(n+1)\max_M u^2\\
		&=(m-1)(n+1).\nonumber
	\end{align}
Thus, since $m=n+1+q$, \eqref{ribeiro: stima borg mazz} in this case becomes
\begin{align*}
	1\leq \frac{n+1}{n+1+q},
\end{align*}
which is false. However, note that for fixed $q$ and $n>>1$ we approach $1$ with any desired precision. This shows that \eqref{ipotesi BM in proof of thm A} is sharp at least for $\p$-vacuum static spaces.
\end{rem}

In what follows, it will be useful to introduce the $2$-covariant, symmetric tensor $Q$, depending on $u\in C^{\infty}(M)$ and $\p:(M,g)\ra(N,h)$, defined as
\begin{align}\label{4.1}
	Q:=\mathring{\hs}(u)-u\trric^\p,
\end{align}
where $\mathring{\hs}$ and $\trric^\p$ denote the traceless Hessian and the traceless $\p$-Ricci tensor, respectively. With a slight  abuse of notation, as we did before, we shall  indicate with $\hess(u)$ also the associated endomorphism. Note that, if $(M,g)$ is a $\p$-SPFST, then $Q\equiv0$.\\
For the proof of Theorem \ref{thm B} we shall also need some  technical results.
\begin{lemma}\label{lemma 4.7}
	Let $(M, g)$ be a Riemannian manifold of dimension $m$, $\alpha\in \erre$, $\p:(M,g)\ra(N,h)$ a smooth map and $u,w \in C^2(M)$, with $u>0$ on $M$. Define the vector field
	\begin{align}\label{4.5}
		X:=&u\hs(w)(\nabla w, \cdot)^{\#}+\frac{w^2}{u}\hs(u)(\nabla u, \cdot)^{\#}-w\hs(u)(\nabla w,\cdot)^{\#}\notag\\
		&-w\hs(w)(\nabla u, \cdot)^{\#}+\frac{1}{m}\pa{u\nabla w-w\nabla u}.
	\end{align}
	Let $Q$ be defined as in \eqref{4.1}.
	Then,
	\begin{align}\label{4.11}
		\diver X=&u\abs{\mathring{\hs}(w)-\frac{w}{u}\mathring{\hs}(u)}^2+\frac{u}{m}\pa{\Delta w-\frac{w}{u}\Delta u}\pa{\Delta w-\frac{w}{u}\Delta u+1}\notag\\
		&-Q\pa{\nabla w-\frac{w}{u}\nabla u,\nabla w-\frac{w}{u}\nabla u}-\frac{1}{m}\abs{\nabla w-\frac{w}{u}\nabla u}^2\pa{\Delta u-S^\p u}\notag\\
		&+ug\pa{\nabla w-\frac{w}{u}\nabla u, \nabla \Delta w-\frac{w}{u}\nabla \Delta u}+\alpha u\abs{d\p\pa{\nabla w-\frac{w}{u}\nabla u}}^2.
	\end{align}	
\end{lemma}
\begin{rem}
	Note that $X$ can also be written in the elegant form
	\begin{align*}
		X=u\sq{\hess(w)-\frac{w}{u}\hess(u)+\frac{1}{m}g}\pa{\nabla w-\frac{w}{u}\nabla u, \cdot}^{\sharp}.
	\end{align*}
\end{rem}
\begin{proof}[Proof of Lemma \ref{lemma 4.7}]
	As a first step, we prove the validity of
	\begin{align}\label{4.6}
		\mathrm{div} X=&u\set{\abs{\hs(w)}^2+\frac{w^2}{u^2}\abs{\hs(u)}^2-2\frac{w}{u}\mathrm{tr}\pa{\hs(w)\circ \hs(u)}}\\
		&-\set{\hs(u)(\nabla w,\nabla w)-2\frac{w}{u}\hs(u)(\nabla w, \nabla u)+\frac{w^2}{u^2}\hs(u)(\nabla u, \nabla u)}\notag\\
		&+u\set{\ric^\p(\nabla w, \nabla w)-2\frac{w}{u}\ric^\p(\nabla w, \nabla u)+\frac{w^2}{u^2}\ric^\p(\nabla u, \nabla u)}\notag\\
		&+ug\pa{ \nabla w-\frac{w}{u}\nabla u, \nabla \Delta w-\frac{w}{u}\nabla\Delta u}+\alpha u \abs{d\p\pa{\nabla w -\frac{w}{u}\nabla u}}^2\notag\\
		&+\frac{u}{m}\pa{\Delta w-\frac{w}{u}\Delta u}.\notag
	\end{align}
	In a local orthonormal coframe, the components of $X$ are given by
	\begin{align*}
		X_k=uw_{ik}w_i+\frac{w^2}{u}u_{ik}u_i-wu_{ik}w_i-ww_{ik}u_i+\frac{1}{m}uw_k-\frac{1}{m}wu_k.
	\end{align*}
	Then,
	\begin{align*}
		\diver X=& X_{kk}\\
		=&u_kw_{ik}w_i+uw_{ikk}w_i+uw_{ik}w_{ik}+2\frac{w}{u}w_ku_{ik}u_i+\frac{w^2}{u}u_{ikk}u_i+\frac{w^2}{u}u_{ik}u_{ik}\\&-\frac{w^2}{u^2}u_ku_{ik}u_i
		-w_ku_{ik}w_i-wu_{ikk}w_i-wu_{ik}w_{ik}-w_kw_{ik}u_i-ww_{ikk}u_i\\
		&-ww_{ik}u_{ik}+\frac{1}{m}uw_{kk}+\frac{1}{m}u_kw_k-\frac{1}{m}wu_{kk}-\frac{1}{m}w_ku_k\\
		=&uw_i(w_{kki}+R_{it}w_t)+\abs{\hs(w)}^2u+2\frac{w}{u}w_ku_{ik}u_i+\frac{w^2}{u}u_i(u_{kki}+R_{it}u_t)\\
		&+\frac{w^2}{u}\abs{\hs(u)}^2-\frac{w^2}{u^2}u_ku_{ik}u_i-w_ku_{ik}w_i-ww_i(u_{kki}+R_{it}u_t)\\
		&-wu_i(w_{kki}+R_{it}w_t)-2ww_{ik}u_{ik}+\frac{1}{m}u\Delta w-\frac{1}{m}w\Delta u\\
		=&u\set{\abs{\hs(w)}^2+\frac{w^2}{u^2}\abs{
				\hs(u)}^2-2\frac{w}{u}w_{ik}u_{ik}+\frac{1}{m}\Delta w-\frac{1}{m}\frac{w}{u}\Delta u}\\
		&-\hs(u)(\nabla w, \nabla w)+2\frac{w}{u}\hs(u)(\nabla w,\nabla u)-\frac{w^2}{u^2}\hs(u)(\nabla u, \nabla u)\\
		&+ug( \nabla\Delta w, \nabla w)+u\ric(\nabla w,\nabla w)+\frac{w^2}{u}g\pa{ \nabla\Delta u, \nabla u}+\frac{w^2}{u}\ric(\nabla u, \nabla u)\\
		&-wg\pa{ \nabla\Delta u,\nabla w}-\ric(\nabla w, \nabla u)-wg\pa{\nabla\Delta w,\nabla u} -w\ric(\nabla u, \nabla w)\\
		=&u\set{\abs{\hs}(w)+\frac{w^2}{u^2}\abs{\hs(u)}^2-2\frac{w}{u}\mathrm{tr}\pa{\hs(u)\circ\hs(w)}}\\
		&+\frac{u}{m}\pa{\Delta w-\frac{w}{u}\Delta u}\\
		&-\set{\hs(u)(\nabla w,\nabla w)-2\frac{w}{u}\hs(u)(\nabla u,\nabla w)+\frac{w^2}{u^2}\hs(u)(\nabla u,\nabla u)}\\
		&+u\set{\ric^\p(\nabla w,\nabla w)-2\frac{w}{u}\ric^\p(\nabla u,\nabla w)+\frac{w^2}{u^2}\ric^\p(\nabla u, \nabla u)}\\
		&+u\alpha\set{\p^*h(\nabla w,\nabla w)-2\frac{w}{u}\p^+h(\nabla u,\nabla w)+\frac{w^2}{u^2}\p^*h(\nabla u,\nabla u)}\\
		&+ug\pa{\nabla w-\frac{w}{u}\nabla u,\nabla \Delta w-\frac{w}{u}\nabla \Delta u},
	\end{align*}
	which gives \eqref{4.6}. \
	To simplify the writing, set
	\begin{align}\label{4.8}
		\mathcal{A}&=\abs{\mathring{\hs}(w)-\frac{w}{u}\mathring{\hs}(u)}^2\\
		&=\abs{\hs(w)}^2+\frac{w^2}{u^2}\abs{\hs(u)}^2-2\frac{w}{u}\mathrm{tr}\,\pa{\hs(w)\circ\hs(u)}-\frac{1}{m}\pa{\Delta w-\frac{w}{u}\Delta u}^2\notag
	\end{align}
	and observe that
	\begin{align}\label{4.9}
		Q\pa{\nabla w-\frac{w}{u}\nabla u,\nabla w-\frac{w}{u}\nabla u}=&\hs(u)(\nabla w,\nabla w)-2\frac{w}{u}\hs(u)(\nabla w,\nabla u)\notag\\
		&+\frac{w^2}{u^2}\hs(u)(\nabla u,\nabla u)-u\left\{\ric^\p(\nabla w,\nabla w)\right.\notag\\
		&\left.-2\frac{w}{u}\ric^\p(\nabla w,\nabla u)+\frac{w^2}{u^2}\ric^\p(\nabla u,\nabla u)\right\}\notag\\
		&-\frac{\Delta u}{m}\abs{\nabla w-\frac{w}{u}\nabla u}^2+u\frac{S^\p}{m}\abs{\nabla w-\frac{w}{u}\nabla u}^2.
	\end{align}
	We check \eqref{4.9}:  by the definition of $Q$, we have

	\begin{align*}
		&Q \Big(\nabla w-\frac{w}{u}\nabla u,\nabla w-\frac{w}{u}\nabla u\Big)=\\
		&\quad=\pa{u_{ij}-u\mathring{R}^\p_{ij}-\frac{1}{m}u_{kk}\delta_{ij}}\pa{w_iw_j-\frac{w}{u}w_iu_j+\frac{w^2}{u^2}u_iu_j-\frac{w}{u}w_ju_i}\\
		&\quad=\hs(u)(\nabla w, \nabla w)-\frac{1}{m}\Delta u\abs{\nabla w}^2-u\trric^\p(\nabla w, \nabla w)\\
		&\qquad-\frac{w}{u}\hs(u)(\nabla u,\nabla w)+\frac{1}{m}\frac{w}{u}\Delta ug(\nabla u,\nabla w)+w\trric^\p(\nabla u,\nabla w)\\
		&\qquad-\frac{w}{u}\hs(u)(\nabla u,\nabla w)\frac{1}{m}\frac{w}{u}\Delta ug(\nabla u,\nabla w)+w\trric^\p(\nabla u,\nabla w)\\
		&\qquad+\frac{w^2}{u^2}\hs(u)(\nabla u, \nabla u)-\frac{1}{m}\frac{w^2}{u^2}\Delta u\abs{\nabla u}^2-\frac{w^2}{u^2}u\trric^\p(\nabla u,\nabla u)\\
		&\quad=\hs(u)(\nabla w,\nabla w)-2\frac{w}{u}\hs(u)(\nabla w,\nabla u)+\frac{w^2}{u^2}\hs(u)(\nabla u,\nabla u)\notag\\
		&\qquad-u\pa{\ric^\p(\nabla w,\nabla w)-2\frac{w}{u}\ric^\p(\nabla w,\nabla u)+\frac{w^2}{u^2}\ric^\p(\nabla u,\nabla u)}\\
		&\qquad+u\pa{\frac{S^\p}{m}\abs{\nabla w}^2-2\frac{S^\p}{m}g\pa{\nabla w,\nabla u}+\frac{w^2}{u^2}\frac{S^\p}{m}\abs{\nabla u}^2}\\
		&\qquad-\frac{\Delta u}{m}\pa{\abs{\nabla w}^2-2\frac{w}{u}g\pa{ \nabla w,\nabla u}+\frac{w^2}{u^2}\abs{\nabla u}^2},
	\end{align*}
	from which we immediately deduce \eqref{4.9}.
	Using \eqref{4.8} and \eqref{4.9} into \eqref{4.6}, we obtain
	\begin{align*}
		\mathrm{div}X=&u\set{\abs{\hs(w)}^2+\frac{w^2}{u^2}\abs{\hs(u)}^2-2\frac{w}{u}\mathrm{tr}\,\pa{\hs(w)\circ\hs(u)}}\\
		&+\frac{u}{m}\pa{\Delta w-\frac{w}{u}\Delta u}-Q\pa{\nabla w-\frac{w}{u}\nabla u,\nabla w-\frac{w}{u}\nabla u}\\
		&-\frac{1}{m}\pa{\Delta u-uS^\p}\abs{\nabla w-\frac{w}{u}\nabla u}^2+ug\pa{ \nabla w-\frac{w}{u}\nabla u,\nabla \Delta w-\frac{w}{u}\nabla\Delta u}\\
		&+\alpha u\abs{d\p\pa{\nabla w-\frac{w}{u}\nabla u}}^2\\
		=&u \mathcal{A}+\frac{u}{m}\pa{\Delta w-\frac{w}{u}\Delta u}\pa{\Delta w-\frac{w}{u}\Delta u+1}-Q\pa{\nabla w-\frac{w}{u}\nabla u,\nabla w-\frac{w}{u}\nabla u}\\
		&-\frac{1}{m}\pa{\Delta u-uS^\p}\abs{\nabla w-\frac{w}{u}\nabla u}^2\\
		&+ug\pa{\nabla w-\frac{w}{u}\nabla u, \nabla\Delta w-\frac{w}{u}\nabla\Delta u} +\alpha u\abs{d\p\pa{\nabla w-\frac{w}{u}\nabla u}}^2.
	\end{align*}
	from which we get \eqref{4.11}.
	
\end{proof}

\begin{lemma}\label{lemma 4.12}
	Assume that there exists $u\in C^{\infty}(M)$, $u>0$, satisfying
	\begin{align}\label{lemma div del primo campo X}
		\begin{cases}
			i)\,\hs(u)-u\set{\ric^\p-\frac{1}{m-1}\pa{\frac{S^\p}{2}-p+U(\p)}g}=0,\\
			ii)\Delta u=\frac{u}{m-1}\pa{mp-mU(\p)+\frac{m-2}{2}S^\p},\\
			iii)\,\mu+U(\p)=\frac{1}{2}S^\p	
		\end{cases}
	\end{align}
	and that, for some relatively compact $\Omega \subset\subset \mathrm{int}(M)$,  $w\in C^{\infty}(\Omega)$ satisfies the differential equation
	\begin{align}\label{4.13}
		\Delta w-\frac{w}{u}\Delta u=-1, \quad \text{on } \Omega.
	\end{align}
	Let $X$ be the vector field on $\Omega$ defined  in \eqref{4.5}. Then
	\begin{align}\label{4.14}
		\mathrm{div}X=&u\abs{\mathring{\hs}(w)-\frac{w}{u}\mathring{\hs}(u)}^2+\alpha u\abs{d\p\pa{\nabla w-\frac{w}{u}\nabla u}}^2\\
		&+\pa{\mu+p}u\abs{\nabla w-\frac{w}{u}\nabla u}^2 \quad \text{on } \Omega.\notag
	\end{align}
\end{lemma}
\begin{proof}
	Since $w$ is a solution of \eqref{4.13}, we have
	\begin{align}\label{4.19}
		\frac{u}{m}\pa{\Delta w-\frac{w}{u}\Delta u}\pa{\Delta w-\frac{w}{u}\Delta u+1}\equiv 0;
	\end{align}
	moreover, by \eqref{lemma div del primo campo X} i), ii), $Q$ defined in \eqref{4.1} is identically null; indeed,
	\begin{align*}
		Q&=\mathring{\hs}(u)-u\mathring{\ric^\p}\\
		&=\hs(u)-\frac{\Delta u}{m}g-u\ric^\p+u\frac{S^\p}{m}g\\
		&=\hs(u)-u\ric^\p-\frac{u}{m(m-1)}\pa{mp-mU(\p)+\frac{m-2}{2}S^\p}g+u\frac{S^\p}{m}g\\
		&=\hs(u)-u\set{\ric^\p-\frac{1}{m-1}\pa{\frac{S^\p}{2}-p+U(\p)}g}=0.
	\end{align*}
	As a consequence, \eqref{4.11} rewrites as
	\begin{align}\label{divXLemma4.12}
		\mathrm{div} X=&u\abs{\mathring{\hs}(w)-\frac{w}{u}\mathring{\hs}(u)}^2-\frac{1}{m}\pa{\Delta u-uS^\p}\abs{\nabla w-\frac{w}{u}\nabla u}^2\\
		&+ug\pa{ \nabla w-\frac{w}{u}\nabla u,\nabla \Delta w-\frac{w}{u}\nabla\Delta u}+\alpha u\abs{d\p\pa{\nabla w-\frac{w}{u}\nabla u}}^2.\notag
	\end{align}
	Equation \eqref{4.13} yields
	\begin{align}\label{lapwlapu}
		ug\pa{ \nabla w-\frac{w}{u}\nabla u,\nabla \Delta w-\frac{w}{u}\nabla\Delta u}&=-ug\pa{\nabla w-\frac{w}{u}\nabla u,\Delta u \frac{\nabla w}{u}-\Delta u \frac{w}{u^2}\nabla u}\\
		&=\Delta u\abs{\nabla w-\frac{w}{u}\nabla u}^2.\notag
	\end{align}
	Inserting \eqref{lapwlapu} into \eqref{divXLemma4.12} and using \eqref{lemma div del primo campo X} ii), we have
	\begin{align*}
		\mathrm{div} X=&u\abs{\mathring{\hs}(w)-\frac{w}{u}\mathring{\hs}(u)}^2+\frac{u}{m}\pa{(m-1)\Delta u+S^\p u}\abs{\nabla w-\frac{w}{u}\nabla u}^2\\
		&+\alpha u\abs{d\p\pa{\nabla w-\frac{w}{u}\nabla u}}^2\\
		=&u\abs{\mathring{\hs}(w)-\frac{w}{u}\mathring{\hs}(u)}^2+\alpha u\abs{d\p\pa{\nabla w-\frac{w}{u}\nabla u}}^2\\
		&+\frac{u}{m}\abs{\nabla w-\frac{w}{u}\nabla u}^2\pa{S^\p-mU(\p)+mp+\frac{m-2}{2}S^\p}\\
		=&u\abs{\mathring{\hs}(w)-\frac{w}{u}\mathring{\hs}(u)}^2+\alpha u\abs{d\p\pa{\nabla w-\frac{w}{u}\nabla u}}^2\\
		&+u\abs{\nabla w-\frac{w}{u}\nabla u}^2\pa{\mu+p},
	\end{align*}
	where in the last equality we have used \eqref{lemma div del primo campo X} iii). Thus, we have the validity of \eqref{4.14}.
\end{proof}
In order to ensure that \eqref{4.13} holds, we recall the following
\begin{lemma}\label{4.25}
	Let $(M,g)$ be a manifold and
	$\Omega\subset\subset \mathrm{int}(M)$ open, $u\in C^{\infty}(M)$, such that $\frac{\Delta u}{u}\in C^{\infty}(\overline{\Omega})$. Then, there exists a solution $w\in C^{\infty}(\Omega)\cap C^{2,\beta}(\ol{\Omega})$, $\beta \in \pa{0,1}$, of the problem
	\begin{align}\label{4.24}
		\begin{cases}
			\Delta w=\frac{\Delta u}{u}w-1\quad\text{on }\ol{\Omega}\\
			w\equiv 0\quad \text{on }\partial \Omega.
		\end{cases}
	\end{align}
	Moreover, $w>0$ on $\Omega$ and $\partial \Omega=\set{x\in M\: :\: w(x)=0}$ is a regular level set of $w$.
\end{lemma}
\noindent
For a proof see \cite{FP}.\\
\noindent

We are now ready to prove Theorem \ref{thm B}, that we recall for the ease of the reader.

\begin{theorem}
	Let $(M,g)$ be a $\p$-SPFST of dimension $m\geq 2$ and let $\Omega \subset\subset \mathrm{int} M$ with smooth boundary.
	Let
	\begin{align*}
		\ol{H}=\frac{1}{m}\frac{\int_{\partial \Omega}u}{\int_{\Omega }u}
	\end{align*}
	and assume
	\begin{align*}
		H\leq-\ol{H},
	\end{align*}
	where $H$ is the mean curvature of $\partial \Omega$ in the direction of the inward unit normal. Furthermore, suppose
	\begin{align*}
		\mu+p\geq 0 \text{ on } M.
	\end{align*}
	Then
	\begin{align*}
		i:\partial\Omega \hookrightarrow M
	\end{align*} 
	is totally umbilical and $\mu, p$ are constant on $\Omega$, with
	\begin{align*}
		\mu=-p.
	\end{align*}
	
\end{theorem}

\begin{proof}[Proof (of Theorem \ref{thm B})]
	Let $\Omega \subset \subset \mathrm{int}(M)$ be a domain with smooth boundary. Let $w$ be the solution of \eqref{4.24} and recall that
	\begin{align*}
		\partial \Omega=w^{-1}(\set{0}),
	\end{align*}
	which we know to be smooth with inward unit normal
	\begin{align*}
		\nu=\frac{\nabla w}{\abs{\nabla w}}.
	\end{align*}
	For the vector field $X$ defined in \eqref{4.5}, we have
	\begin{align}\label{4.34}
		g\pa{ X,\nu}=\frac{u}{\abs{\nabla w}}\hs(w)(\nabla w, \nabla w)+\frac{1}{m}u\abs{\nabla w}
	\end{align}
	on $\partial\Omega$.\\
	A simple computation shows that the mean curvature in the direction of $\nu$ of the level set $\partial \Omega$ of $w$ is given by 
	\begin{align*}
		(m-1)H=\frac{1}{\abs{\nabla w}}\Delta w+g\pa{ \nabla\pa{\frac{1}{\abs{\nabla w}}},\nabla w} \quad\text{on }\partial \Omega
	\end{align*}
	and therefore
	\begin{align}\label{4.35}
		(m-1)H=\frac{1}{\abs{\nabla w}}\Delta w-\frac{1}{\abs{\nabla w}^3}\hs(w)(\nabla w,\nabla w).
	\end{align}
	Using the latter and \eqref{4.24}, we infer
	\begin{align}\label{4.36}
		-\frac{1}{\abs{\nabla w}}\hs(w)(\nabla w,\nabla w)=(m-1)H\abs{\nabla w}^2+\abs{\nabla w} \quad \text{on } \partial \Omega.
	\end{align}
	Integrating $\mathrm{div}X$ on $\Omega$ and using \eqref{4.34} and \eqref{4.36}, we obtain
	\begin{align}\label{4.37}
		\int_{\Omega}\mathrm{div}X=\int_{\partial \Omega}-g\pa{ X, \nu} =\int_{\partial \Omega}\pa{(m-1)Hu\abs{\nabla w}^2+\frac{m-1}{m}u\abs{\nabla w}}.
	\end{align}
	We set
	\begin{align}\label{Lambda e H barrato}
		\Lambda=\frac{\int_{\Omega}u}{\int_{\partial\Omega}u},\quad\ol{H}=\frac{1}{m}\frac{1}{\Lambda}=\frac{1}{m}\frac{\int_{\partial\Omega}u}{\int_{\Omega}u}.
	\end{align}
	Integrating $\mathrm{div}(w\nabla u)$ on $\Omega$ we obtain
	\begin{align}\label{4.39}
		\int_{\Omega}w\Delta u=-\int_{\Omega}g\pa{\nabla u, \nabla w} .
	\end{align}
	Similarly, integrating $\mathrm{div}(u\nabla w)$ on $\Omega$ and using \eqref{4.24} and \eqref{4.39}, we infer
	\begin{align*}
		\int_{\partial \Omega}-u\abs{\nabla w}&=\int_{\Omega}u\Delta w+\int_{\Omega}g\pa{\nabla u,\nabla w} \\
		&=\int_{\Omega}\pa{w\Delta u-u}-\int_{\Omega}w\Delta u\\
		&=-\int_{\Omega}u,
	\end{align*}
	that is
	\begin{align}\label{4.40}
		\int_{\partial \Omega}u\abs{\nabla w}=\int_{\Omega}u.
	\end{align}
	Using the latter and \eqref{Lambda e H barrato}, we have
	\begin{align*}
		\frac{1}{\Lambda}\int_{\partial \Omega}u\pa{\Lambda-\abs{\nabla w}}^2&=\frac{1}{\Lambda}\int_{\partial \Omega}u\abs{\nabla w}^2-\int_{\partial \Omega}2u\abs{\nabla w}+\int_{\partial \Omega}\Lambda u\\
		&=\frac{1}{\Lambda}\int_{\partial\Omega}u\abs{\nabla w}^2-2\int_{\partial \Omega}u\abs{\nabla w}+\int_{\Omega} u\\
		&=\frac{1}{\Lambda}\int_{\partial \Omega}u\abs{\nabla w}^2-\int_{\partial \Omega}u\abs{\nabla w}.
	\end{align*}
	Now we use \eqref{4.14} and the previous equality into \eqref{4.37} to obtain
	\begin{align*}
		(m-1)\int_{\partial \Omega}\set{Hu\abs{\nabla w}^2-\frac{u}{m}\abs{\nabla w}}=&\int_{\Omega}u\abs{\mathring{\hs}(w)-\frac{w}{u}\mathring{\hs}(u)}^2\\
		&+\int_{\Omega}\alpha u\abs{d\p\pa{\nabla w-\frac{w}{u}\nabla u}}^2\\
		&+\int_{\Omega}\abs{\nabla w-\frac{w}{u}\nabla u}^2u\pa{\mu+p},
	\end{align*}
	that is
	\begin{align*}
		(m-1)\int_{\partial \Omega} \set{Hu\abs{\nabla w}^2+\frac{1}{m}\frac{1}{\Lambda}u\abs{\nabla w}^2}=&\int_{\Omega}u\abs{\mathring{\hs}(w)-\frac{w}{u}\mathring{\hs}(u)}^2\\
		&+\int_{\Omega}\alpha u\abs{d\p\pa{\nabla w-\frac{w}{u}\nabla u}}^2\\
		&+\int_{\Omega}\abs{\nabla w-\frac{w}{u}\nabla u}^2u\pa{\mu+p}\\
		&+\frac{m-1}{m}\frac{1}{\Lambda}\int_{\partial \Omega}u\pa{\Lambda-\abs{\nabla w}}^2.
	\end{align*}
	On the other hand, by \eqref{4.39}, we get
	\begin{align}
		\int_{\partial \Omega} (m-1)\set{Hu+\frac{1}{m}\frac{1}{\Lambda}u}\abs{\nabla w}^2=&(m-1)\int_{\partial \Omega} u\abs{\nabla w}^2\pa{\ol{H}+H}.
	\end{align}
	Hence we have
	\begin{align}\label{4.41}
		(m-1)\int_{\partial \Omega} u\abs{\nabla w}^2\pa{\ol{H}+H}=&\int_{\Omega}u\abs{\mathring{\hs}(w)-\frac{w}{u}\mathring{\hs}(u)}^2\notag\\
		&+\int_{\Omega}\alpha u\abs{d\p\pa{\nabla w-\frac{w}{u}\nabla u}}^2\notag\\
		&+\int_{\Omega}\abs{\nabla w-\frac{w}{u}\nabla u}^2u\pa{\mu+p}\notag\\
		&+\frac{m-1}{m}\frac{1}{\Lambda}\int_{\partial \Omega}u\pa{\Lambda-\abs{\nabla w}}^2.
	\end{align}
	Thus, for $\alpha>0$, $\mu+p\geq 0$ and $\ol{H}+H\leq 0$ we deduce
	\begin{align*}
		\abs{\mathring{\hs}(w)-\frac{w}{u}\mathring{\hs}(u)}^2\equiv 0 \quad\text{on }\ol{\Omega};
	\end{align*}
	moreover, since $w\in C^{2,\beta}(\ol{\Omega})$ and $w\equiv 0$ on $\partial \Omega$, we get
	\begin{align}\label{4.42}
		\hs(w)=\frac{\Delta w}{m}g \quad\text{on }\partial\Omega.
	\end{align}
	Let $\se_{ab}$, $1\leq a,\,b,\, ...\,\leq m-1$, be the coefficients of the second fundamental form of $i: \partial \Omega \hookrightarrow M$ in the direction of the inward unit normal $\nu=\frac{\nabla w}{\abs{\nabla w}}$ of $\partial \Omega$; since
	\begin{align*}
		\se_{ab}=-\frac{w_{ab}}{\abs{\nabla w}}
	\end{align*}
	from \eqref{4.42}, we deduce that $i: \partial \Omega \hookrightarrow M$ is totally umbilical.\\		
	Furthermore, by \eqref{4.41}, we infer
	\begin{align}\label{asinpropprima}
		\pa{\mu+p}\abs{\nabla w-\frac{w}{u}\nabla u}^2\equiv 0 \quad \text{on }\Omega.
	\end{align}
	and we deduce that $\mu$ and $p$ are constant on $\Omega$, with $\mu=-p$; indeed, suppose $$A:=\mathrm{int}\pa{\set{\pa{\nabla w-\frac{w}{u}\nabla u}(x)=0}}\neq \emptyset.$$
	Let $\hat{A}$ be a connected component of $A$; since
	\begin{align*}
		\nabla\pa{\frac{w}{u}}=\frac{1}{u}\nabla w-\frac{w}{u^2}\nabla u=\frac{1}{u}\pa{\nabla w-\frac{w}{u}\nabla u}
	\end{align*}
	on $\hat{A}$, then there exists $c\in \erre$ such that $w=cu$, but $w$ solves
	\begin{align*}
		\Delta w-\frac{\Delta u}{u}w=-1,
	\end{align*}
	thus
	\begin{align*}
		0=c\Delta u-\frac{\Delta u}{u}cu=-1,
	\end{align*}
	which is a contradiction.
\end{proof}
Note that $\abs{\nabla w-\frac{w}{u}\nabla u}$ can be zero only on a set with empty interior of $\Omega$. On the other hand, \eqref{4.41} gives $d\p\pa{\nabla w-\frac{w}{u}\nabla u}\equiv 0$ on $\Omega$, so that necessarily $\mathrm{Ker}(d\p)$ is not trivial, at least on the complement in $\Omega$ of a set with empty interior.

	\chapter{Non-Existence results}\label{Sect_ non existence}
In this Chapter we provide non-existence results for $\p$-SPFSTs; in the first section, the method we use is based on the introduction of the elliptic operator
\begin{align*}
	Lu:=\Delta u+\frac{u}{m-1}\pa{2U(\p)-mp-(m-2)\mu}
\end{align*}
and on finding sufficient conditions under which each solution of a related Cauchy problem admits a first zero.
We observe that $Lu=0$ is obtained by
\begin{align}\label{system SPFST non ex}
		\begin{cases}
		i)\, \hs(u)-u\set{\ric^\p-\frac{1}{m-1}\pa{\frac{S^\p}{2}-p+U(\p)}g}=0,\\
		ii)\,\Delta u=\frac{u}{m-1}\sq{mp-mU(\p)+\frac{m-2}{2}S^\p},\\
		iii)\,u\tau(\p)+d\p(\nabla u)=\frac{u}{\alpha}(\nabla U)(\p),\\
		iv)\,\mu+U(\p)=\frac{1}{2}S^\p,\\
		v)\,(\mu+p)\nabla u=-u\nabla p,
	\end{cases}
\end{align}
combining the second and the fourth equations.\\
Using a different technique, the second part of this chapter is devoted to finding an obstruction to the existence on a closed Riemannian manifold $(M,g)$, of a structure satisfying
\begin{align}\label{sistema per KW}
	\begin{cases}
		u\ric^\p-\hs(u)+\Delta ug=\lambda g,\\
		u>0,
	\end{cases}
\end{align}
where $u, \lambda\in \cinf$; note that system \eqref{sistema per KW} holds for every $\p$-SPFST with empty boundary, as it can be seen combining equations i), ii) and iv) of system \eqref{system SPFST non ex}.
\section{A non-existence result \emph{via} oscillation theorey}
As before, and for the rest of this chapter,  $\p:(M,g)\to (N,h)$ is a smooth map, $U\in C^{\infty}(N), \mu, p\in C^{\infty}(M)$ and we use the notation of the previous chapters.
\begin{proposition}\label{non ex: prop1}
	Let $(M,g)$ be a complete Riemannian manifold of dimension $m$, $\partial M=\emptyset$ . For $r\in \erre^+$, let
	\begin{align*}
		v(r):=\mathrm{vol}(\partial B_r),&&A(r):=\frac{1}{v(r)}\int_{\partial B_r}\frac{1}{m-1}\pa{2U(\p)-mp-(m-2)\mu}(x),
	\end{align*}
	where $B_r$ is the geodesic ball of radius $r$ centered at a fixed origin $o\in M$.
	Let $z\in \mathrm{Lip}_{loc}(\erre^+_0)$ be a solution of
	\begin{align}\label{non ex: CP}
		\begin{cases}
			(vz')'+Avz=0 \quad\text{on }\erre^+,\\
			z(0^+)=z_0>0,\\
			(vz')(0^+)=0.
		\end{cases}
	\end{align}
	Suppose $z$ admits a first zero $R_0\in \erre^+$; then there exists no positive $u\in C^\infty(M)$ satisfying
	\begin{align}\label{equation Lu}
		Lu:=&\Delta u+\frac{u}{m-1}\pa{2U(\p)-mp-(m-2)\mu}=0.
	\end{align}
\end{proposition}
\begin{proof}
	By contradiction, assume that \eqref{equation Lu} admits a positive solution on $M$; then by Theorem 1 of \cite{Tshirt_col_brie} and \cite{MossPiepenbrink}, the first eigenvalue of $L$ on $M$, $\lambda_L^1(M)$, is non-negative;
	\noindent
	moreover, by Proposition 4.6 of \cite{BMR} and the definition of $A(t)$ we have that $z\in \mathrm{Lip}_{loc}(R^+_0)$ and its zeros are isolated (if any).\\ Let $r$ be the distance function from the fixed origin $o\in M$ and define
	\begin{align}\label{non ex def of psi}
		\psi=z\circ r.
	\end{align}
	To obtain the desired contradiction, we consider the Rayleigh quotient of $\psi$ on the geodesic ball $B_{R_0}$ centered at $o\in M$ and of radius $R_0$,
	\begin{align}\label{non ex: RQ}
		Q(\psi)=\pa{\int_{B_{R_0}}\psi^2}^{-1}\int_{B_{R_0}}\pa{\abs{\nabla\psi}^2-\frac{1}{m-1}\sq{\pa{2U(\p)-mp-(m-2)\mu}(x)}\psi^2}.
	\end{align}
	Gauss lemma yields
	\begin{align*}
		\abs{\nabla \psi}^2=\pa{z'}^2;
	\end{align*}
	applying the co-area formula (see e.g. \cite{MRSYamabe}) $Q(\psi)$ rewrites as
	\begin{align*}
		Q(\psi)&=\pa{\int_0^{R_0}dt\int_{\partial B_t}z^2}^{-1}\int_0^{R_0}dt\int_{\partial B_t}\pa{\pa{z'}^2-\frac{1}{m-1}{\pa{2U(\p)-mp-(m-2)\mu}(x)z^2}}\\
		&=\pa{\int_0^{R_0}vz^2dt}^{-1}\int_0^{R_0}\pa{v\pa{z'}^2-Avz^2}dt,
	\end{align*}
	where in the last equality we have used the definition of $A$. Integrating by parts, using \eqref{non ex: CP} and the fact that $z(R_0)=0$, we deduce
	\begin{align*}
		\int_0^{R_0}(z')^2vdt&=(zz'v)(R_0)-(zz'v)(0)-\int_0^{R_0}z(z'v)'dt\\
		&=\int_0^{R_0}Avz^2dt.
	\end{align*}
	It follows that $Q(\psi)=0$. Hence $\lambda_L^1(B_{R_0})\leq 0$ and by monotonicity of the eigenvalues of $L$, we have $\lambda_L^1(M)<0$, the desired contradiction.
\end{proof}

	The existence of a first zero for $z$ has been a fundamental tool in the proof of Proposition \ref{non ex: prop1}; as a consequence, it is natural to study sufficient conditions under which a solution of \eqref{non ex: CP} (that always exists by Proposition 3.2 of \cite{BMR}) admits a first zero.

Towards this aim, let $h$ be a function satisfying
\begin{itemize}
	\item[1)] $h\in L^{\infty}_{loc}(\erre^+_0)$;
	\item[2)] $\frac{1}{h}\in L^{\infty}_{loc}(\erre^+_0)$;
	\item[3)] $0\leq v\leq h$ on $\erre^+_0$,
\end{itemize}
with corresponding \emph{critical curve}
\begin{align}\label{non ex: def of chi}
	\chi_h(r):=\set{2h(r)\int_{r}^{+\infty}\frac{1}{h(s)}ds}^{-2}.
\end{align}
By Corollary 6.2 of \cite{BMR}, if
\begin{itemize}
	\item[(A1)] $A\in L^{\infty}_{loc}(R^+_0)$;
	\item[(V1)]$0\leq v(r)\in L^{\infty}_{loc}(R^+_0),\quad \frac{1}{v(r)}\in L^{\infty}_{loc}(R^+_0), \quad\lim_{r\ra 0^+}v(r)=0$
\end{itemize}
are satisfied, $A\geq0$ on $\erre^+$, $A\not\equiv0$ and, for some $h$ satisfying the assumptions above, either
\begin{itemize}
	\item[i)] $\frac{1}{h}\notin L^1(+\infty)$;
	\item[ii)] or otherwise $\frac{1}{h}\in L^1(+\infty)$ and there exists $r>R>0$ such that $A\not\equiv 0$ on $\sq{0,R}$ and
	\begin{align}\label{non ex: condition for zeros}
		\int_{R}^r\pa{\sqrt{A(s)}-\sqrt{\chi_h(s)}}ds>-\frac{1}{2}\pa{\log\int_0^R A(s)v(s)ds+\log\int_R^{+\infty}\frac{1}{h(s)}ds},
	\end{align}
\end{itemize}
then the solution $z$ of \eqref{non ex: CP} admits a first zero.\\
A similar result is given for oscillation in Theorem 6.6 of \cite{BMR}.
\begin{rem}
	Note that conditions $(V1)$ in our case is satisfied since
	\begin{align*}
		v(r)=\mathrm{vol}(\partial B_r)
	\end{align*}
	(see, for instance, Proposition 2.6 of \cite{BMR}), while $(A1)$ is satisfied by
	\begin{align*}
		A(r)=\frac{1}{v(r)}\int_{\partial B_r}\frac{1}{m-1}\pa{2U(\p)-mp-(m-2)\mu}(x),
	\end{align*}
because $\frac{1}{v(r)}\in L^{\infty}_{loc}(R^+_0)$ and $\frac{1}{m-1}\pa{2U(\p)-mp-(m-2)\mu}(x)\in C^{\infty}(M)$.
\end{rem}	

	By Corollary 2.9 of \cite{MMR} and the co-area formula, if $v, 1/v\in L^\infty_{loc}(\erre^+)$, $v>0$, $1/v\notin L^1(+\infty)$ and, for some $r_0\in \erre^+$,
	\begin{align*}
		\lim_{r\ra+\infty}\int_{B_r\setminus B_{r_0}}\frac{1}{m-1}\pa{2U(\p)-mp-(m-2)\mu}(x)=+\infty,
	\end{align*}
	then a solution $z$ of \eqref{non ex: CP} is oscillatory, so that it certainly admits a first zero.

As a consequence, we deduce the validity of the following
\begin{proposition}
	Let $(M,g)$ be a complete Riemannian manifold of dimension $m$. For $r\in\erre^+$, let
	\begin{align*}
		v(r)=\mathrm{vol}(\partial B_r)
	\end{align*}
	and assume that
	\begin{align*}
		\frac{1}{v(r)}\notin L^1(+\infty)
	\end{align*}
	and
	\begin{align*}
		\lim_{r\ra+\infty}\int_{B_r\setminus B_{r_0}}\frac{1}{m-1}\pa{2U(\p)-mp-(m-2)\mu}(x)=+\infty.
	\end{align*}
	Then, there is no positive solution $u$ to \eqref{equation Lu}.
\end{proposition}
\begin{rem}
	Note that it is possible to construct examples of manifolds satisfying
	\begin{align*}
		\frac{1}{v(r)}\notin L^1(+\infty)
	\end{align*}
	and such that the volume grows exponentially (see e.g. \cite{RigoliSetti}). Moreover, the second hypothesis does not require assumptions on the sign of
	\begin{align*}
		\pa{2U(\p)-mp-(m-2)\mu}(x)
	\end{align*}
	and hence, assumptions on the violation of the Strong Energy Condition (SEC) when we consider a solution of \eqref{system SPFST non ex} with $\alpha>0$.
\end{rem}

When $v(r)$ admits an upper bound it is possible to show that under suitable assumptions either \eqref{non ex: CP} admits a first zero or it is oscillatory. Note that, in the first case hypotheses on the sign of $\pa{2U(\p)-mp-(m-2)\mu}(x)$ will be needed; however, to prove that \eqref{non ex: CP} is oscillatory we rely on Theorem 6.6 \cite{BMR} and Proposition 6.9 of \cite{BMR}, whose hypotheses do not force us to make an assumption on the sign of $\pa{2U(\p)-mp-(m-2)\mu}(x)$.

\begin{proposition}\label{non ex: prop1 first zero}
	Let $(M,g)$ be a complete Riemannian manifold of dimension $m$, let $v(r)$, $A(r)$, $\frac{1}{m-1}\pa{2U(\p)-mp-(m-2)\mu}(x)$ be as in Proposition \ref{non ex: prop1}. Assume that
	\begin{align}\label{non ex: bound on v}
		v(r)\leq Cr^{\theta},
	\end{align}
	where $C,\theta\in \erre$, $C>0$, $\theta>1$.
	\begin{itemize}
		\item[1)] If $A\geq0$ on $\erre^+$,
		\begin{align}\label{non ex: q leq 0}
			2U(\p)-mp-(m-2)\mu\geq 0 \text{ on }M
		\end{align}
		and for some $R, D\in \erre^+$, $D>\frac{\theta-1}{2}$, we have
		\begin{align}\label{non ex; prop1 assumption q}
			\int_{\partial B_R}\frac{1}{m-1}\pa{2U(\p)-mp-(m-2)\mu}(x)\geq\frac{D^2}{r^2}v(r),\quad r\geq R,
		\end{align}
		then \eqref{non ex: CP} admits a first zero.
		\item[2)] If $A\geq 0$ in $[r_0,+\infty)$, $\erre\ni r_0>0$,
		\begin{align}\label{non ex: hp E no sgn A}
			\int_{\partial B_r}\frac{1}{m-1}\pa{2U(\p)-mp-(m-2)\mu}(x)\notin L^1(+\infty)
		\end{align}
		and, for some $R,D\in \erre$, $R>r_0$, $D>\frac{\theta-1}{2}$, we have
		\begin{align}\label{non ex: hp su sqrt A no sgn A}
			\int_R^r\sqrt{A(s)}ds\geq D\log\pa{r/R},
		\end{align}
		then \eqref{non ex: CP} is oscillatory.
	\end{itemize}
\end{proposition}

\begin{rem}
	
	Note that, when
	\eqref{non ex: q leq 0} is satisfied, half of the Strong Energy Condition (SEC) is violated.
\end{rem}
\begin{proof}
	By \eqref{non ex: bound on v}, we set
	\begin{align*}
		h(r)=Cr^{\theta}.
	\end{align*}
	When $\theta>1$, $ h^{-1}\in L^1(+\infty)$ and by definition of $h$, we have
	\begin{align}\label{non ex: chih}
		\chi_h(r)&=\set{2h(r)\int_r^{+\infty}\frac{1}{Cs^\theta}ds}^{-2}\\
		&=\set{2Cr^\theta\frac{r^{1-\theta}}{C(\theta-1)}}^{-2}\notag\\
		&=\pa{\frac{\theta-1}{2r}}^2.\notag
	\end{align}
	We now prove the first part of the theorem. We want to show the validity of \eqref{non ex: condition for zeros};
	by \eqref{non ex: chih}, we rewrite \eqref{non ex: condition for zeros} as
	\begin{align*}
		\int_{R}^r\pa{\sqrt{A(s)}-\sqrt{\chi_h(s)}}ds=&\int_{R}^r\sq{\sqrt{A(s)}-\pa{\frac{\theta-1}{2s}}}ds\\
		=&\int_{R}^r\sqrt{A(s)}\,ds-\frac{\theta-1}{2}(\log(r)-\log(R))\\
		>&-\frac{1}{2}\bigg[\log\int_{B_R}\frac{1}{m-1}\pa{2U(\p)-mp-(m-2)\mu}(x)\\
		&+\log\pa{\frac{R^{1-\theta}}{C(\theta-1)}}\bigg],
	\end{align*}
	that is,
	\begin{align*}
		\int_{R}^r\sqrt{A(s)}\,ds-\frac{\theta-1}{2}\log(r)>&\frac{1}{2}\log\sq{C(\theta-1)}\\
		&-\frac{1}{2}\log\int_{B_R}\frac{1}{m-1}\pa{2U(\p)-mp-(m-2)\mu}(x).
	\end{align*}
	By definition of $A(r)$ and \eqref{non ex; prop1 assumption q}, it follows
	\begin{align*}
		\sqrt{A(r)}\geq\frac{D}{r}\quad\text{for }r\geq R.
	\end{align*}
	Therefore, to conclude, it is sufficient to show
	\begin{align*}
		D\int_R^r\frac{ds}{s}-\frac{\theta-1}{2}\log(r)>&\frac{1}{2}\log\sq{C(\theta-1)}\\
		&-\frac{1}{2}\log\int_{B_R}\frac{1}{m-1}\pa{2U(\p)-mp-(m-2)\mu}(x),
	\end{align*}
	that is,
	\begin{align}\label{sufficient to prove}
		\pa{D-\frac{\theta-1}{2}}\log(r)>&\log\sq{R^D\sqrt{C(\theta-1)}}\\
		&-\frac{1}{2}\log\int_{B_R}\frac{1}{m-1}\pa{2U(\p)-mp-(m-2)\mu}(x).\notag
	\end{align}
	Since $D>\frac{\theta-1}{2}$, \eqref{sufficient to prove} holds for a sufficient large $r$; moreover, $A(r)\not\equiv 0$ on $\sq{0,R}$ since equation \eqref{non ex; prop1 assumption q} implies
	\begin{align*}
		\frac{1}{m-1}\pa{2U(\p)-mp-(m-2)\mu}\not\equiv0,
	\end{align*}
	which concludes the first part of the statement.\\
	We are now ready prove the second part of the Theorem. By the validity of \eqref{non ex: hp E no sgn A}, it is sufficient to show that condition $ii)$ of Proposition 6.9 of \cite{BMR} is satisfied to prove that \eqref{non ex: CP} is oscillatory. It follows by \eqref{non ex: hp su sqrt A no sgn A} and our choice of $D$ that
	\begin{align}\label{limsup 1}
		\limsup_{r\ra+\infty}\frac{\int_R^r\sqrt{A(s)}ds}{\int_{R}^{r}\sqrt{\chi_h(s)}ds}&=\limsup_{r\ra+\infty}\frac{\int_R^r\sqrt{A(s)}ds}{\pa{\frac{\theta-1}{2}}\log(r/R)}\notag\\
		&\geq \limsup_{r\ra+\infty}\frac{D\log(r/R)}{\pa{\frac{\theta-1}{2}}\log(r/R)}>1.
	\end{align}
\end{proof}
So far we have considered a polynomial growth of $v(r)$, see assumption \eqref{non ex: bound on v}. Our aim is now to allow a faster growth, even superexponential, as in assumption \eqref{non ex prop 3 hp1} below. In this way, we cover the reasonable ranges of the growth of the volume of geodesic spheres in a  complete manifold.

\begin{proposition}\label{prop3nonex}
	Let $(M,g)$ be a complete Riemannian manifold of dimension $m$, and let $v(r)$, $A(r)$, $\frac{1}{m-1}\pa{2U(\p)-mp-(m-2)\mu}(x)$ be as in Proposition \ref{non ex: prop1}. Assume the validity of \eqref{non ex: q leq 0} on $M$ and
	\begin{align}\label{non ex prop 3 hp1}
		v(r)\leq\Lambda\exp\set{a r^{\gamma}\log^\beta(r)},
	\end{align}
	for some constants $\Lambda, a>0$ and either $\gamma>0$, $\beta\geq 0$ or $\gamma\geq 0$, $\beta>0$.
	\begin{itemize}
		\item[1)] If $A\geq0$ on $\erre^+$, for some $r>r_0>1$, and for some $b\in \erre$, $b>1$ we have
		\begin{align}\label{non ex prop 3 hp2}
			A(r)&=\frac{1}{v(r)}\int_{\partial B_r}\frac{1}{m-1}\pa{2U(\p)-mp-(m-2)\mu}(x)\\
			&\geq b\frac{a^2}{4}\pa{\gamma\log(r)+\beta }^2(r)^{2(\gamma-1)}\log^{2(\beta-1)}(r),\notag
		\end{align}
		then \eqref{non ex: CP} admits a first zero.
		\item[2)] If $A\geq 0$ in $[r_0,+\infty)$, for some  $r_0>0$,
		\begin{align}\label{non ex: hp E no sgn A2}
			\int_{\partial B_r}\frac{1}{m-1}\pa{2U(\p)-mp-(m-2)\mu}(x)\notin L^1(+\infty)
		\end{align}
		and for some $R,C\in \erre$ such that $R>r_0$, $C>1$ we have the validity of
		\begin{align}\label{non ex: hp su sqrt A no sgn A2}
			\int_R^r\sqrt{A(s)}ds\geq Car^\gamma\log^\beta(r),
		\end{align}
		then \eqref{non ex: CP} is oscillatory.
	\end{itemize}
\end{proposition}

%

\begin{proof}
	By \eqref{non ex prop 3 hp1}, we let
	\begin{align*}
		h(r)=\Lambda\exp\set{a r^{\gamma}\log^\beta(r)};
	\end{align*}
	note that $h^{-1}\in L^1(+\infty)$.
	\begin{itemize}
		\item[1)] To prove the validity of the first part of the statement of Proposition \ref{prop3nonex}, it is then sufficient to show that there exists $r>R>r_0$ such that $A\not\equiv 0$ on $\sq{0,R}$ and
		\begin{align*}
			&\int_R^r\pa{\sqrt{A(s)}-\sqrt{\chi_h(s)}}ds>\\
			&\quad>-\frac{1}{2}\pa{\log\int_{B_R}\frac{1}{m-1}\pa{2U(\p)-mp-(m-2)\mu}(x)+\log\int_{R}^{+\infty}\frac{1}{h(s)}ds}.
		\end{align*}
		Let $$\tilde{\chi_h}(s):=\pa{\frac{h'(s)}{2h(s)}}^2;$$
		then
		\begin{align*}
			\lim_{t\ra+\infty}\frac{\sqrt{\tilde{\chi_h}(s)}}{\sqrt{\chi_h(s)}}=1,
		\end{align*}
		and we refer the reader to \cite{BMR09} for a proof; therefore, note that when $R$ is sufficiently large, for every $t\geq R$
		\begin{align*}
			\sqrt{\chi_h(s)}<c\sqrt{\tilde{\chi_h}(s)},
		\end{align*}
		where $c\in\erre$, $c>1$.
		Hence,
		\begin{align}\label{stima per sqrt chi}
			\int_R^r\sqrt{\chi_h(s)}dt&<\int_R^rc\sqrt{\tilde{\chi_h}(s)}ds\\
			&=\int_R^rc\frac{h'(s)}{2h(s)}ds\notag\\
			&=\frac{c}{2}\log(h(r))-\frac{c}{2}\log(h(R))\notag\\
			&=\frac{c}{2}ar^{\gamma}\log^\beta(r)-\frac{c}{2}aR^{\gamma}\log^\beta(R).\notag
		\end{align}
		Moreover, by assumption \eqref{non ex prop 3 hp2}, we have
		\begin{align}\label{stima per sqrt A prop3}
			\sqrt{A(t)}\geq \sqrt{b}\frac{a}{2}\pa{\gamma\log(t)+\beta }t^{(\gamma-1)}\log^{(\beta-1)}(t)=\sqrt{b}\frac{a}{2}\pa{t^{\gamma}\log^\beta(t)}'.
		\end{align}
		It follows by \eqref{stima per sqrt chi} and \eqref{stima per sqrt A prop3} that
		\begin{align}\label{stima per differenza sqrt A e chi}
			&\int_R^r\pa{\sqrt{A(t)}-\sqrt{\chi_h(t)}}dt>\\
			&\quad>\int_R^r\sqrt{b}\frac{a}{2}\pa{t^{\gamma}\log^\beta(t)}'dt-\frac{c}{2}ar^{\gamma}\log^\beta(r)+\frac{c}{2}aR^{\gamma}\log^\beta(R)\notag\\
			&\quad=\frac{\pa{\sqrt{b}-c}}{2}ar^{\gamma}\log^\beta(r)-\frac{\pa{\sqrt{b}-c}}{2}aR^{\gamma}\log^\beta(R).
		\end{align}
		Therefore, choosing $c<\sqrt{b}$ we can conclude as in Proposition \ref{non ex: prop1 first zero}.
		\item[2)] Since we have the validity of \eqref{non ex: hp E no sgn A2}, it is sufficient to show that one of the equivalent hypotheses of Proposition 6.9 of \cite{BMR} is satisfied to show that \eqref{non ex: CP} is oscillatory. As in the previous proposition, we show that
		\begin{align*}
			\limsup_{r\ra+\infty}\frac{\int_R^r\sqrt{A(s)}ds}{\int_{R}^{r}\sqrt{\chi_h(s)}ds}>1
		\end{align*}
		holds.
		By \eqref{stima per sqrt chi} we deduce
		\begin{align*}
			\limsup_{r\ra+\infty}\frac{\int_R^r\sqrt{A(s)}ds}{\int_{R}^{r}\sqrt{\chi_h(s)}ds}&\geq\lim_{r\ra+\infty}\frac{\int_R^r\sqrt{A(s)}ds}{\int_{R}^{r}\sqrt{\chi_h(s)}ds}\\
			&\geq\lim_{r\ra+\infty}\frac{\int_R^r\sqrt{A(s)}ds}{\int_{R}^{r}c\sqrt{\tilde{\chi_h}(s)}ds}\\
			&\geq\lim_{r\ra+\infty}\frac{Car^{\gamma}\log^\beta(r)}{\frac{c}{2}ar^{\gamma}\log^\beta(r)-\frac{c}{2}aR^{\gamma}\log^\beta(R)}>1.
		\end{align*}
		Hence, choosing $C\geq\frac{c}{2}+\eps$, $\eps>0$, we conclude that \eqref{non ex: CP} is oscillatory by Proposition 6.9 of \cite{BMR}.
	\end{itemize}
\end{proof}
	\section{A Kazdan-Warner type obstruction}
	We now look for obstructions, on a compact manifold with empty boundary, to the existence of a positive solution $u$ of the equation
	\begin{align}\label{system non ex KW}
		u\ric^\p-\hs(u)+\Delta u g=u E,
	\end{align}
	for a $2$-covariant, symmetric tensor field $E$ on $M$. Note that the first equation in \eqref{system SPFST non ex} can be written in the form \eqref{system non ex KW}, with
	\begin{align}\label{E for phiSPFST}
		E=(\mu+p)g.
	\end{align}
	From equation \eqref{5.63.1}, we obtain the validity of the next
	\begin{lemma}\label{lemma div Z non ex}
		Let $A$ be a $2$-covariant, symmetric Codazzi tensor on $M$ and $X$ a vector field. Define the vector field $Z$ of components
		\begin{align}\label{components of Z non ex}
			Z_j=\mathring{\sq{P_{k-1}\circ A}}_{ij}X_i,
		\end{align}
		where
		\begin{align*}
			\mathring{\sq{P_{k-1}\circ A}}_{ij}&=\pa{P_{k-1}}_{it}A_{tj}-\frac{1}{m}\mathrm{tr}\pa{{P_{k-1}\circ A}}\delta_{ij}\\
			&=\pa{P_{k-1}}_{it}A_{tj}-\frac{k}{m}S_k\delta_{ij}.
		\end{align*}
		Then,
		\begin{align}\label{divergenza Z non ex}
			\diver Z=\frac{m-k}{m}(S_k)_jX_j+\frac{1}{2}\pa{\mathcal{L}_Xg}_{ij}\mathring{\sq{P_{k-1}\circ A}}_{ij}.
		\end{align}
	\end{lemma}
	\begin{proof}
		Since $A$ commutes with the Newton operator,
		\begin{align*}
			\mathring{\sq{P_{k-1}\circ A}}_{ij}=\mathring{\sq{P_{k-1}\circ A}}_{ji}.
		\end{align*}
		Using \eqref{5.63.1}, the fact that $A$ is Codazzi and that, under this assumption, $P_k$ is divergence free, we obtain
		\begin{align*}
			\diver Z=&X_j\set{\mathring{\sq{P_{k-1}\circ A}}}_{ji,i}+X_{ji}\mathring{\sq{P_{k-1}\circ A}}_{ij}\\
			=&\frac{m-k}{m}\pa{S_k}_jX_j+\frac{1}{2}\pa{X_{ji}+X_{ij}}\mathring{\sq{P_{k-1}\circ A}}_{ij}\\
			=&\frac{m-k}{m}\pa{S_k}_jX_j+\frac{1}{2}(\mathcal{L}_Xg)_{ij}\mathring{\sq{P_{k-1}\circ A}}_{ij}.
		\end{align*}
	\end{proof}
	\begin{rem}
		From here we already obtain an interesting conclusion; indeed, suppose that $X$ is a conformal vector field; then \eqref{divergenza Z non ex} reduces to
		\begin{align}\label{divergenza Z con X conformal}
			\diver Z=\frac{m-k}{m}g(X,\nabla S_k)=\frac{c_k}{m}g(X,\nabla \sigma_k),
		\end{align}
		which reminds us of the Kazdan-Warner-Pohozaev formula. Indeed, the above identity has been proved by Han, \cite{HAN}, for $A$ the classical Schouten tensor, to generalize the Kazdan-Warner-Pohozaev formula to higher order symmetric functions.
	\end{rem}
	We have
	\begin{proposition}
		Let $(M,g)$ be a compact Riemannian manifold with possibly empty boundary. Let $A$ be a $2$-covariant, symmetric, Codazzi tensor on $M$ and $X$ a conformal vector field. Then, for $1\leq k\leq m-1$,
		\begin{align}\label{formula integrale 1}
			\frac{m-k}{m}\int_MX(S_k)=-\int_{\partial M}\mathring{\sq{P_{k-1}\circ A}}_{ij}X_i\nu_j,
		\end{align}
		where $\nu$ is the inward unit normal to $\partial M$.
	\end{proposition}
	Observe that for $k=1$ and $\ric$ Codazzi, that is $(M,g)$  a \emph{harmonic} manifold, equation \eqref{formula integrale 1} yields
	\begin{align}\label{id int 1 per Ricci codazzi}
		\frac{m-1}{m}\int_MX(S)=-\int_{\partial M}\mathring{\ric}(X,\nu).
	\end{align}
	The latter does not coincide with the classic Kazdan-Warner formula obtained without restrictions on $\ric$ and that reads
	\begin{align}\label{classical KW}
		\frac{m-2}{2m}\int_MX(S)=-\int_{\partial M}\mathring{\ric}(X,\nu);
	\end{align}
	however, with this choice of $A$ ($A=\ric$) and for $k=1$ it is possible to avoid to suppose that Ricci is Codazzi: in fact, due to Schur's identity we have
	\begin{align*}
		\mathring{R}_{ij,i}=R_{ij,i}-\frac{S_j}{m}=\frac{m-2}{2m}S_j,
	\end{align*}
	that gives the validity of \eqref{classical KW} just proceeding as in the proof of Lemma \ref{lemma div Z non ex}.\\
	Let now $A^\p$ be the $\p$-Schouten tensor and $u\in \cinf$ a positive solution of
	\begin{align}\label{7.111}
		u\ric^\p-\hs(u)+\Delta u g=u E,
	\end{align}
	for some $2$-covariant, symmetric tensor $E$. We choose the vector field
	\begin{align*}
		X=\nabla u,
	\end{align*}
	so that
	\begin{align*}
		\mathcal{L}_Xg=2\hs(u).
	\end{align*}
	We set, for some $V\in\cinf$,
	\begin{align*}
		A=A^\p-\frac{V}{m-1}g.
	\end{align*}
	From \eqref{divergenza Z non ex} and \eqref{7.111} we obtain
	\begin{align}\label{divergenza Z non ex scelto A}
		\diver Z=&\frac{m-k}{m}g(\nabla S_k,\nabla u)+uR^{\p}_{ij}\mathring{\sq{P_{k-1}\circ A}}_{ij}-uE_{ij}\mathring{\sq{P_{k-1}\circ A}}_{ij}.
	\end{align}
	By the definitions of $A^{\p}$ and $A$ we get
	\begin{align*}
		R^{\p}_{ij}&=A^{\p}_{ij}+\frac{S^\p}{2(m-1)}\delta_{ij}\\
		&=A_{ij}+\frac{1}{m-1}\sq{V+\frac{S^{\p}}{2}}\delta_{ij},
	\end{align*}
	thus, using \eqref{operatori di newton} and \eqref{traccia di A P}, we obtain
	\begin{align*}
		uR^{\p}_{ij}\mathring{\sq{P_{k-1}\circ A}}_{ij}&=uA^{\p}_{ij}\mathring{\sq{\pa{P_{k-1}}\circ A}}_{ij}\\
		&=u\set{A_{ij}(P_{k-1})_{it}A_{tj}-A_{ij}\frac{k}{m}S_k\delta_{ij}}\\
		&=u\set{S_k\delta_{tj}A_{tj}-(P_{k})_{tj}A_{tj}-\frac{k}{m}S_kS_1}\\
		&=u\set{\frac{m-k}{m}S_1S_k-(k+1)S_{k+1}}.
	\end{align*}
	Assuming
	\begin{align*}
		E=\lambda g
	\end{align*}
	then \eqref{divergenza Z non ex scelto A} becomes
	\begin{align}\label{divegenza Z non ex scelto E}
		\diver Z&=\frac{m-k}{m}\nabla u(S_k)+\set{\frac{m-k}{m}S_1S_k-(k+1)S_{k+1}}u\\
		&=\frac{(m-1)!}{k!(m-k-1)!}\set{\nabla u(\sigma_k)+m\pa{\sigma_1\sigma_k-\sigma_{k+1}}u}.\notag
	\end{align}
	Therefore, we have the following
	\begin{theorem}
		Let $(M,g)$ be a compact Riemannian manifold of dimension $m\geq2$ with empty boundary and let $u\in \cinf$ be a solution of
		\begin{align}\label{sistema per simil KW}
			\begin{cases}
				u\ric^{\p}-\hs(u)+\Delta ug=\lambda g,\\
				u>0,
			\end{cases}
		\end{align}
		where $\lambda\in\cinf$. Let $V\in \cinf$ and let $\sigma_k$  be the $k$-th normalized symmetric function in the eigenvalues of the $2$-covariant symmetric tensor
		\begin{align*}
			A=A^{\p}-\frac{V}{m-1}g,
		\end{align*}
		for some $1\leq k\leq m-1$, and assume that $A$ is Codazzi. Then
		\begin{align*}
			\int_M\nabla(u)\pa{\sigma_k}=-\int_Mm\pa{\sigma_1\sigma_k-\sigma_{k+1}}u.
		\end{align*}
		Furthermore, for $k=1$ or $k\geq2$, $\sigma_k$ positive and $A$ with positive eigenvalues at some point $p\in M$, then
		\begin{align*}
			\int_M\nabla u(\sigma_k)\leq 0.
		\end{align*}
	\end{theorem}
	Here the last inequality is due to Lemma \ref{disuguaglianza di Anselli} and $u>0$.\\
	In case of a $\p$-SPFST we can choose
	\begin{align*}
		\lambda=\mu+p,&&V=U(\p)
	\end{align*}
	and thus $A$ is Codazzi if and only if \eqref{U phi Cotton} is satisfied, that is,
	\begin{align*}
		C^{\p}_{ijk}=\frac{U^a}{m-1}(\p^a_k\delta_{ij}-\p^a_j\delta_{ik}).
	\end{align*}
	\begin{cor}
	Let $(M,g)$ be a compact $\p$-SPFST
 	of dimension $m\geq2$ with empty boundary. Let
 	\begin{align*}
 		A=A^{\p}-\frac{U(\p)}{m-1}g
 	\end{align*}
 	and suppose that
 	\begin{align*}
 		C^{\p}=-\frac{1}{2(m-1)}\diver_1\pa{U(\p)g\KN g}.
 	\end{align*}
 	Let $\sigma_k$ be the $k$-th normalized symmetric function in the eigenvalues of $A$, for some $1\leq k\leq m-1$. Then
 	\begin{align*}
 		\int_M\nabla(u)\pa{\sigma_k}=-\int_Mm\pa{\sigma_1\sigma_k-\sigma_{k+1}}u.
 	\end{align*}
 	Furthermore, for $k=1$ or $k\geq2$, $\sigma_k$ positive and $A$ with positive eigenvalues at some point $p\in M$, then
 	\begin{align*}
 		\int_M\nabla u(\sigma_k)\leq 0,
 	\end{align*}
 	where equality holds if and only if $\p$ is constant and $(M,g)$ is isometric to a Euclidean sphere.
 	\end{cor}
	\begin{proof}
		We only have to prove the last part of the statement. In particular, we have
		\begin{align*}
			\int_M\nabla u(\sigma_k)=-\int_Mm\pa{\sigma_1\sigma_k-\sigma_{k+1}}u=0.
		\end{align*}
		Then, Lemma \ref{disuguaglianza di Anselli} implies
		\begin{align*}
			\sigma_1\sigma_k-\sigma_{k+1}=0
		\end{align*}
		and the eigenvalues of $A$ coincide. The proof now continues as the one of Theorem \ref{teo: U Cotton zero conseguenze} to give the desired conclusion.
	\end{proof}

	\bibliographystyle{amsplain}
	\bibliography{bibliogphiSPFST}

\end{document}